\documentclass[a4paper,11pt]{amsart}

\pdfoutput=1

\usepackage[text={400pt,660pt},centering]{geometry}

\usepackage{amsthm, amssymb, amsmath, amsfonts, mathrsfs}
\usepackage{mathtools}
\usepackage{scalerel} 
\usepackage[colorlinks=true, pdfstartview=FitV, linkcolor=blue, citecolor=blue, urlcolor=blue,pagebackref=false]{hyperref}

\usepackage{esint} 
\usepackage{MnSymbol} 

\usepackage{microtype}

\usepackage{tikz} 
\usetikzlibrary{shapes,arrows,positioning,calc}

\tikzstyle{result} = [rectangle, 
minimum width=4cm, 
minimum height=1cm,  
text width=4cm, 
align=flush center,
font = \footnotesize, 
draw=black]

\tikzstyle{input} = [rectangle, 
minimum width=3cm, 
minimum height=1cm, 
text width=3cm, 
align=flush center,
font = \footnotesize,
draw=black]

\tikzstyle{arrow} = [thick, ->,>=stealth]

\newtheorem{proposition}{Proposition}
\newtheorem{theorem}[proposition]{Theorem}
\newtheorem{lemma}[proposition]{Lemma}
\newtheorem{corollary}[proposition]{Corollary}

\theoremstyle{remark}
\newtheorem{remark}[proposition]{Remark}

\theoremstyle{definition}

\newtheorem{condition}[proposition]{Condition}

\newtheorem{hypothesis}[proposition]{Hypothesis}

\numberwithin{equation}{section}
\numberwithin{proposition}{section}

\renewcommand{\le}{\leqslant}
\renewcommand{\ge}{\geqslant}
\renewcommand{\leq}{\leqslant}
\renewcommand{\geq}{\geqslant}

\renewcommand{\subset}{\subseteq}
\newcommand{\mcl}{\mathcal}

\newcommand{\A}{{\mathsf{A}}}
\newcommand{\F}{\mathcal{F}}
\newcommand{\G}{\mathcal{G}}
\renewcommand{\S}{\mathsf{S}}
\newcommand{\DD}{\mathcal{D}}
\newcommand{\D}{{\overbracket[1pt][-1pt]{D}}}
\renewcommand{\c}{\mathbf{c}}
\newcommand{\cc}{\bar{\c}}
\newcommand{\ca}{\hat{\c}}
\newcommand{\Da}{\widehat{D}}

\newcommand{\E}{\mathbb{E}}
\newcommand{\Er}{\mathbb{E}_{\rho}}
\renewcommand{\Pr}{\mathbb{P}_{\rho}}

\renewcommand{\aa}{\mathbf{a}}
\renewcommand{\L}{\mathcal{L}}

\newcommand{\M}{\mathcal{M}}
\newcommand{\N}{\mathbb{N}}

\newcommand{\tna}{\tilde \nabla}
\newcommand{\Ll}{\left}
\newcommand{\Rr}{\right}
\newcommand{\lhs}{left-hand side}
\newcommand{\rhs}{right-hand side}
\newcommand{\1}{\mathbf{1}}
\newcommand{\R}{\mathbb{R}}

\newcommand{\Z}{\mathcal{Z}}
\newcommand{\Zd}{{\mathbb{Z}^d}}
\newcommand{\Td}{\mathbb{T}^d}
\renewcommand{\P}{\mathbb{P}}
\newcommand{\ov}{\overline}
\renewcommand{\bar}{\overline}

\renewcommand{\tilde}{\widetilde}

\newcommand{\de}{\delta}

\renewcommand{\d}{{\mathrm{d}}}
\newcommand{\var}{\mathbb{V}\!\mathrm{ar}}

\renewcommand{\epsilon}{\varepsilon}
\newcommand{\T}{\mathbb{T}}

\newcommand{\X}{\mathcal{X}} 

\newcommand{\tx}{t_{\operatorname{mix}}}

\newcommand{\cu}{{\scaleobj{1.2}{\square}}}

\renewcommand{\fint}{\strokedint}
\newcommand{\Rd}{{\mathbb{R}^d}}

\renewcommand{\r}{\mathbf{r}}
\newcommand{\nub}{\bar{\nu}}

\newcommand{\fil}{\mathscr{F}}
\newcommand{\gil}{\mathscr{G}}

\DeclareMathOperator{\dist}{dist}
\DeclareMathOperator{\supp}{supp}
\DeclareMathOperator{\diam}{diam}
\DeclareMathOperator{\size}{size}
\newcommand{\bbracket}[1]{\left\llangle{#1}\right\rrangle} 
\newcommand{\XX}{\widetilde{\mathcal{X}}} 
\newcommand{\PP}{\widetilde{\mathbb{P}}} 
\newcommand{\EE}{\widetilde{\mathbb{E}}} 
\newcommand{\teta}{\tilde{\eta}} 

\newcommand{\poi}{\text{Poi}}
\newcommand{\Ind}[1]{\mathbf{1}_{\left\{#1\right\}}}
\newcommand{\id}{\mathsf{Id}}
\newcommand{\norm}[1]{\left\Vert{#1}\right\Vert}
\newcommand{\bracket}[1]{\left\langle{#1}\right\rangle}

\newcommand{\ovs}[1]{\overline{{#1}^*}}

\newcommand{\mres}{\mathbin{\vrule height 1.4ex depth 0pt width
		0.13ex\vrule height 0.13ex depth 0pt width 1ex}}

\numberwithin{equation}{section}

\renewcommand{\a}{\alpha}
\renewcommand{\b}{\beta}
\newcommand{\e}{\varepsilon}

\newcommand{\ga}{\gamma}
\renewcommand{\k}{\kappa}
\newcommand{\la}{\lambda}
\renewcommand{\th}{\theta}
\newcommand{\si}{\sigma}

\renewcommand{\t}{\tau}
\newcommand{\om}{\omega}
\newcommand{\De}{\Delta}
\newcommand{\Ga}{\Gamma}
\newcommand{\La}{\Lambda}
\newcommand{\Om}{\Omega}
\newcommand{\lan}{\langle}
\newcommand{\ran}{\rangle}

\def\Tr{\operatorname{Tr}}

\def\lbr{\left(}
\def\rbr{\right)}

\allowdisplaybreaks

\newcommand{\vertiii}[1]{{\left\vert\kern-0.25ex\left\vert\kern-0.25ex\left\vert #1 
	\right\vert\kern-0.25ex\right\vert\kern-0.25ex\right\vert}}

\title[Homogenization of non-gradient exclusion process]{Quantitative homogenization and hydrodynamic limit of non-gradient exclusion process}
\author{Tadahisa Funaki, Chenlin Gu, Han Wang}

\address[Tadahisa Funaki]{Beijing Institute of Mathematical Sciences and Applications, Beijing, China}
\email{funaki@bimsa.cn}

\address[Chenlin Gu]{Yau Mathematical Sciences Center, Tsinghua University, Beijing, China}
\email{gclmath@tsinghua.edu.cn}  

\address[Han Wang]{Qiuzhen College, Tsinghua University, Beijing, China}
\email{wanghan21@mails.tsinghua.edu.cn}

\begin{document}

\begin{abstract}

	For the non-gradient exclusion process, we prove the quantitative homogenization of the diffusion matrix and the conductivity by local functions. The proof relies on the renormalization approach developed by Armstrong, Kuusi, Mourrat, and Smart, while the new challenge here is the hard core constraint of particle number on every site. Therefore, a coarse-grained method is proposed to lift the configuration to a larger space without exclusion, and a gradient coupling between two systems is applied to capture the spatial cancellation. We then strengthen the convergence rate to be uniform concerning the density, and integrate it into the work by Funaki, Uchiyama, and Yau [\textit{IMA Vol. Math. Appl.}, 77 (1996), pp. 1-40.] to yield a quantitative hydrodynamic limit.  Our new approach avoids showing the characterization of closed forms and provides stronger results. The extension is discussed for the model in the presence of disorder on the bonds.
	
	\bigskip
	
	
	\medskip
	
	
\end{abstract}

\maketitle

\tableofcontents

\newpage

%
%
%
%
%
%
%
%

\section{Introduction}\label{sec:1}
\emph{The diffusion matrix} plays an important role in the study of the large-scale behaviors of
interacting particle systems. Among these systems, some are classified as \emph{the gradient models} if the current of the conserved quantity can be written as a sum of the difference between a local function and its spatial shift, and the others are called \emph{the non-gradient models}. Unlike the gradient models, the non-gradient models usually require more techniques to derive the hydrodynamic limit, because the scaling yields a diverging factor and, moreover, the diffusion matrix does not have an explicit expression and one needs to add a correction. This idea goes back to the seminal work \cite{varadhanII} of Varadhan, who studied the Ginzburg--Landau model. Later, the hydrodynamic limit was proved in several classical particle systems of non-gradient type: \emph{the generalized symmetric exclusion process (GSEP)}  by Kipnis, Landim and Olla \cite{KLO94};  \emph{the lattice gas}, also known as \emph{the non-gradient exclusion process} or \emph{speed-change Kawasaki dynamics}, by Funaki, Uchiyama and Yau \cite{fuy}, the general lattice gas with mixing condition by Varadhan and Yau \cite{varadhan1997mixing}; \emph{the multi-type simple symmetric exclusion process (multi-type SSEP)} by Quastel \cite{quastel}, etc. The equilibrium fluctuation in non-gradient models was studied later in \cite{fun96, lu94, cha96}, and the regularity of the diffusion matrix was discussed in a series of work \cite{lov-reg, naga1, naga2, naga3, bernardin}. One can also refer to \cite{spohn2012large,kipnis1998scaling} for the basic background and \cite{sasada2018green} for the relation between the gradient condition and the Green--Kubo formula.

It is natural to ask the convergence rate, and the quantitative hydrodynamic limit has received attentions recently. The sharp estimate of the relative entropy was obtained in \cite{JM} by Jara and Menezes for a weakly asymmetric exclusion, which is an important step to derive the non-equilibrium fluctuation. Based on \cite{natalie2009twoscale}, the quantitative results were developed  by Dizdar, Menz, Otto, and Wu for the Ginzburg--Landau model in \cite{dizdar2018quantitativeA, dizdar2018quantitativeB}. Very recently, Menegaki and Mouhot proposed a \emph{consistence-stability approach} in \cite{menegaki2021} to obtain the quantitative hydrodynamic limit in 1-Wasserstein distance for several models including the zero-range process, the Ginzburg--Landau model, and the simple exclusion process (see \cite[Chapter~5.5]{menegaki2021PhD}). However, all these results are for the gradient model, and there are no quantitative results for the non-gradient model in the literature to the best of our knowledge. This is because, as mentioned in the last paragraph,  a diverging factor appears and the diffusion matrix is more complicated in the non-gradient model. Such obstacle was already observed in the proof in \cite{fuy,fun96}, where several key error terms are finally reduced to the approximation of the diffusion matrix or \emph{the conductivity}; see \cite[(2.5) and Section~5]{fuy}. These two fundamental quantities are defined using variational formulas and are related by the Einstein relation. Therefore, both of them can be approximated qualitatively by a sequence of local functions, but the convergence rate is unknown.

In this paper, we answer the question above using a concrete construction of the desired local functions. As Varadhan observed the link between the interacting systems and homogenization, in the sense of averaging and gradient replacement that kills the diverging factor,  in the earlier work \cite{varadhanII}, the new improvement comes from the recent progress in quantitative homogenization theory; see \cite{armstrong2016quantitative, armstrong2016lipschitz, armstrong2016mesoscopic, armstrong2017additive, ferg1, ferg2, AKMbook, armstrong2022elliptic} based on the renormalization approach, and \cite{NS,vardecay, gloria2011optimal,gloria2012optimal,gloria2015quantification, GO3,gloria2020regularity} based on another approach using spectral inequalities. As an example, for the $\nabla \phi$ interface model studied in \cite{funaki1997interface}, a quantitative hydrodynamic limit is obtained in \cite{ArmstrongDario2024} using the renormalization approach. Also inspired by the renormalization approach, \cite{bulk} studies a similar diffusion matrix problem in continuous configuration space, and a quantitative equilibrium fluctuation is obtained recently in \cite{gu2024quantitative} under the same setting. These works pave the way for the quantitative homogenization theory in interacting particle systems, but the continuous configuration model there relaxes the hard core constraint by allowing arbitrarily large number of particles in the unit volume. As a consequence, to apply the existing results to a lattice particle model of exclusion rule still meets technical challenges in math. The present paper aims to resolve these difficulties and is the first example to establish quantitative homogenization theory on the non-gradient exclusion process. Our main result not only generalizes \cite{bulk} to the non-gradient exclusion process, but also improves in the sense that we construct \emph{one} local corrector to realize uniform convergence of conductivity for \emph{every} particle density. Moreover, this density-uniform homogenization can be integrated into the relative entropy method in the classical work \cite{fuy}, and then establishes a quantitative hydrodynamic limit; namely, our result provides the convergence rate in the hydrodynamic limit for non-gradient exclusion process.  We emphasize that our method is new and, in particular, avoids to prove the characterization of closed forms. The latter is usually required and useful to show the hydrodynamic limit for non-gradient models in qualitative level, but the convergence rate was not accessible. 

We believe that our proof is robust, as it extends to the cases where the external disorder is posed on the bonds. We hope to cover other models in future work, including GSEP, the multi-type SSEP, and the exclusion process reversible under the general Gibbs measure.

\subsection{Main results}
In this part, we state our main results. 

We recall quickly the necessary notations of the exclusion process and the results in the previous work. Let $\Zd$ be the Euclidean lattice and we use $\X := \{0,1\}^{\Zd}$ to stand for the configuration of particles under exclusion rule. The element of $\X$ will be denoted by $\eta = \{\eta_x: x \in \Zd \}$. Here $\eta_x = 0$ means the site $x$ is vacant and $\eta_x = 1$ means the site is occupied by one particle. We denote by $y \sim x$ for $x,y \in \Zd$ if $\vert x - y\vert = 1$. Then $\{x,y\}$ is called an (undirected) bond. For every $\Lambda \subset \Zd$, we denote by $\Lambda^*$ the set of bonds in $\Lambda$ that 
\begin{align}\label{eq.defBond}
	\Lambda^* := \{ \{x,y\}: x,y \in \Lambda, x \sim y\}.
\end{align} 

For $x,y \in \Zd$, the exchange operator $\eta^{x,y}$ is defined as 
\begin{align}\label{eq.exchange}
	(\eta^{x,y})_z := \Ll\{\begin{array}{ll}
		\eta_z, & \qquad z \neq x,y; \\
		\eta_y, & \qquad z = x; \\
		\eta_x, & \qquad z = y.
	\end{array}\Rr.
\end{align}
Especially, when $b = \{x,y\}$ is a bond, we also write $\eta^b$ instead of $\eta^{x,y}$, and define the Kawasaki operator $\pi_b \equiv \pi_{x,y}$ 
\begin{align}\label{eq.Kawasaki}
	\pi_b f(\eta) := f(\eta^b) - f(\eta).
\end{align} 
For every $x \in \Zd$, the translation operator $\tau_x$ is defined as 
\begin{align}\label{eq.translation1}
	(\tau_x \eta)_{y} := \eta_{x+y},
\end{align}
and for function $f$ on $\X$, we also define $\tau_x f$ as 
\begin{align}\label{eq.translation2}
	(\tau_x f)(\eta) = f(\tau_x \eta).
\end{align}

We also consider a family of local functions 
\begin{align}
	\{c_b(\eta) \equiv c_{x,y}(\eta) = c_{y,x}(\eta); b=\{x,y\} \in (\Zd)^*\},
\end{align}
which determine the jump rate of particles on the nearest bonds. We suppose the following conditions for the jump rate throughout the paper without specific explanation.
\begin{hypothesis}\label{hyp} The following conditions are supposed for  $\{c_b\}_{b \in (\Zd)^*}$.
	\begin{enumerate}
		\item Non-degenerate and local: $c_{x,y}(\eta)$ depends only on $\{\eta_z: \vert z - x\vert \leq \r\}$ for some integer $\r > 0$, and is bounded on two sides $1 \leq c_{x,y}(\eta) \leq \lambda$.
		\item Spatially homogeneous: for all $\{x,y\} \in (\Zd)^*$, $c_{x,y} = \tau_x c_{0,y-x}$.
		\item Detailed balance under Bernoulli measures: $c_{x,y}(\eta)$ is independent of $\{\eta_x, \eta_y\}$.
	\end{enumerate}
\end{hypothesis}

The \emph{non-gradient exclusion process} on $\Zd$ is defined by the generator below
\begin{align}\label{eq.Generator}
	\L := \sum_{b \in (\Zd)^*} c_b(\eta) \pi_b = \frac{1}{2}\sum_{x,y \in \Zd: \vert x - y\vert = 1} c_{x,y}(\eta) \pi_{x,y}.
\end{align}
This model is also called \emph{the speed-change Kawasaki dynamics} or \emph{the lattice gas} in the literature, and we will also use these names alternatively from time to time in the paper.

This model is known of non-gradient type, i.e. we cannot find any function $\{h_{i,j}\}_{1\leq i,j\leq d}$ such that $c_{0,e_i} (\eta)(\eta_{e_i} - \eta_0 ) = \sum_{j=1}^d \Ll((\tau_{e_j} h_{i,j})(\eta) - h_{i,j}(\eta)\Rr)$ for general $\{c_{b}\}_{b \in (\Zd)^*}$, with $\{e_i\}_{1 \leq i \leq d}$ the canonical basis of $\Zd$. 

\smallskip

The hydrodynamic limit of the speed-change Kawasaki dynamics on torus is proved in \cite{fuy}. Let $\Td_N := (\mathbb{Z} / N \mathbb{Z})^d$ be the lattice torus of scale $N$, where we can define all the notations by replacing $\Zd$ with $\Td_N$. We denote by $\X_N := \{0,1\}^{\Td_N}$ the configuration space on $\Td_N$, and define $\eta^N(t) = \{\eta^N_x(t), x \in \Td_N\}$ as the $\X_N$-valued Markov jump process on torus governed by the generator $\L_N := N^2 \L$, the counterpart of \eqref{eq.Generator} on $\Td_N$. The macroscopic empirical measure of $\eta^N(t)$ is defined as the following measure on $\Td$
\begin{align}\label{eq.empricalMeasure}
	\rho^N(t, \d v) := N^{-d} \sum_{x \in \Td_N} \eta^N_x(t) \delta_{x/N}(\d v),
\end{align} 
and the limit is the unique (classical) solution of nonlinear diffusion equation 
\begin{align}\label{eq.nonlinear}
	\partial_t \rho(t, v) = \nabla \cdot(D(\rho(t, v)) \nabla \rho(t, v)), \qquad (t, v) \in \R_+ \times \Td,
\end{align}
provided that the initial profile $\rho^N(0,\d v)$ has a smooth limit $\rho(0,v)$. Here $\Td = \Rd / \Zd$ is the continuous torus and $D : (0,1) \to \R^{d \times d}$ is \emph{the diffusion matrix} defined by \emph{the Einstein relation} 
\begin{align}\label{eq.Einstein}
	D(\rho) := \frac{\c(\rho)}{2 \chi(\rho)},
\end{align}
where $\chi(\rho)$ is \emph{the compressibility}
\begin{align}\label{eq.defCompress}
	\chi(\rho) := \rho (1-\rho),
\end{align}  
and  $\c(\rho)$ is \emph{the effective conductivity} defined as follows. Let $\F_0$ stand for the local function space on $\X$ and $\F_0^d := (\F_0)^d$, and $\bracket{\cdot}_{\rho}$ stand for the expectation under Bernoulli product measure of density $\rho \in [0,1]$. We construct at first a quadratic form with respect to the $\Rd$-valued function $F \in \F_0^d$ 
\begin{align}\label{eq.defQuadra}
	\forall \xi \in \Rd, \quad \xi \cdot \c(\rho; F) \xi = \frac{1}{2} \sum_{\vert x\vert = 1} \bracket{c_{0,x}\Ll(\xi \cdot \Ll\{ x(\eta_x - \eta_0) - \pi_{0,x}(\sum_{y \in \Zd} \tau_y F)\Rr\}\Rr)^2}_{\rho}.
\end{align}
Then $\c(\rho)$ is the minimization of $\c(\rho; F)$
\begin{align}\label{eq.defC}
	\xi \cdot \c(\rho) \xi := \inf_{F \in \F_0^d}\xi \cdot \c(\rho; F) \xi.
\end{align}
See \cite[Part II, Proposition~2.2]{spohn2012large}  for the relation between 
the diffusion matrix $D(\rho)$ and the Green--Kubo formula.
It is proved in \cite{bernardin} that $D(\rho)$ is continuously extendable to $\rho=0, 1$ and
\begin{equation} \label{1.D}
	D \in C^\infty([0,1]).
\end{equation}
Furthermore, the diffusion matrix $D(\rho)$ is uniformly positive and bounded:
\begin{align}  \label{eq:1.9}
	\forall \xi \in \R^d \text{ and } \rho\in [0,1], \qquad c_* |\xi|^2 \le \xi \cdot D(\rho)\xi \le c^* |\xi|^2,
\end{align}
where $0<c_*< c^*<\infty$; see \cite[Lemma 12.1]{Fu24}.

The hydrodynamic limit in \cite[Theorem~1.1]{fuy} states the following convergence: let $P$ stand for the probability space of the process $(\eta^N(t))_{t \in \R_+}$ and $\rho(t, \d v) = \rho(t, v) \d v$. For every $\phi \in C^\infty(\Td)$ and $\epsilon > 0$, we have
\begin{align}  \label{eq.P-decay-Qualitative}
	P\Ll[ \Ll|\int_{\Td} \phi(v) \rho^N(t, \d v) -
	\int_{\Td} \phi(v) \rho(t, \d v) \Rr| > \e\Rr] \xrightarrow{N \to \infty} 0,
\end{align}
assuming that the distribution of $\eta^N(0)$ is close to the local equilibrium
state with a density profile $\rho_0(v)$, which satisfies following Hypothesis \ref{hyp:1.2}.

\begin{hypothesis}\label{hyp:1.2} 
	The initial value $\rho_0 = \rho_0(v)$ of the nonlinear diffusion equation 
	\eqref{eq.nonlinear} is smooth (i.e., $\rho_0\in C^\infty(\T^d)$) and satisfies
	$0<\rho_0(v)<1$ for every $v\in \T^d$.
\end{hypothesis}
Hypothesis \ref{hyp:1.2}, together with the properties 
\eqref{1.D} and \eqref{eq:1.9} of $D(\rho)$, ensures the unique smooth classical solution of
the  equation \eqref{eq.nonlinear} in the class ${\rho(t,v)\in C^{1,3}(\R_+\times\T^d)}$; see \cite[Theorem 6.1, p.452]{Ladyzen}, and the discussion in  \cite[Section~10]{Fu24}. Moreover, by the maximum principle (see \cite[p.389]{Evans}), we have  $0<\rho(t,v)<1$. 

The proof of the hydrodynamic limit in \cite{fuy} relies on \emph{the relative entropy method}. Roughly, the relative entropy $h_N(t)$ measures how close the empirical measure $\eta^N(t)$ is to the solution $\rho(t, \cdot)$; see Section~\ref{subsec.hydro_setting} for the definition of $h_N(t)$. The assumption in \eqref{eq.P-decay-Qualitative} about the initial data is $h_N(0) = o(1)$ as $N \to \infty$, and \eqref{eq.P-decay-Qualitative} was proved via $h_N(t) \xrightarrow{N \to \infty} 0$. More precisely, one key step (\cite[Theorem~2.1, Corollary~2.1]{fuy}) proved that, for every $\beta > 0$ and small $\delta > 0$, the following estimate for $h_N(t)$ is valid
\begin{align}\label{eq.EntropyFlow}
	h_N(t) \leq h_N(0) + \frac{1}{\delta}\int_0^t h_N(s) \, \d s + C(\beta +1) \sup_{\rho \in [0,1]} \vert R(\rho; F_N) \vert + \frac{C}{\beta} + Q_N.
\end{align}
One can find the details of the error term $Q_N \, (= Q^{\Om_1}_N+Q^{LD}_N+Q^{En}_N+Q^{\Om_2}_N)$ in Proposition~\ref{prop.after_LD}. The function $F_N$ in \eqref{eq.EntropyFlow} aims to better approximate the reference measure in $h_N(t)$ (see \eqref{eq:2psit} for details), and the quantity $R(\rho; F)$ in \eqref{eq.EntropyFlow} comes from the conductivity
\eqref{eq.defQuadra} and \eqref{eq.defC}, which is defined as
\begin{align}\label{eq.defR}
	R(\rho; F) :=  \c(\rho; F) - \c(\rho).
\end{align}
To conclude the decay of relative entropy $h_N(t)$, we expect naturally that the error term $Q_N$ vanishes as $N \to \infty$. We need to take a large $\beta$ such that $C/\b$ on the {\rhs} of \eqref{eq.EntropyFlow} is small, and then apply Gronwall's inequality and the entropy inequality. Therefore,  the term $R(\rho; F_N)$ in \eqref{eq.EntropyFlow} also contributes the decay, and it was proved in \cite[Lemma~2.1]{fuy} that 
\begin{align}\label{eq.RQualitative}
	\inf_{F \in \F^d_0} \sup_{\rho \in [0,1]} \vert R(\rho; F) \vert = 0,
\end{align}
where $|R|:= \big(\sum_{i,j=1}^d R_{ij}^2\big)^{1/2}$ stands for the norm of
the matrix $R=(R_{ij})_{1\le i,j\le d}$.

\bigskip

The object of this paper is to give a convergence rate for $\rho^N(t, \cdot)\to\rho(t,\cdot)$ as $N\to\infty$, improving the qualitative convergence
in probability stated in \eqref{eq.P-decay-Qualitative}. As explained briefly above, 
one essential step is the decay of $R(\rho; F_N)$, and we need to study \eqref{eq.RQualitative} more precisely. Actually, the refinement of \eqref{eq.RQualitative} can be considered as a quantitative homogenization of the fundamental quantity $\c(\rho)$, which is our main result stated as follows. In the statement, $\Lambda_L := (-\frac{L}{2}, \frac{L}{2})^d \cap \Zd$ stands for a hypercube of side length around $L \in \R_+$, and $\F_0^d(\Lambda_L)$ is the subset of $\F_0^d$ which contains $\sigma(\{\eta_x\}_{x \in \Lambda_L})$-measurable local functions.   

\begin{theorem}\label{thm.mainUniform}
	Under Hypothesis~\ref{hyp}, there exists an exponent $\gamma(d,\lambda, \r) > 0$ and a positive constant $C(d, \lambda, \r) < \infty$, such that  
	\begin{align}\label{eq.mainUniform}
		\inf_{F_L \in \F^d_0(\Lambda_L)} \sup_{\rho \in [0,1]} 	\vert R(\rho; F_L) \vert \leq C L^{-\gamma}.
	\end{align}
\end{theorem}

We will also give the concrete construction of the local function achieving the estimate above in our proof; see Section~\ref{subsec.Stationary} for details. 
Then as expected, we can insert the estimate \eqref{eq.mainUniform} in \eqref{eq.EntropyFlow},
and obtain a quantitative hydrodynamic limit after careful investigation. In the statement, $H^{-\a}(\T^d)$ is the standard Sobolev space on $\T^d$ with negative index $-\a$. We refer to Section~\ref{subsec.hydro_setting} for the details of relative entropy and other related settings.

\begin{theorem} \label{thm.HLQuant}
	Under Hypothesis~\ref{hyp}, and	we assume Hypothesis~\ref{hyp:1.2} for the initial value $\rho_0$ with the initial
	entropy condition $h_N(0) \le C_0 N^{-\kappa_0}$ for some constants
	${C_0, \kappa_0>0}$.  Then, for every $\a>d/2$ and $T>0$, there exist $\kappa(\k_0, d,\lambda, \r)>0$ and ${C(C_0, \a, T, \rho_0, \k_0, d,\lambda, \r)>0}$ such that the following estimate holds
	\begin{align}  \label{eq.P-decay}
		\sup_{t\in [0,T]} \E\Big[ \big\|\rho^N(t, \cdot) - \rho(t, \cdot)\big\|_{H^{-\a}(\T^d)}^2\Big]
		\le C N^{-\kappa}.
	\end{align}
\end{theorem}

The quantitative hydrodynamic limit was shown for several kinds of
gradient models; see \cite{dizdar2018quantitativeA, menegaki2021, JM}.
Our Theorem \ref{thm.HLQuant}, which is shown for non-gradient model,  
is formulated similar to 
\cite{dizdar2018quantitativeA}, in which one-dimensional Ginzburg--Landau 
lattice model of gradient type was studied and, in their case, 
$d=1$, $\a=1$ and $\k=2/3$.  In our case, the range of $\a$ is $(d/2,\infty)$,
which is taken as widely as possible for the measure $\rho^N(t)$ on $\T^d$.
Our $\k>0$ is determined from several factors and small.  The optimal exponent suggested by the central limit theorem should be $\k \simeq d$.

\subsection{Strategy of the proof} 
The outline of this paper is presented in Figure~\ref{fig.outline}. It can be divided into 4 blocks. 
The main body is the orange block in the center, which consists of Sections~\ref{sec.Sub}, ~\ref{sec.Rate}, ~\ref{sec.disorder},  and contains a proof of quantitative homogenization. The argument roughly follows \cite[Chapter~2]{AKMbook}, which was developed by Armstrong, Kuusi, Mourrat, and Smart, and has been applied to many other models. However, several analytic inputs are required to apply this method to the non-gradient exclusion setting. They are resumed in the red block on the left, consisting of Section~\ref{sec.coarse} and Proposition~\ref{prop.Caccioppoli}. They are also the main technical progresses in this paper. Especially, we need to handle the constraint of particle numbers. The blue block on the right corresponds to Section~\ref{sec.Free}, where we improve the homogenization result, so it becomes independent of density as stated in Theorem~\ref{thm.mainUniform}. Finally, the green block at the bottom refers to the quantitative hydrodynamic limit in Section~\ref{sec.hy}. We further explain the 4 blocks one by one in the following paragraphs.

\begin{figure}
	\centering
	\begin{tikzpicture}[node distance=2cm]\label{fig.diagram}
		\node[text=red] (Lift) [input] {Section~\ref{sec.coarse} \textit{Coarse-grained lifting}};
		\node[text=red] (Poincare) [input, below of = Lift, yshift=0.5cm] {Proposition~\ref{prop.WMPoincare} \textit{Weighted multiscale Poincar\'e inequality}};
		\node[text=red] (Caccioppoli) [input, below of = Poincare, yshift=0.5cm] {Proposition~\ref{prop.Caccioppoli} \textit{Modified Caccioppoli inequality}};
		
		\node[text={rgb,255:red, 204; green, 102; blue,0}] (Renormalization) [result, right of=Poincare, xshift=3cm, yshift=-0.5cm] {Section~\ref{sec.Rate} \\ \textit{Quantitative homogenization via renormalization} \\   (\textit{Extension in} Section~\ref{sec.disorder})};
		\node[text={rgb,255:red, 204; green, 102; blue,0}] (Subadditive) [result, above of = Renormalization] {Section~\ref{sec.Sub} \\ \textit{Subadditive quantities}};
		
		\node[text=cyan] (DensityFree) [result, below of=Renormalization] {Section~\ref{subsec.FreeCorrector} \\ \textit{Construction of density-free corrector}};
		\node[text={rgb,255:red, ; green, 153; blue,0}] (Hydrodynamic) [result, below of=DensityFree] {Section~\ref{sec.hy} \\ \textit{Quantitative hydrodynamic limit}};

		\node[text={rgb,255:red, ; green, 153; blue,0}] (RelativeEntropy) [input, below of=Caccioppoli, yshift=-1cm] {Relative entropy method \cite{fuy}};
		
		\node[text=cyan] (Regularity) [input, right of = Subadditive, xshift=3cm] {Section~\ref{subsec.Regularity} \textit{Regularity of local corrector}};
		\node[text=cyan] (CanonicalEnsemble) [input, below of=Regularity] {Section~\ref{subsec.RateCano} \textit{Convergence under canonical ensemble}};
		\node[text=cyan] (LocalEnsemble) [input, below of=CanonicalEnsemble] {Lemma~\ref{lem.localEquiv} \\ \textit{Local equivalence of ensembles}};
		
		\draw [arrow] (Subadditive) -- (Regularity);	
		\draw [arrow] (Subadditive) -- (Renormalization);
		\draw [arrow] (Renormalization) -- (CanonicalEnsemble);	
		\draw [arrow] (DensityFree) -- (Hydrodynamic);	
		\draw [arrow] (Renormalization) -- (DensityFree) ;
		\draw [arrow] (CanonicalEnsemble.east) |- ($(CanonicalEnsemble.east) + (0.2cm,0cm)$) |- ($(CanonicalEnsemble.east) + (0.2cm,-3cm)$) |-  (Hydrodynamic.east);
		\draw [arrow] (Regularity.north) |- ($(Regularity)  + (0cm,0.8cm)$) |- ($(Lift)  + (-1.9cm,0.8cm)$) |- (DensityFree.west);	
		
		\draw [arrow] (Lift) -- (Poincare);	
		\draw [arrow] (Poincare) -- (Renormalization);
		\draw [arrow] (Caccioppoli) -- (Renormalization);
		\draw [arrow] (LocalEnsemble) -- (CanonicalEnsemble);
		\draw [arrow] (LocalEnsemble) -- (DensityFree);
		\draw [arrow] (RelativeEntropy) -- (Hydrodynamic);
	\end{tikzpicture}
	\caption{The outline of proof.}\label{fig.outline}
\end{figure}

\subsubsection{Homogenization via the renormalization approach} This part concerns the orange block in Figure~\ref{fig.outline}.
The quantitative homogenization via the renormalization approach was developed in \cite{armstrong2016quantitative, armstrong2016lipschitz, armstrong2016mesoscopic, armstrong2017additive,ferg1, ferg2}; see monographs \cite{AKMbook, armstrong2022elliptic} and \cite{informal} for a gentle introduction. This renormalization approach has shown its robustness in a number of other settings including the parabolic equations \cite{armstrong2018parabolic},  the finite-difference equations on percolation clusters \cite{armstrong2018elliptic,dario2018optimal,dario2019quantitative}, the differential forms \cite{dario2018differential}, the ``$\nabla \phi$'' interface model \cite{dario2019phi,gradphi2,armstrong2023scaling, ArmstrongDario2024}, and the Villain model \cite{dario2020massless}. Recently, \cite{bulk, giunti2021smoothness, gu2024quantitative} also generalize the theory to an interacting particle system in continuous space \emph{without} exclusion.

The heart of the renormalization approach requires a finite-volume approximation with sub-additivity. Therefore, let us first reformulate  \eqref{eq.defQuadra} and \eqref{eq.defC}. Consider a formal sum $\ell_{\xi} = \sum_{x \in \Zd} (\xi \cdot x) \eta_x$, then we notice that the term $\xi \cdot x(\eta_x - \eta_0)$  in \eqref{eq.defQuadra} is 
\begin{align*}
	\xi \cdot x(\eta_x - \eta_0) = - \pi_{0,x} \ell_{\xi}.
\end{align*} 
Moreover,  we denote by the formal sum $\Gamma_F :=\sum_{y \in \Zd} \tau_y F $ given $F \in \F_0^d$. Thus, \eqref{eq.defQuadra} is the Dirichlet energy of a linear statistic plus some correction term $\Gamma_F$
\begin{equation}\label{eq.defQuadra2}
	\begin{split}
		\xi \cdot \c(\rho; F) \xi &=  \sum_{i=1}^d \bracket{c_{0,e_i}\Ll(\pi_{0,e_i} (\ell_\xi + 	\xi \cdot \Gamma_F\Rr)^2}_{\rho} \\
		&= \frac{1}{\vert \Lambda \vert} \sum_{i=1}^d \sum_{x \in \Lambda} \bracket{c_{x,x+e_i}\Ll(\pi_{x,x+e_i} (\ell_\xi + 	\xi \cdot \Gamma_F\Rr)^2}_{\rho}.
	\end{split}
\end{equation}
Here $\Lambda \subset \Zd$ and the passage to the second line is based on stationarity.

We then propose the following expression as a natural finite-volume approximation of \eqref{eq.defC} 
\begin{align}\label{eq.defC2}
	\frac{1}{2} \xi \cdot \cc(\rho, \Lambda) \xi := \inf_{v \in \ell_{\xi} +  \F_0(\Lambda^-)} \frac{1}{ \vert \Lambda \vert} \bracket{ v (- \L_\Lambda v)}_{\rho}. 
\end{align}
Here $\Lambda^-$ is the interior of $\Lambda$ and $\overline{\Lambda^*}$ is the set of bonds issued from $\Lambda$ (see \eqref{eq.defMinus}, \eqref{eq.defBondLarge} for details). $\F_0(\Lambda^-)$ is the set of $\sigma(\{\eta_x\}_{x \in \Lambda^-})$-measurable local functions, and $\L_{\Lambda}$ is the generator on $\Lambda$
\begin{align}\label{eq.GeneratorDomain}
	\L_\Lambda v := \sum_{b \in \ovs{\Lambda}} c_b \pi_b v.
\end{align}
A well-known integration by part formula under $\bracket{\cdot}_\rho$ also tells us 
\begin{align}\label{eq.IPP}
	\bracket{ v (- \L_\Lambda v)}_{\rho} =\frac{1}{2} \sum_{b \in \ovs{\Lambda}} \bracket{c_b (\pi_b v)^2}_{\rho},
\end{align}
so \eqref{eq.defC2} can be written as 
\begin{align*}
	\xi \cdot \cc(\rho, \Lambda) \xi &= \inf_{v \in \ell_{\xi} +  \F_0(\Lambda^-)} \frac{1}{ \vert \Lambda \vert}\bracket{c_b (\pi_b v)^2}_{\rho}\\
	& \approx \inf_{v \in \ell_{\xi} +  \F_0(\Lambda^-)} \frac{1}{ \vert \Lambda \vert} \sum_{i=1}^d \sum_{x \in \Lambda} \bracket{c_{x, x+e_i} (\pi_{x, x+e_i} v)^2}_{\rho}.
\end{align*}
We throw away the boundary layer in the second line, then this formula is very close to the version in \eqref{eq.defQuadra2}. We expect that  $\cc(\rho, \Lambda)$ converges to $\c(\rho)$ when $\Lambda \nearrow \Zd$.

A similar definition like \eqref{eq.defC2} can also be posed under the canonical ensemble
\begin{align}\label{eq.defC3}
	\frac{1}{2} \xi \cdot \ca(\Lambda, N) \xi := \inf_{v \in \ell_{\xi} +  \F_0(\Lambda^-)}  \frac{1}{\vert \Lambda \vert} \bracket{ v (- \L_\Lambda v)}_{\Lambda, N}, 
\end{align}
where $\bracket{\cdot}_{\Lambda, N}$ is the expectation under the uniform measure of $N$ particles in $\Lambda$. Notice that the quantity $\ca(\Lambda, N)$ still depends on the configuration outside $\Lambda$. On the other hand, because the jump rate $c$ is of finite range $\r$ by Hypothesis~\ref{hyp}, the influence from the boundary layer vanishes  when $\Lambda \nearrow \Zd$. Thus $\ca(\Lambda, N)$ should be close to $\c(N/\vert \Lambda \vert)$ in large scale.

The convergence of these two quantities are intermediate steps to prove Theorem~\ref{thm.mainUniform}, and are also related to the CLT variance estimate in \cite[Section~5]{fuy}.
\begin{theorem}\label{thm.main1}
	Under Hypothesis~\ref{hyp}, there exists a constant $C(d,  \lambda, \r) < \infty$ and two exponents ${\gamma_1(d, \lambda, \r), \gamma_2(d, \lambda, \r) > 0}$ such that for every $L,M \in \N_+$,
	\begin{align}\label{eq.main1A}
		\Ll\vert \cc(\rho, \Lambda_L) - \c(\rho) \Rr\vert  \leq C L^{-\gamma_1},
	\end{align}
	and 
	\begin{align}\label{eq.main1B}
		\Ll\vert \ca( \Lambda_L, M) - \c(M/\vert \Lambda_L \vert) \Rr\vert  \leq C L^{-\gamma_2}.
	\end{align}
\end{theorem}
The convergence of quantities in Theorem~\ref{thm.main1} is largely studied using the renormalization approach, because $\cc(\rho, \Lambda)$ defined in \eqref{eq.defC2} satisfies the sub-additivity. That is,  for disjoint union $\Lambda = \bigsqcup_{i=1}^N \Lambda^{(i)}$, we have
\begin{align}\label{eq.SubC}
	\cc(\rho, \Lambda) \leq \sum_{i=1}^N \frac{\vert \Lambda^{(i)}\vert}{\vert \Lambda \vert} \cc(\rho, \Lambda^{(i)}).
\end{align}
To prove it, it suffices to ``glue" the minimizers in $\{\Lambda^{(i)}\}_{1 \leq i \leq N}$ together, which yields a sub-minimizer for the functional of $\cc(\rho, \Lambda)$. Then, $\cc(\rho, \Lambda)$ admits a limit, saying $\cc(\rho)$, when $\Lambda \nearrow \Zd$. The coincidence $\cc(\rho) = \c(\rho)$ is not contained directly from their definitions, but can be justified later in Proposition~\ref{prop.ProofMain}.

The sub-additivity \eqref{eq.SubC} only implies a qualitative convergence. To quantify the rate, an important observation from \cite[Chapter~2]{AKMbook} is another dual quantity $\cc_*(\rho,  \Lambda)$. It satisfies $\cc_*(\rho,  \Lambda) \leq \cc(\rho, \Lambda)$, so we can use its gap $\Ll(\cc(\rho, \Lambda) - \cc_*(\rho,  \Lambda)\Rr)$ to measure the convergence rate. Surprisingly, the gap itself $\Ll(\cc(\rho, \Lambda) - \cc_*(\rho, \Lambda)\Rr)$ roughly satisfies the sub-additivity like \eqref{eq.SubC} 
\begin{align}\label{eq.SubC_gap}
	\Ll(\cc(\rho, \Lambda) - \cc_*(\rho, \Lambda)\Rr) \leq \sum_{i=1}^N \frac{\vert \Lambda^{(i)}\vert}{\vert \Lambda \vert}  \Ll(\cc(\rho, \Lambda^{(i)}) - \cc_*(\rho, \Lambda^{(i)})\Rr).
\end{align}
See \eqref{eq.Jsub} for its rigorous version. This non-trivial property hints the quantitative homogenization. The complete proof is presented in Section~\ref{sec.Sub} and~\ref{sec.Rate}.

\begin{remark}\label{rmk.Domain}
	We have three remarks about Theorem~\ref{thm.main1}.
	\begin{enumerate}
		\item Recalling the Einstein relation \eqref{eq.Einstein}, the results in Theorem~\ref{thm.main1} also imply the convergence rate of the diffusion matrix $D(\rho)$, for all $\rho \in (0,1)$, by local functions or the finite-volume approximation.
		
		\item The choice of the notation $\ovs{\Lambda}$ here is just for technical convenience. Lemma~\ref{lem.WhitneySub} ensures the stability of Theorem~\ref{thm.main1} in the general domain, so the statement still holds if one replaces $\ovs{\Lambda}$ in \eqref{eq.GeneratorDomain} by the canonical notation $\Lambda^*$ defined in \eqref{eq.defBond}.
		
		\item We also obtain an estimate similar to \eqref{eq.main1A} when the disorder is posed on bonds. To simplify the notation, we leave the related discussion in Section~\ref{sec.disorder}.
	\end{enumerate}
\end{remark}

\bigskip
\subsubsection{Two analytic inequalities and coarse-grained lifting} This part explains the red block in Figure~\ref{fig.outline}.
Besides the observation of sub-additivity \eqref{eq.SubC_gap}, the proof of Theorem~\ref{thm.main1} also needs two analytic inequalities:
\begin{enumerate}
	\item the Caccioppoli inequality;
	\item the multiscale Poincar\'e inequality.
\end{enumerate}
The two key inequalities are well-developed for the elliptic equation on $\Rd$, but can become challenging in other settings. We should also highlight that, the two inequalities are more than the technical estimates, but the essentials of quantitative homogenization, because they characterize the elliptic conditions in the large scale; see the recent work for the homogenization in high contrast \cite{armstrong2024superdiffusive}, \cite{armstrong2024renormalization}. In the exclusion process, these inequalities are not accessible directly. The main contribution in this paper is to relax them respectively as \emph{the modified Caccioppoli inequality} and \emph{the weighted multiscale Poincar\'e inequality}, which are explained in the following paragraphs in detail.

The classical Caccioppoli inequality describes the inner regularity of the elliptic equations, but seems missing in the particle systems due to the influence of particles near the boundary. Therefore, the modified Caccioppoli inequality is developed in the previous work \cite[Proposition 3.9]{bulk}, which differs from the classical one, but captures the same spirit. In the present work, its counterpart in the exclusion process is also recovered in Proposition~\ref{prop.Caccioppoli}. The proof requires more work due to the microscopic behavior; see Lemma~\ref{lem.GlauberExchange} and Lemma~\ref{lem.variationBC}. As new inputs, the Glauber derivative and \emph{the reverse Efron--Stein inequality} are also involved. 

The multiscale Poincar\'e inequality (see \cite[Proposition~1.12 and Corollary~1.14]{AKMbook} for example) improves the estimate of the classical Poincar\'e inequality when the function has the spatial cancellation property, which is the case in homogenization. In particle systems, this inequality was derived in the previous work \cite{bulk}, but meets obstacles in the present paper due to the exclusion rule. Precisely, we need the $H^2$-estimate in the proof (see \cite[Corollary~1.14]{AKMbook}), which is only true for the standard Laplacian operator $\Delta$ rather than a general variable-speed divergence form $-\nabla \cdot \aa(x) \nabla$ in PDE setting. In exclusion processes, one may expect the generator of SSEP to be a natural analogue of $\Delta$. Unfortunately, it is not totally correct; see Remark~\ref{rmk.H2Challenge} for details. A heuristic explanation is that, particles in SSEP are still less free than independent random walks. For this reason, we believe the multiscale Poincar\'e inequality should live in the independent particles $\tilde \X = \N^{\Zd}$ where $\N = \{0,1,2,3, \cdots\}$, and we prove it in Section~\ref{subsec.free}.

Then a crucial problem is how to apply an inequality in $\XX = \N^{\Zd}$ to the functions in $\X = \{0,1\}^{\Zd}$. A similar problem was also posed in the percolation setting, and a possible solution is the coarse-grained strategy; see \cite{armstrong2018elliptic, dario2018optimal, gu2019efficient, dario2019quantitative}. We hope to implement this idea in the exclusion processes: for every function $u : \X \to \R$, we aim to find a coarsened function $[u] : \XX \to \R$ in the larger space such that for every $\eta \in \X$ as the grain of $\teta \in \tilde \X$, it satisfies 
\begin{align}\label{eq.RoughCoarseGrain}
	\teta \simeq \eta \Longrightarrow [u](\teta) \simeq u(\eta).
\end{align} 
A naive candidate for grain $\eta$ is the one close to $\teta$ under some distance. However, different from the Bernoulli percolation setting,  the space of grain $\X$ is a very sparse subset of $\XX$ in particle systems,  so we face \emph{the curse of dimension} which may extremely enlarge the error in \eqref{eq.RoughCoarseGrain}. This is also the key difficulty compared to the previous work \cite{bulk}. Our solution turns out to be not only a coarse-grained method of functions, but also a \emph{lift} from $\X$ to $\XX$, i.e. we can represent the function on Kawasaki dynamics using a coupled system of independent particles. More precisely, for every $\teta \in \XX$, we set its grain $[\teta] \in \X$ as 
\begin{align*}
	\forall x \in \Zd, \qquad [\teta]_x := \Ind{\teta_x \geq 1},
\end{align*}
and define the coarsened function of $u : \X \to \R$ as
\begin{align*}
	[u]:\XX \to \R, \qquad [u](\teta) := u([\teta]).
\end{align*}
This \emph{coarse-grained lifting} is introduced in Section~\ref{sec.coarse}. Based on this technique, we also obtain a \emph{gradient coupling} between two systems (see Propositions~\ref{prop.IdentityGradient} and \ref{prop.IdentityGradient2}), and a weighted multiscale Poincar\'e inequality on Kawasaki dynamics. They are the main tools to evaluate the flatness of the functions in Proposition~\ref{prop.L2Flat}.

\begin{figure}[h]
	\centering
	\includegraphics[width=0.45\linewidth]{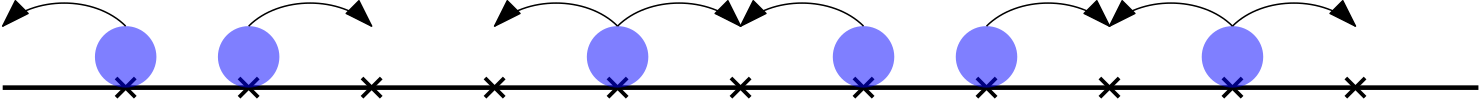}
	\includegraphics[width=0.45\linewidth]{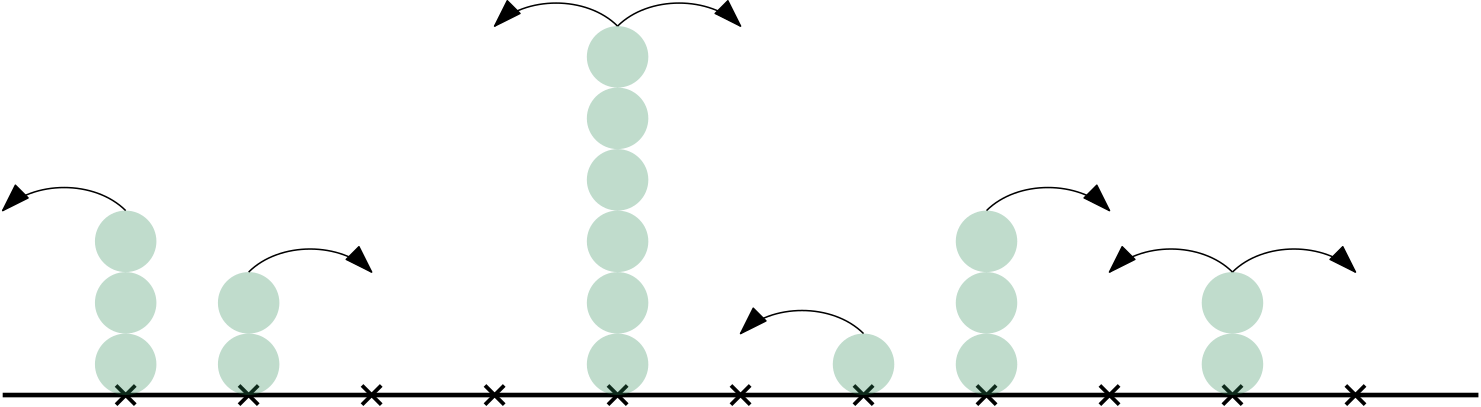}
	\caption{An illustration for the coarse-grained lifting between the Kawasaki dynamics and independent particles.}
\end{figure}

\bigskip
\subsubsection{Regularity and uniform convergence}
This part further develops the idea in the blue block in Figure~\ref{fig.outline}.
Usually the convergence rate depends on the particle density, so let us explain why a uniform convergence is valid. A first qualitative argument is that, our finite-volume approximation decreases to the limit, and the limit function $\rho \mapsto \c(\rho)$ is continuous thanks to \cite{naga1}, so Dini's theorem applies and the convergence is locally uniform. At the quantitative level, we should track the dependence of density carefully in analysis. Especially, the two key inputs above are both robust with respect to the density. We also notice that $\c$ has a trivial bound by $\chi(\rho) = \rho(1-\rho)$, and this can help us at two endpoints.


We still need some more ingredients to pass the results from Theorem~\ref{thm.main1} to Theorem~\ref{thm.mainUniform}. The uniform estimates in \eqref{eq.main1A} can be seen as a weak version
\begin{align*}
	\sup_{\rho \in [0,1]}	\inf_{F_{\rho, L} \in \F^d_0(\Lambda_L)}   R(\rho; F_{\rho, L})  \leq C L^{-\gamma_1},
\end{align*}
while quantity in \eqref{eq.mainUniform} is a strong version, and usually we have
\begin{align*}
	\sup_{\rho \in [0,1]}	\inf_{F_{\rho, L} \in \F^d_0(\Lambda_L)}   R(\rho; F_{\rho, L}) \leq \inf_{F_L \in \F^d_0(\Lambda_L)} \sup_{\rho \in [0,1]}  R(\rho; F_L).
\end{align*}
We do not know whether there exists any duality property in the function $R(\rho; F)$, thus we make the proof by a direct construction. The minimizer $F$ in this variational problem is actually \emph{the correctors} in homogenization theory. In the renormalization step, we already get a candidate $\phi_{\rho, \Lambda, \xi}$ for the problem \eqref{eq.defC2}, but it has dependence on the density. Our main task is to remove this dependence, so we propose a modified version of the local corrector: for disjoint union $\Lambda = \bigsqcup_{i=1}^N \Lambda^{(i)}$, we make a combination of the local correctors on every subset
\begin{align}\label{eq.CorrectorFreeRough}
	\tilde \phi_{\hat{\rho}, \Lambda, \xi} = \sum_{i=1}^N \phi_{\hat{\rho}, \Lambda_i, \xi},
\end{align} 
where $\hat{\rho}$ is an empirical density on the domain $\Lambda$ instead of a fixed designed density. Insert this function in the problem \eqref{eq.defC2}, we obtain a uniform convergence under the grand canonical ensemble by the following reasons.
\begin{enumerate}
	\item The homogenization appears in the large scale, so \eqref{eq.SubC} is nearly an equality and $\sum_{i=1}^N \phi_{\rho, \Lambda_i, \xi}$ nearly equals $\phi_{\rho, \Lambda, \xi}$.
	\item Given an empirical density $\hat{\rho}$, each local corrector lives as if under the grand canonical ensemble thanks to the local equivalence of ensembles. This is also the trick in the proof of \eqref{eq.main1B}.
	\item The empirical density $\hat{\rho}$ may also fluctuate when applying the Kawasaki operator, but this can be handled. On the one hand, we have the regularity of the mapping $\rho \mapsto \phi_{\rho, \Lambda_i, \xi}$, and each fluctuation of density is just $1/\vert \Lambda \vert$. On the other hand, such fluctuation only happens on the boundary layer of $\Lambda$, whose order is dominated by the volume order in \eqref{eq.defC2}.
\end{enumerate}
Similar argument actually has already appeared in the proof of \eqref{eq.RQualitative} in \cite[Lemma~2.1]{fuy}. Besides the quantitative homogenization in (1), we also need to calculate carefully the errors in (2) and (3), i.e. that from  local equivalence of ensembles and the regularity of density. They are discussed in detail in Sections~\ref{subsec.RateCano} and ~\ref{subsec.Regularity}, and then we justify the density-free corrector \eqref{eq.CorrectorFreeRough} in Section~\ref{subsec.FreeCorrector}.

\subsubsection{Hydrodynamic Limit} This part corresponds to the green block in Figure~\ref{fig.outline}.  Overall, Section~\ref{sec.hy} reformulates the relative entropy argument in \cite{fuy}, and integrates the elements from homogenization. As mentioned briefly in \eqref{eq.EntropyFlow}, Theorem~\ref{thm.mainUniform} resolves well the rate in the term $R(\rho; F_N)$. It remains to analyze the error term $Q_N$. The error $Q_N = Q^{\Om_1}_N+Q^{LD}_N+Q^{En}_N+Q^{\Om_2}_N$ has several sources, and most of them can be handled after careful investigation, but the term $Q^{\Om_1}_N$ also involves the homogenization. This term comes from the microscopic current in non-gradient model,  which does not have a gradient form and yields a diverging factor under the scaling.  To overcome this difficulty, we need to show that, under a large space-time domain, one can replace such a term by a well-behaving function of gradient form  asymptotically. This technique is called \emph{the gradient replacement} (see \cite[Theorem~3.2 and Lemma~3.4]{fuy}
and Section~\ref{subsubsec.grad_replacement}). To implement it, a standard approach is Varadhan's lemma (see \cite[Theorem~5.2]{varadhanII} and \cite[Theorem~4.1]{fuy}) which gives the characterization of closed forms defined on a configuration space. We observe a connection between the gradient replacement and the dual quantity employed in the renormalization approach (see Section~\ref{subsec.CLT}), thus our method provides another new route for the non-gradient hydrodynamic limit avoiding Varadhan's lemma.

\subsection{Perspectives}
It is natural to ask how to extend the results in this paper to other non-gradient particles, including GSEP, the multi-type SSEP, etc. Among the four steps in Figure~\ref{fig.outline}, the blue block and the green block respectively require only the local equivalence of the ensemble and the relative entropy, which are also developed in other models. The main problem is to derive Theorem~\ref{thm.main1} in other settings. 

The general framework of renormalization approach is also quite robust as it works in several different statistical physics models \cite{armstrong2018elliptic,dario2019phi, gradphi2, ArmstrongDario2024,dario2020massless}. The only issue is whether one can recover all the technical inputs in the red block in Figure~\ref{fig.outline}. They are quite standard in PDE setting, but face challenges in particle systems due to different microscopical behaviors. One major progress in this paper is the  coarse-grained lifting, which is inspired by \cite{armstrong2018elliptic} in percolation setting. This is not only a technique, but also reveals the essentials of homogenization. Roughly, the homogenization means the diffusion is close to the Brownian motion. If we make its analogue, an interacting particle system should be homogenized to independent random walks. Then the lifting allows us to compare the functions defined in two systems.

Meanwhile, the coarse-grained lifting in this paper still has some limitation, as it uses the specific rule of exclusion that at most $1$ particle per site. For GSEP studied in \cite{KLO94}, since $N$ particles are allowed on every site, then it is hard to find any one-parameter free particle system to implement the lifting. A possible approach is to further enlarger the space via a superposition of several free particle systems of different parameters. This extension is not obvious and requires a careful justification of the steps in Section~\ref{sec.coarse}.

The idea of coarse-graining is not new in conservative particle systems. In the literature, it usually indicates the convergence of local density in hydrodynamic limit. We believe it is also related to our argument in  Section~\ref{sec.coarse}, as the convergence of density plays role in Proposition~\ref{prop.WMPoincare}. We hope to clarify this connection, then establish a more robust coarse-graining argument to cover more examples in the future work.

\subsection{Notations}
We end the introduction by a summary of notations.

\subsubsection{Geometry}\label{subsubsec.Geometry}
We use $\vert \cdot \vert$ to stand for the usual $\ell^2$-norm for the finite dimensional vector or matrices. Meanwhile, for any $x,y \in \Zd$, we define
\begin{align}\label{eq.defDistance}
	\dist(x,y) := \max\{\vert x_1 - y_1 \vert, \vert x_2 - y_2 \vert, \cdots, \vert x_d - y_d\vert\}.
\end{align}
This also generalizes to $\dist(x,\Lambda):= \inf_{y \in \Lambda} \dist(x,y)$ for every subset $\Lambda \subset \Zd$.

We denote by $\Lambda_L := (-\frac{L}{2}, \frac{L}{2})^d \cap \Zd$ the hypercube of side length around $L$, where $L \in \R_+$ is not necessarily an integer for the flexibility. For every $m \in \N = \{0,1,2,\cdots\}$, we also denote by $\cu_m := \Ll(-\frac{3^m}{2}, \frac{3^m}{2}\Rr)^d \cap \Zd$ the hypercube of side length $3^m$. For any $n,m \in \N$ such that $n < m$, we denote by $\Z_{m,n} := 3^n \Zd \cap \cu_m$ and $\Z_n := 3^n \Zd$. Then using $\sqcup$ as the disjoint union, we have the following partition
\begin{align}\label{eq.CubeRenormalization1}
	\cu_m = \bigsqcup_{z \in \Z_{m,n}} (z+\cu_n),
\end{align}
which provides convenience to implement the renormalization. 

For any finite set $\Lambda \subset \Zd$, we denote by $\vert \Lambda\vert$ the number of vertices 
\begin{align}\label{eq.defVolume}
	\vert \Lambda\vert := \#\{x : x\in \Lambda\},
\end{align}
and define the diameter as 
\begin{align}\label{eq.defDiameter}
	\diam(\Lambda) := \max\{\vert x - y\vert:x,y\in \Lambda\}.
\end{align}
We also define $\partial \Lambda$ as the boundary set of $\Lambda$ that 
\begin{align}\label{eq.defBoundary}
	\partial \Lambda := \{ x \in \Lambda : \exists y \notin \Lambda, y \sim x\},
\end{align}
and denote by $\Lambda^-$ the interior of $\Lambda$
\begin{align}\label{eq.defMinus}
	\Lambda^- := \Lambda \setminus \partial \Lambda.
\end{align}
Recall that the set of bonds of $\Lambda$ is defined as $\Lambda^*$ in \eqref{eq.defBond}. We  define its enlarged version 
\begin{align}\label{eq.defBondLarge}
	\overline{\Lambda^*} := \{\{x,y\}: x \in \Lambda, y = x+e_i, i=1,2, \cdots, d\},
\end{align}
where $e_i \in \Zd$ is the $i$-th directed unit vector, and also denote by $\Lambda^+$ the vertices concerned in \eqref{eq.defBondLarge}
\begin{align} \label{eq.defVertexLarge}
	\Lambda^+:=\Lambda\cup\bigcup_{i=1}^d(\Lambda+e_i).
\end{align}
One motivation is that, for $n,m \in \N$ such that $n < m$, despite of \eqref{eq.CubeRenormalization1}, we observe that $\bigcup_{z \in \Z_{m,n}} (z+\cu_n)^* \subsetneq (\cu_m)^*$. On the other hand, \eqref{eq.defBondLarge} provides a better partition structure for bonds 
\begin{align}\label{eq.CubeRenormalization2}
	\overline{(\cu_m)^*} = \bigsqcup_{z \in \Z_{m,n}} \overline{(z+\cu_n)^*}.
\end{align}

For disjoint sets $\Lambda, \Lambda' \subset \Zd, \Lambda \cap \Lambda' = \emptyset$, we define $(\Lambda, \Lambda')^*$ as the set of bonds between $\Lambda$ and $\Lambda'$
\begin{align}
	(\Lambda, \Lambda')^* := 	\{\{x,y\}:x\in \Lambda,\ y\in\Lambda'\}.
\end{align}
Especially, $(\Lambda, \Lambda^c)^*$ is the set of bonds connecting $\Lambda$ and its complement.

\subsubsection{Probability spaces}\label{subsubsec.proba}
For every $\Lambda \subset \Zd$, we denote by $\fil_{\Lambda}$ the $\sigma$-algebra generated by the $(\eta_x)_{x \in \Lambda}$ and we write $\fil$ short for $\fil_{\Zd}$. Given $\rho \in (0,1)$ as the density of particle, and make use of $\Pr$ as the Bernoulli product measure $\operatorname{Ber}(\rho)^{\otimes \Zd}$ on $\X$, thus $(\X, \fil, \Pr)$ is the triplet of probability space most used in this paper. For the expectation under $\Pr$, we use the notation $\bracket{ \cdot }_{\rho}$ or $\Er[\cdot]$. We make use of $\P_{\rho, \Lambda}, \bracket{\cdot}_{\rho, \Lambda}$ when we restrict our measure on $(\eta_x)_{x \in \Lambda}$. We also denote by $\P_{\Lambda, N, \zeta}$ and $\bracket{\cdot}_{\Lambda, N, \zeta}$ the probability and expectation under the canonical ensemble, i.e. $N$ particles distributed uniformly on different sites of $\Lambda$ with the configuration $\zeta$ on $\Lambda^c$.  We sometimes omit $\zeta$ and just write them as  $\P_{\Lambda, N}$ and $\bracket{\cdot}_{\Lambda, N}$. See also the notation $\mathcal{G}_{\Lambda}$ defined in \eqref{eq.defGTildeKawa} about the canonical ensemble.

\subsubsection{Function spaces}\label{subsubsec.Functions}
For every $1 \leq p \leq \infty$, we denote by $\norm{\cdot}_{L^p}$ or $\norm{\cdot}_p$ the $L^p$ norm over the probability space $(\X, \fil, \Pr)$, and denote by $L^p(\X, \fil, \Pr)$ or shortly $L^p$ the set of random variables with finite norm. For any $\Lambda \subset \Zd$, let $\F_0(\Lambda)$ be the set of $\fil_{\Lambda}$-measurable local functions. We also define the Sobolev norm $H^1(\Lambda)$ for a measurable function:
\begin{align}\label{eq.defSobolev}
	\norm{f}^2_{H^1(\Lambda)} = \bracket{f^2}_{\rho} + \sum_{b \in \Lambda^*} \bracket{(\pi_b f)^2}_{\rho}.
\end{align}
For every local functions, we can calculate its $H^1(\Lambda)$ norm, and we also use $H^1(\Lambda)$ to represent the set of functions with finite $H^1(\Lambda)$ norm. Despite of the natural definition of $\F_0(\Lambda)$, the function space $\F_0(\Lambda^-)$ is the proper analogue of the function space $H^1_0$ in the domain $\Lambda$. To see this, we can verify the following identity easily
\begin{align}\label{eq.extension}
	\forall \Lambda \subset \Lambda' \subset \Zd, f \in \F_0(\Lambda^-), \qquad \norm{f}_{H^1(\Lambda)} = \norm{f}_{H^1(\Lambda')}.
\end{align}  
This is the important extension property of $H^1_0$  function, but a general $\F_0(\Lambda)$ function does not necessarily satisfy it. We will not use the notation $H^1_0(\Lambda)$ in the paragraphs for the conciseness of notation, while we keep in mind that $\F_0(\Lambda^-)$ plays the same role. 

Viewing the discussion above, we define the space of harmonic functions with respect to the Kawasaki dynamics
\begin{align}\label{eq.defHarmonic}
	\mcl A (\Lambda) := \{u \in H^1(\Lambda) : \forall v \in \F_0(\Lambda^-), \bracket{v (-\L_\Lambda u)}_\rho = 0\}.
\end{align}
Note that $u \in \mcl A (\Lambda)$ does \emph{not} imply that $u \in \F_0(\Lambda)$ and it can have dependence on the configuration outside $\Lambda$. 

\subsubsection{Operators}
The translation operator, exchange operator and Kawasaki operator are respectively defined in \eqref{eq.translation1}, \eqref{eq.translation2}, \eqref{eq.exchange} and \eqref{eq.Kawasaki}.  For $\eta \in \X$ and $\Lambda \subset \Zd$, we define  $(\eta \mres \Lambda)$ as the configuration restricted on $\Lambda$ that 
\begin{align}\label{eq.Restriction}
	\forall x \in \Zd, \qquad (\eta \mres \Lambda)_x := \eta_x \Ind{x \in \Lambda}.
\end{align}
We sometimes identify $\eta \in \X$ as $\eta = \sum_{x\in\Zd} \eta_x \delta_x$ for the convenience to manipulate.

The affine function defined by
\begin{align}\label{eq.defAffineFormal}
	\ell_p (\eta) := \sum_{x \in \Zd} (p \cdot x) \, \eta_x,
\end{align}
is just a formal sum as there are infinite terms, while $\pi_b \ell_p$ is well-defined as 
\begin{align}\label{eq.defPiAffine}
	\forall b = \{x,y\} \in (\Zd)^*, \qquad (\pi_b \ell_p)(\eta) := p \cdot (y-x)(\eta_x - \eta_y).
\end{align}
A rigorous version of \eqref{eq.defAffineFormal} is a sum restricted on the finite set $\Lambda \subset \Zd$
\begin{align}\label{eq.defAffineRigorous}
	\ell_{p, \Lambda} (\eta) := \sum_{x \in \Lambda} (p \cdot x) \, \eta_x.
\end{align} 

In Kawasaki dynamics, we define the tangent field along the direction $e_i$ at $x$ for $u: \X \to \R$ as
\begin{align}\label{eq.defKawaTagent}
	\nabla_{x, e_i}u := (\pi_{x, x+e_i} u) (\pi_{x, x+e_i} \ell_{e_i}).
\end{align}   
Some simple calculation gives us 
\begin{align}\label{eq.IdTagent}
	(\pi_{x, x+e_i} u) (\pi_{x, x+e_i} \ell_{e_i})(\eta) = \Ll(u(\eta^{x, x+e_i}) - u(\eta)\Rr) (\eta_{x} - \eta_{x+e_i}),
\end{align}
so the term is non-zero if and only if $(\eta_x, \eta_{x+e_i})=(1,0)$ or $(\eta_x, \eta_{x+e_i})=(0,1)$. Moreover, for both two non-zero cases, they evaluate the change of $u$ that a particle jumps from $x$ to $x+e_i$. Similarly, we define the gradient field of $u$ at $x$ as 
\begin{align}\label{eq.defKawaGradient}
	\nabla_x u := (\nabla_{x, e_1}u, \nabla_{x, e_2}u, \cdots, \nabla_{x, e_d}u).
\end{align} 
For every $p \in \Rd$, we also obtain that 
\begin{align}\label{eq.IdGradient}
	p \cdot \nabla_x u = \sum_{i=1}^d (\pi_{x, x+e_i} u) (\pi_{x, x+e_i} \ell_{p}).
\end{align}

The Glauber operator appears naturally in some steps of analysis. We denote by $\eta^x$ the flip operator at $x$ that 
\begin{align}\label{eq.Glauber}
	(\eta^{x})_z = \Ll\{\begin{array}{ll}
		\eta_z, & \qquad z \neq x; \\
		1-\eta_z, & \qquad z = x. 
	\end{array}\Rr.
\end{align}
Then the Glauber derivative for $f: \X \to \R$ is defined by 
\begin{align}
	\pi_x f := f(\eta^x) - f(\eta).
\end{align}
Clearly, $(\pi_x f)^2$ is independent of $\eta_x$.

\subsubsection{Constants}
We usually use $C$ to represent a positive finite constant and $C(\cdots)$ to indicate its dependence on other parameters. The value of $C$ may change from line to line. The following constants will be fixed and used throughout the paper.
\begin{itemize}
	\item $d \in \N_+ = \{1,2,3,\cdots\}$ for the dimension of lattice.
	\item $\lambda$ for the upper bound of the jump rate, i.e. $c_{b}(\eta) \leq \lambda$ for any $b \in (\Zd)^*, \eta \in \X$.
	\item $\r$ for the radius of dependence of jump rate as indicated in Hypothesis~\ref{hyp}.
\end{itemize}

\section{Analytic tools}\label{sec.Pre}

In this part, we collect all the necessary analytic tools in this paper. The two main results of this section are the modified Caccioppoli inequality (Proposition~\ref{prop.Caccioppoli}), and the weighted multiscale Poincar\'e inequality (Proposition~\ref{prop.MPoincare} and then Proposition~\ref{prop.WMPoincare}).

\subsection{Analytic tools on Kawasaki dynamics}

\subsubsection{Glauber operator meets Kawasaki operator}
Our first inequality comes from the observation in \cite[eq.(B.2)]{jlqy}, which states that we can exchange the site of the Glauber derivative by paying the error of the Kawasaki operator.  

\begin{lemma}\label{lem.GlauberExchange}
	Recall the $L^2$ function space defined in Section~\ref{subsubsec.Functions}, then we have
	\begin{align}\label{eq.GlauberExchange}
		\norm{\pi_x f}_{L^2} \leq \norm{\pi_y f}_{L^2} + \frac{1}{\sqrt{2 \chi(\rho)}} \norm{\pi_{x,y}f}_{L^2}.
	\end{align}
\end{lemma}
\begin{proof}
	As we know, $(\pi_x f)^2$ does not depend on $\eta_x$, thus we decompose $\norm{\pi_x f}_{L^2}$ with respect to the state of $\eta_y$
	\begin{align*}
		\norm{\pi_x f}_{L^2} = \bracket{\Ll(\pi_x f\Rr)^2 \Ind{\eta_y = 1} + \Ll(\pi_x f\Rr)^2 \Ind{\eta_y = 0}}_\rho^{\frac{1}{2}}.
	\end{align*}
	We denote by $\tilde{\eta} = (\eta_z)_{z \in \Zd \setminus \{x, y\}}$ and $F(\eta_x, \eta_y, \tilde{\eta}) = f(\eta)$. Then we have 
	\begin{multline*}
		\norm{\pi_x f}_{L^2} \\
		= \Ll(\int_{\X} \rho(F(1,1,\tilde{\eta}) - F(0,1,\tilde{\eta}))^2 + (1-\rho)(F(1,0,\tilde{\eta}) - F(0,0,\tilde{\eta}))^2\, \d \Pr(\tilde{\eta})\Rr)^{\frac{1}{2}}.
	\end{multline*}
	We apply the triangle inequality for this norm. The trick is that we only replace the term involving $\eta_x \neq \eta_y$. For example, in the terms $(F(1,1,\tilde{\eta}) - F(0,1,\tilde{\eta}))^2$, we add an intermediate term $F(1,0,\tilde{\eta})$. This follows exactly the spirit of the Kawasaki operator $\pi_{x,y}$ and we obtain
	\begin{align*}
		&\norm{\pi_x f}_{L^2} \\
		&\leq \Ll(\int_{\X} \rho(F(1,1,\tilde{\eta}) - F(1,0,\tilde{\eta}))^2 + (1-\rho)(F(0,1,\tilde{\eta}) - F(0,0,\tilde{\eta}))^2\, \d \Pr(\tilde{\eta})\Rr)^{\frac{1}{2}} \\
		& \qquad + \Ll(\int_{\X} \rho(F(0,1,\tilde{\eta}) - F(1,0,\tilde{\eta}))^2 + (1-\rho)(F(1,0,\tilde{\eta}) - F(0,1,\tilde{\eta}))^2\, \d \Pr(\tilde{\eta})\Rr)^{\frac{1}{2}}\\
		&= \Ll(\int_{\X} \rho(F(1,1,\tilde{\eta}) - F(1,0,\tilde{\eta}))^2 + (1-\rho)(F(0,1,\tilde{\eta}) - F(0,0,\tilde{\eta}))^2\, \d \Pr(\tilde{\eta})\Rr)^{\frac{1}{2}} \\
		& \qquad + \Ll(\int_{\X} (F(1,0,\tilde{\eta}) - F(0,1,\tilde{\eta}))^2\, \d \Pr(\tilde{\eta})\Rr)^{\frac{1}{2}}. 
	\end{align*}
	We notice the identity in the last equation
	\begin{align*}
		\norm{\pi_y f}_{L^2} &= \Ll(\int_{\X} \rho(F(1,1,\tilde{\eta}) - F(1,0,\tilde{\eta}))^2 + (1-\rho)(F(0,1,\tilde{\eta}) - F(0,0,\tilde{\eta}))^2\, \d \Pr(\tilde{\eta})\Rr)^{\frac{1}{2}},\\
		\norm{\pi_{x,y}f}_{L^2} &= \Ll(\int_{\X} 2\rho(1-\rho)(F(1,0,\tilde{\eta}) - F(0,1,\tilde{\eta}))^2 \, \d \Pr(\tilde{\eta})\Rr)^{\frac{1}{2}},
	\end{align*}
	then we conclude the desired result.
\end{proof}

\subsubsection{Spectral inequality}

The spectral inequality is an important tool to analyze Markov processes. In this part, we summarize several spectral inequalities from the literature.

The most used spectral inequality is the one for independent random variables known as Efron--Stein inequality. Here we state it and  its reverse version.

\begin{lemma}[Efron--Stein inequality]
	Let $(X_i)_{1 \leq i \leq n}$ be i.i.d. random variables on $E$. (Then they can be regarded as random variables on $E^n$.) For any function $f=f(X_1, X_2, \cdots, X_n)$, we denote by 
	\begin{align}\label{eq.EfronSteinE}
		\E_{(i)}[\cdot] := \int_{E} (\cdot) \, \d \P_{X_i}, \qquad \E_{(-i)}[\cdot] := \int_{E^{n-1}} (\cdot) \,  \prod_{1 \leq j \leq n, j \neq i}\d\P_{X_j},  
	\end{align}\label{eq.EfronStein}
	and $\var_{(i)}, \var_{(-i)}$ for the corresponding variances. Then we have
	\begin{align}
		\sum_{i=1}^n \var_{(i)}[\E_{(-i)}[f]] \leq	\var[f] \leq \sum_{i=1}^n \E_{(-i)}[\var_{(i)}[f]].
	\end{align}
\end{lemma}
\begin{proof}
	The upper bound is the classical Efron--Stein inequality, and one can find its proof in \cite[Theorem~3.1]{boucheron2013concentration}. The lower bound, which could be seen as a reverse Efron--Stein inequality, is less well-known, but follows exactly the same strategy of proof. The authors learn the lower bound at first in \cite{EFI2}.
\end{proof}

A direct corollary of Efron--Stein inequality is the spectral inequality of the Glauber operator \eqref{eq.Glauber} under product Bernoulli measure.
\begin{corollary}[Spectral inequality for Glauber dynamics]
	For any $\Lambda \subset \Zd$, we have
	\begin{align}\label{eq.spectralGlauber}
		\var_{\rho,\Lambda}[f] \leq \chi(\rho)\sum_{x \in \Lambda} \bracket{\Ll(\pi_x f\Rr)^2}_{\rho, \Lambda}.
	\end{align}
\end{corollary}
\begin{proof}
	We apply the classical Efron--Stein inequality and the upper bound is the {\rhs} of \eqref{eq.spectralGlauber}.
\end{proof}

With some more treatment, we can also obtain the spectral inequality for the Kawasaki operator \eqref{eq.Kawasaki} under product Bernoulli measure.
\begin{lemma}
	For any bounded set $\Lambda \subset \Zd$ and $f \in \F_0(\Lambda^-)$, we have
	\begin{align}\label{eq.spectralGradient}
		\var_{\rho}[f] \leq \diam(\Lambda)^2 \sum_{b \in \Lambda^*} \bracket{\Ll(\pi_b f\Rr)^2}_{\rho}.
	\end{align}
\end{lemma}
\begin{proof}
	We apply at first the spectral inequality for Glauber dynamics
	\begin{align}\label{eq.stepGlauber}
		\var_{\rho}[f] \leq \chi(\rho)\sum_{x \in \Lambda} \bracket{\Ll(\pi_x f\Rr)^2}_{\rho}.
	\end{align}
	We fix a direction in the canonical basis $e_i$, then for every $x \in \Lambda$, there exists a positive integer $\ell$ depending on $x$,
	\begin{align*}
		\ell(x) := \min \{k \in \N_+: x + k e_i \in \partial \Lambda\}.
	\end{align*}		
	Then we apply Lemma~\ref{lem.GlauberExchange} 
	\begin{align*}
		\norm{\pi_x f}_{L^2} \leq \norm{ \pi_{x + e_i} f }_{L^2} + \frac{1}{\sqrt{2 \chi(\rho)}} \norm{\pi_{x,x+e_i}f}_{L^2}.
	\end{align*}
	We sum this inequality along the path $x \to x+e_i \to x+2 e_i \cdots \to x + \ell(x) e_i$
	\begin{align*}
		\norm{\pi_x f}_{L^2} \leq \norm{\pi_{x + \ell(x) e_i}f}_{L^2} + \frac{1}{\sqrt{2 \chi(\rho)}} \sum_{j=1}^{\ell(x)} \norm{\pi_{x+(j-1)e_i,x+ j e_i}f}_{L^2}.
	\end{align*}
	Notice that $x + \ell(x) e_i \in \partial \Lambda$ and $f \in \F_0(\Lambda^-)$, so $f$ does not depend on $\eta_{x + \ell(x) e_i}$ and $\pi_{x + \ell(x) e_i}f = 0$. This implies 
	\begin{align}\label{eq.GlauberPath}
		\norm{\pi_x f}_{L^2} \leq \frac{1}{\sqrt{2 \chi(\rho)}} \sum_{j=1}^{\ell(x)} \norm{\pi_{x+(j-1)e_i,x+ j e_i}f}_{L^2}.
	\end{align}
	We put \eqref{eq.GlauberPath} back to \eqref{eq.stepGlauber} and apply Cauchy--Schwarz inequality
	\begin{align*}
		\var_{\rho}[f] &\leq \frac{1}{2}\sum_{x \in \Lambda} \ell(x) \sum_{j=1}^{\ell(x)} \bracket{(\pi_{x+(j-1)e_i,x+ j e_i}f)^2}_\rho \\
		&\leq \frac{1}{2} \diam(\Lambda)\sum_{x \in \Lambda} \sum_{j=1}^{\ell(x)} \bracket{(\pi_{x+(j-1)e_i,x+ j e_i}f)^2}_\rho.
	\end{align*}
	Here the factor $\chi(\rho)$ in \eqref{eq.GlauberPath} and \eqref{eq.stepGlauber} compensates, and we also make use of the fact $\ell(x) \leq \diam(\Lambda)$. We now exchange the order of the sum $\sum_{x \in \Lambda} \sum_{j=1}^{\ell(x)}$
	\begin{align*}
		\var_{\rho}[f] \leq \frac{1}{2} \diam(\Lambda) \sum_{b \in \Lambda^*} \sum_{x \in \Lambda: \exists j \in \N_+, x + j e_i \in b} \bracket{(\pi_{b}f)^2}_\rho.
	\end{align*}
	Because every bond $b$ can be counted at most $\diam(\Lambda)$ times along the direction $e_i$, we obtain the desired result.
\end{proof}

\bigskip

In \cite[Theorem 1]{LuYau}, Lu and Yau proved a generalized version of the spectral inequality for the Glauber dynamics. We do not need that one in this paper, but we will make use of \cite[Theorem 2]{LuYau}, the spectral inequality for Kawasaki dynamics under canonical ensemble. This sharp spectral gap estimate is also one of the key estimates in Varahdan's method.

\begin{lemma}[Theorem 2, \cite{LuYau}]
	There exists a positive constant $C=C(d)$, such that for any $L \in \N_+$ and any $N \in \N_+, N \leq \vert \Lambda_L\vert$, we have 
	\begin{align}\label{eq.Poingrand}
		\var_{\Lambda_L, N}[f]\leq C L^2 \sum_{b \in (\Lambda_L)^*}\bracket{(\pi_{b} f)^2}_{\Lambda_L, N}.
	\end{align}  
\end{lemma}

\subsubsection{Modified Caccioppoli  inequality}
In PDE setting, for the elliptic equation 
\begin{align*}
	-\nabla \cdot \aa(x) \nabla v = 0, \qquad \text{ in } B_r,
\end{align*}
with uniform ellipticity $\vert \xi \vert^2 \leq \xi \cdot \aa \xi \leq \lambda \vert \xi \vert^2$, then $v$ satisfies the classical Caccioppoli inequality
\begin{align*}
	\fint_{B_{r/2}} \vert \nabla v\vert^2 \leq \frac{C(\lambda)}{r^2} \fint_{B_r} v^2. 
\end{align*}
This inequality is also one key step to establish the quantitative homogenization, but was missing in the literature for infinite particle systems. A modified Caccioppoli inequality is proved in \cite[Proposition 3.9]{bulk}, which differs from the classical one by an extra term on the {\rhs}; see \eqref{eq.Caccioppoli}. Here we present its version in Kawasaki dynamics. The conditional expectation operator will be used in the following paragraphs. Recall $\fil_\Lambda$ defined in Section~\ref{subsubsec.proba}. For $\Lambda \subset \Zd$ and $f \in L^1$, we define 
\begin{align}\label{eq.defA}
	\A_{\Lambda} f := \Er \Ll[f \, \vert \fil_\Lambda\Rr].
\end{align}
Concretely, using the restriction operator defined in \eqref{eq.Restriction}, it is calculated as 
\begin{align*}
	\A_{\Lambda} f (\eta) = \int_{\X} f(\eta \mres \Lambda + \eta' \mres \Lambda^c) \, \d \Pr(\eta').
\end{align*}
We usually denote by $\A_{L} f \equiv \A_{\Lambda_L} f$ for short.

\begin{proposition}[Modified Caccioppoli inequality]\label{prop.Caccioppoli}
	There exists ${\theta(d, \lambda) \in (0,1)}$, finite positive constants $C(d, \lambda)$, and $R_0(d, \lambda, \r)$ such that for every $L \ge R_0$ and ${u \in \mcl A(\Lambda_{3L})}$ (defined in \eqref{eq.defHarmonic}), we have
	\begin{align}
		\label{eq.Caccioppoli}
		\frac{1}{\vert \Lambda_L \vert} \bracket{\A_{L + 2\r} u (-\L_{\Lambda_L} \A_{L + 2\r} u)}_\rho \leq  \frac{C L^{-2}}{ \vert \Lambda_{3L} \vert} \bracket{u^2}_\rho  + \frac{\theta}{ \vert \Lambda_{3L} \vert} \bracket{u (-\L_{\Lambda_{3L}} u)}_\rho  .
	\end{align}
\end{proposition}

Its proof is similar to \cite[Proposition 3.9]{bulk}, which can be summarized in the following three steps.
\begin{enumerate}
	\item Test the harmonic function ${u \in \mcl A(\Lambda_{3L})}$ with its cutoff version $\A_L u$.
	\item Obtain the $L^2$-term using the quadratic variation of the martingale $(\A_n u)_{n \in \N_+}$.
	\item Bootstrap the result from its weak version with a correct normalization volume factor.
\end{enumerate} 
The first and second step can be seen as the main difference between particle system and PDE setting, where we make the cutoff in order to reduce the influence of particles from the boundary, and we also need the nice $L^2$-isometry of martingale to recover the $L^2$ term of $u$.

In Kawasaki dynamics, the quadratic variation structure is less obvious compared to the setting of continuous configuration space in \cite{bulk}. Let us make some explicit calculation at first. It is clear that 
\begin{equation}\label{eq.AdInOut}
	\begin{split}
		\forall b \in (\Lambda_n)^*, \qquad &\pi_b \A_n f = \A_n \pi_b f, \\
		\forall b \in (\Lambda_n^c)^*, \qquad &\pi_b \A_n f = 0.
	\end{split}
\end{equation}
That is to say, the operators $\A_n$ and $\pi_b$ are commutative when the bond $b$ stays in $(\Lambda_n)^*$, and the influence is $0$ when $b$ is outside $\Lambda_n$. When $b \in (\Lambda_n, \Lambda_n^c)^*$, the situation is subtle as we will see the perturbation 
\begin{align*}
	\forall x \in \partial \Lambda_n,  y \notin \Lambda_n, y \sim x, \qquad \pi_{x,y} \A_n f(\eta) = \A_n f(\eta^{x,y}) - \A_n f(\eta). 
\end{align*}
Since we apply the conditional expectation, the information of $\eta_x$ is no longer accessible in $\A_n f(\eta^{x,y})$ and we have  
\begin{align*}
	\A_n f(\eta^{x,y}) = \A_n f(\eta \mres (\Lambda_n \setminus \{x\}) + \eta_y \delta_x). 
\end{align*}
Thus, near the boundary the Kawasaki operator is like resampling the state at $x$ and we have 
\begin{align}\label{eq.AdBC}
	\forall x \in \partial \Lambda_n,  y \notin \Lambda_n, y \sim x, \qquad \bracket{(\pi_{x,y} \A_n f)^2 }_\rho = 2\chi(\rho) \bracket{\Ll(\pi_x \A_n f \Rr)^2}_\rho.
\end{align}
By the spectral inequality \eqref{eq.spectralGlauber}, we know that 
\begin{align*}
	\chi(\rho)\sum_{x \in \partial \Lambda_n} \bracket{\Ll(\pi_x  \A_n f \Rr)^2}_\rho \geq \bracket{(\A_n f - \A_{n-2} f)^2}_{\rho}.
\end{align*}
The inequality is not on the desired direction, because we hope to give an upper bound for the boundary perturbation. For this reason, we would like to study how to control this Glauber derivative near the boundary at first.

\begin{lemma}\label{lem.variationBC}
	For $\A_n \equiv \A_{\Lambda_n} $ defined in \eqref{eq.defA} and $f \in L^2$, the following estimate holds 
	\begin{align}
		\sum_{b \in (\Lambda_n, \Lambda_n^c)^*} \bracket{(\pi_{b} \A_n f)^2}_\rho \leq 4 \Ll(\sum_{b \in (\Lambda_n, \Lambda_n^c)^*} \bracket{( \A_{n+2} \pi_{b} f)^2}_\rho +  \bracket{(\A_{n+2} f - \A_n f)^2}_\rho\Rr).
	\end{align}
\end{lemma}
\begin{proof}
	The {\lhs} can be expressed with \eqref{eq.AdBC}
	\begin{equation}\label{eq.variationStep0}
		\begin{split}
			\sum_{b \in (\Lambda_n, \Lambda_n^c)^*} \bracket{(\pi_{b} \A_n f)^2}_\rho &= \sum_{x \in \partial \Lambda_n,  y \notin \Lambda_n, y \sim x} \bracket{(\pi_{x,y} \A_n f)^2}_\rho \\
			&= \sum_{x \in \partial \Lambda_n,  y \notin \Lambda_n, y \sim x} 2\chi(\rho) \bracket{\Ll(\pi_x \A_n f\Rr)^2}_\rho .
		\end{split}
	\end{equation}
	We notice that 
	\begin{equation}\label{eq.variationStep1}
		\begin{split}
			\pi_x \A_n f(\eta) &= \A_n f(\eta^x) -  \A_n f(\eta) \\
			&= \int_{\{0,1\}} \Ll(\A_{\Lambda_n \sqcup \{y\}} f(\eta^x) -  \A_{\Lambda_n \sqcup \{y\}}  f(\eta)\Rr) \, \d \Pr(\eta_y)\\
			&= \int_{\{0,1\}} \pi_x \A_{\Lambda_n \sqcup \{y\}} f  \, \d \Pr(\eta_y).
		\end{split}
	\end{equation}
	Using Jensen's inequality, we have 
	\begin{align}\label{eq.addOne}
		\bracket{\Ll(\pi_x \A_n f\Rr)^2}_\rho \leq \bracket{\Ll(\pi_x \A_{\Lambda_n \sqcup \{y\}} f \Rr)^2}_\rho
	\end{align}
	Then we apply Lemma~\ref{lem.GlauberExchange} to $\A_{\Lambda_n \sqcup \{y\}} f$ and obtain that 
	\begin{align}\label{eq.variationStep2}
		\norm{\pi_x \A_{\Lambda_n \sqcup \{y\}} f}_{L^2} \leq  \norm{\pi_y \A_{\Lambda_n \sqcup \{y\}} f}_{L^2} + \frac{1}{\sqrt{2 \chi(\rho)}} \norm{\pi_{x,y}\A_{\Lambda_n \sqcup \{y\}} f}_{L^2}.
	\end{align}
	We put \eqref{eq.variationStep2} and \eqref{eq.addOne} back to \eqref{eq.variationStep0}, and obtain that 
	\begin{multline}\label{eq.variationStep2Plus}
		\sum_{b \in (\Lambda_n, \Lambda_n^c)^*} \bracket{(\pi_{b} \A_n f)^2}_\rho \\
		\leq \sum_{x \in \partial \Lambda_n,  y \notin \Lambda_n, y \sim x} 4\chi(\rho) \Ll(\bracket{\Ll(\pi_y \A_{\Lambda_n \sqcup \{y\}} f\Rr)^2}_\rho + \frac{1}{2 \chi(\rho)}\bracket{\Ll(\pi_{x,y}\A_{\Lambda_n \sqcup \{y\}} f\Rr)^2}_\rho \Rr).
	\end{multline}
	For the first term on {\rhs}, it is exactly the fluctuation on $\partial \Lambda_{n+2}$, so we apply the reverse Efron--Stein inequality (the first inequality of \eqref{eq.EfronStein}) to $\A_{n+2} f$ under the expectation over the $\{\eta_y\}_{y \in \partial \Lambda_{n+2}}$
	\begin{align*}
		\sum_{y \in \partial \Lambda_{n+2}} \chi(\rho) \bracket{\Ll(\pi_y \A_{\Lambda_n \sqcup \{y\}} f\Rr)^2}_{\rho, \partial \Lambda_{n+2}} &= \sum_{y \in \partial \Lambda_{n+2}} \var_{\rho, \partial \Lambda_{n+2}}[ \E_{\rho, \partial \Lambda_{n+2}}[\A_{n+2} f \,\vert \eta_y]] \\
		&\leq \var_{\rho, \partial \Lambda_{n+2}}[\A_{n+2} f ]\\
		&=\bracket{(\A_{n+2} f - \A_{n} f)^2}_{\rho, \partial \Lambda_{n+2}}.
	\end{align*} 
	Recall that $\bracket{\cdot}_{\rho, \partial \Lambda_{n+2}}$ is defined in Section~\ref{subsubsec.proba}, and note that $\chi(\rho)$ also appears similarly in \eqref{eq.spectralGlauber} in a reversed inequality. We also use the identity $\E_{\rho, \partial \Lambda_{n+2}}[\A_{n+2} f \,\vert \eta_y] = \A_{\Lambda_n \sqcup \{y\}} f$ here. Then we take the expectation of other variables to yield the estimate of the first term on the \rhs of \eqref{eq.variationStep2Plus}
	\begin{align*}
		\sum_{x \in \partial \Lambda_n,  y \notin \Lambda_n, y \sim x} 4\chi(\rho) \bracket{\Ll(\pi_y \A_{\Lambda_n \sqcup \{y\}} f\Rr)^2}_\rho  \leq 4\bracket{(\A_{n+2} f - \A_{n} f)^2}_\rho.
	\end{align*}
	For the second term on the \rhs of \eqref{eq.variationStep2Plus}, we apply \eqref{eq.AdInOut} and once again Jensen's inequality to obtain that 
	\begin{align*}
		\sum_{x \in \partial \Lambda_n,  y \notin \Lambda_n, y \sim x} \bracket{\Ll(\pi_{x,y}\A_{\Lambda_n \sqcup \{y\}} f\Rr)^2}_\rho &= \sum_{x \in \partial \Lambda_n,  y \notin \Lambda_n, y \sim x} \bracket{\Ll(\A_{\Lambda_n \sqcup \{y\}} \pi_{x,y} f\Rr)^2}_\rho\\ &\leq \sum_{b \in (\Lambda_n, \Lambda_n^c)^*} \bracket{( \A_{n+2} \pi_{b} f)^2}_\rho.
	\end{align*}
	This concludes the desired result.
\end{proof}
\begin{remark}
	The proof above is inspired by the example in \cite[Proposition~3.9]{bulk} and \cite[Lemma~6.1, eq.(B.2)]{jlqy}, while similar arguments also appeared before in the intermediate steps in Varadhan's characterization of closed form (see \cite[Lemma~4.15, Appendix~A.4]{kipnis1998scaling}).  
\end{remark}

Once we develop Lemma~\ref{lem.variationBC}, the rest of the proof follows that in \cite[Proposition 3.9]{bulk}.
\begin{proof}[Proof of Proposition~\ref{prop.Caccioppoli}] The proof can be divided into three steps.
	
	\textit{Step 1: construction of test function.} In the first step, we do some preparation. Our object is to regularize $u$ such that it becomes a function in $\mcl F_0(\Lambda_{3L}^-)$. A very natural idea is to apply the conditional expectation operator \eqref{eq.defA}, then the information outside $3L$ will be averaged. In order to make this cutoff more smooth, we propose the following regularized version  
	\begin{align}\label{eq.defAs}
		\A_{s,\ell} f := \frac{1}{\ell} \int_{0}^{\ell} \A_{s + t} f \, \d t.
	\end{align}
	Recall $\Lambda_L$ is defined for $L \in \R_+$, so $\A_{s+t}$ does for $s,t \in \R_+$. Its Kawasaki derivative can be calculated using \eqref{eq.AdInOut}
	\begin{align}
		\label{eq.ADerivative}
		\pi_b \A_{s,\ell} f  = \left\{ 
		\begin{array}{ll}
			\A_{s,\ell} \pi_b f & \text{if } b \in (\Lambda_s)^*;\\
			\frac{1}{\ell} \int_{\tau(b)+}^{s+\ell} \A_{t} \pi_b f \, \d t + \frac{(\tau_b - s)\wedge 1}{\ell} \pi_b \A_{\tau_b} f \mathbf{1}_{b \in \Ll(\Lambda_{\tau(b)}, \Lambda_{\tau(b)}^c\Rr)^*}  & \text{if } b \in \Ll(\Lambda_s, \Lambda_{s+\ell}^c\Rr)^*;\\
			0 & \text{if } b \in (\Lambda_{s+\ell}^c)^*.
		\end{array} \right. 
	\end{align}
	Here the notation $\tau(b)$ is defined as 
	\begin{align}\label{eq.stopTime}
		\tau(b) := \inf \{s \in \R_+ : b \in (\Lambda_s)^*\}.
	\end{align}
	From the definition of hypercube in Section~\ref{subsubsec.Geometry}, we know $b \in \Lambda_{\tau(b)+}$ but $b \notin \Lambda_{\tau(b)}$.
	
	We will also make use of the following operator 
	\begin{align}\label{eq.Atilde}
		\tilde{\A}_{s,\ell} f := (\A_{s,\ell} \circ  \A_{s,\ell} )(f) =  \frac{2}{\ell^2} \int_{0}^{\ell} (\ell - t) \A_{s+ t} f \, \d t,
	\end{align}
	The motivation comes from the following identity that 
	\begin{align}\label{eq.AtildeIdentity}
		\bracket{(\A_{s,\ell} f)^2}_{\rho} = \bracket{f (\tilde{\A}_{s,\ell} f)}_\rho = \frac{2}{\ell^2} \int_{0}^{\ell} (\ell - t) \bracket{(\A_{s+ t} f)^2}_\rho \, \d t.
	\end{align}
	Similar to \eqref{eq.ADerivative}, we also calculate its Kawasaki derivative as preparation
	\begin{multline}
		\label{eq.AtildeDerivative}
		\pi_b \tilde{\A}_{s,\ell}f =
		\\
		\left\{ \begin{array}{ll}
			\A_{s,\ell} \pi_b f & \text{if } b \in (\Lambda_s)^*;\\
			\frac{2}{\ell^2}  \int_{\tau(b)+}^{s+\ell} (s + \ell - t) \A_{t} \pi_b f \, \d t &\\
			\qquad \qquad + \frac{2}{\ell^2}  \int_{(\tau(b)-2) \vee s}^{\tau(b)} (s + \ell - t) \pi_b \A_{\tau(b)} f \mathbf{1}_{b \in \Ll(\Lambda_{\tau(b)}, \Lambda_{\tau(b)}^c\Rr)^*} & \text{if } b \in \Ll(\Lambda_s, \Lambda_{s+\ell}^c\Rr)^*;\\
			0 & \text{if } b \in (\Lambda_{s+\ell}^c)^*.\end{array} \right. 
	\end{multline}

	\textit{Step 2: weak Caccioppoli inequality.} We then prove the weak Caccioppoli inequality at first. Fix $\theta'(\lambda) := \frac{10\lambda}{1 + 10\lambda} \in (0,1)$; recall that $\lambda \geq 1$ is the constant in Hypothesis~\ref{hyp}. For every $L > 0$, $s \geq L+2\r, \ell > 1, s + \ell < 3L$ and ${u \in \mcl A(\Lambda_{s+\ell+2})}$, we claim that 
	\begin{multline}\label{eq.weakCaccioppoli}
		\ell^{-2}\bracket{(\A_{s} u)^2}_\rho + \bracket{\A_{s, \ell} u (-\L_{\Lambda_L} \A_{s, \ell} u) }_\rho \\ \leq \theta' \Ll(\ell^{-2}\bracket{(\A_{s+\ell}u)^2}_\rho + \bracket{u (-\L_{\Lambda_{s+\ell}} u) }_\rho\Rr).
	\end{multline} 
	
	The main idea is to use the conditional expectation $ \tilde{\A}_{s,\ell} u$ given in \eqref{eq.Atilde}, because it provides a cutoff that $\tilde{\A}_{s,\ell} u \in \F_0 (\Lambda_{s+\ell+2}^{-})$. Then we test it with $u$
	\begin{align*}
		0 = \bracket{\tilde{\A}_{s,\ell} u(-\L_{s+\ell+2} u)}_\rho = \sum_{b \in (\Lambda_{s+\ell+2})^*} \bracket{c_b (\pi_b \tilde{\A}_{s,\ell} u)(\pi_b u)}_{\rho} = \mathbf{I} + \mathbf{II} + \mathbf{III}.
	\end{align*}
	Here decompose the {\rhs} into the sum of three terms
	\begin{equation}\label{eq.WeakDecom1}
		\begin{split}
			\mathbf{I} &:= \sum_{b \in (\Lambda_{s - 2\r})^*} \bracket{c_b (\pi_b \tilde{\A}_{s,\ell} u)(\pi_b u)}_{\rho},  \\
			\mathbf{II} &:= \sum_{b \in (\Lambda_s)^* \setminus (\Lambda_{s - 2\r})^*}  \bracket{c_b (\pi_b \tilde{\A}_{s,\ell} u)(\pi_b u)}_{\rho}, \\
			\mathbf{III} &:= \sum_{b \in (\Lambda_{s+\ell})^* \setminus (\Lambda_{s})^*} \bracket{c_b (\pi_b \tilde{\A}_{s,\ell} u)(\pi_b u)}_{\rho}, 
		\end{split}
	\end{equation}
	and we have the estimate 
	\begin{align}\label{eq.WeakDecom2}
		\vert \mathbf{I} \vert \leq \vert \mathbf{II} \vert  + \vert \mathbf{III} \vert. 
	\end{align}
	
	The term $ \mathbf{I} $ is the easiest one to treat
	\begin{equation}\label{eq.WeakI}
		\begin{split}
			\mathbf{I} &= \sum_{b \in (\Lambda_{s - 2\r})^*} \bracket{c_b (\pi_b \tilde{\A}_{s,\ell} u)(\pi_b u)}_{\rho}\\
			&= \sum_{b \in (\Lambda_{s - 2\r})^*} \bracket{(\pi_b \A_{s,\ell} u) \A_{s,\ell}(c_b \pi_b u)}_{\rho}\\
			&= \sum_{b \in (\Lambda_{s - 2\r})^*} \bracket{c_b (\pi_b \A_{s,\ell} u)^2}_{\rho}.
		\end{split}
	\end{equation}
	From the first line to the second line, we use the fact $\tilde{\A}_{s,\ell} = \A_{s,\ell} \circ \A_{s,\ell}$ and the reversibility of $\A_{s,\ell}$. From the second line to the third line, we use the fact $c_b$ is $\fil_{\Lambda_s}$-measurable when $b \in  (\Lambda_{s - 2\r})^*$ and $ \A_{s,\ell} \pi_b =  \pi_b \A_{s,\ell}$ from \eqref{eq.AtildeDerivative}.
	
	The identity \eqref{eq.WeakI} does not apply directly to $ \mathbf{II} $, because for $b \in (\Lambda_s)^* \setminus (\Lambda_{s - 2\r})^*$, the jump rate $c_b$ is no longer $\fil_{\Lambda_s}$-measurable. Therefore, we make use of the exact expression \eqref{eq.Atilde}
	\begin{equation}\label{eq.WeakII}
		\begin{split}
			\vert \mathbf{II}\vert &=  \sum_{b \in (\Lambda_s)^* \setminus (\Lambda_{s - 2\r})^*}\Ll\vert \frac{2}{\ell^2} \int_{0}^{\ell} (\ell - t) \bracket{c_b (\pi_b \A_{s+ t} u) (\pi_b u) }_\rho \, \d t \Rr\vert \\
			&\leq  \sum_{b \in (\Lambda_s)^* \setminus (\Lambda_{s - 2\r})^*}\frac{\lambda}{\ell^2} \int_{0}^{\ell} (\ell - t) \bracket{(\pi_b \A_{s+ t} u)^2 +  (\pi_b u)^2 }_\rho \, \d t  \\
			&\leq  \sum_{b \in (\Lambda_s)^* \setminus (\Lambda_{s - 2\r})^*}\frac{2\lambda}{\ell^2} \int_{0}^{\ell} (\ell - t) \bracket{(\pi_b u)^2 }_\rho \, \d t \\
			& = \sum_{b \in (\Lambda_s)^* \setminus (\Lambda_{s - 2\r})^*}  \lambda  \bracket{(\pi_b u)^2}_{\rho}.
		\end{split}
	\end{equation}
	From the first line to the second line, we make use of Young's inequality and $c_b \leq \lambda$. From the second line to the third line,  $ \A_{s,\ell} \pi_b =  \pi_b \A_{s,\ell}$ and $\bracket{( \A_{s+ t} \pi_b u)^2 }_\rho \leq \bracket{ (\pi_b u)^2 }_\rho$ is also applied thanks to Jensen's inequality.
	
	The term  $ \mathbf{III} $ has two integrals following \eqref{eq.AtildeDerivative}, which can be noted respectively by $ \mathbf{III.1} $ and $ \mathbf{III.2} $. The first part is similar to \eqref{eq.WeakII}
	\begin{multline}\label{eq.WeakIII1}
		\begin{split}
			\vert  \mathbf{III.1}  \vert &\leq \sum_{b \in (\Lambda_{s+\ell})^* \setminus (\Lambda_{s})^*}\Ll\vert \frac{2}{\ell^2}  \int_{\tau(b)+}^{s+\ell} (s + \ell - t) \bracket{c_b (\A_{t} \pi_b u) (\pi_b u)}_\rho \, \d t \Rr\vert \\
			&\leq \sum_{b \in (\Lambda_{s+\ell})^* \setminus (\Lambda_{s})^*} \lambda  \bracket{(\pi_b u)^2}_{\rho}.
		\end{split}
	\end{multline}
	The second part is the key to make appear the $L^2$ term
	\begin{equation*}
		\begin{split}
			\vert \mathbf{III.2}\vert &\leq	\sum_{b \in (\Lambda_{s+\ell})^* \setminus (\Lambda_{s})^*}\Ll\vert \frac{2}{\ell^2}   \int_{(\tau(b)-2) \vee s}^{\tau(b)} (s + \ell - t) \bracket{c_b(\pi_b \A_{\tau(b)} u)(\pi_b u)}_\rho  \mathbf{1}_{b \in \Ll(\Lambda_{\tau(b)}, \Lambda_{\tau(b)}^c\Rr)^*}  \, \d t \Rr\vert \\
			&\leq \sum_{b \in (\Lambda_{s+\ell})^* \setminus (\Lambda_{s})^* } \frac{2 \lambda}{\ell}    \bracket{\gamma^{-1}(\pi_b \A_{\tau(b)} u)^2 + \gamma(\pi_b u)^2}_\rho  \mathbf{1}_{b \in \Ll(\Lambda_{\tau(b)}, \Lambda_{\tau(b)}^c\Rr)^*}.  
		\end{split}
	\end{equation*}
	From the first line to the second line, Young's inequality is applied with $\gamma > 0$ to be fixed, together with the trivial bound $(s + \ell - t) \leq \ell$ for $t$ defied above. We rearrange the sum, and obtain 
	\begin{equation}\label{eq.WeakIII2}
		\begin{split}
			\vert \mathbf{III.2}\vert &\leq \sum_{n=0}^{\lfloor \ell \rfloor - 2} \sum_{b \in  \Ll(\Lambda_{n}, \Lambda_{n}^c\Rr)^*} \frac{2 \lambda}{\ell}    \bracket{\gamma^{-1}(\pi_b \A_{n} u)^2 + \gamma(\pi_b u)^2}_\rho \\
			&\leq  \sum_{n=0}^{\lfloor \ell \rfloor - 2} \Ll(\frac{8 \lambda}{\gamma \ell} \bracket{(\A_{s+n+2} u - \A_{s+n} u)^2}_\rho + \sum_{b \in (\Lambda_n, \Lambda_n^c)^*} \Ll(\frac{8 \lambda}{\gamma \ell} + \frac{2 \gamma \lambda}{\ell} \Rr)\bracket{(\pi_{b} u)^2}_\rho \Rr)\\
			&=  \frac{8 \lambda}{\gamma \ell} \bracket{(\A_{s+\ell} u - \A_{s} u)^2}_\rho + \sum_{n=0}^{\lfloor \ell \rfloor - 2} \sum_{b \in (\Lambda_n, \Lambda_n^c)^*} \Ll(\frac{8 \lambda}{\gamma \ell} + \frac{2 \gamma \lambda}{\ell} \Rr)\bracket{(\pi_{b} u)^2}_\rho.
		\end{split}
	\end{equation}
	Here we insert the estimate from Lemma~\ref{lem.variationBC} in the second line to handle the perturbation of boundary term, and then make use of the orthogonal decomposition of martingale from the second line to the third line. We choose $\gamma = \ell$, and put \eqref{eq.WeakI}, \eqref{eq.WeakII}, \eqref{eq.WeakIII1} and \eqref{eq.WeakIII2} back to \eqref{eq.WeakDecom2}, which concludes that 
	\begin{multline}\label{eq.WeakToFill}
		\sum_{b \in (\Lambda_{s - 2\r})^*} \bracket{c_b (\pi_b \A_{s,\ell} u)^2}_{\rho} \\
		\leq 10\lambda \Ll(\ell^{-2}\bracket{(\A_{s+\ell} u - \A_{s} u)^2}_\rho +  \sum_{b \in (\Lambda_{s+\ell})^* \setminus (\Lambda_{s-2\r})^* } \bracket{c_b(\pi_{b} u)^2}_\rho\Rr).
	\end{multline}
	Then Widman's filling-hole argument applies by adding $10\lambda \sum_{b \in (\Lambda_{s - 2\r})^*} \bracket{c_b (\pi_b \A_{s,\ell} u)^2}_{\rho}$ on the two sides (Jensen's inequality is also applied to the {\rhs})
	\begin{multline}\label{eq.WeakToFill2}
		(1+10\lambda)\sum_{b \in (\Lambda_{s - 2\r})^*} \bracket{c_b (\pi_b \A_{s,\ell} u)^2}_{\rho} \\
		\leq 10\lambda \Ll(\ell^{-2}\bracket{(\A_{s+\ell} u - \A_{s} u)^2}_\rho +  \sum_{b \in (\Lambda_{s+\ell})^*} \bracket{c_b(\pi_{b} u)^2}_\rho\Rr).
	\end{multline}
	We note the martingale property of $(\A_s u)_{s \geq 0}$ and divide $(1+10\lambda)$ on the two sides to obtain \eqref{eq.weakCaccioppoli} with $\theta' = \frac{10\lambda}{1+10\lambda}$.
	
	\textit{Step 3: bootstrap.} When we take $L$ large enough, $\ell = L$, and $s = 2L$ in  \eqref{eq.weakCaccioppoli}, we obtain 
	\begin{align}\label{eq.sBad}
		\frac{1}{\vert \Lambda_L \vert}\bracket{\A_{2L} u (-\L_{\Lambda_L} \A_{2L} u) }_\rho \leq 3^d \theta' \Ll(\frac{L^{-2}}{\vert \Lambda_{3L} \vert}\bracket{(\A_{3L}u)^2}_\rho + \frac{1}{\vert \Lambda_{3L} \vert}\bracket{u (-\L_{\Lambda_{3L}} u) }_\rho\Rr).
	\end{align}
	The factor before the term $\bracket{(\A_{3L}u)^2}_\rho$ is correct, but an additional volume factor is added before $\bracket{u (-\L_{\Lambda_{3L}} u) }_\rho$ compared to the desired result \eqref{eq.Caccioppoli}. This is not good as we do not necessarily have  $3^d \theta' < 1$. However, if we choose carefully $s = (1 + \delta)L$ and $\ell = \delta L$, we obtain  
	\begin{multline}\label{eq.sGood}
		\frac{1}{\vert \Lambda_L \vert}\bracket{\A_{L} u (-\L_{\Lambda_L} \A_{L} u) }_\rho \\
		\leq (1+2\delta)^d \theta' \Ll(\frac{(\delta L)^{-2}}{\vert \Lambda_{(1+2\delta)L} \vert}\bracket{(\A_{(1+2\delta)L}u)^2}_\rho + \frac{1}{\vert \Lambda_{(1+2\delta)L} \vert}\bracket{u (-\L_{\Lambda_{(1+2\delta)L}} u) }_\rho\Rr).
	\end{multline}
	We can choose $\delta$ small such that $(1+2\delta)^d \theta' < 1$, and then iterate the Dirichlet energy term on the {\rhs}, such that the domain increases progressively to $\Lambda_{3L}$. See \cite[eq.(56)-(60)]{bulk} for details, since this step is an algebraic iteration and independent of model. 
\end{proof}

\subsection{Analytic tools on independent particles}\label{subsec.free}
In this part, we develop the analytic tools on independent particles. We denote by $\XX := \N^{\Zd}$ for the configuration space of independent particles on $\Zd$, which does not have constraint for the number of particles on every site. For $\teta \in \tilde \X$ such that $\teta_x \geq 1$, we define the jump operator 
\begin{align*}
	\teta^{x,y} := \teta - \delta_x + \delta_y,
\end{align*}
and 
\begin{align*}
	(\tilde \pi_{x,y} \tilde u)(\teta) := \tilde u (\teta^{x,y}) - \tilde u(\teta)
\end{align*}
for any function $\tilde u$ on $\tilde \X$.

The generator of the independent particles is 
\begin{align}\label{eq.defIRW}
	(\tilde \L \tilde u)(\teta) := \sum_{x \in \Zd} \teta_x \sum_{y \sim x} \tilde \pi_{x,y} \tilde u.
\end{align}
This generates a dynamic that every particle jumps independently with rate $1$ to the nearest neighbor site.

Fix an $\alpha > 0$, we denote by $\poi(\alpha)$ the Poisson distribution on $\N$ with mean $\alpha$ and $\PP_{\alpha} := \poi(\alpha)^{\otimes \Zd}$ the probability measure on $\XX$ which is stationary with respect to $\tilde \L$ in \eqref{eq.defIRW},  and $\bbracket{\cdot}_{\alpha}$ its associated expectation. We usually write $\teta \in \XX$ as a canonical random variable sampled under $\PP_{\alpha}$. We also denote by $\tilde \P_{\Lambda, N}$ and $\bbracket{\cdot}_{\Lambda, N}$ for the probability and expectation under the canonical ensemble, i.e. $N$ particles distributed independently and uniformly in $\Lambda$.

\subsubsection{Mecke's identity}
The first lemma is about Mecke's identity, which simplifies some expectation under $\bbracket{\cdot}_{\alpha}$ by adding one additional particle. This identity is inspired from the reference \cite[Theorem~4.1]{bookPoisson}.
\begin{lemma}[Mecke's identity]\label{lem.Mecke1}
	For any $\alpha > 0$, any $x\in\Zd$ and any measurable function $F : \tilde \X\to \R$. If $\tilde \eta_x F$ is integrable under $\P_{\alpha}$, then the following identity holds
	\begin{align}\label{eq.Mecke}
		\bbracket{\teta_x F(\teta)}_{\alpha} = \alpha \bbracket{F(\teta+\delta_x)}_{\alpha}.
	\end{align}
\end{lemma}
\begin{proof}
	We make the calculation directly
	\begin{align*}
		\bbracket{\teta_x F(\teta)}_{\alpha} &= e^{-\alpha} \sum_{k=0}^\infty \frac{\alpha^k}{k!} \bbracket{k  F(\teta) \, \vert \teta_x = k}_{\alpha}\\
		&= \alpha e^{-\alpha}\sum_{k=1}^\infty \frac{\alpha^{(k-1)}}{(k-1)!} \bbracket{ F(\teta) \, \vert \teta_x = k}_{\alpha}\\
		&= \alpha e^{-\alpha}\sum_{k=1}^\infty \frac{\alpha^{(k-1)}}{(k-1)!} \bbracket{ F(\teta+\delta_x) \, \vert \teta_x = k-1}_{\alpha} \\
		&= \alpha \bbracket{F(\teta+\delta_x)}_{\alpha}.
	\end{align*}
\end{proof}

The calculation over the independent particle system has a close connection with the finite difference operator on lattice $\Zd$ or $\Td_L$. Given a function $\tilde u: \XX \to \R$ which only depends on the configuration in $\Lambda$, and using the expression $\teta \mres \Lambda = \sum_{i=1}^N \delta_{x_i}$, then we have the following canonical projection $u_N:\Lambda^N\rightarrow\R$ written as
\begin{align}\label{eq.IRWHighDim}
	\tilde u_N(x_1, x_2, \cdots, x_N) := \tilde u(\teta),
\end{align}
where $\teta \mres \Lambda = \sum_{i=1}^N \delta_{x_i}$.

Moreover, $(\tilde u_N)_{N\in\N_+}$ are glued together as a function $\tilde u$  on $\N^\Lambda$ if and only if $\tilde u_N$ is invariant under permutation for all $N \in \N_+$; see \cite[Lemma~A.1]{bulk} for similar discussions.

We state some more properties using the expression \eqref{eq.IRWHighDim}. To better treat the high dimensional function, we define the following notation for shorthand, 
\begin{align*}
	\fint_{\Lambda^N}\Ll(\cdot\Rr) := \frac{1}{\vert \Lambda \vert^N} \sum_{x_1, \cdots, x_N \in \Lambda}\Ll(\cdot\Rr) ,
\end{align*}
then using the notation \eqref{eq.IRWHighDim},  we observe
\begin{align}\label{eq.IRWHighDimIntegral}
	\bbracket{\tilde u}_{\Lambda, N} = \fint_{\Lambda^N} \tilde u_N.
\end{align}
For any integer $1 \leq i \leq N$ and $e \in U:=\{\pm e_1, \pm e_2, \cdots, \pm e_d\}$, the finite difference operator $\DD_{x_i,e}$ is defined for $\tilde u_N$ as 
\begin{align}\label{eq.FiniteDifference}
	(\DD_{x_i,e} \tilde u_N)(x_1, \cdots, x_N) := \tilde u_N(x_1, \cdots, x_i+e, \cdots, x_N) - \tilde u_N(x_1, \cdots, x_i, \cdots, x_N), 
\end{align}
which is commutative in the sense
\begin{align}\label{eq.commutativity}
	\forall 1 \leq i,j \leq N,  \forall e,e' \in U, \qquad \DD_{x_i, e} \DD_{x_j, e'} = \DD_{x_j, e'}\DD_{x_i, e}. 
\end{align}
Combing the canonical projection \eqref{eq.IRWHighDim}, the finite difference operator is related to the generator $\tilde \L$ in \eqref{eq.defIRW} by the following identity
\begin{align}\label{eq.IDGeneratorIRW}
	\tilde \L \tilde u(\teta) = \sum_{i=1}^N \sum_{e \in U} \DD_{x_i, e} \tilde u_N(x_1, \cdots, x_N).
\end{align}
Then a lot of analytic tools on Euclidean space can be applied to independent particles.

\subsubsection{$H^2$-estimate}
In the following paragraphs, we will recall $H^2$ estimate for independent particles (indeed identity), which will be used in the multiscale Poincar\'e estimate later in Proposition~\ref{prop.MPoincare}. The proof follows \cite[Lemma~B.19]{AKMbook} and \cite[Proposition~3.10]{bordas2018homogeneisation} after a careful review. We will also explain in detail in Remark~\ref{rmk.H2Challenge} the difficulty met when developing the counterpart for Kawasaki dynamics.

\begin{lemma}[Dimension-free $H^2$ estimate on torus]\label{lem.H2Torus}
	For any $L, N \in \N_+$ and any function $u, f : (\Td_L)^N \to \R$ satisfying 
	\begin{align}\label{eq.harmonicHigh}
		-\sum_{i=1}^N \sum_{e \in U} \DD_{x_i, e} u = f,
	\end{align}
	the following identity holds
	\begin{align}\label{eq.H2Torus}
		\fint_{(\Td_L)^N} 	\sum_{i,j=1}^N \sum_{e,e' \in U} (\DD_{x_j,e'}\DD_{x_i, e} u)^2 = 4 \fint_{(\Td_L)^N} f^2.
	\end{align}
\end{lemma}
\begin{proof}
	It is easy to verify the identity 
	\begin{align*}
		\sum_{i=1}^N \sum_{e \in U} \DD_{x_i, e} = -\frac{1}{2} \sum_{i=1}^N \sum_{e \in U} \DD_{x_i, -e} \DD_{x_i, e},
	\end{align*}
	so \eqref{eq.harmonicHigh} is equivalent to 
	\begin{align*}
		-\frac{1}{2} \sum_{i=1}^N \sum_{e \in U} \DD_{x_i, -e} \DD_{x_i, e} u = f.
	\end{align*}
	We evaluate the $L^2$ sum of the two sides
	\begin{equation}\label{eq.L2Torus}
		\begin{split}
			\fint_{(\Td_L)^N} f^2 &= \frac{1}{4} \fint_{(\Td_L)^N} \Ll(\sum_{i=1}^N \sum_{e \in U} \DD_{x_i, -e}\DD_{x_i, e} u \Rr)^2\\
			&= \frac{1}{4} \fint_{(\Td_L)^N} \Ll(\sum_{i=1}^N \sum_{e \in U} \DD_{x_i, -e}\DD_{x_i, e} u \Rr)\Ll(\sum_{j=1}^N \sum_{e' \in U} \DD_{x_j, -e'}\DD_{x_j, e'} u \Rr).
		\end{split}
	\end{equation}
	On torus, it is easy to verify the integration by part formula for $u,v : (\Td_L)^N \to \R$
	\begin{align*}
		\forall 1 \leq i \leq N, e \in U, \qquad \fint_{(\Td_L)^N} (\DD_{x_i, e} u)v = \fint_{(\Td_L)^N} u(\DD_{x_i, -e} v). 
	\end{align*}
	We apply it and the commutativity \eqref{eq.commutativity} to the {\rhs} of \eqref{eq.L2Torus}
	\begin{align*}
		\fint_{(\Td_L)^N} \Ll( \DD_{x_i, -e}\DD_{x_i, e} u \Rr)\Ll(\DD_{x_j, -e'}\DD_{x_j, e'} u \Rr) 	
		&= \fint_{(\Td_L)^N} \Ll( \DD_{x_i, e} u \Rr)\Ll(\DD_{x_i, e} \DD_{x_j, -e'}\DD_{x_j, e'} u \Rr) \\
		&= \fint_{(\Td_L)^N} \Ll( \DD_{x_i, e} u \Rr)\Ll(\DD_{x_j, -e'} \DD_{x_i, e} \DD_{x_j, e'} u \Rr) \\
		&= \fint_{(\Td_L)^N} \Ll( \DD_{x_j, e'} \DD_{x_i, e} u \Rr)\Ll( \DD_{x_i, e} \DD_{x_j, e'} u \Rr) \\
		&= \fint_{(\Td_L)^N} \Ll( \DD_{x_i, e} \DD_{x_j, e'} u \Rr)^2.   
	\end{align*}
	We put it back to the {\rhs} of \eqref{eq.L2Torus} and conclude \eqref{eq.H2Torus}.
\end{proof}
\begin{remark}\label{rmk.H2Challenge}
	We give a more conceptual proof of $H^2$ estimate on torus, so one notices the challenges facing exclusion. Consider the following equation on $\mathbb{T}^N$
	\begin{align*}
		-\Delta u = f.
	\end{align*}
	Then we have the following calculation using the $L^2$ estimate for the two sides
	\begin{align*}
		\int_{\mathbb{T}^N} f^2  = \int_{\mathbb{T}^N} (-\Delta u)^2 = \int_{\mathbb{T}^N} \Ll(\sum_{i=1}^N \partial_{i} \partial_{i} u\Rr)^2 =  \sum_{i,j=1}^N \int_{\mathbb{T}^N} (\partial_{i} \partial_{i} u)(\partial_{j} \partial_{j} u). 
	\end{align*}
	Using the integration by part, then we have 
	\begin{align*}
		\int_{\mathbb{T}^N} (\partial_{i} \partial_{i} u)(\partial_{j} \partial_{j} u) = \int_{\mathbb{T}^N} (\partial_{i} \partial_{j} u)^2.
	\end{align*}
	Then we conclude that 
	\begin{align*}
		\int_{\mathbb{T}^N} f^2  = \sum_{i,j=1}^N \int_{\mathbb{T}^N} (\partial_{i} \partial_{j} u)^2.
	\end{align*}
	
	Lemma~\ref{lem.H2Torus} is just a discrete version of the above proof. If one hopes to recover a similar identity on Kawasaki dynamics for $u,f : \X \to \R$, such that $-\sum_{b \in (\Td_L)^*} \pi_b u = f$, then we also have $\pi_b = -\frac{1}{2}\pi_b \pi_b$ and integration by part formula. However, the main difficulty appears in the commutativity. The identity  
	\begin{align*}
		\pi_b \pi_{b'} = \pi_{b'} \pi_b,
	\end{align*}
	holds when $b \cap b' = \emptyset$. Otherwise, when $b,b'$ shares common endpoint, as the symmetry group is not Abelian, some exotic term will generate and pose a challenge.
\end{remark}

We then extend the result above to the discrete Poisson equation on cube $\Lambda_L$. Here we add the Neumann boundary condition \eqref{eq.Neumann}, and the indicator in \eqref{eq.H2Cube} excludes the second-order finite difference outside $\Lambda_L^+$.   
\begin{corollary}[Dimension-free $H^2$ estimate on cube]\label{cor.H2}
	For any $L, N \in \N_+$ and any function $u, f : (\Lambda_L)^N \to \R$ satisfying 
	\begin{align*}
		-\sum_{i=1}^N \sum_{e \in U} \DD_{x_i, e} u = f,
	\end{align*}
	with the Neumann boundary condition
	\begin{align}\label{eq.Neumann}
		\DD_{x_i, e} u \Ind{x_i \in \partial \Lambda_L, x_i + e \notin  \Lambda_L} = 0,
	\end{align}
	we have
	\begin{align}\label{eq.H2Cube}
		\fint_{(\Lambda_L)^N} 	\sum_{i,j=1}^N \sum_{e,e' \in U} (\DD_{x_j,e'}\DD_{x_i, e} u)^2 \Ind{x_i = x_j, e=e', x_i + 2e \notin \Lambda_L} = 4 \fint_{(\Lambda_L)^N} f^2.
	\end{align}
\end{corollary}
\begin{proof}
	The main idea of proof is periodization. Only in this proof, we work on half integers. With a translation and scaling, we denote by $\cu$ the cube of size $L$ that $\cu=(\mathbb Z+\frac{1}{2})^d\cap[0,L]^d$, and set $\tilde \cu=(\mathbb Z+\frac{1}{2})^d\cap[-L,L]^d$. We still denote by $u,f: (\cu)^N \to \R$ and extend them to $\tilde u, \tilde f: (\tilde \cu)^N \to \R$ by mirror symmetry
	\begin{align}\label{eq.Mirror}
		x_1, \cdots, x_N \in (\tilde \cu)^N, \qquad \tilde u(x_1, \cdots, x_N) := u (\vert x_1 \vert, \cdots, \vert x_N \vert).
	\end{align}
	Next, we identify the opposite sides of the cube $[-L,L]^d$ together to get a torus, and view $\tilde\cu$ as a lattice torus $\mathbb T^d_{2L}$. Then functions $\tilde u, \tilde f$ on $\tilde \cu$ can be regarded as functions on $\mathbb T^d_{2L}$. Moreover, the Neaumann boundary condition and the mirror symmetry implies that on the whole $\mathbb T^d_{2L}$, we have
	\begin{align*}
		\sum_{i=1}^N \sum_{e \in U} \DD_{x_i, e} \tilde u = \tilde f.
	\end{align*}
	Therefore, Lemma~\ref{lem.H2Torus} applies to obtain
	\begin{align*}
		\fint_{(\Td_{2L})^N} 	\sum_{i,j=1}^N \sum_{e,e' \in U} (\DD_{x_j,e'}\DD_{x_i, e} \tilde u)^2 = 4 \fint_{(\Td_{2L})^N} \tilde f^2.
	\end{align*}
	Each side counts $2^d$ times in the integration over $\cu$, which gives us 
	\begin{align*}
		\fint_{\cu^N} 	\sum_{i,j=1}^N \sum_{e,e' \in U} (\DD_{x_j,e'}\DD_{x_i, e} \tilde u)^2 = 4 \fint_{\cu^N} \tilde f^2.
	\end{align*}
	We realize that the second-order derivative of $\tilde u$ near the boundary vanishes due to the Neumann boundary condition and the mirror symmetry, then we conclude \eqref{eq.H2Cube}.

\end{proof}

\subsubsection{Multiscale Poincar\'e inequality}

In this part, we introduce the notion of spatial average and use it to develop some kind of multiscale Poincar\'e inequality.

We define the gradient by adding one more particle 
\begin{equation}\label{eq.defGradientFree}
	\begin{split}
		(\partial_k \tilde u)(\teta, x) &:= \tilde u(\teta + \delta_{x+e_k}) - \tilde u(\teta + \delta_{x}),\\
		(\tna \tilde u)(\teta, x) &:= \Ll((\partial_1 \tilde u)(\teta, x), (\partial_2 \tilde u)(\teta, x), \cdots, (\partial_d \tilde u)(\teta, x)\Rr).
	\end{split}
\end{equation}
This is related to \eqref{eq.FiniteDifference} by Lemma \ref{lem.Mecke1} and is particularly useful when we pass the gradient to the exclusion dynamics in the next section.

Recalling the definition of the enlarged domain in \eqref{eq.defVertexLarge}, we define the filtration
\begin{align}\label{eq.defGTilde}
	\tilde \G_{\Lambda^+} :=\sigma\Ll(\sum_{x\in \Lambda^+}\teta_x, \{\teta_y, y\in (\Lambda^+)^c\}\Rr).
\end{align}
Using $\Z_{m,n}$ and $\Z_n$ defined in Section~\ref{subsubsec.Geometry}, we also define the spatial average operator
\begin{align}\label{eq.defSn}
	\S_{n}(\tna \tilde u)(\teta, x) := \sum_{z \in \Z_n}  \Ll(\frac{1}{\vert \cu_n \vert} \sum_{y \in z + \cu_n} \bbracket{(\tna \tilde u)(\teta, y) \, \vert \tilde \G_{(z+\cu_n)^+}}\Rr)\Ind{x \in z + \cu_n}. 
\end{align}
That is, this operator makes spatial average over the added particle and over the local configuration. The enlarged domain $(z+\cu_n)^+$ is needed to include all the sites where the particle in $(\tna \tilde u(\tilde\eta,x))_{x \in z+\cu_n}$ is added (see \eqref{eq.defGradientFree}). Then note that $\bbracket{\cdot \, \vert \tilde \G_{(z+\cu_n)^+}}_\alpha$ does not depend on $\alpha > 0$, so we drop the density $\alpha$. This leads to the multiscale Poincar\'e inequality.

\begin{proposition}[Multiscale Poincar\'e inequality]\label{prop.MPoincare}
	There exists a finite positive constant $C{=C(d)}$ such that for all bounded functions $\tilde u : \XX \to \R$ such that $\bbracket{\tilde u  \, \vert \tilde \G_{\cu_m^+}} = 0$, we have
	\begin{multline}\label{eq.multiscale}
		\bbracket{\frac{1}{\vert \cu_m \vert}\tilde u^2}^{\frac{1}{2}}_{\alpha} \leq C\alpha^\frac{1}{2}\frac{1}{|\cu_m|}\sum_{x\in \cu_m}\bbracket{\vert\tilde\nabla\tilde u (\tilde\eta,x)\vert^2}_\alpha^\frac{1}{2} 
		\\+C\alpha^{\frac{1}{2}}\sum_{n=0}^m 3^n \Ll( \frac{1}{\vert \Z_{m,n}\vert}\sum_{z \in \Z_{m,n}}\bbracket{\Ll\vert \S_{n}\tna \tilde u\Rr\vert^2(\teta, z)}_{\alpha}\Rr)^{\frac{1}{2}}.  
	\end{multline}
\end{proposition}

\begin{proof}
	We are only interested in the case that $\tilde u$ is integrable under $\bbracket{\cdot}_{\alpha}$. The proof follows \cite[Proposition~3.5]{bulk} and we give a sketch here. By Jensen's inequality, we may replace $u$ by $\tilde\E_\alpha[u|\sigma(\tilde\eta_x)_{x\in\cu_m^+}]$ and assume $u$ is a local function on $\cu_m^+$. Let $\tilde w$ solve the following equation: 
	\begin{align*}
		-\sum_{x \in \cu_m} \teta_x \sum_{y\sim x} \tilde \pi_{x, y} \tilde w =- \tilde u, 
	\end{align*}
	in the sense of the Poisson equation with the same Neumann boundary condition \eqref{eq.Neumann} for $\tilde w$ in Corollary~\ref{cor.H2}. More precisely, we can fix the number of particles $N$ in $\cu_m^+$ and the information outside $\cu_m^+$. Then by the identity \eqref{eq.IDGeneratorIRW}, the equation is written in the form of  Corollary~\ref{cor.H2}, which is solvable recursively. This will give us the $H^2$ estimate for $\tilde w$ using the projection. We have 
	\begin{align*}
		\frac{1}{\vert \cu_m \vert}\bbracket{\tilde u^2}_{\alpha} &= \frac{1}{\vert \cu_m \vert}  \sum_{i=1}^d \sum_{x \in \cu_m}  \bbracket{\teta_x (\tilde \pi_{x, x+e_i} \tilde u)(\tilde \pi_{x, x+e_i} \tilde w)}_\alpha \\
		&= \frac{\alpha}{\vert \cu_m \vert} \sum_{x \in \cu_m} \bbracket{(\tna \tilde u)(\tna \tilde w)(\teta, x)}_\alpha.
	\end{align*}
	Here we use the Neumann boundary condition and Mecke's identity in Lemma~\ref{lem.Mecke1}. Then we add the local averages
	\begin{align*}
		\frac{1}{\vert \cu_m \vert}\bbracket{\tilde u^2}_{\alpha} &= \frac{\alpha}{\vert \cu_m \vert} \sum_{x \in \cu_m} \bbracket{(\tna \tilde u)(\tna \tilde w - \S_0 \tna \tilde w)(\teta, x)}_\alpha \\
		& \qquad +  \sum_{n=0}^{m-1} \frac{\alpha}{\vert \cu_m \vert} \sum_{x \in \cu_m} \bbracket{(\S_n \tna \tilde u)( \S_n \tna \tilde w - \S_{n+1} \tna \tilde w)(\teta, x)}_\alpha \\
		& \qquad +  \frac{\alpha}{\vert \cu_m \vert} \sum_{x \in \cu_m} \bbracket{(\S_m \tna \tilde u)( \S_m \tna \tilde w) (\teta, x)}_\alpha.
	\end{align*}
	Here we use the tower property of conditional expectations.
	
	We just focus on one scale, which gives us 
	\begin{align*}
		&\Ll\vert \frac{1}{\vert \cu_m \vert} \sum_{x \in \cu_m} \bbracket{(\S_n \tna \tilde u)( \S_n \tna \tilde w - \S_{n+1} \tna \tilde w)(\teta, x)}_\alpha \Rr\vert \\
		&= \Ll\vert \frac{1}{\vert \Z_{m,n}\vert} \sum_{z \in \Z_{m,n}} \bbracket{(\S_n \tna \tilde u)( \S_n \tna \tilde w - \S_{n+1} \tna \tilde w)(\teta, z)}_\alpha \Rr\vert\\
		&\leq \Ll( \frac{1}{\vert \Z_{m,n}\vert}\sum_{z \in \Z_{m,n}}\bbracket{\Ll\vert \S_{n}\tna \tilde u\Rr\vert^2(\teta, z)}_{\alpha}\Rr)^{\frac{1}{2}} \Ll( \frac{1}{\vert \Z_{m,n}\vert}\sum_{z \in \Z_{m,n}}\bbracket{\Ll\vert \S_n \tna \tilde w - \S_{n+1} \tna \tilde w\Rr\vert^2(\teta, z)}_{\alpha}\Rr)^{\frac{1}{2}}\\
		&\leq C \alpha^{-\frac{1}{2}} 3^n \Ll( \frac{1}{\vert \Z_{m,n}\vert}\sum_{z \in \Z_{m,n}}\bbracket{\Ll\vert \S_{n}\tna \tilde u\Rr\vert^2(\teta, z)}_{\alpha}\Rr)^{\frac{1}{2}}\bbracket{\frac{1}{\vert \cu_m \vert}\tilde u^2}^{\frac{1}{2}}_{\alpha}.
	\end{align*}
	From the third line to the fourth line, we use the Poincar\'e inequality of independent particles for the term $\Ll\vert \S_n \tna \tilde w - \S_{n+1} \tna \tilde w\Rr\vert^2$, and this will give the factor $3^n$. The output is the second-order derivatives of $\tilde w$, which will be bounded by $\bbracket{\tilde u^2}_{\alpha}$ using the dimension-free $H^2$ estimate in Corollary~\ref{cor.H2}. Here the factor $\alpha^{-\frac{1}{2}}$ comes from another application of Mecke's identity \eqref{eq.Mecke}, and this concludes \eqref{eq.multiscale}.
\end{proof}

\section{Coarse-grained lifting}\label{sec.coarse}
This part introduces the key technique to handle the constraint of particle numbers in Kawasaki dynamics, which is the coarse-grained lifting to independent particles. Let  $\teta \in \XX = \N^{\Zd}$ stand for the configuration of independent particles, and $\eta \in \X = \{0,1\}^{\Zd}$ for the configuration with exclusion. We aim to embed $\X$ into $\XX$, and a natural idea is to define the following projection $[] : \XX \to \X$ that 
\begin{align}\label{eq.defCoarsen}
	\forall x \in \Zd, \qquad [\teta]_x := \Ind{\teta_x \geq 1}.
\end{align}
Therefore, $[\teta]$ only indicates whether the site is occupied, but does not care about the exact number of particles. This projection operator also induces an extension for every function $u:\X \to \R$  by pull-back  that
\begin{align}\label{eq.CoarsenFunction}
	[u]:\XX \to \R, \qquad [u](\teta) := u([\teta]),
\end{align}
and we call $[u]$ the coarsened function of $u$. In the following paragraphs, we will use $[\teta]$ to represent the configuration with exclusion, and explore several identities under this projection/coarsen.

\subsection{Grand canonical ensemble} Let the configuration of independent particles  $\teta$ follow the law $\PP_\alpha$, which is independent Poisson distribution of parameter $\alpha >0$. In order to make $[\teta]$ of the same law as sampled from $\P_\rho$, we make a specific choice between the parameters such that $e^{-\alpha} = 1-\rho$, i.e.
\begin{align}\label{eq.CouplePara}
	\forall \rho \in (0,1), \qquad \alpha(\rho) := - \log(1-\rho).
\end{align}
Under this specific choice of parameter, we can see the coarsen function $[u]$ as a \emph{lift} of $u$ thanks of the following proposition.
\begin{proposition}[Coarsen-grained lifting]\label{prop.coarsen}
	Given $\rho \in (0,1)$ and $\alpha(\rho)$ defined as \eqref{eq.CouplePara} and $\teta$ sampled from $\PP_{\alpha(\rho)}$, then $[\teta] \in \X$ follows the law $\P_{\rho}$. As a consequence, for every $u : \X \to \R$ integrable under $\P_{\rho}$, its coarsen function satisfies 
	\begin{align}\label{eq.IdCoarsen}
		\bbracket{[u]}_{\alpha(\rho)} = \bracket{u}_{\rho}.
	\end{align}
\end{proposition} 
\begin{proof}
	With the choice of the parameter \eqref{eq.CouplePara} and the definition of the projection operator \eqref{eq.defCoarsen}, we have 
	\begin{align*}
		\forall x  \in \Zd, \qquad \PP_{\alpha(\rho)}[[\teta]_x = 0] &= \PP_{\alpha(\rho)}[\teta_x = 0] = e^{-\alpha(\rho)} = 1-\rho,\\
		\PP_{\alpha(\rho)}[[\teta]_x = 1] &= \PP_{\alpha(\rho)}[\teta_x \geq 1] = 1 - \PP_{\alpha(\rho)}[\teta_x = 0] = \rho.
	\end{align*}
	Therefore, $[\teta]_x$ follows the Bernoulli law with parameter $\rho$. Since $(\teta_x)_{x \in \Zd}$ are i.i.d. random variables, $[\teta]$ has the same law as $\P_\rho$. Combing this fact and  \eqref{eq.CoarsenFunction}, the identity \eqref{eq.IdCoarsen} is a direct corollary
	\begin{align*}
		\bbracket{[u]}_{\alpha(\rho)} = \bbracket{u([\teta])}_{\alpha(\rho)} = \bracket{u(\eta)}_{\rho}.
	\end{align*}
\end{proof}

\bigskip

Although we can couple the static configuration between the two systems and obtain the nice identity \eqref{eq.IdCoarsen}, similar result cannot be extended to the Dirichlet energy and we cannot expect a similar identity like 
\begin{align}\label{eq.IdWrong}
	\sum_{x \in \Lambda}  \sum_{y \in \Lambda, y \sim x}\bbracket{\teta_x (\tilde \pi_{x,y}[u])^2}_{\alpha(\rho)} = C(\rho)\sum_{x,y \in \Lambda, x \sim y} \bracket{(\pi_{x,y}u)^2}_{\rho}.
\end{align}
On the other hand, the coarse-grained lifting can be very useful when evaluating the spatial average of the gradient. We establish the following identity, which can be used as the gradient coupling. We recall the definition of the tangent field $\nabla_{x, e_i}$ for Kawasaki dynamics defined in \eqref{eq.defKawaTagent}.
\begin{proposition}[Gradient coupling]\label{prop.IdentityGradient}
	For every $\rho \in (0,1), \Lambda \subset \Zd$ and $u:\X \to \R$, the following identity holds for every $i \in \{1, \cdots, d\}$
	\begin{align}
		\sum_{x \in \Lambda} \bbracket{\teta_x \tilde \pi_{x,x+e_i} [u]}_{\alpha(\rho)} = \frac{\alpha(\rho)}{2 \rho} \sum_{x \in \Lambda} \bracket{\nabla_{x, e_i} u}_\rho .
	\end{align}
\end{proposition}
Proposition~\ref{prop.IdentityGradient} is the result of the following lemmas. We state them separately as they can be useful in other proofs. The first lemma is a similar version of Mecke's identity in Kawasaki dynamics. Here, instead of adding a particle, the added particle is understood as ``forcing the site occupied".
\begin{lemma}[Mecke's identity in Kawasaki dynamics]\label{lem.Mecke2}
	For every $\rho \in (0,1)$ and $u:\X \to \R$, the following identity holds
	\begin{align}\label{eq.GradientKawasaki}
		\bracket{\nabla_{x, e_i} u}_\rho = 2 \rho \Ll(\bracket{u \,\vert \eta_{x+e_i} = 1}_\rho - \bracket{u \,\vert \eta_x = 1}_\rho\Rr).
	\end{align}
\end{lemma}
\begin{proof}
	Recall the identity \eqref{eq.IdTagent} and the term is non-zero if and only if $(\eta_x, \eta_{x+e_i})=(1,0)$ or $(\eta_x, \eta_{x+e_i})=(0,1)$. Then we obtain that 
	\begin{align*}
		\bracket{(\pi_{x, x+e_i} u) (\pi_{x, x+e_i} \ell_{e_i})}_\rho &=  2\bracket{u(\eta)\Ind{(\eta_x, \eta_{x+e_i})=(0,1)} - u(\eta)\Ind{(\eta_x, \eta_{x+e_i})=(1,0)}}_\rho.
	\end{align*}
	Here noticing that the term $(\eta_x, \eta_{x+e_i})=(1,1)$ is canceled on the {\rhs} of the equation, we can rewrite it as 
	\begin{align*}
		&\bracket{u(\eta)\Ind{(\eta_x, \eta_{x+e_i})=(0,1)} - u(\eta)\Ind{(\eta_x, \eta_{x+e_i})=(1,0)}}_\rho \\
		&=\bracket{u(\eta)\Ind{(\eta_x, \eta_{x+e_i})=(0,1)} + u(\eta)\Ind{(\eta_x, \eta_{x+e_i})=(1,1)}}_\rho \\
		& \qquad - \bracket{u(\eta)\Ind{(\eta_x, \eta_{x+e_i})=(1,0)} + u(\eta)\Ind{(\eta_x, \eta_{x+e_i})=(1,1)}}_\rho\\
		&= \bracket{u(\eta)\Ind{\eta_{x+e_i}=1} - u(\eta)\Ind{\eta_x=1}}_\rho.
	\end{align*}
	Therefore, we obtain the identity
	\begin{align*}
		\bracket{(\pi_{x, x+e_i} u) (\pi_{x, x+e_i} \ell_{e_i})}_\rho 
		&= 2\bracket{u(\eta)\Ind{\eta_{x+e_i}=1} - u(\eta)\Ind{\eta_x=1}}_\rho\\
		&= 2 \rho \Ll(\bracket{u \,\vert \eta_{x+e_i} = 1}_\rho - \bracket{u \,\vert \eta_x = 1}_\rho\Rr).
	\end{align*}
\end{proof}

The second lemma makes the bridge between Lemma~\ref{lem.Mecke1} and Lemma~\ref{lem.Mecke2}, and is a corollary from Proposition~\ref{prop.coarsen}. 
\begin{lemma}[Change of variable]\label{lem.ChangeVariable}
	For any $\rho \in (0,1)$ and $x,y\in \Zd$, let $\eta \in \X$ be sampled from $\P_\rho$ and $\tilde \eta \in \tilde \X$ sampled from $\tilde \P_{\alpha(\rho)}$, then we have
	\begin{align}\label{eq.ChangeVariable}
		([\tilde \eta + \delta_x], [\tilde \eta + \delta_y ]) \stackrel{(d)}{=} (\eta \vee \delta_x,\eta \vee \delta_y),
	\end{align}
	where $\eta \vee \delta_x \in \X$ is defined as $(\eta \vee \delta_x)_x = \eta_x \vee 1$ and $(\eta \vee \delta_x)_z = \eta_z$ for $z \neq x$.
\end{lemma}
\begin{proof}
	The proof follows the observation that 
	\begin{align*}
		[\tilde \eta + \delta_x] = [\tilde \eta ] \vee \delta_x.
	\end{align*}
	Thus, we apply Proposition~\ref{prop.coarsen}
	\begin{align*}
		([\tilde \eta + \delta_x], [\tilde \eta + \delta_y ]) = ([\tilde \eta ] \vee \delta_x, [\tilde \eta ] \vee \delta_y) \stackrel{(d)}{=}  (\eta \vee \delta_x,\eta \vee \delta_y).
	\end{align*}
\end{proof}

\begin{proof}[Proof of Proposition~\ref{prop.IdentityGradient}]
	We combine the results in Lemma~\ref{lem.Mecke1}, Lemma~\ref{lem.Mecke2} and Lemma~\ref{lem.ChangeVariable}
	\begin{align*}
		\sum_{x \in \Lambda} \bbracket{\teta_x \tilde \pi_{x,x+e_i} [u]}_{\alpha(\rho)} &\stackrel{\eqref{eq.Mecke}}{=} \alpha(\rho)\sum_{x \in \Lambda} \bbracket{[u] (\tilde \eta + \delta_{x+e_i}) - [u] (\tilde \eta + \delta_{x})}_{\alpha(\rho)}\\
		&\stackrel{\eqref{eq.CoarsenFunction}}{=}\alpha(\rho)\sum_{x \in \Lambda} \bbracket{u ([\tilde \eta + \delta_{x+e_i}]) - u ([\tilde \eta + \delta_{x}])}_{\alpha(\rho)}\\
		&\stackrel{\eqref{eq.ChangeVariable}}{=}\alpha(\rho)\sum_{x \in \Lambda} \bracket{u (\eta \vee \delta_{x+e_i}) - u (\eta \vee \delta_{x})}_{\rho}\\
		&\stackrel{\eqref{eq.GradientKawasaki}}{=}\frac{\alpha(\rho)}{2 \rho} \sum_{x \in \Lambda} \bracket{(\pi_{x, x+e_i} u) (\pi_{x, x+e_i} \ell_{e_i})}_\rho.
	\end{align*}
	Here from the third line to the fourth line, we also use the fact $ \bracket{u \,\vert \eta_{x} = 1}_\rho = \bracket{ u (\eta \vee \delta_{x})}_{\rho}$. 
\end{proof}

At the end of this subsection, we give another application of Lemma~\ref{lem.Mecke2}. We recall the notation $\Lambda^-$ defined in \eqref{eq.defMinus} and the gradient $\nabla_x$ defined in \eqref{eq.defKawaGradient}.
\begin{corollary} For every bounded set $\Lambda \subset \Zd$ and $u \in \F_0(\Lambda^-)$, we have 
	\begin{align}\label{eq.MeanZeroLocal}
		\sum_{x \in \Lambda} \bracket{\nabla_x u}_\rho = 0.
	\end{align}
\end{corollary}
\begin{proof}
	We just focus on the gradient field along one direction $e_i$ and apply \eqref{eq.GradientKawasaki}
	\begin{align*}
		\sum_{x \in \Lambda} \bracket{\nabla_{x,e_i} u}_\rho &=2 \rho	\sum_{x \in \Lambda} \Ll(\bracket{u (\eta \vee \delta_{x+e_i})}_\rho - \bracket{u (\eta \vee \delta_{x})}_\rho\Rr)\\
		&=2 \rho \Ll(\sum_{x\notin\Lambda,x-e_i\in\Lambda} \bracket{u (\eta \vee \delta_{x})}_\rho - \sum_{x \in \Lambda, x - e_i \notin \Lambda} \bracket{u (\eta \vee \delta_{x})}_\rho\Rr)\\
		&=2 \rho \Ll(\sum_{x\notin\Lambda,x-e_i\in\Lambda} \bracket{u (\eta)}_\rho - \sum_{x \in \Lambda, x - e_i \notin \Lambda} \bracket{u (\eta)}_\rho\Rr)\\
		&=0.
	\end{align*}
	In the first line, we also use the observation $\bracket{u \,\vert \eta_{x} = 1}_\rho = \bracket{ u (\eta \vee \delta_{x})}_{\rho}$. Then the proof is similar to the discrete Stokes' formula that all the terms except those on the boundary cancel, which yields the second line. The condition $u \in \F_0(\Lambda^-)$ implies that $u (\eta \vee \delta_{x}) = u (\eta)$ and passes the result from the second line to the third line. Finally, as the terms in $\sum_{x \in \Lambda, x + e_i \notin \Lambda}$ and $\sum_{x \in \Lambda, x - e_i \notin \Lambda}$ are coupled, we obtain $0$.
\end{proof}

\subsection{Canonical ensemble} We establish a gradient coupling similar to  Proposition~\ref{prop.IdentityGradient} under the canonical ensemble. Recall $\Lambda^+$ defined in \eqref{eq.defVertexLarge}.
\begin{proposition}[Gradient coupling]\label{prop.IdentityGradient2} 
	For every $M \in \N, \Lambda \subset \Zd$ and $u:\X \to \R$ a $\mcl F_{\Lambda^+}$-measurable function, let $(\teta_x)_{x \in \Lambda^+}$ be sampled from the canonical ensemble of $M$ independent particles $\tilde \P_{\Lambda^+, M}$ and $(P_{\Lambda, M, N})_{N \in \N}$ be the probability in $N$ of the number of occupied sites defined as 
	\begin{align}\label{eq.defPMN}
		P_{\Lambda, M, N} := \tilde \P_{\Lambda^+, M}\Ll[\sum_{z \in \Lambda^+ \setminus \{y\}} \Ind{\teta_z > 0} = N-1\Rr], \qquad y \in \Lambda^+,
	\end{align}
	then the following identity holds for every $i \in \{1, \cdots, d\}$
	\begin{align}\label{eq.IdentityGradient2}
		\sum_{x \in \Lambda} \bbracket{\partial_i [u](\teta, x)}_{\Lambda^+, M} = \sum_{N=1}^{M+}  \frac{\vert \Lambda^+ \vert P_{\Lambda, M, N}}{2 N} \sum_{x \in \Lambda} \bracket{\nabla_{x, e_i} u}_{\Lambda^+, N}.
	\end{align}
\end{proposition}
\begin{remark}
	One can check that \eqref{eq.defPMN} is well-defined and does not depend on the choice of $y$. One can pick a specific $y$ and put it back to \eqref{eq.IdentityGradient2}. This will make the identity there a little strange, but it transfers the symmetry used in the proof.
\end{remark}

To prove this proposition, we also give a Mecke's identity like Lemma~\ref{lem.Mecke2} under the canonical ensemble.
\begin{lemma}[Mecke's identity under canonical ensemble]\label{lem.Mecke3}
	Given $u:\X \to \R$, then for every $N \in \N$, a finite subset $\Lambda \subset \Zd$ and $x \in \Lambda$, the following identity holds 
	\begin{align}\label{eq.GradientKawasaki2}
		\bracket{\nabla_{x, e_i} u}_{\Lambda^+, N} =  \frac{2 N}{\vert \Lambda^+\vert} \Ll(\bracket{u \,\vert \eta_{x+e_i} = 1}_{\Lambda^+, N} - \bracket{u \,\vert \eta_x = 1}_{\Lambda^+, N}\Rr).
	\end{align}
\end{lemma}
\begin{proof}
	The proof follows the similar strategy as Lemma~\ref{lem.Mecke2}
	\begin{align*}
		\bracket{(\pi_{x, x+e_i} u) (\pi_{x, x+e_i} \ell_{e_i})}_{\Lambda^+, N} &=  2\bracket{u(\eta)\Ind{(\eta_x, \eta_{x+e_i})=(0,1)} - u(\eta)\Ind{(\eta_x, \eta_{x+e_i})=(1,0)}}_{\Lambda^+, N} \\
		&= 2\bracket{u(\eta)\Ind{\eta_{x+e_i}=1} - u(\eta)\Ind{\eta_x=1}}_{\Lambda^+, N}\\
		&= \frac{2 N}{\vert \Lambda^+ \vert} \Ll(\bracket{u \,\vert \eta_{x+e_i} = 1}_{\Lambda^+, N} - \bracket{u \,\vert \eta_x = 1}_{\Lambda^+, N}\Rr).
	\end{align*}
	Here from the first line to the second line, we use the fact that the case $(\eta_x, \eta_{x+e_i})=(1,1)$ cancels. From the second line to the third line, we notice that $x, x+e_i \in \Lambda^+$, so 
	\begin{align*}
		\P_{\Lambda^+, N} [\eta_x = 1] = \P_{\Lambda^+, N} [\eta_{x + e_i} = 1] = \frac{{(\vert \Lambda^+\vert - 1) \choose (N-1)}}{{\vert \Lambda^+\vert \choose N}} = \frac{N}{\vert \Lambda^+\vert}.
	\end{align*}
\end{proof}

\begin{proof}[Proof of Proposition~\ref{prop.IdentityGradient2}]
	Notice that, for every $x \in \Lambda^+$, we have
	\begin{equation}\label{eq.IdentityGradientStep1}
		\begin{split}
			&\bbracket{[u](\teta + \delta_{x})}_{\Lambda^+, M}\\
			&= \bbracket{u([\teta + \delta_{x}])}_{\Lambda^+, M} \\
			&= \sum_{N=1}^{M+1} \tilde \P_{\Lambda^+, M}\Ll[\sum_{z \in \Lambda^+ } \Ind{\teta_z + \delta_x > 0} = N\Rr] \bbracket{u([\teta + \delta_{x}]) \, \big\vert \sum_{z \in \Lambda^+ } \Ind{\teta_z + \delta_x > 0} = N}_{\Lambda^+, M} \\
			& = \sum_{N=1}^{M+1} P_{\Lambda, M, N} \bracket{u \, \vert \eta_x = 1}_{\Lambda^+, N}.
		\end{split}
	\end{equation}
	Here we apply the definition of the projection operator \eqref{eq.defCoarsen} from the first line to the second line, and the total probability formula from the second line to the third line. The passage from the third line to the fourth line requires some work. Firstly, we recall the identity \eqref{eq.defPMN}
	\begin{align*}
		\tilde \P_{\Lambda^+, M}\Ll[\sum_{z \in \Lambda^+ } \Ind{\teta_z + \delta_x > 0} = N\Rr] = \tilde \P_{\Lambda^+, M}\Ll[\sum_{z \in \Lambda^+ \setminus \{x\}} \Ind{\teta_z > 0} = N-1\Rr] = P_{\Lambda, M, N}. 
	\end{align*}
	We also remark that $P_{\Lambda, M, N}$ is well-defined and does not depend on the excluded site $x$. Secondly, we notice that, for every choice of vertex set $V \subset \Lambda^+ \setminus \{x\}$ such that  $\vert V\vert = N-1$, in the following set
	\begin{align*}
		\Ll\{ \sum_{z \in \Lambda^+ } \teta_z  = M : \teta_z > 0 \text{ for all } z \in V , \text{ and } \teta_z = 0 \text{ for all } z \in \Lambda^+ \setminus (\{x\} \cup V) \Rr\},
	\end{align*}
	the number of configurations is the same. Therefore, the configuration $[\teta + \delta_{x}]$ contains a particle on $x$, and the other occupied positions are uniformly distributed conditioned on the number of occupied sites. This generalizes the coarse-grained lifting under the canonical ensemble, and gives the passage concerning the conditional expectation in the third line of \eqref{eq.IdentityGradientStep1}.
	
	Applying \eqref{eq.IdentityGradientStep1}, we obtain 
	\begin{align*}
		&\sum_{x \in \Lambda} \bbracket{[u](\teta + \delta_{x+e_i}) - [u](\teta + \delta_x)}_{\Lambda^+, M} \\
		&= \sum_{N=1}^{M+1}  P_{\Lambda, M, N}  \sum_{x \in \Lambda}  \Ll(\bracket{u \,\vert \eta_{x+e_i} = 1}_{\Lambda^+, N} - \bracket{u \,\vert \eta_x = 1}_{\Lambda^+, N}\Rr)\\
		&= \sum_{N=1}^{M+1}  \frac{\vert \Lambda^+ \vert P_{\Lambda, M, N}}{2 N} \sum_{x \in \Lambda} \bracket{(\pi_{x, x+e_i} u) (\pi_{x, x+e_i} \ell_{e_i})}_{\Lambda^+, N}. 
	\end{align*}
	Here from the second line to the third line, we also use \eqref{eq.GradientKawasaki2}. This yields the desired result.
\end{proof}

\subsection{Weighted multiscale Poincar\'e inequality}
As the main result of this subsection, we combine the coarse-grained lifting and gradient coupling to obtain the weighted multiscale Poincar\'e inequality on Kawasaki dynamics. Here we define $\mathcal{G}_{\Lambda^+}$ similar to \eqref{eq.defGTilde}
\begin{equation}\label{eq.defGTildeKawa}
	\mathcal{G}_{\Lambda^+} :=\sigma\Ll(\sum_{x\in \Lambda^+}\eta_x, \{\eta_y, y\in (\Lambda^+)^c\}\Rr).
\end{equation}
Recall the notations about the canonical ensemble in Section~\ref{subsubsec.proba}, we actually have ${\P_{\Lambda, N, \zeta}[\cdot] = \Pr[ \cdot  \vert \mathcal{G}_{\Lambda}](N, \zeta)}$ for every $\rho \in [0,1]$.

\begin{proposition}[Weighted multiscale Poincar\'e inequality]\label{prop.WMPoincare} There exists a finite positive constant $C=C(d)$ such that for all measurable functions $u: \X \to \R$ such that $\bracket{u \, \vert \G_{\cu_m^+}} = 0$, the following estimate is established
	\begin{multline}\label{eq.WMPoincare2}
		\bracket{\frac{1}{|\cu_{m}|} u^2}_\rho^\frac{1}{2}  \\
		\leq C \sum_{n=0}^{m} 3^n \Ll( \frac{1}{\vert \Z_{m,n}\vert}\sum_{z \in \Z_{m,n}}\bracket{\frac{\vert \cu_n^+ \vert }{2 \mathbf{N}^*_{z,n}} \Ll\vert\frac{1}{\vert \cu_n \vert}\sum_{x \in z+\cu_n} \bracket{\nabla_{x}u|\G_{z+\cu_n^+}}\Rr\vert^2}_{\rho} \Rr)^{\frac{1}{2}},
	\end{multline}
	where the random variable $\mathbf{N}_{z,n}^*$ is defined as
	\begin{equation*}
		\mathbf{N}_{z,n}^*(\eta) :=\begin{cases}
			\sum_{x\in z+\cu_n^+}\eta_x,\ &\rho\leq\frac{1}{2},
			\\\sum_{x\in z+\cu_n^+}(1-\eta_x),\ &\rho>\frac{1}{2}.
		\end{cases}
	\end{equation*}
	The integrand in the expectation is defined to be zero if $\mathbf N_{z,n}^*=0$.
\end{proposition}

\begin{proof}
	\emph{Step~1: forward procedure.} In order to better estimate its $L^2$ norm, we need the multiscale Poincar\'e inequality \eqref{eq.multiscale}. Since this is only established in independent particle systems, we make use of the coarse-grained lifting developed in Proposition~\ref{prop.coarsen} as a bridge, that is
	\begin{align*}
		\bracket{\frac{1}{|\cu_{m}|} u^2}_\rho^\frac{1}{2} 
		&=\bbracket{\frac{1}{|\cu_{m}|} [u]^2}_{\alpha(\rho)}^\frac{1}{2}\\
		&\leq C\alpha^{\frac{1}{2}}(\rho)\sum_{n=1}^{m} 3^n \Ll( \frac{1}{\vert \Z_{m,n}\vert}\sum_{z \in \Z_{m,n}}\bbracket{\Ll\vert\S_{n}\tna [u]\Rr\vert^2(\teta, z)}_{\alpha(\rho)}\Rr)^{\frac{1}{2}}.
	\end{align*} 
	We also remark that $\bracket{u \, \vert \G_{\cu_m^+}} = 0$ ensures the condition $\bracket{[u] \, \vert \tilde \G_{\cu_m^+}} = 0$ to apply \eqref{eq.multiscale}.
	
	\emph{Step~2: backward procedure.}
	We focus on one term $\bbracket{\Ll(\S_{n}\tna [ u]\Rr)^2(\teta, z)}_{\alpha(\rho)}$, which is the spatial average in the independent particles. Our main task in this step is to bring $(\S_{n}\tna [ u])(\teta, z)$ back to the original Kawasaki dynamics from the coarsen operator. We apply the definition of $\S_n$ in \eqref{eq.defSn}
	\begin{multline}\label{eq.L2Step21}
		\bbracket{\Ll\vert\S_{n}\tna [ u] \Rr\vert^2(\teta, z)}_{\alpha(\rho)} \\
		= \sum_{M=0}^\infty \tilde \P_{\alpha(\rho)}\Ll[\sum_{y \in z+\cu_n^+}\teta_y=M\Rr]  \bbracket{\Ll\vert\frac{1}{\vert \cu_n \vert} \sum_{x \in z + \cu_n} \bbracket{\tna [u](\teta, x) }_{z+\cu_n^+, M}\Rr\vert^2}_{\alpha(\rho)}.
	\end{multline}
	More precisely, by $\bbracket{\tna [u](\teta, x) }_{z+\cu_n^+, M}$ we mean 
	\begin{align*}
		\bbracket{\tna [u](\teta, x) }_{z+\cu_n^+, M} = \int_{\XX} \tna [u](\teta' \mres (z+\cu_n^+) + \teta \mres (z+\cu_n^+)^c , x) \, \d \tilde \P_{z+\cu_n^+, M}(\teta'). 
	\end{align*}
	Then we apply the gradient coupling \eqref{eq.IdentityGradient2} to the term $\teta'$ in the {\rhs}
	\begin{align*}
		&\Ll(\frac{1}{\vert \cu_n \vert} \sum_{x \in z + \cu_n} \bbracket{\partial_i [u](\teta, x) }_{z+\cu_n^+, M}\Rr)^2 \\
		&= \Ll(\sum_{N=1}^{M+1} P_{z+\cu_n, M, N}  \frac{\vert \cu_n^+ \vert }{2 N}\Rr. \\
		& \qquad \Ll.\frac{1}{\vert \cu_n \vert}\sum_{x \in z+\cu_n} \bracket{(\nabla_{x, e_i}u) (\eta \mres (z+\cu_n^+) + [\teta] \mres  (z+\cu_n^+)^c )}_{z+\cu_n^+, N}\Rr)^2 \\
		&\leq \sum_{N=1}^{M+1} P_{z+\cu_n, M, N}  \Ll(\frac{\vert \cu_n^+ \vert }{2 N}\Rr)^2 \\
		& \qquad \Ll(\frac{1}{\vert \cu_n \vert}\sum_{x \in z+\cu_n} \bracket{(\nabla_{x, e_i}u) (\eta \mres (z+\cu_n^+) + [\teta] \mres  (z+\cu_n^+)^c )}_{z+\cu_n^+, N}\Rr)^2. \\
	\end{align*}
	From the second line to the third line, we make use of Jensen's inequality. We also remark that, our function is not $\mcl F_{z+\cu_n^+}$-measurable, thus we keep $[\teta] \mres  (z+\cu_n^+)^c$ after the local average. We put this result back to \eqref{eq.L2Step21} and apply  Proposition~\ref{prop.coarsen} to $[\teta] \mres  (z+\cu_n^+)^c$, which gives us  
	\begin{multline}\label{eq.L2Step22}
		\bbracket{\Ll(\frac{1}{\vert \cu_n \vert}\sum_{x \in z+\cu_n} \bracket{(\nabla_{x, e_i}u) (\eta \mres (z+\cu_n^+) + [\teta] \mres  (z+\cu_n^+)^c )}_{z+\cu_n^+, N}\Rr)^2}_{\alpha(\rho)} \\
		= \underbrace{\bracket{\Ll(\frac{1}{\vert \cu_n \vert}\sum_{x \in z+\cu_n} \bracket{\nabla_{x, e_i}u}_{z+\cu_n^+, N}\Rr)^2}_{\rho}}_{=:F(N)}.
	\end{multline}
	Using the definition \eqref{eq.defPMN} with a specific choice $y  = z$ there, we have  
	\begin{align*}
		P_{z+\cu_n, M, N} = \tilde \P_{z+\cu_n^+, M}\Ll[\sum_{x \in (z+\cu_n^+) \setminus \{z\}} \Ind{\teta_x > 0} = N-1\Rr].
	\end{align*}
	Then we apply Fubini's lemma to the double sum
	\begin{align*}
		&\bbracket{\Ll(\S_{n}(\partial_i  [ u])(\teta, z)\Rr)^2}_{\alpha(\rho)}\\
		&\leq \sum_{M=0}^\infty \tilde \P_{\alpha(\rho)}\Ll[\sum_{y \in z+\cu_n^+}\teta_y=M\Rr] \sum_{N=1}^{M+1} P_{z+\cu_n, M, N} \Ll(\frac{\vert \cu_n^+ \vert }{2 N}\Rr)^2  F(N)\\
		&= \sum_{N=1}^\infty  \Ll( \sum_{M\geq N-1} \tilde \P_{\alpha(\rho)}\Ll[\sum_{y \in z+\cu_n^+}\teta_y=M\Rr] \tilde \P_{z+\cu_n^+, M}\Ll[\sum_{x \in (z+\cu_n^+) \setminus \{z\}} [\teta]_x = N-1\Rr]\Rr) \Ll(\frac{\vert \cu_n^+ \vert }{2 N}\Rr)^2  F(N)\\
		&= \sum_{N=1}^\infty  \tilde \P_{\alpha(\rho)}\Ll[\sum_{x \in (z+\cu_n^+) \setminus \{z\}} [\teta]_x = N-1\Rr] \Ll(\frac{\vert \cu_n^+ \vert }{2 N}\Rr)^2  F(N)\\
		&= \sum_{N=1}^\infty  \P_{\rho}\Ll[\sum_{x \in (z+\cu_n^+) \setminus \{z\}} \eta_x = N-1\Rr] \Ll(\frac{\vert \cu_n^+ \vert }{2 N}\Rr)^2  F(N)\\
		&=\sum_{N=1}^\infty  \tbinom{|\cu_n^+|-1}{N-1}\rho^{N-1}(1-\rho)^{|\cu_n^+|-N} \Ll(\frac{\vert \cu_n^+ \vert }{2 N}\Rr)^2  F(N)\\
		&=\sum_{N=1}^\infty \frac{1}{2\rho} \tbinom{|\cu_n^+|}{N}\rho^{N}(1-\rho)^{|\cu_n^+|-N} \frac{\vert \cu_n^+ \vert }{2 N}  F(N)\\
		&=\frac{1}{2\rho}\sum_{N=1}^\infty  \P_{\rho}\Ll[\sum_{x \in (z+\cu_n^+)} \eta_x = N\Rr]\frac{\vert \cu_n^+ \vert }{2 N}  F(N).
	\end{align*}
	The identity from the third line to the fourth line comes from the definition of conditional probability. From the fourth line to the fifth line, we apply the coarse-grained lifting in Proposition~\ref{prop.coarsen} once again. We define the random variable ${\mathbf N_{z,n}=\sum_{x\in z+\cu_n^+}\eta_x}$ to be the number of particles in the cube $z+\cu_n^+$. Using this notation, the last term becomes $\frac{1}{2\rho}\bracket{\frac{\vert \cu_n^+ \vert }{2 \mathbf{N}_{z,n}}F(\mathbf{N}_{z,n})}_{\rho}$, and it gives us 
	\begin{multline}\label{eq.L2Step23}
		\bbracket{\Ll(\S_{n}(\partial_i  [ u])(\teta, z)\Rr)^2}_{\alpha(\rho)}
		\\ \leq \frac{1}{2\rho}\bracket{\frac{\vert \cu_n^+ \vert }{2 \mathbf{N}_{z,n}} \Ll(\frac{1}{\vert \cu_n \vert}\sum_{x \in z+\cu_n} \bracket{(\nabla_{x, e_i}u) }_{z+\cu_n^+, \mathbf{N}_{z,n}}\Rr)^2}_{\rho}.
	\end{multline}
	We remark that we define this integrand to be zero if $\mathbf N$ is zero. Plugging this back, we have
	\begin{multline}\label{eq.WMPoincare}
		\bracket{\frac{1}{|\cu_{m}|} u^2}_\rho^\frac{1}{2}  \\
		\leq C\left(\frac{\alpha(\rho)}{\rho}\right)^{\frac{1}{2}} \sum_{n=0}^{m} 3^n \Ll( \frac{1}{\vert \Z_{m,n}\vert}\sum_{z \in \Z_{m,n}}\bracket{\frac{\vert \cu_n^+ \vert }{2 \mathbf{N}_{z,n}} \Ll\vert\frac{1}{\vert \cu_n \vert}\sum_{x \in z+\cu_n} \bracket{\nabla_{x}u}_{z+\cu_n^+, \mathbf{N}_{z,n}}\Rr\vert^2}_{\rho} \Rr)^{\frac{1}{2}}.
	\end{multline}
	
	\textit{Step 3: flipping.} The result \eqref{eq.WMPoincare} is close to what we expect, since we retract the spatial average under the Kawasaki dynamics with a weight of particle numbers. The only issue is that the factor $\alpha(\rho)$ behaves badly when $\rho$ is close to $1$. However, as we pointed out, this lifting procedure $[u]$ breaks the symmetry $\bracket{u(\eta)}_\rho=\bracket{u(1-\eta)}_{1-\rho}$, where $(1-\eta)_x=1-\eta_x$. Therefore, we actually have another version by flipping $\rho$ and $1-\rho$.
	
	We define $\check u:\mcl X\rightarrow \R$ as $\check u(\eta)=u(1-\eta)$ and apply \eqref{eq.WMPoincare} to it under $\bracket{\cdot}_{1-\rho}$ to obtain
	\begin{align*}
		&\bracket{\frac{1}{|\cu_{m}|} u^2}_\rho^\frac{1}{2}=\bracket{\frac{1}{|\cu_{m}|} \check u^2}_{1-\rho}^\frac{1}{2}  \\
		\leq& C\left(\frac{-\log\rho}{1-\rho}\right)^{\frac{1}{2}} \sum_{n=0}^{m} 3^n \Ll( \frac{1}{\vert \Z_{m,n}\vert}\sum_{z \in \Z_{m,n}}\bracket{\frac{\vert \cu_n^+ \vert }{2 \mathbf{N}_{z,n}} \Ll\vert\frac{1}{\vert \cu_n \vert}\sum_{x \in z+\cu_n} \bracket{\nabla_{x}\check u}_{z+\cu_n^+, \mathbf{N}_{z,n}}\Rr\vert^2}_{1-\rho} \Rr)^{\frac{1}{2}}.
	\end{align*}
	Let us work on a single term.
	\begin{align*}
		&\bracket{\frac{\vert \cu_n^+ \vert }{2 \mathbf{N}_{z,n}(\eta)} \Ll\vert\frac{1}{\vert \cu_n \vert}\sum_{x \in z+\cu_n} \bracket{\nabla_{x}\check u}_{z+\cu_n^+, \mathbf{N}_{z,n}(\eta)}\Rr\vert^2}_{1-\rho}
		\\=&\bracket{\frac{\vert \cu_n^+ \vert }{2 \mathbf{N}_{z,n}(1-\eta)} \Ll\vert\frac{1}{\vert \cu_n \vert}\sum_{x \in z+\cu_n} \bracket{\nabla_{x}\check u}_{z+\cu_n^+, \mathbf{N}_{z,n}(1-\eta),(1-\eta)\mres (\cu_n^+)^c}\Rr\vert^2}_{\rho}
		\\= &\bracket{\frac{\vert \cu_n^+ \vert }{2 \left(|\cu_n^+|- \mathbf{N}_{z,n}(\eta)\right)} \Ll\vert\frac{1}{\vert \cu_n \vert}\sum_{x \in z+\cu_n} \bracket{\nabla_{x} u}_{z+\cu_n^+,\mathbf{N}_{z,n}(\eta), \eta \mres (\cu_n^+)^c}\Rr\vert^2}_{\rho}.
	\end{align*}
	Here from the first line to the second line, we use the observation that $\eta$ has the law $\P_\rho$ if and only if $1-\eta$ has the law $\P_{1-\rho}$. From the second line to the third line, we use the similar observation that $\eta$ has the law $\P_{\Lambda, N,\zeta}$ if and only if $1-\eta$ has the law $\P_{\Lambda, |\Lambda|-N,1-\zeta}$. This gives us a flipped version of previous result:
	\begin{multline}\label{eq.WMPoincareflip}
		\bracket{\frac{1}{|\cu_{m}|} u^2}_\rho^\frac{1}{2}  \\
		\leq C\left(\frac{-\log\rho}{1-\rho}\right)^{\frac{1}{2}} \sum_{n=0}^{m} 3^n \Ll( \frac{1}{\vert \Z_{m,n}\vert}\sum_{z \in \Z_{m,n}}\bracket{\frac{\vert \cu_n^+ \vert }{2 (|\cu_n^+|-\mathbf{N}_{z,n})} \Ll\vert\frac{1}{\vert \cu_n \vert}\sum_{x \in z+\cu_n} \bracket{\nabla_{x}u}_{z+\cu_n^+, \mathbf{N}_{z,n}}\Rr\vert^2}_{\rho} \Rr)^{\frac{1}{2}}.
	\end{multline}
	
	We combine the two results, and observe that the factor $\frac{-\log(1-\rho)}{\rho}$ is bounded on $[0,\frac{1}{2}]$, while $\frac{-\log\rho}{1-\rho}$ is bounded on $[\frac{1}{2},1]$, we can combine the two estimates into a single one and drop the dependence on $\rho$. We use $\eqref{eq.WMPoincare}$ for $\rho\leq\frac{1}{2}$ and \eqref{eq.WMPoincareflip} for $\rho>\frac{1}{2}$, and also simplify the notation using the conditional expectation with respect to $\G_{z+\cu_n^+}$. This proves \eqref{eq.WMPoincare2}.
\end{proof}

\section{Subadditive quantities}\label{sec.Sub}

In this section, we start the renormalization approach. We define at first several subadditive quantities and develop their elementary properties.

\subsection{Subadditive quantities $\nub$ and $\nub_*$}
For every finite set $\Lambda\subset\Zd$ and $p,q\in\Rd$, we define the quantities
\begin{equation}\label{eq.defNu}
	\begin{split}
		\nub(\rho,\Lambda,p) &:= \inf_{v\in\ell_{p,\Lambda^+} +\F_0(\Lambda^-)} \Ll\{ \frac{1}{2 \chi(\rho)\vert\Lambda\vert} \sum_{b\in\ov{\Lambda^*}} \bracket{ \frac{1}{2}c_b(\pi_b v)^2}_\rho \Rr\},\\
		\nub_*(\rho,\Lambda,q) &:=\sup_{v \in  \F_0} \Ll\{ \frac{1}{2 \chi(\rho)\vert\Lambda\vert}\sum_{b\in\ov{\Lambda^*}}  \bracket{ (\pi_b \ell_{q,\Lambda^+})(\pi_b v) - \frac{1}{2} c_b(\pi_b v)^2}_{\rho}\Rr\}.
	\end{split}
\end{equation}
Recall that the affine function $\ell_{p,\Lambda^+}$ defined in \eqref{eq.defAffineRigorous}, and $\Lambda^-$, $\ov{\Lambda^*}$ defined respectively in \eqref{eq.defMinus}, \eqref{eq.defBondLarge}. Compared to \eqref{eq.defC2}, here we add the factor $\chi(\rho)$ in the normalization,  which will make the notation lighter in the homogenization step.

In the following context, we also use $\ell_p$ defined in \eqref{eq.defAffineFormal} instead of $\ell_{p,\Lambda^+}$ to simplify the notation. This formal notation has well-defined gradients and one should consider it as $\ell_{p,\Lambda^+}$ when talking about the function itself.

We record some elementary properties satisfied by $\nub$ and $\nub_*$.

\begin{proposition}[Elementary properties of $\nub$ and $\nub_*$]\label{prop.Element}
	The following properties hold for every bounded $\Lambda\subset\Zd$ and $p,p',q,q'\in\Rd$.
	
	(1) There exists a unique solution for the optimization problem of $\nub(\rho,\Lambda,p)$ satisfying $\bracket{v-\ell_{p, \Lambda^+}}_\rho=0$; we denote it by $v(\cdot,\rho,\Lambda,p)$. For the optimization problem of $\nub_*(\rho,\Lambda.q)$, there exists a unique maximizer $u(\cdot,\Lambda,q)$ being independent of  $\rho$ and belonging to $\F_0(N_\r(\Lambda^+))$ satisfying $\mathbb{E}_\rho[u|\mathcal{G}_{\Lambda^+}]=0$, where $\G_{\Lambda^+}$ is defined in \eqref{eq.defGTildeKawa}	and $N_\r(\Lambda^+)$ is defined as
	\begin{equation*}
		N_\r(\Lambda^+):=\{x\in\mathbb{Z}^d:  \dist(x,\Lambda^+)\leq\r\}.
	\end{equation*}
	Moreover, the two optimizers are both harmonic in the sense $v(\cdot,\rho,\Lambda,p), u(\cdot,\Lambda,q)\in\mcl A(\Lambda)$, where $\mcl A(\Lambda)$ is defined in \eqref{eq.defHarmonic}.
	
	(2) There exist two $d\times d$ positive symmetric matrices $\D(\rho,\Lambda)$ and $\D_*(\rho,\Lambda)$ such that for every $p,q\in\mathbb{R}^d$
	\begin{align}\label{eq.quadraticNu}
		\nub(\rho,\Lambda,p)=\frac{1}{2}p\cdot \D(\rho,\Lambda) p, \qquad  \nub_*(\rho,\Lambda,q)=\frac{1}{2}q\cdot \D_*^{-1}(\rho,\Lambda) q       
	\end{align}
	and these matrices satisfy $\id \leq \D(\rho,\Lambda), \D_*(\rho,\Lambda) \leq \lambda \id$. Moreover, for ever $p', q' \in \Rd$, we have
	\begin{equation}\label{eq.bilinearNu}
		\begin{split}
			p'\cdot \D(\rho,\Lambda) p &=\bracket{\frac{1}{2\chi(\rho)|\Lambda|}\sum_{b\in\ov{\Lambda^*}}c_b (\pi_b \ell_{p'})(\pi_b v(\cdot,\rho,\Lambda,p))}_\rho, \\
			q'\cdot \D_*^{-1}(\rho,\Lambda) q &=\bracket{\frac{1}{2\chi(\rho)|\Lambda|}\sum_{b\in\ov{\Lambda^*}}(\pi_b \ell_{q'})(\pi_b u(\cdot,\Lambda,q))}_\rho.
		\end{split}
	\end{equation}

	(3) For every $v'\in\ell_{p,\Lambda^+}+\F_0(\Lambda^-)$, we have 
	\begin{equation} \label{eq.qrnu}
		\bracket{\frac{1}{2\chi(\rho)|\Lambda|}\sum_{b\in\ov{\Lambda^*}}\frac{1}{2}c_b(\pi_b(v-v'))^2}_\rho=\bracket{\frac{1}{2\chi(\rho)|\Lambda|}\sum_{b\in\ov{\Lambda^*}}\frac{1}{2}c_b(\pi_bv')^2}_\rho-\nub(\rho,\Lambda,p),
	\end{equation} 
	where $v = v(\cdot,\rho,\Lambda,p)$ is the minimizer defined in (1).
	
	Similarly, for every $u'\in\F_0$ and the maximizer $u = u(\cdot, \Lambda, q)$ defined in (1), we have 
	\begin{multline} \label{eq.qrnu*}
		\bracket{\frac{1}{2\chi(\rho)|\Lambda|}\sum_{b\in\ov{\Lambda^*}} \frac{1}{2}c_b(\pi_b(u-u'))^2}_\rho
		\\=\nub_*(\rho,\Lambda,q)-\bracket{\frac{1}{2\chi(\rho)|\Lambda|}\sum_{b\in\ov{\Lambda^*}}\Ll(-\frac{1}{2}c_b(\pi_b u')^2+(\pi_b \ell_{q})(\pi_b u')\Rr)}_\rho.
	\end{multline}
	
	(4) For any partition of vertices $\Lambda=\bigsqcup_{i=1}^m\Lambda_i$, we have 
	\begin{align}
		\nub(\rho,\Lambda,p) &\leq\sum_{i=1}^m\frac{|\Lambda_i|}{|\Lambda|}\nub(\rho,\Lambda_i,p), \label{eq.subadditivenu}\\
		\nub_*(\rho,\Lambda,p) &\leq\sum_{i=1}^m\frac{|\Lambda_i|}{|\Lambda|}\nub_*(\rho,\Lambda_i,p). \label{eq.subadditivenu*}
	\end{align}
	In particular, we have $\nub(\rho,\cu_{n+1},q)\leq\nub(\rho,\cu_n,q)$ and $\nub_*(\rho,\cu_{n+1},q)\leq\nub_*(\rho,\cu_n,q)$ for every $n \in \N_+$.

\end{proposition}

\begin{proof}
	The proof of this proposition is elementary and standard. We give the complete proof of (1), which concerns some details. For other statements, we sketch the main idea of the proof. Readers may look for details in a similar setting in \cite[Proposition~4.1]{bulk}.
	
	(1) We first study $\nub(\rho,\Lambda,p)$. By a variational calculus, the minimizer can be characterized by the equation that for any $\phi\in\F_0(\Lambda^-)$,
	\begin{equation}\label{eq.variationnu}
		\bracket{\sum_{b\in\ov{\Lambda^*}}c_b(\pi_b(v-\ell_{p,\Lambda^+}))(\pi_b\phi)}_\rho=\bracket{\sum_{b\in\ov{\Lambda^*}}c_b(\pi_b(-\ell_{p,\Lambda^+}))(\pi_b\phi)}_\rho.
	\end{equation}
	We can define its solution in the space 
	\begin{equation*}
		V=\{f\in\F_0(\Lambda^-):\bracket{f}_\rho=0\}.
	\end{equation*}
	The coercivity is ensured by Poincar\'e inequality \eqref{eq.spectralGradient} and we can apply Lax--Milgram theorem to get the unique minimizer $v(\cdot,\rho,\Lambda,p)$. Moreover, \eqref{eq.variationnu} implies that $v(\cdot,\rho,\Lambda,p) \in\mcl A(\Lambda)$.

	Then we turn to $\nub_*(\rho,\Lambda,q)$. A first observation is that the maximizer can be found in $\F_0(N_\r(\Lambda^+))$. Because for any $u\in\F_0$, its conditional expectation $\A_{N_\r(\Lambda^+)}u$ (defined in \eqref{eq.defA}) reaches a larger value for the functional. More precisely, 
	\begin{align}
		\begin{split}
			&\bracket{\sum_{b\in\ov{\Lambda^*}}\left(-\frac{1}{2}c_b(\pi_b\A_{N_\r(\Lambda^+)}u)^2+(\pi_b \ell_{q})(\pi_b\A_{N_\r(\Lambda^+)}u\right)}_\rho
			\\&=\bracket{\sum_{b\in\ov{\Lambda^*}}\left(-\frac{1}{2}c_b(\A_{N_\r(\Lambda^+)} \pi_bu)^2+(\pi_b\ell_{q})(\pi_b\A_{N_\r(\Lambda^+)}u\right)}_\rho
			\\&\geq \bracket{\sum_{b\in\ov{\Lambda^*}}\left(-\frac{1}{2}c_b(\pi_bu)^2+(\pi_b\ell_{q})(\pi_bu)\right)}_\rho,
		\end{split}
	\end{align}
	where we use the locality of $c_b$ and \eqref{eq.AdInOut} from the first line to the second line, and Jensen's inequality from the second line to the third line.
	
	Similar to the discussion above, the maximizer of the functional can be characterized by the variational equation that for any $\phi\in\F_0$,
	\begin{align}\label{eq.variationnu*}
		\bracket{\sum_{b\in\ov{\Lambda^*}}c_b(\pi_bu)(\pi_b\phi)}_\rho=\bracket{\sum_{b\in\ov{\Lambda^*}}(\pi_b\ell_{q})(\pi_b\phi)}_\rho
	\end{align}
	Notice that for any $\mathcal{G}_{\Lambda^+}$ measurable function $\phi'$, the difference $\pi_b\phi'$ becomes zero and thus the above equation automatically holds. By replacing $\phi$ by $\phi-\mathbb{E}_\rho[\phi|\mathcal{G}_{\Lambda^+}]$ we may only solve the problem for all $\phi$ in the space $W=\{ f\in\F_0: \mathbb{E}_\rho[f|\mathcal{G}_{\Lambda^+}]=0\}$. Moreover, testing the equation with $\phi\mathbf{1}_{\{\sum_{x\in \Lambda^+}\eta_x=N\}}\mathbf{1}_{\{\eta_y=\epsilon_y, \text{ for finitely many } y\in (\Lambda^+)^c\}}$ for arbitrary $N\in\mathbb{N}$ and $\epsilon_y\in\{0,1\}$, we actually reinforce the equation \eqref{eq.variationnu*} in a stronger way
	\begin{align}\label{eq.variationnu*2}
		\mathbb{E}_\rho\left[\sum_{b\in\ovs\Lambda}c_b(\pi_bu)(\pi_b\phi)\ |\ \mathcal{G}_{\Lambda^+}\right]=\mathbb{E}_\rho\left[\sum_{b\in\ovs\Lambda}(\pi_b\ell_{q})(\pi_b\phi)\ |\ \mathcal{G}_{\Lambda^+}\right],
	\end{align} 
	and we may seek for the solution in the space $W$. In this space the Poincar\'e inequality \eqref{eq.Poingrand} ensures the coercivity, so Lax--Milgram theorem applies and we get the unique maximizer $u(\cdot,\Lambda,q)$. We emphasize that after taking conditional expectation with respect to proper $\G_{\Lambda^+}$, the equation \eqref{eq.variationnu*2} is independent of the choice of $\rho$. Thus, the optimizer is found separately for every conditional expectation $\mathbb{E}[\cdot|\G_{\Lambda^+}]$ and in particular, $u$ is an optimizer for all $\rho$. Moreover, testing \eqref{eq.variationnu*} with $\phi\in\F_0(\Lambda^-)$, the right hand side becomes zero thanks to \eqref{eq.MeanZeroLocal} and we get $u\in\mcl A(\Lambda)$.
	
	(2) We test \eqref{eq.variationnu} with $v(\cdot,\rho,\Lambda,p')-\ell_{p',\Lambda^+}$ and get 
	\begin{equation}
		\bracket{\sum_{b\in\ov{\Lambda^*}}c_b(\pi_bv(\cdot,\rho,\Lambda,p))(\pi_bv(\cdot,\rho,\Lambda,p'))}_\rho=\bracket{\sum_{b\in\ov{\Lambda^*}}c_b(\pi_bv(\cdot,\rho,\Lambda,p))(\pi_b\ell_{p',\Lambda^+})}_\rho.
	\end{equation}
	In particular, the above formula is linear in $p'$. For the same reason, it is also linear in $p$. In particular, $\nub$ is quadratic in $p$.
	To obtain the bound of $\D(\rho,\Lambda)$, we use the condition $1 \leq c_b \leq \lambda$ in Hypothesis~\ref{hyp}, 
	\begin{equation*}
		\begin{split}
			&\inf_{v\in \ell_{p,\Lambda^+}+\F_0(\Lambda^-)}\bracket{\frac{1}{4\chi(\rho)|\Lambda|}\sum_{b\in\ov{\Lambda^*}}(\pi_b v)^2}_\rho
			\\&\leq\inf_{v\in \ell_{p,\Lambda^+}+\F_0(\Lambda^-)}\bracket{\frac{1}{4\chi(\rho)|\Lambda|}\sum_{b\in\ov{\Lambda^*}}c_b(\pi_b v)^2}_\rho=\frac{1}{2}p\cdot \D(\rho,\Lambda)\cdot p
			\\&\leq\inf_{v\in \ell_{p,\Lambda^+}+\F_0(\Lambda^-)}\bracket{\frac{1}{4\chi(\rho)|\Lambda|}\sum_{b\in\ov{\Lambda^*}}\lambda(\pi_b v)^2}_\rho.
		\end{split}
	\end{equation*}
	We can observe that $\ell_{p,\Lambda^+}$ is the minimizer for $\inf_{v\in \ell_{p,\Lambda}+\F_0(\Lambda^-)}\bracket{\frac{1}{4\chi(\rho)|\Lambda|}\sum_{b\in\Lambda^*}(\pi_b v)^2}_\rho$, whose energy is precisely $\vert p\vert^2$. One may see this from the variational characterization \eqref{eq.variationnu}. This proves the first part of the proposition.
	
	The similar argument works for $\nub_*(\rho,\Lambda,q)$. Testing \eqref{eq.variationnu*} with $u(\cdot,\Lambda,q')$ yields the linearity. 
	Concerning the bound for $\D_*$, we use the bound for $c_b$ to obtain
	\begin{align*}
		&\sup_{v\in\F_0}\frac{1}{2\chi(\rho)|\Lambda|}\bracket{-\frac{\lambda}{2}\sum_{b\in\Lambda^*}(\pi_b v)^2+\sum_{b\in\Lambda^*}(\pi_b \ell_{q,\Lambda^+})(\pi_b v)}_\rho
		\\&\leq \nub_*(\rho,\Lambda,q)
		\\ &\leq \sup_{v\in\F_0}\frac{1}{2\chi(\rho)|\Lambda|}\bracket{-\frac{1}{2}\sum_{b\in\Lambda^*}(\pi_b v)^2+\sum_{b\in\Lambda^*}(\pi_b \ell_{q,\Lambda^+})(\pi_b v)}_\rho.
	\end{align*}
	One can see in the lower bound, the maximizer is $\ell_{\frac{q}{\lambda}.\Lambda^+}$, while in the upper bound, the maximizer is $\ell_{q,\Lambda^+}$. This gives the bound we want.

	(3) This is a direct calculation. We test \eqref{eq.variationnu} with $(v'-\ell_{p,\Lambda^+})$ to get the first equation, and test the equation \eqref{eq.variationnu*} with $u'$ to get the second equation.
	
	(4) For the quantity $\nub(\rho,\Lambda,p)$, 
	\begin{equation*}
		v'= \ell_{p,\Lambda^+} + \sum_{i=1}^n (v(\cdot,\rho,\Lambda_i,p)-\ell_{p,\Lambda_i^+})
	\end{equation*} 
	is a sub-minimizer of $\nub(\rho,\Lambda,p)$, and we use it to prove the sub-additivity of $\nub$. Concerning the quantity $\nub_*(\rho,\Lambda,q)$, we can not ``glue'' the local optimizers, but we use the fact that $u(\cdot,\Lambda,q)$ is a sub-maximizer for every $\nub_*(\rho,\Lambda_i,q)$ to get the sub-additivity of $\nub_*$. We also highlight the identity \eqref{eq.CubeRenormalization2}, which helps avoid the boundary layer and is the motivation we use $\ovs{\Lambda}$ in the definition \eqref{eq.defNu}.
	
\end{proof}
\begin{remark}\label{rmk.DualCanonical}
	Intuitively, the last part of the proof of item (1) actually says that the maximizer $u(\cdot,\Lambda,q)$ can be found in the following way: first, fix the number of particles inside $\Lambda^+$ (denoted by $n$) and the environment outside $\Lambda^+$ (denoted by $\zeta$), and seek for a maximizer $u_{n,\zeta}$ in this condition. Then the maximizer is exactly obtained by pasting all the $u_{n,\zeta}$'s.	
\end{remark}

Viewing the sub-additivity \eqref{eq.subadditivenu}, we have the following corollary.
\begin{corollary}\label{cor.defDLimit}
	For every $\rho \in (0,1)$, the following limit is well-defined
	\begin{align}\label{eq.defDLimit}
		\D(\rho) := \lim_{m \to \infty} \D(\rho, \cu_m).
	\end{align}
\end{corollary}

Recall the $L^\infty$ norm over $\X$ defined as $\norm{F}_{\infty}:= \sup_{\eta \in \X} \vert F(\eta)\vert$. Here we give an upper bound estimate of $v(\cdot, \rho, \Lambda, p)$ and $u(\cdot, \Lambda, q)$ defined in Proposition~\ref{prop.Element}.
\begin{lemma}[$L^{\infty}$ estimate]\label{lem.supNorm}
	There exists a constant $C(\lambda, d)$ such that for any connected domain $\Lambda$ of diameter $L$ and $p,q \in B_1:=\{p\in\Rd, |p|=1\}$, the following estimate is valid:
	\begin{align}\label{eq.supNorm}
		\norm{v(\cdot, \rho, \Lambda, p)}_{\infty} + \norm{u(\cdot, \Lambda, q)}_{\infty} \leq C L^{d+2} \log L.
	\end{align}
\end{lemma}
The proof relies on the mixing time of our non-gradient dynamic; see Appendix~\ref{appendix.B}. Note that for $1 \leq s < \infty$, we have a better estimate that 
\begin{align*}
	\norm{v(\cdot, \rho, \Lambda, p)}_{s} + \norm{u(\cdot, \Lambda, q)}_{s} \leq C L^{d+2}.
\end{align*}
See Remark~\ref{rmk.Lp} for details.

\subsection{Master quantities $J$}
We continue to explore the dual property between $\nub$ and $\nub_*$ in this subsection. For every bonded domain $\Lambda\subset\Zd$ and $p,q\in\Rd$, we define the quantities
\begin{equation} \label{eq.defJ}
	J(\rho,\Lambda,p,q):=\nub(\rho,\Lambda,p)+\nub_*(\rho,\Lambda,q)- p\cdot q.
\end{equation}
We first describe $J$ with a variational formula.
\begin{lemma} \label{lem.repJ}
	(1) For each $p,q\in\mathbb{R}^d$, we have the variational representation
	\begin{align} \label{eq.varJ}
		J(\rho,\Lambda,p,q) &:= \sup_{w\in\mcl A(\Lambda)}J(\rho,\Lambda,p,q; w),
	\end{align}
	where $J(\rho,\Lambda,p,q; w)$ is a functional defined as 
	\begin{multline}\label{eq.defJFunctional}
		J(\rho,\Lambda,p,q; w) \\
		:= \bracket{\frac{1}{2\chi(\rho)|\Lambda|}\sum_{b\in\ov{\Lambda^*}}\Ll(-\frac{1}{2}c_b(\pi_b w)^2-c_b(\pi_b\ell_{p})(\pi_b w)+(\pi_b\ell_{q})(\pi_b w)\Rr)}_\rho. 
	\end{multline}
	One maximizer is $(u(\cdot, \Lambda, q)-v(\cdot, \rho, \Lambda,p))$ with $v,u$ defined in (1) of Proposition~\ref{prop.Element}.
	
	(2) The master quantity $J$ is always positive, i.e.  $J(\rho,\Lambda,p,q) \geq 0$.
\end{lemma}

\begin{proof}
	(1) In the following paragraph, we use $v(\cdot,\rho,\Lambda,p)$ to denote the minimizer in the definition of $\nub$ and write $v=v(\cdot,\rho,\Lambda,p)$ for short. Since we have deduced that the maximizer of $\nub_*$ can be found in $\mcl A(\Lambda)$ in (1) of Proposition~\ref{prop.Element}, we may write 
	\begin{align*}
		J(\rho,\Lambda,p,q)=&\bracket{\frac{1}{4\chi(\rho)|\Lambda|}\sum_{b\in\ov{\Lambda^*}}c_b(\pi_b v)^2}_\rho
		\\&+\sup_{u\in\mcl A(\Lambda)}\bracket{\frac{1}{2\chi(\rho)|\Lambda|}\sum_{b\in\ov{\Lambda^*}}\Ll(-\frac{1}{2}c_b(\pi_b u)^2+(\pi_b\ell_{q})(\pi_b u)\Rr)}_\rho
		\\&-p\cdot q.
	\end{align*}
	Since $(v-\ell_{p,\Lambda^+})\in\F_0(\Lambda^-)$, we have 
	\begin{equation}\label{eq.MeanZerov}
		\bracket{\sum_{b\in\ov{\Lambda^*}}(\pi_b (v-\ell_{p,\Lambda^+}))(\pi_b \ell_{q,\Lambda^+})}_\rho = \sum_{x \in \Lambda}\bracket{q \cdot \nabla_x (v-\ell_{p,\Lambda^+})}_\rho = 0.
	\end{equation}
	Here the first equality comes from \eqref{eq.defKawaTagent}, \eqref{eq.defKawaGradient}, \eqref{eq.IdGradient} and the second equality comes from \eqref{eq.MeanZeroLocal}. Moreover, for any $u\in\mcl A(\Lambda)$, by its definition \eqref{eq.defHarmonic} and noting $(v-\ell_{p,\Lambda^+})\in\F_0(\Lambda^-)$ again, we have
	\begin{equation*}
		\bracket{\sum_{b\in\ov{\Lambda^*}}c_b(\pi_b u)(\pi_b v)}_\rho=\bracket{\sum_{b\in\ovs{\Lambda}}c_b(\pi_b u)(\pi_b \ell_{p,\Lambda^+})}_\rho,
	\end{equation*}
	and in particular this equation is true when taking $u=v$.
	
	Combining these results, we obtain
	\begin{align*}
		&J(\rho,\Lambda,p,q)\\
		&=\sup_{u\in\mcl A(\Lambda)}\Ll(\bracket{\frac{1}{4\chi(\rho)|\Lambda|}\sum_{b\in\ov{\Lambda^*}}c_b(\pi_b v)^2}_\rho+\bracket{\frac{1}{2\chi(\rho)|\Lambda|}\sum_{b\in\ovs{\Lambda}}\Ll(-\frac{1}{2}c_b(\pi_b u)^2+(\pi_b\ell_{q})(\pi_b u)\Rr)}_\rho \Rr.\\
		&\quad -\Ll.\bracket{\frac{1}{2\chi(\rho)|\Lambda|}\sum_{b\in\Lambda^*}(\pi_b\ell_{p})(\pi_b\ell_{q})}_\rho\Rr)\\
		&=\sup_{u\in\mcl A(\Lambda)}\bracket{\frac{1}{2\chi(\rho)|\Lambda|}\sum_{b\in\ovs{\Lambda}}\Ll(-\frac{1}{2}c_b(\pi_b(u-v))^2-c_b(\pi_b\ell_{p})\pi_b(u-v)+(\pi_b\ell_{q})\pi_b(u-v)\Rr)}_\rho\\
		&=\sup_{w\in\mcl A(\Lambda)}\bracket{\frac{1}{2\chi(\rho)|\Lambda|}\sum_{b\in\ovs{\Lambda}}\Ll(-\frac{1}{2}c_b(\pi_bw)^2-c_b(\pi_b\ell_{p})(\pi_bw)+(\pi_b\ell_{q})(\pi_b w)\Rr)}_\rho.
	\end{align*}
	From the definition of $u(\cdot,\Lambda,q)$, we conclude that $w=u(\cdot,\Lambda,q)-v(\cdot,\rho,\Lambda,p)$ is exactly a maximizer.
	
	(2) We test the functional in the definition of $\nub_*(\rho,\Lambda,q)$ with the minimizer $v=v(\cdot,\rho,\Lambda,p)$ of $\nub(\rho,\Lambda,p)$ and obtain
	\begin{align*}
		\nub_*(\rho,\Lambda,q)&\geq\bracket{\frac{1}{2\chi(\rho)|\Lambda|}\sum_{b\in\ovs{\Lambda}}\Ll(-\frac{1}{2}c_b(\pi_b v)^2+(\pi_b\ell_{q})(\pi_b v)\Rr)}_\rho\\
		&=-\nub(\rho,\Lambda,p)+\bracket{\frac{1}{2\chi(\rho)|\Lambda|}\sum_{b\in \ovs{\Lambda}}(\pi_b\ell_{q})(\pi_b \ell_{p})}_\rho \\
		&=-\nub(\rho,\Lambda,p)+ p \cdot q, 
	\end{align*}
	which proves $J(\rho,\Lambda,p,q)\geq0$. Here from the first line to the second line, we also use the identity \eqref{eq.MeanZerov}.
\end{proof}

Now we explain why $J$ is convenient in estimating the convergence rate. 

\begin{lemma} \label{lem.ratecontrolJ}
	(1) For any bounded $\Lambda \subset \Zd$ and $\rho \in (0,1)$, we have $\D(\rho,\Lambda)\geq \D_*(\rho,\Lambda)$.
	
	(2) There exists a constant $C(d,\lambda)$ such that for every symmetric matrix $\widetilde D$,
	we have 
	\begin{equation} \label{eq.ratecontrolJ}
		|\widetilde D-\D(\rho,\Lambda)|+|\widetilde D-\D_*(\rho,\Lambda)|\leq C\sup_{|p|=1}J(\rho,\Lambda,p,\widetilde Dp)^{\frac{1}{2}}.
	\end{equation}
\end{lemma}

\begin{proof}
	(1) Recall $J(\rho,\Lambda,p,q)\geq0$ from (2) of Lemma~\ref{lem.repJ}, and insert $q=\D_*(\rho,\Lambda)p$ to obtain
	\begin{align*}
		0\leq &J(\rho,\Lambda,p,\D_*(\rho,\Lambda)p)
		\\=&\frac{1}{2}p\cdot \D(\rho,\Lambda)p+\frac{1}{2}p\cdot \D_*(\rho,\Lambda) p-p\cdot \D_*(\rho,\Lambda) p
		\\=&\frac{1}{2}p\cdot \D(\rho,\Lambda) p-\frac{1}{2}p\cdot \D_*(\rho,\Lambda) p,
	\end{align*}
	so we have $\D(\rho,\Lambda)\geq\D_*(\rho,\Lambda)$.
	
	(2) Using the property $\D(\rho,\Lambda)\geq\D_*(\rho,\Lambda)$, we have 
	\begin{align*}
		J(\rho,\Lambda,p,q)=&\frac{1}{2}p\cdot \D(\rho,\Lambda) p+\frac{1}{2}q\cdot \D_*^{-1}(\rho,\Lambda) q-p\cdot q
		\\ \geq&\frac{1}{2}p\cdot \D(\rho,\Lambda) p+\frac{1}{2} q\cdot \D^{-1}(\rho,\Lambda) q-p\cdot q
		\\=&\frac{1}{2}(\D(\rho,\Lambda)p-q)\cdot \D^{-1}(\rho,\Lambda)(\D(\rho,\Lambda)p-q).
	\end{align*}
	Setting $|p|=1$ and $q=\widetilde D p$, we conclude that
	\begin{equation*}
		|\D(\rho,\Lambda)-\widetilde D|\leq C\sup_{|p|=1}J(\rho,\Lambda,p,\widetilde D p)^{\frac{1}{2}}.
	\end{equation*}
	The proof of the statement concerning $|\D_*(\rho,\Lambda)-\widetilde D|$ is similar.
\end{proof}

This lemma allows us to control the convergence rate of $ \D(\rho,\cu_m)$ by showing the convergence rate of $J(\rho,\cu_m,p,\D_*(\rho,\cu_m)p)$. Finally, we summarize some more properties of $J$ similar to the properties for $\nub$ and $\nub_*$.

\begin{proposition}[Elementary properties of $J$]\label{prop.ElementJ}
	For every $\rho\in(0,1)$, every bounded $\Lambda \subsetneq \mathbb{Z}^d$ and every $p,q\in\mathbb{R}^d$, the quantity $J(\rho,\Lambda,p,q)$ defined as \eqref{eq.defJ} satisfies the following properties:
	
	(1) First order variation and optimizer: the optimization problem in \eqref{eq.varJ} admits a unique solution $v(\cdot,\rho,\Lambda,p,q)\in\mcl A(\Lambda)$ such that $\mathbb{E}_\rho[v(\cdot,\rho,\Lambda.p,q)|\mathcal{G}_{\Lambda^+}]=0$, which can be expressed in terms of the optimizers of $\nub$ and $\nub_*$ as
	\begin{equation}\label{eq.defvJ}
		v(\cdot,\rho,\Lambda,p,q)=u(\cdot,\Lambda,q)-v(\cdot,\rho,\Lambda,q)-\mathbb{E}_\rho[u(\cdot,\Lambda,q)-v(\cdot,\rho,\Lambda,p)|\mcl G_{\Lambda^+}].
	\end{equation}
	This solution $v(\cdot,\rho,\Lambda,p,q)$ satisfies that for every $w\in\mcl A(\Lambda)$,
	\begin{equation}\label{eq.var2J}
		\bracket{\sum_{b\in\ovs\Lambda}(c_b\pi_b v(\cdot,\rho,\Lambda,p,q))(\pi_b w)}_\rho=\bracket{\sum_{b\in\ovs\Lambda}(-c_b(\pi_b\ell_{p})(\pi_b w)+(\pi_b\ell_{q})(\pi_b w))}_\rho,
	\end{equation}
	and the mapping $(p,q)\mapsto v(\cdot,\rho,\Lambda,p,q)$ is linear.

	(2) Quadratic response:	we have a quadratic expression for $J$
	\begin{equation}\label{eq.JQuadratic}
		J(\rho,\Lambda,p,q)=\bracket{\frac{1}{4\chi(\rho)|\Lambda|}\sum_{b\in\ovs\Lambda}c_b(\pi_bv(\cdot,\rho,\Lambda,p,q))^2}_\rho.
	\end{equation}
	For every $w\in\mcl A(\Lambda)$, with the functional defined in \eqref{eq.defJFunctional}, we have
	\begin{align} \label{eq.quadraresponseJ}
		\bracket{\frac{1}{4\chi(\rho)|\Lambda|}\sum_{b\in\ovs\Lambda}c_b(\pi_b(w-v(\cdot,\rho,\Lambda,p,q)))^2}_\rho
		=J(\rho,\Lambda,p,q)-J(\rho,\Lambda,p,q;w).
	\end{align}
	
	(3) For any partition of vertices $\Lambda=\bigsqcup_{i=1}^m\Lambda_i$, we have 
	\begin{equation}\label{eq.Jsub}
		J(\rho,\Lambda,p,q)\leq\sum_{i=1}^m\frac{|\Lambda_i|}{|\Lambda|}J(\rho,\Lambda_i,p,q).
	\end{equation}
	
	(4) Slope property: with the gradient operator $\nabla_x$ defined in \eqref{eq.defKawaGradient}, the optimizer $v(\cdot,\rho,\Lambda,p,q)$ satisfies the slop property
	\begin{align}\label{eq.SlopeJ}
		\bracket{\frac{1}{2\chi(\rho)|\Lambda|} \sum_{x\in\Lambda}\nabla_x v(\cdot,\rho,\Lambda,p,q)}_\rho = \D_*^{-1}(\rho, \Lambda) q- p.
	\end{align}
\end{proposition}

\begin{proof}
	(1) The equation \eqref{eq.var2J} comes directly from the first order variation calculus. The proof of existence and uniqueness of the solution $v(\cdot,\rho,\Lambda,p,q)$ is similar to the one for $\nub_*(\rho,\Lambda,q)$.        
	
	One can check directly that $v(\cdot,\rho,\Lambda,p_1,q_1)+v(\cdot,\rho,\Lambda,p_2,q_2)$ is the solution for the problem \eqref{eq.var2J} with parameter $(p_1+p_2,q_1+q_2)$ and it also satisfies the conditional expectation condition, so we have $v(\cdot,\rho,\Lambda,p_1+p_2,q_1+q_2)=v(\cdot,\rho,\Lambda,p_1,q_1)+v(\cdot,\rho,\Lambda,p_2,q_2)$ by the uniqueness, which implies that the mapping $(p,q)\mapsto v(\cdot,\rho,\Lambda,p,q)$ is linear. 
	
	The exact expression \eqref{eq.defvJ} follows directly from the fact that $\Ll(u(\cdot,\Lambda,q)-v(\cdot,\rho,\Lambda,p)\Rr)$ is a maximizer in \eqref{eq.varJ}, and we add the necessary regularization condition.

	(2)  We put $v=v(\cdot,\rho,\Lambda,p,q)$ in the first order variation \eqref{eq.var2J} to get
	\begin{align*}
		\bracket{\sum_{b\in\Lambda^*}(c_b(\pi_b v)^2+c_b(\pi_b\ell_{p})(\pi_b v)-(\pi_b \ell_{q,\Lambda^+})(\pi_b v))}_\rho=0.
	\end{align*}
	Then we plug this into \eqref{eq.varJ} to get the quadratic expression. The quadratic response \eqref{eq.quadraresponseJ} follows directly from \eqref{eq.qrnu*} by testing with $u=w-v(\cdot,\rho,\Lambda,p)$.
	
	(3) This is a consequence of \eqref{eq.subadditivenu} and \eqref{eq.subadditivenu*}.
	
	(4) Using the exact expression of $v(\cdot,\rho,\Lambda,p,q)$ in \eqref{eq.defvJ}, we have
	\begin{align*}
		&\bracket{\frac{1}{2\chi(\rho)|\Lambda|} \sum_{x\in\Lambda}\nabla_x v(\cdot,\rho,\Lambda,p,q)}_\rho  \\
		&=  \bracket{\frac{1}{2\chi(\rho)|\Lambda|} \sum_{x\in\Lambda}\nabla_x u(\cdot,\Lambda,q)}_\rho - \bracket{\frac{1}{2\chi(\rho)|\Lambda|} \sum_{x\in\Lambda}\nabla_x v(\cdot,\rho,\Lambda,q)}_\rho \\
		&= \D_*^{-1}(\rho, \Lambda) q- p.
	\end{align*}
	Here in the second line, we apply \eqref{eq.bilinearNu} to the first term and \eqref{eq.MeanZerov} to the second term.
\end{proof}

\section{Renormalization under grand canonical ensemble}\label{sec.Rate}

Based on the diffusion matrix $\D(\rho, \Lambda), \D_*(\rho,\Lambda)$ defined in Proposition~\ref{prop.Element} and using the Einstein relation \eqref{eq.Einstein}, we define the conductivity and its dual quantity on $(0,1)$ as 
\begin{align}\label{eq.defCGrand}
	\cc(\rho,  \Lambda) := 2\chi(\rho) \D(\rho, \Lambda), \qquad \cc_*(\rho, \Lambda) := 2 \chi(\rho) \D_*(\rho,\Lambda).
\end{align}
By default, when $\rho \in \{0,1\}$, we set $\cc(\rho,  \Lambda) = \cc_*(\rho,  \Lambda) = 0$. Especially, $\cc(\rho,  \Lambda)$ coincides with the definition \eqref{eq.defC2}. Then Corollary~\ref{cor.defDLimit} also defines a limit
\begin{align}\label{eq.defCLimit}
	\cc(\rho) := 2\chi(\rho)\D(\rho) = \lim_{m \to \infty} \cc(\rho, \cu_m).
\end{align}
In this section, we are ready to prove the convergence rate.
\begin{proposition}\label{prop.GrandCanonicalEnsemble}
	There exists an exponent $\gamma_1(d, \lambda, \r) > 0$ and a positive constant ${C(d, \lambda, \r) < \infty}$ such that for every $L \in \N_+$,
	\begin{align}\label{eq.main1A_2}
		\sup_{\rho \in [0,1]} \Ll(\Ll\vert \cc(\rho,  \Lambda_L) - \cc(\rho) \Rr\vert + \Ll\vert \cc_*(\rho,  \Lambda_L) - \cc(\rho) \Rr\vert\Rr)  \leq C L^{-\gamma_1}.
	\end{align}
\end{proposition}

Although Proposition~\ref{prop.GrandCanonicalEnsemble} is stated for $\cc(\rho,  \Lambda_L)$ and $\cc_*(\rho,  \Lambda_L)$, we will still rely on  $\D(\rho, \Lambda)$ and $\D_*(\rho,\Lambda)$ in the intermediate steps as they are already normalized; see (2) of Proposition~\ref{prop.Element}. Note that $\chi(\rho)$ in $\cc(\rho,  \Lambda_L)$ helps to prove the uniform convergence in $\rho \in [0,1]$. Throughout this section, we define the following shorthand expression
\begin{align}\label{eq.defDm}
	D_n :=\D_*(\rho,\cu_n).
\end{align}
By the subadditive quantity \eqref{eq.subadditivenu} and $\eqref{eq.subadditivenu*}$, we know that $\D(\rho,\cu_m)$ and $\D_*^{-1}(\rho,\cu_m)$ are decreasing. Therefore, it suffices to show the convergence rate for $|\D_*(\rho,\cu_m)-\D(\rho,\cu_m)|$, which is reduced to the decay rate of $\sup_{|p|=1}J(\rho,\cu_n,p,D_mp)$ after applying \eqref{eq.ratecontrolJ} with $\tilde D=D_m$. 

In this section, we will heavily use the master quantity $J(\rho, \Lambda, p, q)$ defined in \eqref{eq.varJ} and its optimizer $v(\cdot, \rho, \Lambda, p, q)$ defined in (1) of Proposition~\ref{prop.ElementJ}. The notation $\Z_{m,n} = 3^n \Zd \cap \cu_m$ is usually involved to make comparison between $J$ in different scales. We will also use the following gap of the master quantities very often in the proof
\begin{equation}\label{eq.deftau}
	\tau_n=\tau_n(\rho) :=\sup_{p,q\in B_1} (J(\rho,\cu_n,p,q)- J(\rho,\cu_{n+1},p,q)),
\end{equation}
where $B_1 := \{p \in \Rd: \vert p\vert = 1\}$.

Our first lemma makes use of the elliptic regularity, especially the modified Caccioppoli inequality \eqref{eq.Caccioppoli}, in our  master quantity. 
\begin{lemma}\label{lem.Jestimate}
	There exists a finite positive constant $C(d,\lambda)$ such that, for every $m \in \N_+$ satisfying $3^m>R_0$ for the constant in Proposition \ref{prop.Caccioppoli}, we have the following estimate
	\begin{align} \label{eq.Jestimate}
		J(\rho,\cu_m,p,D_mp)^\frac{1}{2} \leq C\tau_m^\frac{1}{2}+C3^{-m}\bracket{\frac{1}{2 \chi(\rho)|\cu_{m+1}|}v(\cdot,\rho,\cu_{m+1},p,D_mp)^2}_\rho^\frac{1}{2}.
	\end{align}
\end{lemma}

\begin{proof}
	We write $v_m=v(\cdot,\rho,\cu_m,p,D_mp)$ as a shortened version in the proof. Recall the operator $\A_L$ defined in \eqref{eq.defA}, and denote by 
	\begin{align}\label{eq.defLocalL}
		L:=3^m + 2\r, 
	\end{align}
	as the length of localization in this proof with $3^m > 2 \r$. Then, by \eqref{eq.JQuadratic},  we decompose our quantity $J(\rho,\cu_m,p,D_mp)$ as 
	\begin{equation}\label{eq.JestimateDecom}
		\begin{split}
			&J(\rho,\cu_m,p,D_mp)^\frac{1}{2}\leq (\mathbf{I})^{\frac{1}{2}} + (\mathbf{II})^{\frac{1}{2}},\\
			\mathbf{I} &:= \bracket{\frac{1}{4\chi(\rho)|\cu_m|}\sum_{b\in\overline{\cu_m^*}}c_b(\pi_b (v_m-\A_Lv_{m+1}))^2}_\rho,\\
			\mathbf{II} &:= \bracket{\frac{1}{4\chi(\rho)|\cu_m|}\sum_{b\in \overline{\cu_m^*}}c_b(\pi_b \A_Lv_{m+1})^2}_\rho.
		\end{split}
	\end{equation}
	
	We justify at first $\A_Lv_{m+1}\in\mathcal A(\cu_m)$ by testing an arbitrary function $f\in\F_0(\cu_{m}^-)$
	\begin{align*}
		\sum_{b\in\overline{\cu_m^*}}\bracket{c_b(\pi_b \A_Lv_{m+1})(\pi_b f)}_\rho &= \sum_{b\in\overline{\cu_m^*}}\bracket{\A_L \Ll(c_b( \pi_bv_{m+1})(\pi_b f)\Rr)}_\rho \\
		&=\sum_{b\in\overline{\cu_m^*}}\bracket{c_b(\pi_bv_{m+1})(\pi_b f)}_\rho=0.
	\end{align*}
	Here in the first line, since $c_b$ and $f$ are both $\fil_{\Lambda_{L}}$-measurable, they commute with $\A_L$ operator. In the second line, we use the fact that $v_{m+1} \in \mcl A(\cu_{m+1}) \subset \mcl A(\cu_{m})$ from (1) of Proposition~\ref{prop.ElementJ}.

	We now turn to the two terms in \eqref{eq.JestimateDecom}. As we have proved $\A_Lv_{m+1}\in\mathcal A(\cu_m)$, we apply \eqref{eq.quadraresponseJ} to $\mathbf{I}$ and get
	\begin{align}\label{eq.JestimateTerm1}
		\mathbf{I} = J(\rho,\cu_m,p,D_mp) - J(\rho,\cu_m,p,D_mp; \A_Lv_{m+1}).
	\end{align}
	We investigate the expression of $J(\rho,\cu_m,p,D_mp; \A_Lv_{m+1})$ by \eqref{eq.defJFunctional}. By Jensen's inequality, we have 
	\begin{align*}
		\bracket{\frac{1}{2\chi(\rho)|\cu_m|}\sum_{b\in\overline{\cu_m^*}}\frac{1}{2}c_b(\pi_b\A_Lv_{m+1})^2}_\rho\leq\bracket{\frac{1}{2\chi(\rho)|\cu_m|}\sum_{b\in\overline{\cu_m^*}}\frac{1}{2}c_b(\pi_bv_{m+1})^2}_\rho,
	\end{align*}
	and by the definition of conditional expectation, we also have
	\begin{align*}
		\bracket{\frac{1}{2\chi(\rho)|\cu_m|}\sum_{b\in\overline{\cu_m^*}}(c_b(\pi_b\ell_{p})(\pi_b\A_Lv_{m+1})-(\pi_b\ell_{D_mp})(\pi_b\A_Lv_{m+1}))}_\rho
		\\=\bracket{\frac{1}{2\chi(\rho)|\cu_m|}\sum_{b\in\overline{\cu_m^*}}(c_b(\pi_b\ell_{p})(\pi_bv_{m+1})-(\pi_b\ell_{D_mp})(\pi_bv_{m+1}))}_\rho.
	\end{align*}
	These imply that $J(\rho,\cu_m,p,D_mp; \A_Lv_{m+1}) \geq J(\rho,\cu_m,p,D_mp; v_{m+1})$. Putting it back to \eqref{eq.JestimateTerm1} and using \eqref{eq.quadraresponseJ} again, we have 
	\begin{align*}
		\mathbf{I} &\leq J(\rho,\cu_m,p,D_mp) - J(\rho,\cu_m,p,D_mp; v_{m+1})\\ &=\bracket{\frac{1}{4\chi(\rho)|\cu_m|}\sum_{b\in\overline{\cu_m^*}}c_b(\pi_b (v_m-v_{m+1}))^2}_\rho.
	\end{align*}
	We claim that the right hand side of the inequality can be controlled by $C\tau_m$ defined in \eqref{eq.deftau} with some constant only depending on the dimension $d$. It suffices to apply \eqref{eq.quadraresponseJ} to $z+\cu_m$ with $z \in \Z_{m+1, m}$. We denote by $v_{m,z} =v(\cdot,\rho,z+\cu_m,p,D_mp)$ to simplify the notation in the following calculation
	\begin{align*} 
		&\bracket{\frac{1}{4\chi(\rho)|\cu_m|}\sum_{b\in\overline{\cu_m^*}}c_b(\pi_b (v_m-v_{m+1}))^2}_\rho\\
		&\leq \sum_{z\in \Z_{m+1, m}}\bracket{\frac{1}{4\chi(\rho)|\cu_m|}\sum_{b\in \ovs{(z+\cu_m)}}c_b(\pi_b (v_{m,z}-v_{m+1}))^2}_\rho\\
		&= \sum_{z\in \Z_{m+1, m}} \Ll(J(\rho,z+\cu_m,p,D_mp) - J(\rho,z+\cu_m,p,D_mp; v_{m+1}) \Rr)\\
		&=3^d(J(\rho,\cu_m,p,D_mp)-J(\rho,\cu_{m+1},p,D_mp)).
	\end{align*}
	This concludes that 
	\begin{align}\label{eq.JestimateDecom1}
		\mathbf{I} \leq 3^d \lambda \tau_m.
	\end{align}
	
	Concerning $\mathbf{II}$, as we have verified the condition in Proposition~\ref{prop.Caccioppoli} that $\A_Lv_{m+1}\in\mathcal A(\cu_m)$ and $3^m > R_0$ is assumed, we can apply the modified Caccioppoli inequality \eqref{eq.Caccioppoli} to  $\A_Lv_{m+1}$ to get
	\begin{align}\label{eq.JestimateDecom2}
		\mathbf{II} \leq \bracket{\frac{C3^{-2m}}{2\chi(\rho)|\cu_{m+1}|}v_{m+1}^2}_\rho+\theta J(\rho,\cu_{m+1},p,D_mp).
	\end{align}
	Here the factor $\theta \in (0,1)$ and the constant $C$ all come from Proposition~\ref{prop.Caccioppoli}, and only depend on $d, \lambda$. 
	
	Combining the two estimates \eqref{eq.JestimateDecom1} and \eqref{eq.JestimateDecom2}, we have 
	\begin{multline*}
		J(\rho,\cu_m,p,D_mp)^\frac{1}{2} \\ \leq \theta^\frac{1}{2}J(\rho,\cu_m,p,D_mp)^\frac{1}{2}+C3^{-m}\bracket{\frac{1}{2\chi(\rho)|\cu_{m+1}|}v(\cu_{m+1})^2}_\rho^\frac{1}{2}+ C \tau_m^{\frac{1}{2}}.
	\end{multline*}
	Noting the fact $\theta \in (0,1)$, we rearrange the expression above and conclude \eqref{eq.Jestimate}.
\end{proof}

\subsection{Flatness estimate}
This part gives a flatness estimate and this is the main challenge compared to its counterpart in \cite[Lemma~5.1]{bulk}. Due to the exclusion rule, we need the weighted multiscale Poincar\'e inequality to show the spatial cancellation of $v(\rho, \cu_{m+1},  p, D_m p)$. The whole Section~\ref{sec.coarse} is devoted to overcome the technical difficulty here. 

\begin{proposition}[$L^2$-flatness estimate]\label{prop.L2Flat}
	There exists an exponent $\beta(d) \in (0, \frac{1}{4})$ and a constant $C(d,\lambda, \r) < \infty$  such that for every $p,q \in B_1$ and $m \in \mathbb{N}$,
	\begin{multline}\label{eq.mainFlatness}
		\frac{1}{2 \chi(\rho)\vert \cu_{m+1}\vert} \bracket{\Ll(v(\rho, \cu_{m+1},  p, q) - \ell_{D_m^{-1}q - p, \cu_{m+1}}\Rr)^2}_{\rho} \\
		\leq C 3^{2m} \Ll(\frac{1}{\chi(\rho)^2}3^{-\beta m} +  \sum_{n=0}^m 3^{-\beta(m-n)}\tau_n\Rr).
	\end{multline}
\end{proposition}
\begin{proof}
	In the proof, we write $v_{m+1}$ as a shorthand of $v(\rho, \cu_{m+1},  p, q)$. We apply the weighted multiscale Poincar\'e inequality in Proposition~\ref{prop.WMPoincare}, to obtain
	\begin{multline}\label{eq.FlatPoincare}
		\bracket{\frac{1}{2 \chi(\rho)|\cu_{m+1}|}(v_{m+1} - \ell_{D_m^{-1}q - p, \cu_{m+1}})^2}_\rho^\frac{1}{2} \\
		\leq C \sum_{n=0}^{m+1} 3^n \Ll( \frac{1}{\vert \Z_{m+1,n}\vert}\sum_{z \in \Z_{m+1,n}} \bracket{I_{z,n}}_\rho \Rr)^{\frac{1}{2}},
	\end{multline}
	where the term $I_{z,n}$ is the shorthand of 
	\begin{equation*} 
		I_{z,n} := \frac{1}{2 \chi(\rho)}\frac{1}{2 \mathbf{N}^*_{z,n}}\frac{1}{|\cu_n|} \Ll \vert\sum_{x \in z+\cu_n} \bracket{\nabla_{x} (v_{m+1} - \ell_{D_m^{-1}q - p})|\G_{z+\cu_n^+}}\Rr \vert^2,
	\end{equation*} 
	and $\mathbf{N}^*_{z,n}$ is the same as defined in Proposition~\ref{prop.WMPoincare}.	Here, the factor $\frac{|\cu_n^+|}{|\cu_n|}$ is absorbed into the constant. Recall \eqref{eq.defKawaGradient} for $\nabla_x$.  We need to treat the terms $I_{z,n}$.

	We use the comparison of Dirichlet energy with $v_{n,z} = v(\rho, z+\cu_n,  p, q)$ to write
	\begin{equation}\label{eq.L2Decom}
		\begin{split}
			\bracket{  I_{z,n}}_{\rho} & \leq \frac{3}{2 \chi(\rho)}\bracket{\frac{1}{\mathbf{N}^*_{z,n}}\frac{1}{|\cu_n|} \left\vert\sum_{x\in z+\cu_n}\bracket{\nabla_x(v_{n,z}-\ell_{D_{n}^{-1}q-p})|\G_{z+\cu_n^+}}\right\vert^2}_\rho\\
			& \quad +\frac{3}{2 \chi(\rho)}\bracket{\frac{1}{\mathbf{N}^*_{z,n}}\frac{1}{|\cu_n|} \left\vert \sum_{x\in z+\cu_n}\bracket{\nabla_x(v_{n,z}-v_{m+1})|\G_{z+\cu_n^+}}\right\vert^2}_\rho\\
			& \quad +\frac{3}{2 \chi(\rho)}\bracket{\frac{1}{\mathbf{N}^*_{z,n}}\frac{1}{|\cu_n|} \left\vert \sum_{x\in z+\cu_n}\bracket{\nabla_x(\ell_{D_m^{-1}q-D_n^{-1}q})|\G_{z+\cu_n^+}}\right\vert^2}_\rho.
		\end{split}
	\end{equation}
	
	We will estimate the first term by Lemma~\ref{lem.mainVariance} below, which studies the decay of variance of $\nabla_x(v_{n,z}-\ell_{D_n^{-1}q-p})$. Before applying the inequality, we first make use of some concentration of $\mathbf N_{z,n}^*$. We denote by 
	\begin{align*}
		\rho^* :=\bracket{\frac{\mathbf N_{z,n}^*}{|\cu_n^+|}}_\rho=\min\{\rho,(1-\rho)\},    
	\end{align*}
	as the mean of spatial average of $\mathbf N^*$.
	
	\textit{Case 1: high density $\mathbf N^*_{z,n}>\frac{1}{2}\rho^*|\cu_n^+|$}. For this case, the weight is of typical size and does not degenerate, so we have
	\begin{equation}\label{eq.L2Case11Weight}
		\begin{split}
			& \frac{3}{2 \chi(\rho)}\bracket{\1_{\{\mathbf N^*_{z,n}>\frac{1}{2}\rho^*|\cu_n^+|\}}\frac{1}{\mathbf{N}^*_{z,n}}\frac{1}{|\cu_n|} \left\vert\sum_{x\in z+\cu_n}\bracket{\nabla_x(v_{n,z}-\ell_{D_{n}^{-1}q-p})|\G_{z+\cu_n^+}}\right\vert^2}_\rho
			\\\leq & \frac{6}{2 \chi(\rho)\rho^*}\bracket{\left\vert\frac{1}{|\cu_n|} \sum_{x\in z+\cu_n}\bracket{\nabla_x(v_{n,z}-\ell_{D_{n}^{-1}q-p})|\G_{z+\cu_n^+}}\right\vert^2}_\rho
			\\\leq & \frac{6}{\chi(\rho)^2}\bracket{\left\vert\frac{1}{|\cu_n|} \sum_{x\in z+\cu_n}\nabla_x(v_{n,z}-\ell_{D_{n}^{-1}q-p})\right\vert^2}_\rho
			\\\leq &  \frac{C}{\chi(\rho)^2} 3^{-\beta n} + C\sum_{k=0}^{n-1}3^{-\beta(n-k)}\tau_k.
		\end{split}
	\end{equation}
	Here we observe that $\rho^*\geq\chi(\rho)\geq\rho^*/2$. The last line is from Lemma~\ref{lem.mainVariance} below. Note that the exponent $\beta$ only depends on $d$ and the constant $C$ depends on $\lambda, \r, d$ as stated in Lemma~\ref{lem.mainVariance}. We highlight this step, as it is where we gain the spatial cancellation from the weighted multiscale Poincar\'e inequality.

	\smallskip
	\textit{Case 2: low density $1 \leq \mathbf{N}^*_{z,n} < \frac{1}{2} \rho^* \vert \cu_n^+ \vert$.} 
	This case is quite rare by Hoeffding inequality (see \cite[Theorem~2.8]{boucheron2013concentration})
	\begin{align*}
		\mathbb{P}_\rho \Ll[\mathbf{N}^*_{z,n} < \frac{ \rho^* \vert \cu_n^+ \vert}{2} \Rr] \leq \exp\Ll(-\frac{(\rho^*)^2 |\cu_n^+|}{2}\Rr).
	\end{align*}
	We also make use of the $L^\infty$ estimate from Lemma~\ref{lem.supNorm}, and obtain 
	\begin{equation}\label{eq.L2Case12}
		\begin{split}
			&\frac{3}{2 \chi(\rho)}\bracket{\1_{\{1\leq\mathbf N^*_{z,n}<\frac{1}{2}\rho^*|\cu_n^+|\}}\frac{1}{\mathbf{N}^*_{z,n}}\frac{1}{|\cu_n|} \left\vert\sum_{x\in z+\cu_n}\bracket{\nabla_x(v_{n,z}-\ell_{D_{n}^{-1}q-p})|\G_{z+\cu_n^+}}\right\vert^2}_\rho 
			\\= & \frac{3}{2 \chi(\rho)}\bracket{\1_{\{1\leq\mathbf N^*_{z,n}<\frac{1}{2}\rho^*|\cu_n^+|\}}\frac{\mathbf{N}^*_{z,n}}{|\cu_n|} \left\vert\frac{1}{\mathbf{N}^*_{z,n}} \sum_{x\in z+\cu_n}\bracket{\nabla_x(v_{n,z}-\ell_{D_{n}^{-1}q-p})|\G_{z+\cu_n^+}}\right\vert^2}_\rho
			\\ \leq& \frac{3}{2 \chi(\rho)}\bracket{\1_{\{1\leq\mathbf N^*_{z,n}<\frac{1}{2}\rho^*|\cu_n^+|\}}\frac{1}{|\cu_n|} \sum_{x\in z+\cu_n}\left\vert\nabla_x(v_{n,z}-\ell_{D_{n}^{-1}q-p})\right\vert^2}_\rho
			\\\leq & \frac{3}{2\chi(\rho)}\P_\rho\left[\mathbf{N}^*_{z,n} < \frac{ \rho^* \vert \cu_n^+ \vert}{2}\right] 4(\Vert v_{n,z}\Vert_\infty^2+ |D_n^{-1}q-p|^2)
			\\ \leq &\frac{C}{\chi(\rho)}  \exp\Ll(-\frac{(\rho^*)^2 3^{nd}}{2}\Rr) 3^{2(d+3)n}.
		\end{split}
	\end{equation}
	Here $\mathbf N_{z,n}^*$ can be seen as the correct normalization factor, because the number of nonzero terms in the sum is no more than $\mathbf N_{z,n}^*$. Then  Jensen's inequality from the second line to the third line. In the case of large scale $3^\frac{n}{4}\geq\frac{1}{\rho^*}$, we have a natural bound
	\begin{align*}
		\frac{C}{\chi(\rho)}  \exp\Ll(-\frac{(\rho^*)^2 3^{nd}}{2}\Rr) 3^{2(d+3)n}\leq \frac{C}{\chi(\rho)}\exp(-\frac{3^{n(d-\frac{1}{2})}}{2})3^{(2d+6)n}\leq \frac{C'}{\chi(\rho)} 3^{-n}
	\end{align*}
	with the choice of constant $C'=C'(d)=C\sup_{t\geq 0}e^{-(d-\frac{1}{2})}t^{2d+6}<\infty$.
	
	When the scale is small $3^{\frac{n}{4}}<\frac{1}{\rho^*}$, we do not expect much concentration, so we use a trivial bound
	\begin{align*}
		&\frac{3}{2 \chi(\rho)}\bracket{\1_{\{1\leq\mathbf N^*_{z,n}<\frac{1}{2}\rho^*|\cu_n^+|\}}\frac{1}{|\cu_n|} \sum_{x\in z+\cu_n}\left\vert\nabla_x(v_{n,z}-\ell_{D_{n}^{-1}q-p})\right\vert^2}_\rho
		\\\leq & \frac{3}{2 \chi(\rho)}\bracket{\frac{1}{|\cu_n|} \sum_{x\in z+\cu_n}\left\vert\nabla_x(v_{n,z}-\ell_{D_{n}^{-1}q-p})\right\vert^2}_\rho
		\\ \leq &C(\lambda)\leq \frac{C(\lambda)}{(\rho^*)^2}3^{-\frac{n}{2}}\leq \frac{C}{\chi(\rho)^2}3^{-\frac{n}{2}}.
	\end{align*}
	
	\smallskip
	Combining the two cases above, we end up with the following bound on the first term:
	\begin{align*}
		&\frac{3}{2 \chi(\rho)}\bracket{\frac{1}{\mathbf{N}^*_{z,n}}\frac{1}{|\cu_n|} \left\vert\sum_{x\in z+\cu_n}\bracket{\nabla_x(v_{n,z}-\ell_{D_{n}^{-1}q-p})|\G_{z+\cu_n^+}}\right\vert^2}_\rho
		\\\leq &\frac{C}{\chi(\rho)^2}3^{-\beta n}+C\sum_{k=0}^{n-1}3^{-\beta(n-k)}\tau_k
	\end{align*}

	The second term in \eqref{eq.L2Decom} can be controlled by the gap of the subadditive quantities using \eqref{eq.quadraresponseJ} as the proof of Lemma~\ref{lem.Jestimate}
	\begin{align*}
		&\frac{1}{|\Z_{m+1,n}|}\sum_{z\in\Z_{m+1,n}}\frac{1}{2 \chi(\rho)}\bracket{\frac{1}{|\cu_n|}\frac{1}{\mathbf{N}^*_{z,n}} \left\vert \sum_{x\in z+\cu_n}\bracket{\left.\nabla_x(v_{n,z}-v_{m+1})\right|\G_{z+\cu_n^+}}\right\vert^2}_\rho\\
		&\leq \frac{1}{|\Z_{m+1,n}|}\bracket{\sum_{z\in\Z_{m+1,n}}\frac{1}{2 \chi(\rho)}\frac{1}{|\cu_n|}\bracket{  \left.\sum_{x\in z+\cu_n}\left\vert\nabla_x(v_{n,z}-v_{m+1})\right\vert^2\right|\G_{z+\cu_n^+}}}_\rho
		\\
		&\leq \frac{1}{|\Z_{m+1,n}|}\sum_{z\in\Z_{m+1,n}}\frac{1}{2 \chi(\rho)}\bracket{ \frac{1}{|\cu_n|}\sum_{x\in z+\cu_n} \vert \nabla_x(v_{n,z}-v_{m+1}) \vert^2 }_\rho\\
		&\leq J(\rho, \cu_{n},p,q)-J(\rho, \cu_{m+1},p,q)\\
		&\leq \sum_{k=n}^m\tau_k.
	\end{align*}
	Here from the first line to the second line, we use Jensen's inequality by noting that there are no more than $\mathbf N^*_{z,n}$ nonzero terms in the sum $\sum_{x\in z+\cu_n}\nabla_x (\cdot)$; see the similar trick in the first three lines in \eqref{eq.L2Case12}. From the second line to the third line, we use \eqref{eq.quadraresponseJ}.
	
	Similarly, the third term  in \eqref{eq.L2Decom} is also naturally bounded
	\begin{align*}
		&\frac{1}{2\chi(\rho)}\bracket{\frac{1}{\mathbf{N}_{z,n}^*}\frac{1}{|\cu_n|} \left\vert \sum_{x\in z+\cu_n}\bracket{\nabla_x(\ell_{D_m^{-1}q-D_n^{-1}q})\vert \G_{z+\cu_n^+}}\right\vert^2}_\rho\\
		\leq&\frac{1}{2\chi(\rho)}\bracket{\frac{1}{|\cu_n|}\sum_{x\in z+\cu_n}\vert\nabla_x\ell_{D_m^{-1}q-D_n^{-1}q}\vert^2}_\rho\\
		=&
		|D_m^{-1}q-D_{n}^{-1}q|^2\leq C(\lambda)\sum_{k=n}^{m-1}\tau_k.
	\end{align*}
	


	\smallskip

	\smallskip
	Plugging the estimates on the three terms into \eqref{eq.FlatPoincare}, we have
	\begin{multline*}
		\bracket{\frac{1}{2 \chi(\rho)|\cu_{m+1}|}(v_{m+1} - \ell_{D_m^{-1}q - p})^2}_\rho^\frac{1}{2}
		\leq C \sum_{n=0}^{m} 3^n \Ll(\frac{1}{\chi(\rho)^2}3^{-\beta n}+\sum_{k=0}^{n-1}3^{-\beta(n-k)}\tau_k+\sum_{k=n}^m\tau_k\Rr)^\frac{1}{2} 
	\end{multline*}
	Squaring this equation with Cauchy--Schwartz inequality, and shrinking $\beta$ if necessary,  we prove \eqref{eq.mainFlatness}.

\end{proof}

\subsection{Variance decay of the averaged gradient}
In this part, we prove the variance decay of gradient. Since its proof only uses the spatial independence, all its parameters are independent of $\rho$. In particular, the decay exponent $\beta$ only depends on the dimension $d$. 
\begin{lemma}[Variance decay]\label{lem.mainVariance}
	There exist $\beta(d) > 0$ and $C(d,\lambda, \r) < \infty$ such that for every $p,q \in B_1$ and $n \in \mathbb{N}$, we have 
	\begin{multline}\label{eq.mainVariance}
		\frac{1}{\chi(\rho)^2}\bracket{\Ll\vert\frac{1}{\vert \cu_n \vert} \sum_{x \in \cu_n} \nabla_x \Ll(v(\rho, \cu_n,  p, q) - \ell_{D_n^{-1}q - p}\Rr)\Rr\vert^2}_{\rho} \\
		\leq \frac{C}{\chi(\rho)^2} 3^{-\beta n} + C\sum_{k=0}^{n-1}3^{-\beta(n-k)}\tau_k.
	\end{multline}
\end{lemma}
\begin{proof}
	Since $\rho, p ,q$ does not change, we write $v_n$ as a shorthand for $v(\rho, \cu_n,  p, q)$ and $v_{n-1, z}$ as a shorthand of $v(\rho, z+\cu_{n-1},  p, q)$. Recall that $\rho^*=\min\{\rho,1-\rho\}$, which implies
	\begin{align}\label{eq.rho_trivial_bound}
		\forall \rho \in [0,1], \qquad \frac{\rho^*}{2} \leq \chi(\rho) \leq \rho^*.
	\end{align}
	From now on, we focus on the large $n$ and make the assumption that 
	\begin{align}\label{eq.mainVariance_assumption}
		n \geq n_0, \qquad n_0 := \min\Ll\{k \in \N_+: 3^k = 10\r + \frac{1}{(\rho^*)^4} \Rr\}.
	\end{align}
	We use a comparison between scale $n$ and scale $(n-1)$
	\begin{equation}\label{eq.mainVariance_ThreeTerms}
		\begin{split}
			&\bracket{\frac{1}{\chi(\rho)^2}\Ll\vert\frac{1}{\vert \cu_n \vert} \sum_{x \in \cu_n} \nabla_x \Ll(v_n - \ell_{D_n^{-1}q - p}\Rr)\Rr|^2}_{\rho}^\frac{1}{2}\\
			& \leq \bracket{\frac{1}{\chi(\rho)^2} \Ll\vert\frac{1}{\vert \cu_n \vert}\sum_{z\in\Z_{n,n-1}} \sum_{x \in z+\cu_{n-1}} \nabla_x \Ll(v_{n-1, z} - \ell_{D_{n-1}^{-1}q - p}\Rr)\Rr\vert^2}_{\rho}^\frac{1}{2}\\
			& \qquad + \bracket{\frac{1}{\chi(\rho)^2}\Ll\vert\frac{1}{\vert \cu_n \vert}\sum_{z\in\Z_{n,n-1}} \sum_{x \in z+\cu_{n-1}} \nabla_x \left(  v_{n-1,z} -  v_n\right)\Rr\vert^2}_{\rho}^\frac{1}{2}\\
			& \qquad + \bracket{\frac{1}{\chi(\rho)^2}\Ll\vert\frac{1}{\vert \cu_n \vert}\sum_{z\in\Z_{n,n-1}} \sum_{x \in z+\cu_{n-1}} \nabla_x  \ell_{D_n^{-1}q-D_{n-1}^{-1}q}\Rr\vert^2}_{\rho}^\frac{1}{2}.
		\end{split}
	\end{equation}

	We treat the three terms above separately. 
	
	Concerning the third term in \eqref{eq.mainVariance_ThreeTerms}, we observe that there are in total no more than $\mathbf N_{0,n}^*(\eta)$ nonzero term in the sum, where recall the definition of $\mathbf{N}^*_{0,n}$ in Proposition~\ref{prop.WMPoincare}.
	\begin{equation}\label{eq.mainVariance_3}
		\begin{split}
			&\bracket{\frac{1}{ \chi(\rho)^2}\Ll\vert\frac{1}{\vert \cu_n \vert}\sum_{z\in\Z_{n,n-1}} \sum_{x \in z+\cu_{n-1}} \nabla_x  \ell_{D_n^{-1}q-D_{n-1}^{-1}q}\Rr\vert^2}_{\rho}^\frac{1}{2}
			\\\leq&\frac{1}{\chi(\rho)^2}\bracket{\frac{1}{|\cu_n|^2}\mathbf N_{0,n}^*\sum_{x\in\cu_n}\vert\nabla_x \ell_{D_n^{-1}q-D_{n-1}^{-1}q}\vert^2}^\frac{1}{2}_\rho
			\\\leq& \frac{1}{\chi(\rho)}\frac{2}{\rho^*}\bracket{\frac{1}{|\cu_n|^2}(\mathbf N_{0,n}^*)^2 C |D_n^{-1}q-D_{n-1}^{-1}q|^2}_\rho^\frac{1}{2}
			\\\leq& C |D_n^{-1}q-D_{n-1}^{-1}q|
			\leq C(d, \lambda)\tau_{n-1}^\frac{1}{2}.
		\end{split}
	\end{equation}
	Here we make use of the fact that $\rho^*|\cu_n^+|$ is sufficiently large so that $$\bracket{(\mathbf N_{0,n}^*)^2}_\rho=(\rho^*|\cu_n^+|)^2+\chi(\rho)|\cu_n^+|\leq 2(\rho^*|\cu_n^+|)^2.$$
	
	Concerning the second term in \eqref{eq.mainVariance_ThreeTerms}, we use Jensen's inequality to obtain 
	\begin{align*}
		&\bracket{\frac{1}{\chi(\rho)^2}\Ll\vert\frac{1}{\vert \cu_n \vert}\sum_{z\in\Z_{n,n-1}} \sum_{x \in z+\cu_{n-1}} \nabla_x \left(  v_{n-1,z} -  v_n\right)\Rr\vert^2}_{\rho}\\
		&\leq \bracket{\frac{1}{\chi(\rho)}\frac{2}{\rho^*}\frac{1}{\vert \cu_n \vert^2}\mathbf N_{0,n}^*\sum_{z\in\Z_{n,n-1}} \sum_{x \in z+\cu_{n-1}} \Ll\vert\nabla_x   \Ll(v_{n-1,z} -  v_n\Rr) \Rr\vert^2}_{\rho}.
	\end{align*}
	Here, we need to use some concentration to give a control on $\mathbf N_{0,n}^*$. We consider two different cases of $\mathbf N_{0,n}^*$. When $\mathbf N_{0,n}^*\leq 2\rho^*|\cu_n^+|$, we have
	\begin{align*}
		&\bracket{\frac{1}{\chi(\rho)}\frac{2}{\rho^*}\frac{1}{\vert \cu_n \vert^2}\mathbf{1}_{\{\mathbf N_{0,n}\leq 2\rho^*|\cu_n^+|\}}\mathbf N_{0,n}^*\sum_{z\in\Z_{n,n-1}} \sum_{x \in z+\cu_{n-1}} \Ll\vert\nabla_x   \Ll(v_{n-1,z} -  v_n\Rr) \Rr\vert^2}_{\rho}
		\\\leq& \bracket{\frac{1}{\chi(\rho)}\frac{8}{\vert \cu_n \vert}\mathbf{1}_{\{\mathbf N_{0,n}\leq 2\rho^*|\cu_n^+|\}}\sum_{z\in\Z_{n,n-1}} \sum_{x \in z+\cu_{n-1}} \Ll\vert\nabla_x   \Ll(v_{n-1,z} -  v_n\Rr) \Rr\vert^2}_{\rho}
		\\\leq & \bracket{\frac{1}{\chi(\rho)}\frac{8}{\vert \cu_n \vert}\sum_{z\in\Z_{n,n-1}} \sum_{x \in z+\cu_{n-1}} \Ll\vert\nabla_x   \Ll(v_{n-1,z} -  v_n\Rr) \Rr\vert^2}_{\rho}
		\\\leq &8 \tau_{n-1},
	\end{align*}
	where the last line comes from the same argument as in Lemma~\ref{lem.Jestimate}. 
	
	Otherwise, when $\mathbf N_{0,n}>2\rho|\cu_n^+|$, we use the $L^\infty$ estimate from Lemma~\ref{lem.supNorm} to obtain
	
	\begin{align*}
		&\bracket{\frac{1}{\chi(\rho)}\frac{2}{\rho^*}\frac{1}{\vert \cu_n \vert^2}\mathbf{1}_{\{\mathbf N^*_{0,n}> 2\rho^*|\cu_n^+|\}}\mathbf N_{0,n}^*\sum_{z\in\Z_{n,n-1}} \sum_{x \in z+\cu_{n-1}} \Ll\vert\nabla_x   \Ll(v_{n-1,z} -  v_n\Rr) \Rr\vert^2}_{\rho}
		\\ \leq & \frac{1}{\chi(\rho)}\frac{1}{\rho^*}\P_\rho[\mathbf N^*_{0,n}> 2\rho^*|\cu_n^+|] 4(\Vert v_n\Vert_\infty^2+\Vert v_{n-1}\Vert_\infty^2)
		\\ \leq & \frac{1}{\chi(\rho)\rho^*}\exp(-2\times3^{nd}(\rho^*)^2)3^{2(d+3)n}
		\\ \leq & 3^{\frac{n}{2}}\exp(-2\times3^{nd}3^{-\frac{n}{2}})3^{2(d+3)n}
		\\ \leq & C 3^{-n},
	\end{align*}
	where $C=C(d)=\sup_{n\geq 0}\exp(-2\times3^{dn-\frac{n}{2}})3^{(2d+8)n}<\infty$. The concentration inequality is applied in the third line and the assumption \eqref{eq.mainVariance_assumption} is utilized in the fourth line above. Combining the two estimates, we conclude that
	\begin{align}\label{eq.mainVariance_2}
		\bracket{\frac{1}{ \chi(\rho)^2}\Ll\vert\frac{1}{\vert \cu_n \vert}\sum_{z\in\Z_{n,n-1}} \sum_{x \in z+\cu_{n-1}} \nabla_x \left(  v_{n-1,z} -  v_n\right)\Rr\vert^2}_{\rho}^\frac{1}{2}
		\leq  C\tau_{n-1}^\frac{1}{2}+C3^{-\frac{n}{2}}.
	\end{align}
	
	For the first term in \eqref{eq.mainVariance_ThreeTerms}, we define 
	\begin{equation*}
		X_z:=\frac{1}{|\cu_{n-1}|}\sum_{x\in z+\cu_{n-1}}\nabla_x \Ll(v(z+\cu_{n-1},p,q)) - \ell_{D_{n-1}^{-1}q - p}\Rr),
	\end{equation*}
	and expand the square
	\begin{align*}
		&\bracket{\Ll\vert\frac{1}{\vert \cu_n \vert}\sum_{z\in\Z_{n,n-1}} \sum_{x \in z+\cu_{n-1}} \nabla_x \Ll(v_{n-1,z} - \ell_{D_{n-1}^{-1}q - p}\Rr)\Rr\vert^2}_{\rho} \\
		&= \bracket{\left\vert\frac{1}{|\Z_{n,n-1}|}\sum_{z\in\Z_{n,n-1}}X_z\right\vert^2}_\rho\\
		&= \frac{1}{|\Z_{n,n-1}|^2} \sum_{z,w \in \Z_{n,n-1}}\bracket{X_z \cdot X_w}_\rho.
	\end{align*}
	Note that $X_z$ is an $\Rd$-valued random vector, and here $X_z \cdot X_w$ is the inner product between $X_z$ and $X_w$.
	
	For sufficiently large $n$ such that $3^n > 10 \r$, there exist two cubes $z+\cu_{n-1}$ and $w+\cu_{n-1}$ with distance greater than $2\r$. Recall that $X_z \in \F_0(N_\r(z+\cu_n^+))$ from (1) of Proposition~\ref{prop.Element}. Then we can use the local property and independence for such pair $X_z, X_w$ to get
	\begin{align*}
		\bracket{X_z \cdot X_w}_\rho=\bracket{X_z}_\rho \cdot \bracket{X_w}_\rho=0,
	\end{align*}
	where the last equal sign follows from the average slope property \eqref{eq.SlopeJ}. For other pairs, we just use Cauchy--Schwarz inequality and stationary property, which concludes
	\begin{align}\label{eq.mainVariance_1}
		\bracket{\left\vert\frac{1}{|\Z_{n,n-1}|}\sum_{z\in\Z_{n,n-1}}X_z\right\vert^2}_\rho \leq \frac{3^{2d}-1}{3^{2d}}\bracket{\vert X_0 \vert^2}_\rho.
	\end{align}
	
	We denote left hand side of \eqref{eq.mainVariance} by $\sigma_n^2$, then under the assumption \eqref{eq.mainVariance_assumption}, the estimates  \eqref{eq.mainVariance_3}, \eqref{eq.mainVariance_2} and \eqref{eq.mainVariance_1} hold and yield
	\begin{equation*}
		\forall n \geq n_0, \qquad \sigma_n\leq \theta\sigma_{n-1}+ C\tau_{n-1}^\frac{1}{2}+C3^{-\frac{n}{2}}
	\end{equation*}
	for some positive constant $\theta := \frac{3^{2d}-1}{3^{2d}} <1$. We implement an iteration to obtain that
	\begin{align*}
		\forall n \geq n_0, \qquad \sigma_n \leq \theta^{n-n_0}\sigma_{n_0}+C(d, \lambda)\sum_{k=n_0}^{n-1}\theta^{n-k}\tau_k^\frac{1}{2}.
	\end{align*}
	Here we observe the fact that $\theta>3^{-\frac{1}{4}}$, which implies that 
	\begin{align*}
		\theta^{-n_0} \leq 3^{\frac{n_0}{4}} \leq 3 \times \Ll(10\r + \frac{1}{(\rho^*)^4}\Rr)^{\frac{1}{4}} \leq C\left( 10\r + \frac{1}{\rho^*}\right).
	\end{align*}
	We also use the trivial bound $\sigma_0 \leq \lambda$, which ensures that
	\begin{align*}
		\forall n \geq n_0, \qquad \sigma_n \leq \frac{C(d,\lambda,\r)}{\chi(\rho)}\theta^{n} +C(d,\lambda)\sum_{k=0}^{n-1}\theta^{n-k}\tau_k^\frac{1}{2}. 
	\end{align*}
	Squaring this and shrinking $\theta$, we conclude that
	\begin{equation*}
		\forall n \geq n_0, \qquad  \sigma_n^2\leq \frac{C}{\chi(\rho)^2}3^{-\beta n}+C\sum_{k=0}^{n-1}3^{-\beta(n-k)}\tau_k.
	\end{equation*}
	It remains to prove the result when $n\leq n_0$. In this case, the bound is automatically true since we have a natural bound $\sigma_n^2\leq \frac{\lambda}{\chi(\rho)}\leq \frac{C}{\chi(\rho)^2}3^{-\beta n}$. This completes the proof.
\end{proof}

\begin{remark}
	Note that we can replace the bound for small $n$ as $\sigma_n^2\leq \frac{\lambda}{\chi(\rho)}\leq \frac{C}{\chi(\rho)^{1+\delta}}3^{-\beta n}$ if we choose sufficiently small $\beta$. This refined argument can be used to prove the quantitative convergence of $\chi(\rho)^{1+\delta} J(\rho,\cu_m,p,D_mp)$. But since we do not need this, we keep the factor $\chi(\rho)^2$.
\end{remark}

\subsection{Iterations and track of parameters}
In this part, we summarize the previous steps and conclude the proof. Especially, here we need to track the dependence of the density $\rho$ in our convergence. Recall \eqref{eq.defCGrand}, and combine the result \eqref{eq.ratecontrolJ} with $\tilde D = D_m$ defined in \eqref{eq.defDm},
\begin{equation}\label{eq.JResume}
	\begin{split}
		&\Ll\vert \cc(\rho,  \cu_m) - \cc_*(\rho,  \cu_m) \Rr\vert \\
		&= 2\chi(\rho) \Ll\vert \D(\rho,  \cu_m) - \D_*(\rho, \cu_m) \Rr\vert\\
		&\leq C(d,\lambda)\chi(\rho)\Ll(\sup_{|p|=1}J(\rho,\cu_m,p, D_m p)^{\frac{1}{2}}\Rr).
	\end{split}
\end{equation} 
Therefore, we need to study the uniform convergence of $\chi^2(\rho) J(\rho,\cu_m,p, D_m p)$, where the compressibility $\chi(\rho)$ helps control the convergence near two endpoints. We state the following lemma, which controls $\chi^2(\rho) J(\rho,\cu_m,p, D_m p)$ by the energy gap $\tau_n$ defined in \eqref{eq.deftau} uniformly in $\rho$.

\begin{proof}[Proof of Proposition~\ref{prop.GrandCanonicalEnsemble}]
	We combine the estimates \eqref{eq.Jestimate} and \eqref{eq.mainFlatness}, which yields 
	\begin{equation*}
		J(\rho,\cu_m,p, D_m p)
		\leq C \tau_m + C\Ll(\chi(\rho)^{-2}3^{-\beta m} +  \sum_{n=0}^m 3^{-\beta(m-n)}\tau_n \Rr).
	\end{equation*}
	Then we times $\chi^2(\rho)$ and obtain
	\begin{align}\label{eq.Juniform}
		\chi^2(\rho) J(\rho,\cu_m,p, D_m p) \leq C  \Ll(3^{- \beta m} +  \sum_{n=0}^m 3^{- \beta (m-n)}\chi^2(\rho)\tau_n\Rr). 
	\end{align}
	
	The remaining part is similar to the proof of \cite[Proposition~2.11]{AKMbook} and we give its sketch. We define 
	\begin{align}\label{eq.defFm}
		F_m := \sum_{i=1}^d \chi^2(\rho)  J(\rho, \cu_m, e_i, D_m e_i),
	\end{align}
	and its weighted version with $\beta > 0$ from \eqref{eq.Juniform} (which inherits from \eqref{eq.mainFlatness})
	\begin{align*}
		\tilde{F}_m := \sum_{n=0}^{m}3^{-\frac{\beta}{2}(m-n)}F_n.
	\end{align*}
	We also define a modified version of the gap in \eqref{eq.deftau} 
	\begin{align}
		\tilde{\tau}_n=\tilde{\tau}_n(\rho) :=\sup_{p,q\in B_1} \chi^2(\rho) (J(\rho,\cu_n,p,q)- J(\rho,\cu_{n+1},p,q))
	\end{align}
	Since  \eqref{eq.deftau} and \eqref{eq.Juniform} all satisfy the estimate independent of $\rho$, it suffices to follow the iteration in \cite[Proposition~2.11]{AKMbook}.
	
	Recalling the definition $$J(\rho,\cu_m,p,q)=\frac{1}{2}(q-D_mp)\cdot D_m^{-1}(q-\D_*(\rho,\cu_m)p)+\frac{1}{2}p\cdot D_mp,$$
	then we have $$F_{m+1}= \sum_{i=1}^d \chi^2(\rho)  J(\rho, \cu_{m+1}, e_i, D_{m+1} e_i)\leq \sum_{i=1}^d \chi^2(\rho)  J(\rho, \cu_{m+1}, e_i, D_{m} e_i)\leq F_m.$$
	
	Therefore, we may apply \eqref{eq.Juniform} to get
	\begin{align*}
		\tilde F_{m+1}&=\sum_{n=0}^{m+1}3^{-\frac{\beta}{2}(m+1-n)}F_n
		\\&\leq C3^{-\frac{\beta m}{2}}+\sum_{n=1}^{m+1}3^{-\frac{\beta}{2}(m+1-n)}F_{n-1}
		\\&\leq C3^{-\frac{\beta m}{2}}+C\sum_{n=0}^{m}3^{-\frac{\beta}{2}(m-n)}\left(3^{-\beta n}+\sum_{k=0}^n3^{-\beta(n-k)}\tilde{\tau}_k\right)
		\\&\leq C3^{-\frac{\beta m}{2}}+C3^{-\frac{\beta m}{2}}\sum_{k=0}^m\tilde{\tau}_k\sum_{n=k}^m3^{\beta k-\frac{\beta}{2}n}
		\\&\leq C3^{-\frac{\beta m}{2}}+C\sum_{k=0}^m3^{-\frac{\beta}{2}(m-k)}\tilde{\tau}_k.
	\end{align*}
	Here all constants are finite positive and depend on $\lambda,d,\r$ only.
	
	On the other hand, we have a lower bound
	\begin{align*}
		F_m-F_{m+1}\geq \sum_{i=1}^d\chi^2(\rho)(J(\rho,\cu_m,e_i,D_m e_i)-J(\rho,\cu_{m+1},e_i,D_m e_i))
		\geq C\tilde{\tau}_m, 
	\end{align*}
	for some finite positive constant $C=C(d)$.
	
	Combining the formulas above, we conclude that there exists a constant ${C(d, \lambda,\r) < \infty}$ such that
	\begin{align*}
		\tilde{F}_{m+1} \leq C(\tilde{F}_{m} - \tilde{F}_{m+1}) + C 3^{-\frac{\beta m}{2}}.
	\end{align*}
	This implies a contraction 
	\begin{align*}
		\tilde{F}_{m+1} \leq \Ll(\frac{C}{1+C}\Rr) \tilde{F}_{m}  + \Ll(\frac{C}{1+C}\Rr) 3^{-\frac{\beta m}{2}},
	\end{align*}
	which gives the decay of $\tilde{F}_{m}$. Following \eqref{eq.JResume} and \eqref{eq.defFm}, there exists $\gamma_1(d,\lambda, \r)>0$ and $C(d, \lambda,\r) < \infty$ such that
	\begin{align}\label{eq.cGapRate}
		\forall \rho \in [0,1], \qquad \Ll\vert \cc(\rho,  \cu_m) - \cc_*(\rho,  \cu_m) \Rr\vert \leq F_m \leq \tilde{F}_{m} \leq C 3^{-\gamma_1 m}.
	\end{align}
	Recall the monotone convergence in Corollary~\ref{cor.defDLimit} for a fixed $\rho$, thus as $m \to \infty$, the sequence $\cc(\rho,  \cu_m)$ decreases and  $\cc_*(\rho,  \cu_m)$ increases to the same limit $\cc(\rho)$. Then \eqref{eq.cGapRate} gives the desired uniform convergence rate along the triadic cubes, and Lemma~\ref{lem.WhitneySub} extends the result to a general cube, which concludes the proof. 	
\end{proof}

\section{Density-free local corrector with uniform convergence}\label{sec.Free}
There still remains some difference between our homogenization result in Proposition~\ref{prop.GrandCanonicalEnsemble} and Theorem~\ref{thm.mainUniform}. The optimizer $F_L$ in Theorem~\ref{thm.mainUniform} is a corrector-type function in homogenization, and we have already obtained a natural candidate from previous sections (see \eqref{eq.defNu} and (1) of Proposition~\ref{prop.Element})
\begin{equation}\label{eq.defCorrector}
	\phi_{\rho, \Lambda, \xi} := v(\rho, \Lambda, \xi) - \ell_{\Lambda^+, \xi}.
\end{equation}
This local corrector is of $\F_0(\Lambda^-)$, but has the dependence on $\rho$. In this part, we explore various properties about this function at first in Sections~\ref{subsec.Regularity} and ~\ref{subsec.RateCano}, then improve it by removing the dependence of density in Section~\ref{subsec.FreeCorrector}, and finally prove the main theorem (Theorem~\ref{thm.mainUniform}) in Section~\ref{subsec.Stationary} and ~\ref{subsec.pfMain}.

Section~\ref{subsec.CLT} is an application of the results obtained, and prepares for the hydrodynamic limit in Section~\ref{sec.hy}.

\subsection{Regularity of local corrector}\label{subsec.Regularity}

We summarize at first some properties about our local corrector $\phi_{\rho, \Lambda, \xi}$ from previous sections.

\begin{proposition}The local corrector $\phi_{\rho, \Lambda, \xi}$ satisfies the following properties.
	\begin{enumerate}\label{prop.ElementCorrector}
		\item (Elementary properties) $\phi_{\rho, \Lambda, \xi}$ is a $\F_0(\Lambda^-)$ function with $\bracket{\phi_{\rho, \Lambda, \xi}}_\rho = 0$, and $\xi \mapsto \phi_{\rho, \Lambda, \xi}$ is linear.
		\item (Approximation of conductivity) For every $L \in \N_+$, we have 
		\begin{align}\label{eq.rhoCorrectorHomo}
			\sup_{\rho \in [0,1], \xi \in B_1} \Ll\vert \frac{1}{\vert \Lambda_L \vert} \bracket{\sum_{b \in \overline{\Lambda_L^*}} \frac{1}{2} c_b (\pi_b (\ell_\xi + \phi_{\rho, \Lambda_L, \xi}))^2}_\rho - \frac{1}{2} \xi \cdot \cc(\rho) \xi \Rr\vert   \leq C L^{-\gamma_1}.
		\end{align}
		Here the exponent $\gamma_1(d, \lambda, \r)$ and the constant $C(d, \lambda, \r)$ are the same as in Proposition~\ref{prop.GrandCanonicalEnsemble}.
		\item (Sublinearity) There exists a constant $C(d, \lambda, \r)<\infty$, such that for every $L \in \N_+$, the following estimate holds for $\gamma_1(d, \lambda, \r)$ from Proposition~\ref{prop.GrandCanonicalEnsemble}
		\begin{align}\label{eq.Sublinear}
			\sup_{\rho \in [0,1], \xi \in B_1}\bracket{ \frac{1}{\vert \Lambda_L \vert}  \phi^2_{\rho, \Lambda_L, \xi} }_\rho \leq C L^{2-\gamma_1}.
		\end{align}
	\end{enumerate}
\end{proposition}
\begin{proof}
	(1) can be deduced from Proposition~\ref{prop.Element}. (2) is the consequence from Proposition~\ref{prop.GrandCanonicalEnsemble}. (3) comes from Proposition~\ref{prop.L2Flat}, once we put the convergence rate \eqref{eq.main1A_2} to \eqref{eq.deftau}.
\end{proof}

The main task of this part is to explore the regularity on $\rho$. We propose the following factor to measure the one-sided bias when changing the probability:
\begin{align}\label{eq.defTheta}
	\forall \rho', \rho \in (0,1), \qquad \Theta_{\rho', \rho} := \max\Ll\{\frac{\rho'}{\rho}, \frac{1-\rho'}{1-\rho} \Rr\},
\end{align}
This quantity is used to control the continuity with respect to a small change of the density $\rho$. We remark that, the role of $\rho$ and $\rho'$ is not same here, since we should always view the second parameter as the targeted density and the first parameter as its perturbation. The following observation is obvious
\begin{align}\label{eq.RMKTheta}
	|\rho'-\rho|< \epsilon \min\{\rho,1-\rho\} \Longrightarrow \Theta_{\rho',\rho} \in [1, 1+\epsilon].
\end{align}
Especially, we only require the targeted density $\rho$ not to be degenerated to $0$ or $1$ to get a good estimate, but the perturbation $\rho'$ can be degenerate. We also define the two-sided bias factor
\begin{align}\label{eq.defTheta2}
	\tilde \Theta_{\rho', \rho} := \max \{ \Theta_{\rho', \rho},   \Theta_{\rho, \rho'}\},
\end{align}
where the role of $\rho$ and $\rho'$ is symmetric.

The following lemma is an example to apply the one-sided bias factor $\Theta_{\rho', \rho}$, and it will be useful throughout the section.
\begin{lemma}\label{lem.BiasProba}
	For bounded domain $\Lambda \subset \Zd$ and every local function $f \in \F_0(\Lambda)$, the following inequality holds for every $\rho', \rho \in (0,1)$ 
	\begin{align}\label{eq.BiasProba}
		\Ll\vert \bracket{f}_{\rho'} - \bracket{f}_\rho \Rr\vert \leq \Ll(\Theta_{\rho', \rho}^{\vert \Lambda \vert} - 1\Rr) \bracket{\vert f \vert}_\rho.
	\end{align}
	\begin{proof}
		We make the decomposition using the canonical ensemble. Denote by $X$ the random variable $X := \sum_{x\in\Lambda} \eta_x$
		\begin{align*}
			\Ll\vert \bracket{f}_{\rho'} - \bracket{f}_\rho \Rr\vert &= \Ll\vert \sum_{M=0}^{\vert \Lambda\vert} \Ll(\P_{\rho'}\Ll[X = M\Rr] - \P_{\rho}\Ll[X = M\Rr]\Rr) \bracket{f}_{\Lambda, M} \Rr\vert\\
			&\leq \sum_{M=0}^{\vert \Lambda\vert} \Ll\vert \P_{\rho'}\Ll[X = M\Rr] - \P_{\rho}\Ll[X = M\Rr]\Rr\vert \bracket{\vert f \vert}_{\Lambda, M} \\
			&= \sum_{M=0}^{\vert \Lambda\vert} \Ll\vert \frac{\P_{\rho'}\Ll[X = M\Rr]}{\P_{\rho}\Ll[X = M\Rr]} - 1 \Rr\vert \P_{\rho}\Ll[X = M\Rr]\bracket{\vert f \vert}_{\Lambda, M}. 
		\end{align*}
		Using the expression of Binomial distribution and the bound $M \leq \vert \Lambda \vert$, we obtain 
		\begin{align*}
			\Ll\vert \frac{\P_{\rho'}\Ll[X = M\Rr]}{\P_{\rho}\Ll[X = M\Rr]} - 1 \Rr\vert = \Ll\vert \frac{(\rho')^M (1-\rho')^{\vert \Lambda\vert - M}}{{\rho}^M (1-\rho)^{\vert \Lambda\vert - M}} - 1\Rr\vert \leq \Theta_{\rho', \rho}^{\vert \Lambda \vert} - 1.
		\end{align*}
		This concludes the desired result.
	\end{proof}
\end{lemma}

We now study the dependence of the density $\rho$ in several quantities. A similar argument can be found in \cite[Proposition~6.1]{giunti2021smoothness}.
\begin{proposition}[Regularity on density] For every $L \in N_+$, every $\rho, \rho', \rho'' \in (0,1)$ and $\xi \in B_1$, we have the following estimates using the factors $\Theta$ and $\tilde \Theta$ defined respectively in \eqref{eq.defTheta} and \eqref{eq.defTheta2}.
	\begin{enumerate}
		\item Regularity of conductivity: we have 
		\begin{align}\label{eq.cRegularity1}
			\cc(\rho', \Lambda_L)  \leq  \Theta_{\rho', \rho}^{(L+2\r)^d}  \cc(\rho, \Lambda_L), .
		\end{align} 
		and 
		\begin{align}\label{eq.cRegularity2}
			\vert \cc(\rho, \Lambda_L) - \cc(\rho', \Lambda_L)\vert 
			\leq  \Ll(\tilde \Theta_{\rho', \rho}^{(L+2\r)^d}  - 1\Rr) \max \Ll\{\vert \cc\vert(\rho, \Lambda_L), \vert\cc\vert(\rho', \Lambda_L)\Rr\}.
		\end{align}
		\item Regularity of mean:
		\begin{align}\label{eq.MeanRegularity}
			\frac{1}{\vert \Lambda_L \vert} \bracket{\phi_{\rho, \Lambda_L, \xi}}^2_{\rho'} \leq L^2 \Ll(\Theta_{\rho', \rho}^{(L+2\r)^d} -1\Rr)^2\vert\cc\vert(\rho, \Lambda_L)
		\end{align}
		\item Regularity of Dirichlet energy and $L^2$: if $\tilde \Theta_{\rho, \rho''}^{(L+2\r)^d}, \tilde \Theta_{\rho', \rho''}^{(L+2\r)^d} \leq 2$, we have 
		\begin{multline}\label{eq.L2Regularity}
			L^{-2} \bracket{ \frac{1}{\vert \Lambda_L \vert}  (\phi_{\rho', \Lambda_L, \xi} - \phi_{\rho, \Lambda_L, \xi})^2}_{\rho''} + \frac{1}{\vert \Lambda_L \vert} \bracket{\sum_{b \in \overline{\Lambda_L^*}} c_b (\pi_b \phi_{\rho', \Lambda_L, \xi} -  \pi_b \phi_{\rho, \Lambda_L, \xi})^2}_{\rho''} \\ \leq  10 \Ll(\max \Ll\{\tilde \Theta_{\rho, \rho''}^{(L+2\r)^d}, \tilde \Theta_{\rho', \rho''}^{(L+2\r)^d}\Rr\} -1\Rr) \max \Ll\{\vert \cc\vert(\rho, \Lambda_L), \vert\cc\vert(\rho', \Lambda_L), \vert\cc\vert(\rho'', \Lambda_L)\Rr\}.
		\end{multline}
	\end{enumerate}
\end{proposition}
\begin{proof}
	(1) We test $\phi_{\rho', \Lambda_L, \xi}$ in the variational problem of $\cc(\rho, \Lambda)$, and apply \eqref{eq.BiasProba} to change the parameter of the grand canonical ensemble
	\begin{equation}\label{eq.RegularityBase}
		\begin{split}
			\frac{1}{2} \xi \cdot\cc(\rho, \Lambda_L)\xi &\leq  \frac{1}{\vert \Lambda_L \vert} \bracket{\sum_{b \in \overline{\Lambda_L^*}} \frac{1}{2} c_b \Ll(\pi_b (\ell_\xi + \phi_{\rho', \Lambda_L, \xi})\Rr)^2}_\rho \\
			&\leq \Theta_{\rho', \rho}^{(L+2\r)^d}\frac{1}{\vert \Lambda_L \vert} \bracket{\sum_{b \in \overline{\Lambda_L^*}} \frac{1}{2} c_b \Ll(\pi_b (\ell_\xi + \phi_{\rho', \Lambda_L, \xi})\Rr)^2}_{\rho'}\\
			&= \Theta_{\rho', \rho}^{(L+2\r)^d} \frac{1}{2} p \cdot \cc(\rho', \Lambda_L)p.
		\end{split}
	\end{equation}
	Here because of the correlation length $\r$, the integration $\sum_{b \in \overline{\Lambda_L^*}} \frac{1}{2} c_b \pi_b (\ell_\xi + \phi_{\rho', \Lambda_L, \xi})^2$ is in $\F_0(\Lambda_{L+2\r})$ and we need to enlarge the power for the factor $\Theta_{\rho', \rho}$. This proves \eqref{eq.cRegularity1}. By exchanging the role of $\rho$ and $\rho'$, we can obtain the estimate on the other direction similarly, which concludes \eqref{eq.cRegularity2} using $\tilde \Theta_{\rho', \rho}$ defined in \eqref{eq.defTheta2}.
	
	\smallskip
	
	(2) Recall that $\bracket{\phi_{\rho, \Lambda_L, \xi}}_{\rho} = 0$ from (1) of Proposition~\ref{prop.ElementCorrector}, then we apply \eqref{eq.BiasProba} to  $\frac{1}{\vert \Lambda_L \vert} \bracket{\phi_{\rho, \Lambda_L, \xi}}^2_{\rho'}$
	\begin{align*}
		\frac{1}{\vert \Lambda_L \vert} \bracket{\phi_{\rho, \Lambda_L, \xi}}^2_{\rho'} &= \frac{1}{\vert \Lambda_L \vert}\Ll\vert \bracket{\phi_{\rho, \Lambda_L, \xi}}_{\rho'} - \bracket{\phi_{\rho, \Lambda_L, \xi}}_{\rho}  \Rr\vert^2 \\
		&\leq \frac{1}{\vert \Lambda_L \vert}\Ll(\Theta_{\rho', \rho}^{(L+2\r)^d} -1\Rr)^2\bracket{\vert \phi_{\rho, \Lambda_L, \xi} \vert}^2_{\rho}\\
		&\leq \frac{1}{\vert \Lambda_L \vert}\Ll(\Theta_{\rho', \rho}^{(L+2\r)^d} -1\Rr)^2\bracket{\phi^2_{\rho, \Lambda_L, \xi}}_{\rho}\\
		&\leq L^2\Ll(\Theta_{\rho', \rho}^{(L+2\r)^d} -1\Rr)^2\vert\cc\vert(\rho, \Lambda_L).
	\end{align*}
	Here we apply Jensen's inequality from the second line to the third line, and the spectral inequality \eqref{eq.spectralGradient} from the third line to the fourth line.
	
	\smallskip
	
	(3) To compare the Dirichlet energy, we add $\phi_{\rho'', \Lambda_L, \xi}$ as an intermediate term, which gives
	\begin{multline*}
		\frac{1}{\vert \Lambda_L \vert} \bracket{\sum_{b \in \overline{\Lambda_L^*}} c_b (\pi_b \phi_{\rho', \Lambda_L, \xi} -  \pi_b \phi_{\rho, \Lambda_L, \xi})^2}_{\rho''} \\
		\leq  \frac{2}{\vert \Lambda_L \vert} \bracket{\sum_{b \in \overline{\Lambda_L^*}} c_b (\pi_b \phi_{\rho', \Lambda_L, \xi} - \pi_b\phi_{\rho'', \Lambda_L, \xi})^2}_{\rho''} +  \frac{2}{\vert \Lambda_L \vert} \bracket{\sum_{b \in \overline{\Lambda_L^*}} c_b (\pi_b\phi_{\rho, \Lambda_L, \xi} - \pi_b\phi_{\rho'', \Lambda_L, \xi})^2}_{\rho''}.
	\end{multline*}
	For each term, we can repeat the argument in \eqref{eq.RegularityBase}, and conclude that 
	\begin{multline}\label{eq.DirichletRegularity}
		\frac{1}{\vert \Lambda_L \vert} \bracket{\sum_{b \in \overline{\Lambda_L^*}} c_b (\pi_b \phi_{\rho', \Lambda_L, \xi} -  \pi_b \phi_{\rho, \Lambda_L, \xi})^2}_{\rho''} \\
		\leq \Ll(\tilde \Theta_{\rho, \rho''}^{(L+2\r)^d} + \tilde \Theta_{\rho', \rho''}^{(L+2\r)^d} -2\Rr) \max \Ll\{\vert\cc\vert(\rho, \Lambda_L), \vert\cc\vert(\rho', \Lambda_L), \vert\cc\vert(\rho'', \Lambda_L)\Rr\}.
	\end{multline}
	
	The $L^2$ term can be done similarly,
	\begin{multline*}
		\bracket{ \frac{1}{\vert \Lambda_L \vert}  ( \phi_{\rho', \Lambda_L, \xi} -   \phi_{\rho, \Lambda_L, \xi})^2}_{\rho''} \\
		\leq \bracket{ \frac{2}{\vert \Lambda_L \vert}  (\phi_{\rho', \Lambda_L, \xi} - \phi_{\rho'', \Lambda_L, \xi})^2}_{\rho''} + \bracket{ \frac{2}{\vert \Lambda_L \vert}  (\phi_{\rho, \Lambda_L, \xi} - \phi_{\rho'', \Lambda_L, \xi})^2}_{\rho''}.
	\end{multline*}
	Here, we hope to use Poincar\'e inequality, but the constant part of the function should be truncated. We take the term involving $\rho, \rho''$ for example
	\begin{multline*}
		\bracket{ \frac{2}{\vert \Lambda_L \vert}  (\phi_{\rho, \Lambda_L, \xi} - \phi_{\rho'', \Lambda_L, \xi})^2}_{\rho''}\\
		\leq \bracket{ \frac{4}{\vert \Lambda_L \vert}  (\phi_{\rho, \Lambda_L, \xi} - \bracket{\phi_{\rho, \Lambda_L, \xi}}_{\rho''} - \phi_{\rho'', \Lambda_L, \xi})^2}_{\rho''} + \frac{4}{\vert \Lambda_L \vert} \bracket{\phi_{\rho, \Lambda_L, \xi}}^2_{\rho''}.
	\end{multline*} 
	The Poincar\'e inequality \eqref{eq.spectralGradient} applies to the first term, and then we can use \eqref{eq.DirichletRegularity} 
	\begin{align*}
		&\bracket{ \frac{4}{\vert \Lambda_L \vert}  (\phi_{\rho, \Lambda_L, \xi} - \bracket{\phi_{\rho, \Lambda_L, \xi}}_{\rho''} - \phi_{\rho'', \Lambda_L, \xi})^2}_{\rho''} \\
		&\leq \frac{4 L^2}{\vert \Lambda_L \vert} \bracket{\sum_{b \in \overline{\Lambda_L^*}} c_b ( \pi_b \phi_{\rho, \Lambda_L, \xi} - \pi_b \phi_{\rho'', \Lambda_L, \xi})^2}_{\rho''}\\
		&\leq 4L^2\Ll(\tilde \Theta_{\rho, \rho''}^{(L+2\r)^d} -1\Rr) \max \Ll\{\vert\cc\vert(\rho, \Lambda_L),  \vert\cc\vert(\rho'', \Lambda_L)\Rr\}.
	\end{align*}
	Concerning the term $\frac{4}{\vert \Lambda_L \vert} \bracket{\phi_{\rho, \Lambda_L, \xi}}^2_{\rho''}$, we apply directly \eqref{eq.MeanRegularity}. Under the assumption $\tilde \Theta_{\rho, \rho''}^{(L+2\r)^d} \leq 2$, then we have$\Ll(\Theta_{\rho, \rho''}^{(L+2\r)^d} -1\Rr)^2 \leq \Ll(\Theta_{\rho, \rho''}^{(L+2\r)^d} -1\Rr)$ and the leading order should be $\Ll(\tilde \Theta_{\rho, \rho''}^{(L+2\r)^d} -1\Rr)$. This concludes the proof.
\end{proof}

\subsection{Convergence rate under canonical ensemble}\label{subsec.RateCano}
In this part, we give the convergence rate of the conductivity under the canonical ensemble. Inspired by the subadditive quantities in \eqref{eq.defNu}, we can also define their counterparts under the canonical ensemble that 
\begin{equation}\label{eq.DualCanonical}
	\begin{split}
		\frac{1}{2}p \cdot \Da(\Lambda, N) p &:= \inf_{v\in\ell_{p,\Lambda^+} +\F_0(\Lambda^-)} \Ll\{ \frac{1}{2 \chi(N / \vert \Lambda \vert)\vert\Lambda\vert} \sum_{b\in\ov{\Lambda^*}} \bracket{ \frac{1}{2}c_b(\pi_b v)^2}_{\Lambda, N} \Rr\}, \\
		\frac{1}{2}q\cdot \Da_*^{-1}(\Lambda, N) q  \\
		&:= \sup_{v \in  \F_0} \Ll\{ \frac{1}{2 \chi(N / \vert \Lambda \vert)\vert\Lambda\vert}\sum_{b\in \ovs{\Lambda}}  \bracket{ (\pi_b \ell_{q, \Lambda})(\pi_b v) - \frac{1}{2} c_b(\pi_b v)^2}_{\Lambda, N}\Rr\}.
	\end{split}
\end{equation}
Similarly to \eqref{eq.defCGrand}, we define the conductivity using the Einstein relation \eqref{eq.Einstein}
\begin{align}\label{eq.defCCanonical}
	\ca(\Lambda,N) := 2\chi(N/\vert \Lambda \vert) \Da(\Lambda, N), \qquad \ca_*(\Lambda,N) := 2 \chi(N/\vert \Lambda \vert) \Da_*(\Lambda, N).
\end{align}
Here $\ca(\Lambda,N)$ also coincides with the definition \eqref{eq.defC3}. Our main result in this subsection is the following proposition.
\begin{proposition}\label{prop.CanonicalEnsemble}
	Under Hypothesis~\ref{hyp}, there exists an exponent $\gamma_2(d, \lambda, \r) > 0$ and a positive constant $C(d, \lambda, \r) < \infty$ such that for every $L,M \in \N_+$,
	\begin{align}\label{eq.main1B_2}
		\Ll\vert \ca( \Lambda_L, M) - \cc(M/\vert \Lambda_L \vert) \Rr\vert + \Ll\vert \ca_*( \Lambda_L, M) - \cc(M/\vert \Lambda_L \vert) \Rr\vert  \leq C L^{-\gamma_2},
	\end{align}
	where $\cc(\rho)$ is the same as that defined in \eqref{eq.defCLimit}.
\end{proposition}

We establish at first a result of the local equivalence of ensembles. Similar result can be found in \cite[Appendix 2]{kipnis1998scaling} and other references. In our setting, for any $\Lambda \subset \Zd$ and $\epsilon \in (0,1)$, we define the following set of integers such that the empirical density is not degenerate 
\begin{align}\label{eq.defEmpricalTrunc}
	\M_{\epsilon}(\Lambda) := \{M \in \N_+: \epsilon \leq M/\vert \Lambda\vert \leq 1 - \epsilon\}.
\end{align}

\begin{lemma}\label{lem.localEquiv}
	Let $L, \ell \in \N_+$ and $M \in \N$, we denote by $\hat{\rho}$ the empirical density $\hat{\rho} := \frac{M}{\vert \Lambda_L \vert}$.
	\begin{enumerate}
		\item If $10 \ell^2 \leq L$ and $0 \leq M \leq L^d$, then for any local function $f \in \F_0(\Lambda_\ell)$ such that ${\bracket{f}_{\Lambda_{\ell}, N} \geq 0}$ for any $N \in \N$, we have 
		\begin{align}\label{eq.localEquiv1}
			\bracket{f}_{\Lambda_L, M} \leq \Ll(1 + 4 \Ll(\frac{\ell^2}{L}\Rr)^d\Rr) \bracket{f}_{\hat{\rho}}.
		\end{align}
		
		\item Given $\epsilon \in (0,1)$, if $10 \ell^{2d} \leq \epsilon L^d$ and $M \in \M_{\epsilon}(\Lambda_L)$, then for any local function $f \in \F_0(\Lambda_\ell)$, we have 
		\begin{align}\label{eq.localEquiv2}
			\Ll\vert \bracket{f}_{\Lambda_L, M} - \bracket{f}_{\hat{\rho}} \Rr\vert  \leq  \frac{1}{\epsilon}\Ll(\frac{\ell^2}{L}\Rr)^d \bracket{\vert f \vert }_{\hat{\rho}}.
		\end{align}
	\end{enumerate}
\end{lemma}

\begin{proof}
	The proof is similar to Lemma~\ref{lem.BiasProba}. For $f \in \F_0(\Lambda_\ell)$, we decompose the expectation as 
	\begin{align}\label{eq.localEquivDecom}
		\bracket{f}_{\Lambda_L, M} - \bracket{f}_{\hat{\rho}} =  \sum_{M=0}^{\vert \Lambda_L\vert} \Ll( \frac{ \P_{\Lambda_L, M}\Ll[\sum_{x\in \Lambda_\ell} \eta_x = N\Rr]}{\P_{\hat{\rho}}\Ll[\sum_{x\in \Lambda_\ell} \eta_x = N\Rr]} - 1 \Rr) \P_{\hat{\rho}}\Ll[\sum_{x\in \Lambda_\ell} \eta_x = N\Rr]\bracket{ f }_{\Lambda_\ell, N}. 
	\end{align}
	It remains to analyze the Radon--Nikodym derivative under different setting.
	\smallskip
	
	(1) Under this setting,  because $\bracket{f}_{\Lambda_{\ell}, N}$ is positive, it suffices to have an upper bound for the following probability
	\begin{align}\label{eq.RNStep1}
		\P_{\Lambda_L, M}\Ll[\sum_{x\in \Lambda_\ell} \eta_x = N\Rr] = \frac{{\vert \Lambda_L \setminus \Lambda_\ell\vert \choose M-N}}{{\vert \Lambda_L \vert \choose M}} 
		= \frac{\Ll(\frac{M!}{(M-N)!}\Rr)\Ll(\frac{(\vert \Lambda_L \vert - M)!}{(\vert \Lambda_L \setminus \Lambda_\ell\vert - (M-N))!}\Rr)}{\frac{\vert \Lambda_L \vert !}{\vert \Lambda_L \setminus \Lambda_\ell\vert !}}.
	\end{align}
	Notice that 
	\begin{equation}\label{eq.UpDown1}
		\begin{split}
			&\frac{\vert \Lambda_L \vert !}{\vert \Lambda_L \setminus \Lambda_\ell\vert !} \geq \vert \Lambda_L \setminus \Lambda_\ell\vert^{\vert \Lambda_\ell \vert}, \qquad \frac{M!}{(M-N)!} \leq M^N, \\
			&\frac{(\vert \Lambda_L \vert - M)!}{(\vert \Lambda_L \setminus \Lambda_\ell\vert - (M-N))!} \leq (\vert \Lambda_L \vert - M)^{\vert \Lambda_\ell\vert - N}, 
		\end{split}
	\end{equation}
	thus we have an upper bound for \eqref{eq.RNStep1}
	\begin{equation}\label{eq.RNStep2}
		\begin{split}
			\frac{\P_{\Lambda_L, M}\Ll[\sum_{x\in \Lambda_\ell} \eta_x = N\Rr]}{\P_{\hat{\rho}}\Ll[\sum_{x\in \Lambda_\ell} \eta_x = N\Rr]}
			&\leq \Ll(\frac{M^N (\vert \Lambda_L \vert - M)^{\vert \Lambda_\ell\vert - N}}{\vert \Lambda_L \setminus \Lambda_\ell\vert^{\vert \Lambda_\ell \vert}}\Rr) \Big/ \Ll(\frac{M^N (\vert \Lambda_L \vert - M)^{\vert \Lambda_\ell\vert - N}}{\vert \Lambda_L \vert^{\vert \Lambda_\ell \vert}}\Rr) \\
			&=  \Ll(\frac{\vert \Lambda_L \vert}{\vert \Lambda_L \setminus \Lambda_\ell\vert}\Rr)^{\vert \Lambda_\ell \vert}.
		\end{split}
	\end{equation} 
	It suffices to give an estimate for the term $\Ll(\frac{\vert \Lambda_L \vert}{\vert \Lambda_L \setminus \Lambda_\ell\vert}\Rr)^{\vert \Lambda_\ell \vert}$. Using the condition ${\frac{\ell^2}{L} \in (0, \frac{1}{10})}$, we have 
	\begin{align*}
		\Ll(\frac{\vert \Lambda_L \vert}{\vert \Lambda_L \setminus \Lambda_\ell\vert}\Rr)^{\vert \Lambda_\ell \vert} \leq \Ll(1 + 2 \Ll(\frac{\ell}{L}\Rr)^d\Rr)^{\ell^d}
		\leq e^{\log\Ll(1 + 2 \Ll(\frac{\ell}{L}\Rr)^d\Rr)\ell^d} \leq e^{2 \Ll(\frac{\ell^2}{L}\Rr)^d} \leq 1 + 4 \Ll(\frac{\ell^2}{L}\Rr)^d.
	\end{align*}
	This estimate together with \eqref{eq.RNStep1}, \eqref{eq.RNStep2} and \eqref{eq.localEquiv1} gives us the desired result for positive function $f$.
	
	\smallskip
	
	(2) Under this setting, $\bracket{f}_{\Lambda_{\ell}, N}$ can be negative, so we need two-sided estimate for the Radon--Nikodym derivative. The upper bound is already proved in \eqref{eq.RNStep2}, and we only need the lower bound, which is as follows
	\begin{equation}\label{eq.UpDown2}
		\begin{split}
			&\frac{\vert \Lambda_L \vert !}{\vert \Lambda_L \setminus \Lambda_\ell\vert !} \leq \vert \Lambda_L \vert^{\vert \Lambda_\ell \vert}, \qquad \frac{M!}{(M-N)!} \geq (M-N)^N, \\
			&\frac{(\vert \Lambda_L \vert - M)!}{(\vert \Lambda_L \setminus \Lambda_\ell\vert - (M-N))!} \geq ((\vert \Lambda_L \vert - M)-(\vert \Lambda_\ell \vert - N))^{\vert \Lambda_\ell\vert - N}.
		\end{split}
	\end{equation}
	This gives us a lower bound of the proportion
	\begin{align*}
		\frac{\P_{\Lambda_L, M}\Ll[\sum_{x\in \Lambda_\ell} \eta_x = N\Rr]}{\P_{\hat{\rho}}\Ll[\sum_{x\in \Lambda_\ell} \eta_x = N\Rr]} \geq \Ll(1 - \frac{N}{M}\Rr)^N \Ll(1 - \frac{\vert \Lambda_\ell \vert - N}{\vert \Lambda_L \vert - M}\Rr)^{\vert \Lambda_\ell\vert - N}.
	\end{align*}
	Because $M \in \M_\epsilon(\Lambda_L)$, we have 
	\begin{align*}
		0 \leq  \max \Ll\{\frac{N}{M}, \frac{\vert \Lambda_\ell \vert - N}{\vert \Lambda_L \vert - M}\Rr\} \leq \frac{\ell^d}{\epsilon L^{ d}},
	\end{align*}
	which results in
	\begin{equation}\label{eq.RNStep3}
		\frac{\P_{\Lambda_L, M}\Ll[\sum_{x\in \Lambda_\ell} \eta_x = N\Rr]}{\P_{\hat{\rho}}\Ll[\sum_{x\in \Lambda_\ell} \eta_x = N\Rr]} \geq \Ll(1 - \frac{\ell^d}{\epsilon L^{d}}\Rr)^{\ell^d} \geq 1- \frac{1}{\epsilon}\Ll(\frac{\ell^2}{L}\Rr)^d.
	\end{equation}
	Here we also make use of the condition $10 \ell^{2d} \leq \epsilon L^d$. This estimate and \eqref{eq.localEquivDecom} together conclude \eqref{eq.localEquiv2}.
\end{proof}

With this local equivalence of ensembles result, we can obtain the convergence rate of diffusion matrix under the canonical ensemble.
\begin{proof}[Proof of Proposition~\ref{prop.CanonicalEnsemble}]
	By similar analysis as Proposition~\ref{prop.Element} and Lemma~\ref{lem.ratecontrolJ} in previous sections, we obtain that 
	\begin{align}\label{eq.CanonicalPriDual}
		\id \leq \Da_*(\Lambda, N) \leq \Da(\Lambda, N) \leq \lambda \id.
	\end{align}
	The strategy is the following \emph{sandwich argument}: we prove that in scale $1 \ll \ell \ll L$, we have
	\begin{multline}\label{eq.CanonicalCompare}
		\D_*(M/\vert \Lambda_L\vert, \Lambda_\ell) - C \ell L^{-1}  \id \leq \Da_*(\Lambda_L, M) \\
		\leq \Da(\Lambda_L, M) \leq \D(M/\vert \Lambda_L\vert, \Lambda_\ell) + C \ell L^{-1} \id.
	\end{multline}
	Then the distance from $\D(M/\vert \Lambda_L \vert, \Lambda_\ell)$ (resp. $\D_*(M/\vert \Lambda_L\vert, \Lambda_\ell)$) to $\D(M/\vert \Lambda_L \vert)$ bounds that from $\Da(\Lambda_L, M)$ (resp. $\Da_*(\Lambda_L, M)$) to $\D(M/\vert \Lambda_L \vert)$.
	
	In the following, for the convenience to implement the renormalization step, we justify the two sides in \eqref{eq.CanonicalCompare} with $L=3^m, \ell=3^n, {m,n \in \N_+}$, but one can easily adapt it to the general case. We always suppose that the parameters $L=3^m, \ell=3^n$ satisfy the conditions of Lemma~\ref{lem.localEquiv}. We also denote by $\hat{\rho} := \frac{M}{\vert \cu_m \vert}$ to lighten the notation, which can be interpreted as the empirical density of particles under the canonical ensemble.
	
	\textit{Step 1: comparison between $\Da(\cu_m, M)$ and $\D(\hat{\rho}, \cu_n)$. } We propose a sub-minimizer $\ell_\xi + \tilde{\phi}_{\hat{\rho}, \cu_m, \xi}$ for the optimization problem of $\Da(\cu_m, M)$ that 
	\begin{align*}
		\tilde{\phi}_{\hat{\rho}, \cu_m, \xi} := \sum_{\substack{z \in \Z_{m,n}, \\ \dist(z, \partial \cu_m) > 3^n}} \phi_{\hat{\rho}, z + \cu_n, \xi}.
	\end{align*}
	Here we require the $\dist(z, \partial \cu_m) > 3^n$ in order to cutoff the influence of particles from the domain outside $\cu_m$. Since it is a sub-minimizer, we have 
	\begin{equation}\label{eq.CanoCompare1}
		\begin{split}
			&\frac{1}{2} \xi \cdot \Da(\cu_m, M) \xi \\
			&\leq \frac{1}{2 \chi(\hat{\rho})\vert \cu_m \vert} \bracket{ \frac{1}{2}\sum_{b\in\overline{\cu_m^*}}c_b(\pi_b(\ell_\xi + \tilde{\phi}_{\hat{\rho}, \cu_m, \xi}))^2}_{\cu_m, M}\\
			&\leq \frac{1}{\vert \Z_{m,n}\vert} \sum_{\substack{z \in \Z_{m,n}, \\ \dist(z, \partial \cu_m) > 3^n}}\frac{1}{2 \chi(\hat{\rho})\vert \cu_n \vert} \bracket{\frac{1}{2} \sum_{b\in\ovs{(z+\cu_n)}}c_b(\pi_b(\ell_\xi + \tilde{\phi}_{\hat{\rho}, z+\cu_n, \xi}))^2}_{\cu_m, M} \\
			& \qquad + \lambda 3^{-(m-n)}\vert p\vert^2.
		\end{split}
	\end{equation} 
	Here the last error term $\lambda 3^{-(m-n)}\vert p\vert^2$ comes from the calculation of Dirichlet energy in the boundary layer of width $3^n$, where we do not pose any corrector and the affine function can be calculated directly. For the Dirichlet energy in the interior and when $3^n > \r$, then for every $z \in \Z_{m,n}, \dist(z, \partial \cu_m) > 3^n$, we have that 
	\begin{align*}
		\sum_{b\in\ovs{(z+\cu_n)}}c_b(\pi_b(\ell_\xi + \tilde{\phi}_{\hat{\rho}, z+\cu_n, \xi}))^2 \in \F_0({z + \cu_{n+1}}),
	\end{align*}
	and $z+\cu_{n+1} \subset \cu_m$. Then Lemma~\ref{lem.localEquiv}-(1) applies to this case ($z \in \Z_{m,n}, {\dist(z, \partial \cu_m) > 3^n > \r}$) and we have 
	\begin{align*}
		&\frac{1}{2 \chi(\hat{\rho})\vert \cu_n \vert}\bracket{ \frac{1}{2}\sum_{b\in\ovs{(z+\cu_n)}}c_b(\pi_b(\ell_\xi + \tilde{\phi}_{\hat{\rho}, z+\cu_n, \xi}))^2}_{\cu_m, M} \\
		& \leq \frac{(1+3^{d(2n-m)})}{2 \chi(\hat{\rho})\vert \cu_n \vert}\bracket{\frac{1}{2} \sum_{b\in\ovs{(z+\cu_n)}}c_b(\pi_b(\ell_\xi + \tilde{\phi}_{\hat{\rho}, z+\cu_n, \xi}))^2}_{\hat{\rho}}\\
		& \leq (1+3^{d(2n-m)}) \Ll(\frac{1}{2} \xi \cdot \D(\cu_n, \hat{\rho})\xi\Rr). 
	\end{align*}
	We put this back to \eqref{eq.CanoCompare1} and concludes that 
	\begin{align}\label{eq.CanoCompare11}
		\Da(\cu_m, M) \leq \D(\cu_n, \hat{\rho}) + C(\lambda)(3^{-(m-n)} + 3^{-d(m-2n)})\id.
	\end{align}
	
	\textit{Step 2: comparison between $\Da_*(\cu_m, M)$ and $\D_*(\hat{\rho}, \cu_n)$. } This part is quite close to that in Step 1.
	We make the decomposition that 
	\begin{equation}\label{eq.CanoCompare2}
		\begin{split}
			&\frac{1}{2 \chi(\hat{\rho})\vert\cu_m\vert}\sum_{b \in \overline{\cu_m^*}}  \bracket{ (\pi_b \ell_{q})(\pi_b v) - \frac{1}{2} c_b(\pi_b v)^2}_{\cu_m, M} \\
			&= \frac{1}{\vert \Z_{m,n}\vert} \sum_{\substack{z \in \Z_{m,n}, \\ \dist(z, \partial \cu_m) > 3^n}} \frac{1}{2 \chi(\hat{\rho})\vert\cu_n\vert}\sum_{b\in\ovs{(z+\cu_n)}}  \bracket{  (\pi_b \ell_{q})(\pi_b v) - \frac{1}{2} c_b(\pi_b v)^2}_{\cu_m, M} \\
			& \qquad + \frac{1}{2 \chi(\hat{\rho})\vert\cu_m\vert}\sum_{\substack{b \in \overline{\cu_m^*}, \\\text{other bonds}}}  \bracket{ (\pi_b \ell_{q})(\pi_b v) - \frac{1}{2} c_b(\pi_b v)^2}_{\cu_m, M}
		\end{split}
	\end{equation}
	Let $u_{z+\cu_n, \hat{\rho}, q}$ be the maximizer for the problem $\nub_*(z+\cu_n, \hat{\rho},q)$. Then for the case $z \in \Z_{m,n}, {\dist(z, \partial \cu_m) > 3^n > \r}$ and by the definition of $\nub_*(z+\cu_n, \hat{\rho},q)$, we also have
	\begin{align*}
		&\frac{1}{2 \chi(\hat{\rho})\vert\cu_n\vert}\sum_{b\in\ovs{(z+\cu_n)}}  \bracket{ (\pi_b \ell_{q})(\pi_b v) - \frac{1}{2} c_b(\pi_b v)^2}_{\cu_m, M} \\
		&\leq \frac{1}{2 \chi(\hat{\rho})\vert\cu_n\vert}\sum_{b\in\ovs{(z+\cu_n)}}  \bracket{ (\pi_b \ell_{q})(\pi_b u_{z+\cu_n, \hat{\rho}, q}) - \frac{1}{2} c_b(\pi_b u_{z+\cu_n, \hat{\rho}, q})^2}_{\cu_m, M}\\
		&\leq (1+3^{d(2n-m)}) \Ll( \frac{1}{2} q \cdot \D_*^{-1}(\cu_n, \hat{\rho})q \Rr).
	\end{align*}
	Here from the first line to the second line, we use the fact that $u_{z+\cu_n, \hat{\rho}, q}$ is also the maximizer under the canonical ensemble and gives positive functional; see \eqref{eq.variationnu*2} and Remark~\ref{rmk.DualCanonical} for details. Then from the second line to the third line, note the positivity of the integral, Lemma~\ref{lem.localEquiv}-(1) applies. We only needs to treat the boundary layer and the layer between the small cubes appearing in the third line of \eqref{eq.CanoCompare2}, which can be solved by the following uniform point-wise estimate that 
	\begin{align*}
		(\pi_b \ell_{q})(\pi_b v) - \frac{1}{2} c_b(\pi_b v)^2 \leq \frac{1}{2}	(\pi_b \ell_{q})^2 + \frac{1}{2}	(\pi_b v)^2 -  \frac{1}{2} (\pi_b v)^2 = \frac{1}{2}	(\pi_b \ell_{q})^2.
	\end{align*}
	Therefore, \eqref{eq.CanoCompare2} can be reduced to 
	\begin{align*}
		\Da_*^{-1}(\cu_m, M) \leq \D_*^{-1}(\cu_n, \hat{\rho}) + C(\lambda)(3^{-(m-n)} + 3^{-d(m-2n)})\id,
	\end{align*}
	which implies, noting $\D_*(\cu_n, \hat{\rho}) \leq \lambda \id$, that 
	\begin{align}\label{eq.CanoCompare22}
		\Da_*(\cu_m, M) \geq \D_*(\cu_n, \hat{\rho}) - C(\lambda)\lambda^2(3^{-(m-n)} + 3^{-d(m-2n)})\id.
	\end{align}
	
	\textit{Step 3: conclusion. } Combining \eqref{eq.CanonicalPriDual}, \eqref{eq.CanoCompare11} and \eqref{eq.CanoCompare22}, we obtain 
	\begin{multline*}
		\D_*(\cu_n, \hat{\rho}) - C(\lambda)(3^{-(m-n)} + 3^{-d(m-2n)})\id \leq \Da_*(\cu_m, M) \\ 
		\leq \Da(\cu_m, M) \leq \D(\cu_n, \hat{\rho}) + C(\lambda)(3^{-(m-n)} + 3^{-d(m-2n)})\id.
	\end{multline*}
	Applying \eqref{eq.defCCanonical} and \eqref{eq.defCLimit}, this is interpreted as 
	\begin{multline*}
		\cc_*(\cu_n, \hat{\rho}) - C(\lambda)\chi(\hat{\rho})(3^{-(m-n)} + 3^{-d(m-2n)})\id \leq \ca_*(\cu_m, M) \\ 
		\leq \ca(\cu_m, M) \leq \cc(\cu_n, \hat{\rho}) + C(\lambda)\chi(\hat{\rho})(3^{-(m-n)} + 3^{-d(m-2n)})\id.
	\end{multline*}
	Inserting the result of homogenization \eqref{eq.main1A_2}, and by a choice of $n$ that
	\begin{align*}
		{d(m-2n) = \gamma_1 n \Longleftrightarrow n = \frac{d m}{2d + \gamma_1}},
	\end{align*}
	we obtain that 
	\begin{align*}
		\vert \ca(\cu_m, M) - \cc(\hat{\rho})\vert + \vert \ca_*(\cu_m, M) - \cc(\hat{\rho})\vert \leq C 3^{- \left(\frac{d\gamma_1}{2d+\gamma_1}\right)m}.
	\end{align*}
	This is the desired result \eqref{eq.main1B_2} by setting $\gamma_2 := \frac{d\gamma_1}{2d+\gamma_1}$.
\end{proof}

\subsection{Construction of density-free local corrector}\label{subsec.FreeCorrector}
This subsection is devoted to removing the dependence of density in the local corrector. A first natural idea is to consider the following variational problem for bounded set $\Lambda \subset \Zd$ and $\xi \in \Rd$
\begin{align}\label{eq.defmu}
	\mu(\Lambda, \xi) := \inf_{v \in \F_0(\Lambda^-)} \sup_{\rho \in [0,1]} \Ll\{ \frac{1}{\vert\Lambda\vert} \sum_{b\in\ov{\Lambda^*}} \bracket{ \frac{1}{2}c_b(\pi_b (\ell_\xi + v))^2}_\rho - \frac{1}{2} \xi \cdot \cc(\rho)\xi\Rr\}. 
\end{align}
As expected, the uniform convergence can be improved as follows. 
\begin{proposition}\label{prop.StrongCanonicalEnsemble}
	There exists an exponent $\gamma_3(d, \lambda, \r) > 0$ and a constant ${C(d, \lambda, \r) < \infty}$ such that for every $L \in \N_+$ and $\xi \in B_1$, we have 
	\begin{align}\label{eq.StrongCanonicalEnsemble}
		0 \leq \mu(\Lambda_L, \xi) \leq C L^{-\gamma_3}.
	\end{align}
\end{proposition}

Compared with \eqref{eq.defNu}, the variational problem \eqref{eq.defmu} is more complicated. Therefore, instead of attacking \eqref{eq.defmu} directly, we turn to construct \emph{one} sub-minimizer with a good uniform convergence for $\rho \in [0,1]$. This is the main task in the remaining part of this subsection.  We are inspired by the function used to prove the qualitative version of \eqref{eq.mainUniform} in the previous work \cite[Lemma~2.1]{fuy}, where the main idea is to make the linear combination of the local corrector using the empirical density. Thus, we propose our density-free local corrector built on \eqref{eq.defCorrector}
\begin{align}\label{eq.defFreeCorrector1}
	\hat{\phi}^{(\epsilon)}_{m,n,\xi} := \sum_{\substack{z \in \Z_{m,n}, \\ \dist(z, \partial \cu_m) > 3^n}}  \phi_{\bar{\eta}_0  \vee \epsilon \wedge (1- \epsilon), z+\cu_n, \xi}.
\end{align}
Recall $\Z_{m,n} = 3^n \Zd \cap \cu_m$. Here $\bar{\eta}_0$ is a (random) empirical density defined as ${\bar{\eta}_0 := \frac{1}{\vert \cu_m^-\vert} \sum_{x \in  \cu_m^-} \eta_x}$. The mapping $\rho \mapsto \rho \vee \epsilon \wedge (1- \epsilon)$ restricts the density away from $0$ and $1$. That is, we make the truncation both for the spatial boundary layer and for the density, in order to avoid the perturbation from the rare but degenerate cases. More explicitly, recall the notation $\M_\epsilon(\Lambda)$ defined in \eqref{eq.defEmpricalTrunc}, and denote by $M_*, M^*$
\begin{align}\label{eq.defMMinMax}
	M_* := \min \M_\epsilon(\cu_m^-), \qquad M^* := \max \M_\epsilon(\cu_m^-),
\end{align}
then we can rewrite this corrector as
\begin{equation}\label{eq.defFreeCorrector2}
	\hat{\phi}^{(\epsilon)}_{m,n,\xi} =  \sum_{M =0}^\infty \Ll(\sum_{\substack{z \in \Z_{m,n}, \\ \dist(z, \partial \cu_m) > 3^n}}  \phi_{\frac{M \vee M_* \wedge M^*}{\vert \cu_m^-\vert}, z+\cu_n, \xi}\Rr)\Ind{\sum_{x \in  \cu_m^-} \eta_x = M} . 
\end{equation}
When $n, \epsilon$ is well-chosen with respect to $m$, this corrector will give us uniform convergence under all grand canonical ensembles. 
\begin{proposition}\label{prop.rhoFreeCorrectorHomo}
	There exist an exponent $\gamma_4(d, \lambda, \r) > 0$ and a constant ${C(d, \lambda, \r) < \infty}$ such that for every $m \in \N_+$, by setting $n := \lfloor \frac{m}{9d+3}\rfloor$ and $\epsilon := 3^{- \frac{2d m}{9d+3}}$, we have 
	\begin{align}\label{eq.rhoFreeCorrectorHomo}
		\sup_{\rho \in [0,1], \xi \in B_1} \Ll\vert \frac{1}{\vert \cu_m \vert} \bracket{\sum_{b \in \overline{\cu_m^*}} \frac{1}{2} c_b (\pi_b (\ell_\xi + \hat{\phi}^{(\epsilon)}_{m,n,\xi}))^2}_\rho - \frac{1}{2} \xi \cdot \cc(\rho) \xi \Rr\vert   \leq C 3^{-\gamma_4 m}.
	\end{align}
\end{proposition}
\begin{proof}
	We assume $\r \ll 3^n \ll 3^m$ and $0 < \epsilon \ll 1$ with $n, \epsilon$ to be determined in the end.  Since $\hat{\phi}^{(\epsilon)}_{m,n,\xi} \in \F_0(\cu_m^-)$, we compare it with the variational problem \eqref{eq.defNu}, which gives
	\begin{align}\label{eq.FreeLower}
		\frac{1}{\vert \cu_m \vert} \bracket{\sum_{b \in \overline{\cu_m^*}} \frac{1}{2} c_b (\pi_b (\ell_\xi + \hat{\phi}^{(\epsilon)}_{m,n,\xi}))^2}_\rho \geq \frac{1}{2} \xi \cdot \cc(\rho, \cu_m) \xi \geq \frac{1}{2} \xi \cdot \cc(\rho) \xi.
	\end{align}
	Thus it suffices to give the upper bound, and we can decompose the sum as
	\begin{equation}\label{eq.FreeDecom}
		\begin{split}
			&\frac{1}{\vert \cu_m \vert} \sum_{b \in \overline{\cu_m^*}} \bracket{ \frac{1}{2} c_b (\pi_b (\ell_\xi + \hat{\phi}^{(\epsilon)}_{m,n,\xi}))^2}_\rho = \mathbf{I} + \mathbf{II} + \mathbf{III},\\
			\mathbf{I}& :=  \sum_{b \in \overline{\cu_m^*}, \dist(b, \partial \cu_m) > 3^n}\frac{1}{\vert \cu_m \vert} \bracket{ \frac{1}{2} c_b (\pi_b (\ell_\xi + \hat{\phi}^{(\epsilon)}_{m,n,\xi}))^2}_\rho,\\
			\mathbf{II}& :=  \sum_{\substack{b \in \overline{\cu_m^*} \setminus (\cu_m, \cu_m^-)^* \\\dist(b, \partial \cu_m) \leq 3^n}}\frac{1}{\vert \cu_m \vert} \bracket{ \frac{1}{2} c_b (\pi_b (\ell_\xi + \hat{\phi}^{(\epsilon)}_{m,n,\xi}))^2}_\rho,\\
			\mathbf{III}& := \sum_{b \in (\cu_m, \cu_m^-)^*}\frac{1}{\vert \cu_m \vert} \bracket{ \frac{1}{2} c_b (\pi_b (\ell_\xi + \hat{\phi}^{(\epsilon)}_{m,n,\xi}))^2}_\rho.
		\end{split}
	\end{equation}
	Roughly, $\mathbf{I}$ is the main contribution, while the terms $\mathbf{II}$ and $\mathbf{III}$ come from the boundary layer. Since the order of boundary layer is $3^{(d-1)m}$, they vanish after the normalization $\frac{1}{\vert \cu_m\vert}$ when $m \nearrow \infty$. See Figure~\ref{fig.Free} for an illustration. We treat these three terms one by one in the remaining paragraphs.
	
	\begin{figure}
		\centering
		\includegraphics[height=150pt]{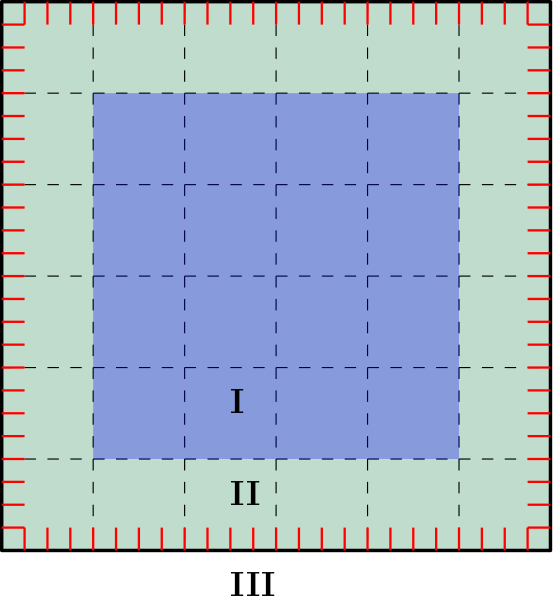}
		\caption{An illustration for the decomposition in \eqref{eq.FreeDecom}: the bonds in the term $\mathbf{I}$, $\mathbf{II}$ and $\mathbf{III}$ are respectively marked in blue, green and red.}\label{fig.Free}
	\end{figure}

	\bigskip
	
	\textit{Step~1: term $\mathbf{I}$.} This is the case that the bonds stay in the interior, and is the main contribution of the Dirichlet energy. This term can be further decomposed as 
	\begin{align*}
		\mathbf{I} \leq \frac{1}{\vert \Z_{m,n} \vert} \sum_{\substack{z \in \Z_{m,n}, \\ \dist(z, \partial \cu_m) > 3^n}} \frac{1}{\vert \cu_n\vert}\sum_{b \in \ovs{(z+\cu_n)}} \bracket{ \frac{1}{2} c_b (\pi_b (\ell_\xi + \hat{\phi}^{(\epsilon)}_{m,n,\xi}))^2}_\rho.
	\end{align*}
	We focus on one term $(z+\cu_n)$  above and $b \in \ovs{(z+\cu_n)}$, and have the following observation 
	\begin{equation}\label{eq.FreeIDiff}
		\pi_b \hat{\phi}^{(\epsilon)}_{m,n,\xi} (\eta) =  \sum_{M = 0}^\infty \Ind{\sum_{x \in  \cu_m^-} \eta_x = M}  \pi_b  \phi_{\frac{M \vee M_* \wedge M^*}{\vert \cu_m^-\vert}, z+\cu_n, \xi}(\eta),
	\end{equation}
	because the Kawasaki operator $\pi_b$ is conservative for the number of particles in $\cu_m^-$ and only one local corrector is perturbed. Therefore, we apply the decomposition of canonical ensemble for the sum over $\ovs{(z+\cu_n)}$
	\begin{align*}
		&\frac{1}{\vert \cu_n\vert}\sum_{b \in \ovs{(z+\cu_n)}} \bracket{ \frac{1}{2} c_b (\pi_b (\ell_\xi + \hat{\phi}^{(\epsilon)}_{m,n,\xi}))^2}_\rho\\
		&= \sum_{M =0}^\infty \P_\rho\Ll[\sum_{x \in  \cu_m^-} \eta_x = M\Rr] \frac{1}{\vert \cu_n\vert}\sum_{b \in \ovs{(z+\cu_n)}}\bracket{\frac{1}{2} c_b \Ll(\pi_b v\Ll(M \vee M_* \wedge M^*/\vert \cu_m^-\vert, z+\cu_n, \xi\Rr)\Rr)^2}_{M, \cu_m^-} \\
		&\leq (1+4 \cdot 3^{d(2n-m)}) \sum_{M =0}^\infty \P_\rho\Ll[\sum_{x \in  \cu_m^-} \eta_x = M\Rr]\\
		&\qquad \times\frac{1}{\vert \cu_n\vert} \sum_{b \in \ovs{(z+\cu_n)}}\bracket{\frac{1}{2} c_b \Ll(\pi_b v\Ll(M \vee M_* \wedge M^*/\vert \cu_m^-\vert, z+\cu_n, \xi\Rr)\Rr)^2}_{\frac{M}{\vert \cu_m^-\vert}}.
	\end{align*}
	We use the definition \eqref{eq.defCorrector} that $\ell_\xi + \phi_{\frac{M}{\vert \cu_m^-\vert}, z+\cu_n, \xi} =  v\Ll(M/\vert \cu_m^-\vert, z+\cu_n, \xi\Rr)$ from the first line to the second line. From the second line to the third line, we use the local equivalence of ensembles in \eqref{eq.localEquiv1} and need to assume
	\begin{align}\label{eq.FreeAssumption1}
		1 \leq n < \frac{m}{2}.
	\end{align} 
	We then study the last line and distinguish it in $3$ cases
	
	\smallskip
	
	\textit{Step~1.1: case $M \in \M_{\epsilon}(\cu_m^-)$.} Then there is no bias between the probability space and the associated corrector, thus we have 
	\begin{multline}\label{eq.FreeDirichletCase1}
		\forall M \in \M_{\epsilon}(\cu_m^-), \quad \frac{1}{\vert \cu_n\vert} \sum_{b \in \ovs{(z+\cu_n)}}\bracket{\frac{1}{2} c_b \Ll(\pi_b v\Ll(M \vee M_* \wedge M^*/\vert \cu_m^-\vert, z+\cu_n, \xi\Rr)\Rr)^2}_{\frac{M}{\vert \cu_m^-\vert}} \\
		=\Ll(\frac{1}{2} \xi \cdot \cc(M/\vert \cu_m^-\vert, z+\cu_n) \xi\Rr).
	\end{multline}
	
	\smallskip
	\textit{Step~1.2: case $0 \leq M \leq M_*$.} For this case, we have the following estimate
	\begin{align*}
		&\frac{1}{\vert \cu_n\vert} \sum_{b \in \ovs{(z+\cu_n)}}\bracket{\frac{1}{2} c_b \Ll(\pi_b v\Ll(M \vee M_* \wedge M^*/\vert \cu_m^-\vert, z+\cu_n, \xi\Rr)\Rr)^2}_{\frac{M}{\vert \cu_m^-\vert}} \\
		&= \frac{1}{\vert \cu_n\vert} \sum_{b \in \ovs{(z+\cu_n)}}\bracket{\frac{1}{2} c_b \Ll(\pi_b v\Ll(M_*/\vert \cu_m^-\vert, z+\cu_n, \xi\Rr)\Rr)^2}_{\frac{M}{\vert \cu_m^-\vert}}\\
		&\leq \Theta_{\frac{M}{\vert \cu_m^-\vert}, \frac{M_*}{\vert \cu_m^-\vert}}^{(3^n+2\r)^d} \frac{1}{\vert \cu_n\vert} \sum_{b \in \ovs{(z+\cu_n)}}\bracket{\frac{1}{2} c_b \Ll(\pi_b v\Ll(M_*/\vert \cu_m^-\vert, z+\cu_n, \xi\Rr)\Rr)^2}_{\frac{M_*}{\vert \cu_m^-\vert}}.
	\end{align*}
	Here from the second line to the third line, we use the regularity estimate \eqref{eq.BiasProba}, which gives an one-sided bias factor $\Theta_{\frac{M}{\vert \cu_m^-\vert}, \frac{M_*}{\vert \cu_m^-\vert}}^{(3^n+2\r)^d}$. This factor will not explode because the targeted density $\frac{M_*}{\vert \cu_m^-\vert}$ does not degenerate; recall the remark around \eqref{eq.RMKTheta}. More precisely, using the definition of \eqref{eq.defTheta}, $M \leq M_*$ and $\frac{M_*}{\vert \cu_m^-\vert} \simeq \epsilon \ll 1$, we have 
	\begin{align*}
		\Theta_{\frac{M}{\vert \cu_m^-\vert}, \frac{M_*}{\vert \cu_m^-\vert}} = \max\Ll\{ \frac{M}{M_*}, \frac{1- \frac{M}{\vert \cu_m^-\vert}}{1-\frac{M_*}{\vert \cu_m^-\vert}}\Rr\} = \frac{1- \frac{M}{\vert \cu_m^-\vert}}{1-\frac{M_*}{\vert \cu_m^-\vert}} \leq  1 + \frac{\frac{M_*}{\vert \cu_m^-\vert}}{1-\frac{M_*}{\vert \cu_m^-\vert}} \leq 1+2\epsilon,
	\end{align*}
	which results in 
	\begin{align*}
		\Theta_{\frac{M}{\vert \cu_m^-\vert}, \frac{M_*}{\vert \cu_m^-\vert}}^{(3^n+2\r)^d} \leq (1+2\epsilon)^{(3^n+2\r)^d} \leq 1 + 10 \cdot 3^{dn} \cdot \epsilon.
	\end{align*}
	Here we also need to assume 
	\begin{align}\label{eq.FreeAssumption2}
		0 < \epsilon \ll 3^{-dn}.
	\end{align} 
	Using the stationarity, this concludes that 
	\begin{multline}\label{eq.FreeDirichletCase2}
		\forall 0 \leq M < M_*, \quad \frac{1}{\vert \cu_n\vert} \sum_{b \in \ovs{(z+\cu_n)}}\bracket{\frac{1}{2} c_b \Ll(\pi_b v\Ll(M \vee M_* \wedge M^*/\vert \cu_m^-\vert, z+\cu_n, \xi\Rr)\Rr)^2}_{\frac{M}{\vert \cu_m^-\vert}} \\
		\leq \Ll(1 + 10 \cdot 3^{dn} \cdot \epsilon \Rr) \Ll(\frac{1}{2} \xi \cdot \cc(M_*/\vert \cu_m^-\vert, z+\cu_n) \xi\Rr).
	\end{multline}
	
	\smallskip
	
	\textit{Step~1.3: case $M_* < M < \vert \cu_m^- \vert$.} This case is similar to  Step~1.2, and we have 
	\begin{multline}\label{eq.FreeDirichletCase3}
		\forall M_* < M < \vert \cu_m^- \vert, \quad \frac{1}{\vert \cu_n\vert} \sum_{b \in \ovs{(z+\cu_n)}}\bracket{\frac{1}{2} c_b \Ll(\pi_b v\Ll(M \vee M_* \wedge M^*/\vert \cu_m^-\vert, z+\cu_n, \xi\Rr)\Rr)^2}_{\frac{M}{\vert \cu_m^-\vert}} \\
		\leq \Ll(1 + 10 \cdot 3^{dn} \cdot \epsilon \Rr) \Ll(\frac{1}{2} \xi \cdot \cc(M^*/\vert \cu_m^-\vert, z+\cu_n) \xi\Rr).
	\end{multline}
	
	\smallskip
	
	Combing \eqref{eq.FreeDirichletCase1}, \eqref{eq.FreeDirichletCase2}, \eqref{eq.FreeDirichletCase3} and the spatial homogeneity from (2) of Hypothesis~\ref{hyp}, we obtain an estimate of the term $\mathbf{I}$ under the assumption \eqref{eq.FreeAssumption1} and \eqref{eq.FreeAssumption2}
	\begin{multline}\label{eq.FreeIQuasi}
		\mathbf{I} \leq \Ll(1+10 \cdot 3^{d(2n-m)}+10 \cdot 3^{dn} \cdot \epsilon\Rr) \sum_{M =0}^\infty \P_\rho\Ll[\sum_{x \in  \cu_m^-} \eta_x = M\Rr]\\
		\Ll(\frac{1}{2} \xi \cdot \cc\Ll(\frac{M \vee M_* \wedge M^*}{\vert \cu_m^-\vert}, \cu_n\Rr) \xi\Rr).
	\end{multline}
	We aim to compare this quantity with the target $\frac{1}{2} \xi \cdot \cc\Ll(\rho, \cu_n\Rr) \xi$, which relies on the concentration of measure and the regularity of $\cc\Ll(\rho, \cu_n\Rr)$. We summarize this as Step~1.4. 
	
	\smallskip
	\textit{Step~1.4: regularity on density.} Recall the Markov inequality of density
	\begin{align}\label{eq.densityMarkov}
		\P_\rho\left[\Ll\vert \frac{1}{\vert \cu_m^-\vert}\sum_{x\in\cu_m^-}\eta_x - \rho \Rr\vert > \delta  \right] \leq \frac{\var_\rho[\frac{1}{\vert \cu_m^-\vert}\sum_{x\in\cu_m^-}\eta_x]}{\delta^2} = \frac{\rho(1-\rho)}{\vert \cu_m^-\vert \delta^2}.
	\end{align}
	We choose $\delta > 0$ in function of $\rho$ and distinguish two cases. Viewing the assumption \eqref{eq.FreeAssumption2}, here the threshold is tighter than $\epsilon$.
	
	\textit{Step~1.4.1: case $\rho \in [3^{-dn}, 1-3^{-dn}]$.}  We choose the window $\delta := 3^{-3dn}$ for such case, then treat the regime $\Ll\vert \frac{M}{\vert \cu_m^-\vert} - \rho\Rr\vert \leq \delta$ at first in \eqref{eq.FreeIQuasi}
	\begin{align*}
		&\sum_{M \in \N, \Ll\vert \frac{M}{\vert \cu_m^-\vert} - \rho\Rr\vert \leq \delta} \P_\rho\Ll[\sum_{x \in  \cu_m^-} \eta_x = M\Rr]	\Ll(\frac{1}{2} \xi \cdot \cc\Ll(\frac{M \vee M_* \wedge M^*}{\vert \cu_m^-\vert}, \cu_n\Rr) \xi\Rr)\\
		&\leq \sum_{M \in \N, \Ll\vert \frac{M}{\vert \cu_m^-\vert} - \rho\Rr\vert \leq \delta} \P_\rho\Ll[\sum_{x \in  \cu_m^-} \eta_x = M\Rr]	\Ll(\frac{1}{2} \xi \cdot \cc\Ll(\rho, \cu_n\Rr) \xi\Rr) \Theta_{\frac{M}{\vert \cu_m^-\vert}, \rho}^{(3^n+2\r)^d}\\
		&\leq (1+3^{dn}\delta)^{2\cdot3^{dn}}\Ll(\frac{1}{2} \xi \cdot \cc\Ll(\rho, \cu_n\Rr) \xi\Rr)\\
		&\leq (1+4 \cdot 3^{-dn})\Ll(\frac{1}{2} \xi \cdot \cc\Ll(\rho, \cu_n\Rr) \xi\Rr).
	\end{align*}
	Here because of the assumption \eqref{eq.FreeAssumption2}, we have ${M \vee M_* \wedge M^* = M}$ for all $\Ll\vert \frac{M}{\vert \cu_m^-\vert} - \rho\Rr\vert \leq \delta$. Then we apply \eqref{eq.cRegularity1} from the first line to the second line. From the second line to the third line, we apply \eqref{eq.RMKTheta} to give an upper bound for $\Theta_{\frac{M}{\vert \cu_m^-\vert}, \rho}^{(3^n+2\r)^d}$
	\begin{align*}
		\forall \rho \in [3^{-dn}, 1-3^{-dn}], \quad \Ll\vert \frac{M}{\vert \cu_m^-\vert} - \rho\Rr\vert \leq \delta &\Longrightarrow \frac{\Ll\vert\frac{M}{\vert \cu_m^-\vert}-\rho\Rr\vert}{\min\{\rho,1-\rho\}} < 3^{dn} \delta\\
		&\stackrel{\eqref{eq.RMKTheta}}{\Longrightarrow} \Theta_{\frac{M}{\vert \cu_m^-\vert}, \rho}^{(3^n+2\r)^d} \leq (1+3^{dn}\delta)^{2\cdot3^{dn}}. 
	\end{align*}
	Finally, we insert the choice $\delta = 3^{-3dn}$ from the third line to the fourth line.
	
	For the regime $\Ll\vert \frac{M}{\vert \cu_m^-\vert} - \rho\Rr\vert \geq \delta$, we just use the Markov inequality \eqref{eq.densityMarkov} and the trivial bound of $\cc$
	\begin{align*}
		\sum_{M \in \N, \Ll\vert \frac{M}{\vert \cu_m^-\vert} - \rho\Rr\vert \geq \delta} \P_\rho\Ll[\sum_{x \in  \cu_m^-} \eta_x = M\Rr]	\Ll(\frac{1}{2} \xi \cdot \cc\Ll(\frac{M \vee M_* \wedge M^*}{\vert \cu_m^-\vert}, \cu_n\Rr) \xi\Rr) \leq \lambda 3^{d(6n-m)}.
	\end{align*}
	Combing the two estimates above and \eqref{eq.FreeIQuasi}, we conclude the estimate for $\mathbf{I}$ under the condition  $\rho \in [3^{-dn}, 1-3^{-dn}]$
	\begin{align}\label{eq.FreeIQuasi1}
		\mathbf{I} \leq \Ll(\frac{1}{2} \xi \cdot \cc\Ll(\rho, \cu_n\Rr) \xi\Rr) + C(\lambda) \Ll( 3^{dn} \epsilon + 3^{-dn} + 3^{d(6n-m)}\Rr).
	\end{align}
	
	\textit{Step~1.4.2: case $\rho \in [0, 3^{-dn}) \cup (1-3^{-dn}, 1]$.} We choose the window $\delta := 3^{-\frac{dm}{2}}$, then for ${\Ll\vert \frac{M}{\vert \cu_m^-\vert} - \rho\Rr\vert \leq \delta}$, we have $ \frac{M}{\vert \cu_m^-\vert}  \in [0,2\cdot3^{-dn}] \cup [1-2\cdot3^{-dn},1]$ since we assume \eqref{eq.FreeAssumption1}. Then with the \textit{a priori} bound $\Ll(\frac{1}{2} \xi \cdot \cc\Ll(\frac{M}{\vert \cu_m^-\vert}, \cu_n\Rr) \xi\Rr) \leq \chi(\frac{M}{\vert \cu_m^-\vert})$, we have
	\begin{align*}
		&\sum_{M \in \N, \Ll\vert \frac{M}{\vert \cu_m^-\vert} - \rho\Rr\vert \leq \delta} \P_\rho\Ll[\sum_{x \in  \cu_m^-} \eta_x = M\Rr]
		\Ll(\frac{1}{2} \xi \cdot \cc\Ll(\frac{M \vee M_* \wedge M^*}{\vert \cu_m^-\vert}, \cu_n\Rr) \xi\Rr) \\
		&\leq \sum_{M \in \N, \Ll\vert \frac{M}{\vert \cu_m^-\vert} - \rho\Rr\vert \leq \delta} \P_\rho\Ll[\sum_{x \in  \cu_m^-} \eta_x = M\Rr] \chi\Ll(\frac{M}{\vert \cu_m^-\vert}\Rr)\\
		&\leq C(\lambda) 3^{-dn}.
	\end{align*}
	
	For ${\Ll\vert \frac{M}{\vert \cu_m^-\vert} - \rho\Rr\vert \geq \delta}$, we use Markov inequality \eqref{eq.densityMarkov} to give
	\begin{align*}
		\sum_{M \in \N, \Ll\vert \frac{M}{\vert \cu_m^-\vert} - \rho\Rr\vert \geq \delta} \P_\rho\Ll[\sum_{x \in  \cu_m^-} \eta_x = M\Rr]
		\Ll(\frac{1}{2} \xi \cdot \cc\Ll(\frac{M \vee M_* \wedge M^*}{\vert \cu_m^-\vert}, \cu_n\Rr) \xi\Rr) \leq C(\lambda) 3^{-dn}.
	\end{align*}
	Therefore, under \eqref{eq.FreeAssumption1} and \eqref{eq.FreeAssumption2}, the estimate in \eqref{eq.FreeIQuasi} has an upper bound for $\rho \in [0, 3^{-dn}) \cup (1-3^{-dn}, 1]$
	\begin{align}\label{eq.FreeIQuasi2}
		\mathbf{I} \leq  C(\lambda) 3^{-dn}.
	\end{align}
	
	Combing \eqref{eq.FreeIQuasi1} and \eqref{eq.FreeIQuasi2}, we conclude the uniform estimate for Step~1 that 
	\begin{align}\label{eq.FreeI}
		\mathbf{I} \leq \Ll(\frac{1}{2} \xi \cdot \cc\Ll(\rho, \cu_n\Rr) \xi\Rr) + C(\lambda) \Ll( 3^{dn} \epsilon + 3^{-dn} + 3^{d(6n-m)}\Rr).
	\end{align}
	
	\bigskip

	\textit{Step~2: term $\mathbf{II}$.} This term is the contribution of the boundary layer. For this case, because the local correctors and the indicator are both invariant when applying $\pi_b$, we have $\pi_b \hat{\phi}^{(\epsilon)}_{m,n,\xi} (\eta) = 0$. The only contribution comes from the affine function, which gives us 
	\begin{align}\label{eq.FreeII}
		\mathbf{II} \leq \sum_{\substack{b \in \overline{\cu_m^*} \setminus (\cu_m, \cu_m^-)^* \\\dist(b, \partial \cu_m) \leq 3^n}}\frac{1}{\vert \cu_m \vert} \bracket{ \frac{1}{2} c_b (\pi_b \ell_\xi )^2}_\rho \leq \lambda \chi(\rho) 3^{(n-m)}.
	\end{align}
	This term is small under the assumption \eqref{eq.FreeAssumption1}.
	
	\textit{Step~3: term $\mathbf{III}$.} This term is also the contribution of the boundary layer, but its estimate is more delicate. We denote by $b = \{y_1, y_2\}$, where $y_1 \in \cu_m^-$ and $y_2 \in \partial \cu_m$. The operator $\pi_{y_1, y_2}$ makes difference only when $\eta_{y_1} \neq \eta_{y_2}$. For example, for the case $\eta_{y_1} = 1$, $\eta_{y_2} = 0$, we have 
	\begin{align*}
		&(\pi_{y_1,y_2} \hat{\phi}^{(\epsilon)}_{m,n,\xi})(\eta) \Ind{\eta_{y_1} = 1, \eta_{y_2} = 0} \\
		&= \Ll(\hat{\phi}^{(\epsilon)}_{m,n,\xi}(\eta^{y_1,y_2}) - \hat{\phi}^{(\epsilon)}_{m,n,\xi}(\eta)\Rr) \Ind{\eta_{y_1} = 1, \eta_{y_2} = 0}\\
		&=  \sum_{M \in \M_\epsilon(\cu_m^-)} \Ll(\sum_{\substack{z \in \Z_{m,n}, \\ \dist(z, \partial \cu_m) > 3^n}}   \Ll(\phi_{\frac{M}{\vert \cu_m^-\vert}, z+\cu_n, \xi} -\phi_{\frac{M+1}{\vert \cu_m^-\vert}, z+\cu_n, \xi}\Rr)\Rr)\Ind{\sum_{x \in  \cu_m^-, x \neq y_1} \eta_x = M} \Ind{\eta_{y_1} = 1, \eta_{y_2} = 0}.
	\end{align*} 
	Here we have $\phi_{\rho, z+\cu_n, \xi}(\eta^{y_1, y_2}) = \phi_{\rho, z+\cu_n, \xi}(\eta)$, because $y_1, y_2$ are both far from the support of the local function when $3^n \geq \r$. Meanwhile, the operator $\pi_{y_1, y_2}$ will change the empirical density, that is the perturbation in the third line. We should also remark that such perturbation will vanish when $M \notin \M_\epsilon(\cu_m^-)$, since there is a regularization of density $\rho \mapsto \rho \vee \epsilon \wedge (1- \epsilon)$ in our definition \eqref{eq.defFreeCorrector1}.
	
	The case $\eta_{y_1} = 0, \eta_{y_2} = 1$ is similar. As the indicator $ \Ind{\eta_{y_1} = 1, \eta_{y_2} = 0}$ is independent of the other terms in the last line under $\P_\rho$, we obtain 
	\begin{multline}\label{eq.DensityPerturbation}
		\bracket{ c_{y_1,y_2}(\pi_{y_1,y_2}  (\ell_\xi + \hat{\phi}^{(\epsilon)}_{m,n,\xi}))^2}_\rho \\
		\leq 2 \lambda \chi(\rho) \Ll(1+ \sum_{M \in \M_\epsilon(\cu_m^-)} \P_\rho\Ll[\sum_{x \in  \cu_m^-, x \neq y_1} \eta_x = M\Rr] \bracket{\Ll(\sum_{\substack{z \in \Z_{m,n}, \\ \dist(z, \partial \cu_m) > 3^n}}\Delta_{M,z}\Rr)^2}_{\cu^-_m \setminus \{y_1\}, M}\Rr),
	\end{multline}
	where we define $\Delta_{M,z}$ to simplify the notation
	\begin{align}\label{eq.defDeltaMz}
		\Delta_{M,z} :=  \phi_{\frac{M+1}{\vert \cu_m^-\vert}, z+\cu_n, \xi} - \phi_{\frac{M}{\vert \cu_m^-\vert}, z+\cu_n, \xi},  \qquad M\in \M_\epsilon(\cu_m^-).
	\end{align}
	Because there are $3^{(d-1)m}$ terms like \eqref{eq.DensityPerturbation} from boundary layer $b \in (\cu_m, \cu_m^-)^*$ and the normalization factor is $\frac{1}{\vert \cu_m \vert}$, our object is to show that each term above is of order $O(3^{K n + s m})$ with some $s \in (0,1)$ and $K \in \R_+$. Then roughly we get an estimate 
	\begin{align}\label{eq.FreeIIIRough}
		\mathbf{III} \leq \sum_{b \in (\cu_m, \cu_m^-)^*}\frac{1}{\vert \cu_m \vert} \bracket{ \frac{1}{2} c_b (\pi_b (\ell_\xi + \hat{\phi}^{(\epsilon)}_{m,n,\xi}))^2}_\rho \leq O(3^{K n + (s-1) m}).
	\end{align}
	By a careful choice of $n \ll m$, we can make the contribution from  $\mathbf{III}$ small.
	
	To obtain the estimate above, we also need to treat the spatial cancellation in $\Ll(\sum_{\substack{z \in \Z_{m,n}, \\ \dist(z, \partial \cu_m) > 3^n}}\Delta_{M,z}\Rr)^2$. Thus, we develop it as
	\begin{multline*}
		\bracket{\Ll(\sum_{\substack{z \in \Z_{m,n}, \\ \dist(z, \partial \cu_m) > 3^n}}\Delta_{M,z}\Rr)^2}_{\cu^-_m \setminus \{y_1\}, M} 
		= \sum_{\substack{z \in \Z_{m,n}, \\ \dist(z, \partial \cu_m) > 3^n}} \bracket{\Ll(\Delta_{M,z}\Rr)^2}_{\cu^-_m \setminus \{y_1\}, M} \\
		+ \sum_{\substack{z,z' \in \Z_{m,n}, z \neq z'\\ \dist(z, \partial \cu_m) > 3^n, \dist(z', \partial \cu_m) > 3^n}} \bracket{\Delta_{M,z}\Delta_{M,z'}}_{\cu^-_m \setminus \{y_1\}, M}.
	\end{multline*}
	We treat the diagonal terms and off-diagonal terms separately.
	
	\smallskip
	
	\textit{Step~3.1: the diagonal term $\bracket{\Ll(\Delta_{M,z}\Rr)^2}_{\cu^-_m \setminus \{y_1\}, M}$.} For this term, we use the local equivalence of ensembles from \eqref{eq.localEquiv1}
	\begin{align}\label{eq.DeltaMzL2DiagonalPre}
		\bracket{\Ll(\Delta_{M,z}\Rr)^2}_{\cu^-_m \setminus \{y_1\}, M} &\leq (1+4 \cdot 3^{d(2n-m)}) \bracket{\Ll(\Delta_{M,z}\Rr)^2}_{\frac{M}{\vert \cu^-_m\vert - 1}}.
	\end{align}
	Assuming that $n < \frac{m}{2}$ as \eqref{eq.FreeAssumption1}, the factor $ (1+4 \cdot 3^{d(2n-m)})$ is smaller than $2$.
	To simplify the notation, we introduce the following shorthand expression 
	\begin{align}\label{eq.defrhoShort}
		\hat{\rho} := \frac{M}{\vert \cu_m^-\vert}, \qquad \hat{\rho}' := \frac{M+1}{\vert \cu_m^-\vert}, \qquad \hat{\rho}'' := \frac{M}{\vert \cu_m^-\vert-1}.
	\end{align}
	Then for $M\in \M_\epsilon(\cu_m^-)$, as the difference between $\hat{\rho}, \hat{\rho}', \hat{\rho}''$ is roughly $3^{-dm}$ and the three terms are not degenerate at $0$ or $1$, we have an estimate for the two-sided bias factor of \eqref{eq.defTheta2} 
	\begin{align}\label{eq.densityBias}
		\forall M\in \M_\epsilon(\cu_m^-), \qquad 1 \leq \tilde \Theta_{\hat{\rho}, \hat{\rho}'}, \tilde \Theta_{\hat{\rho}, \hat{\rho}''} \leq 1+  3^{-dm}\epsilon^{-1}.
	\end{align}
	Insert the notation \eqref{eq.defrhoShort} to \eqref{eq.defDeltaMz}, the quantity $\Delta_{M,z}$ has a clear expression 
	\begin{align}\label{eq.defDeltaMzClear}
		\Delta_{M,z} = \phi_{\hat{\rho}', z+\cu_n, \xi} - \phi_{\hat{\rho}, z+\cu_n, \xi},
	\end{align}
	and $\bracket{\Ll(\Delta_{M,z}\Rr)^2}_{\frac{M}{|\cu_m^-| -1}}$ is just the expectation under the grand canonical ensemble of density $\hat{\rho}''$ for difference between the correctors of densities $\hat{\rho}$ and $\hat{\rho}'$
	\begin{align}\label{eq.defDeltaMzSquare}
		\bracket{\Ll(\Delta_{M,z}\Rr)^2}_{\frac{M}{|\cu_m^-| -1}} = \bracket{(\phi_{\hat{\rho}', z+\cu_n, \xi} - \phi_{\hat{\rho}, z+\cu_n, \xi})^2}_{\hat{\rho}''}.
	\end{align}
	Therefore,  we use the continuity in density \eqref{eq.L2Regularity} and \eqref{eq.densityBias}
	\begin{equation}\label{eq.DeltaMzL2Diagonal}
		\begin{split}
			\bracket{\Ll(\Delta_{M,z}\Rr)^2}_{\cu^-_m \setminus \{y_1\}, M} &\leq 2 \bracket{\Ll(\Delta_{M,z}\Rr)^2}_{\hat{\rho}''}\\
			&\leq 10\lambda \cdot 3^{(d+2)n}\Ll(\tilde \Theta_{\hat{\rho}, \hat{\rho}''}^{(3^n+2\r)^d} + \tilde \Theta_{\hat{\rho}', \hat{\rho}''}^{(3^n+2\r)^d} -2\Rr)\\
			&\leq 20\lambda \cdot 3^{(d+2)n} \cdot \Ll(\Ll(1+  3^{-dm}\epsilon^{-1}\Rr)^{3^{dn}} - 1\Rr)\\
			&\leq 20\lambda \cdot 3^{(2d+2)n} \cdot 3^{-dm} \cdot \epsilon^{-1}.
		\end{split}
	\end{equation}
	Here we also need to add an assumption
	\begin{align}\label{eq.FreeAssumption3}
		3^{dn} \ll 3^{dm} \cdot \epsilon. 
	\end{align}
	Under this assumption, the contribution of the diagonal terms are 
	\begin{align}\label{eq.FreeIIIDiagonal}
		\sum_{\substack{z \in \Z_{m,n}, \\ \dist(z, \partial \cu_m) > 3^n}} \bracket{\Ll(\Delta_{M,z}\Rr)^2}_{\cu^-_m \setminus \{y_1\}, M}  &\leq 20\lambda \cdot 3^{d(m-n)} \cdot 3^{(2d+2)n} \cdot 3^{-dm} \cdot \epsilon^{-1}\\
		&= 20\lambda \cdot 3^{(d+2)n}  \cdot \epsilon^{-1}. \nonumber
	\end{align}
	
	\smallskip
	\textit{Step~3.2: the off-diagonal term $\bracket{\Delta_{M,z}\Delta_{M,z'}}_{\cu^-_m \setminus \{y_1\}, M}$.} Before starting the proof, we remark that the error in each off-diagonal term should be smaller than that of the diagonal term. A naive application of Cauchy--Schwarz inequality and \eqref{eq.DeltaMzL2Diagonal} does not work here, because the factor $3^{-dm}$ in \eqref{eq.DeltaMzL2Diagonal} is not enough to dominate the total number of $3^{2d(m-n)}$ off-diagonal terms. 
	
	We observe that $\bracket{\Delta_{M,z}\Delta_{M,z'}}_{\cu^-_m \setminus \{y_1\}, M}$  is actually asymptotically independent, and we need a decoupling inequality to justify it. Thus we make the following decomposition
	\begin{equation}\label{eq.FreeIIIOffDiagonalDecom}
		\begin{split}
			&\bracket{\Delta_{M,z}\Delta_{M,z'}}_{\cu^-_m \setminus \{y_1\}, M} \\
			&= \bracket{\Delta_{M,z}\Delta_{M,z'}}_{\frac{M}{\vert \cu^-_m\vert - 1}} + \Ll(\bracket{\Delta_{M,z}\Delta_{M,z'}}_{\cu^-_m \setminus \{y_1\}, M} - \bracket{\Delta_{M,z}\Delta_{M,z'}}_{\frac{M}{\vert \cu^-_m\vert - 1}}\Rr) \\
			&= \bracket{\Delta_{M,z}}^2_{\frac{M}{\vert \cu^-_m\vert - 1}} + \Ll(\bracket{\Delta_{M,z}\Delta_{M,z'}}_{\cu^-_m \setminus \{y_1\}, M} - \bracket{\Delta_{M,z}\Delta_{M,z'}}_{\frac{M}{\vert \cu^-_m\vert - 1}}\Rr).
		\end{split}
	\end{equation}
	From the second line to the third line, we use the fact that under $\P_\rho$, the variables $\Delta_{M,z}$ and $\Delta_{M,z'}$ are independent and of the same law. We continue to estimate the two terms respectively.

	For the first term in \eqref{eq.FreeIIIOffDiagonalDecom}, we make use of the definition \eqref{eq.defDeltaMzClear}. It transforms the estimate as the regularity of mean
	\begin{align*}
		\bracket{\Delta_{M,z}}^2_{\frac{M}{\vert \cu^-_m\vert - 1}} = \bracket{\phi_{\hat{\rho}', z+\cu_n, \xi} - \phi_{\hat{\rho}, z+\cu_n, \xi}}^2_{\hat{\rho}''},
	\end{align*}
	so \eqref{eq.MeanRegularity} applies
	\begin{equation}\label{eq.FreeIIIOffDiagonal1}
		\begin{split}
			\bracket{\Delta_{M,z}}^2_{\frac{M}{\vert \cu^-_m\vert - 1}} &\leq 2\bracket{\phi_{\hat{\rho}', z+\cu_n, \xi}}^2_{\hat{\rho}''} + 2\bracket{\phi_{\hat{\rho}, z+\cu_n, \xi}}^2_{\hat{\rho}''}\\
			& \leq 2\lambda\cdot 3^{(d+2)n} \Ll((\tilde \Theta_{\hat{\rho}, \hat{\rho}''}^{(3^n+2\r)^d} - 1)^2 + (\tilde \Theta_{\hat{\rho}', \hat{\rho}''}^{(3^n+2\r)^d} - 1)^2\Rr)\\
			& \leq 4\lambda \cdot 3^{(d+2)n}  \cdot \Ll(\Ll(1+  3^{-dm}\epsilon^{-1}\Rr)^{3^{dn}} - 1\Rr)^2\\
			& \leq 4\lambda \cdot 3^{(3d+2)n} \cdot 3^{-2dm} \cdot \epsilon^{-2}.
		\end{split}
	\end{equation}
	Here from the second line to the third line, we use the estimate of the two-sided bias factor in \eqref{eq.densityBias} for $ M\in \M_\epsilon(\cu_m^-)$. We also need to assume \eqref{eq.FreeAssumption3} from the third line to the fourth line. Compare the regularity of mean \eqref{eq.MeanRegularity} and  of $L^2$ \eqref{eq.L2Regularity}, we gain another factor of type $(\tilde \Theta_{\hat{\rho}, \hat{\rho}''}^{(3^n+2\r)^d} - 1)$, that is why we have $3^{-2dm}$ in the last line. 
	
	For the second term in \eqref{eq.FreeIIIOffDiagonalDecom}, we use the local equivalence of ensembles. Especially, since  $\Delta_{M,z}\Delta_{M,z'}$ is not necessarily positive, we need apply the version of \eqref{eq.localEquiv2}, which requires some supplementary conditions. These conditions are satisfied as we recall that $ M\in \M_\epsilon(\cu_m^-)$ and assume \eqref{eq.FreeAssumption3}. Then we obtain
	\begin{equation}\label{eq.FreeIIIOffDiagonal2}
		\begin{split}
			&\Ll\vert \bracket{\Delta_{M,z}\Delta_{M,z'}}_{\cu^-_m \setminus \{y_1\}, M} - \bracket{\Delta_{M,z}\Delta_{M,z'}}_{\frac{M}{\vert \cu^-_m\vert - 1}}\Rr\vert\\
			&\leq 3^{d(2n-m)} \cdot \epsilon^{-1} \cdot \bracket{\vert \Delta_{M,z}\Delta_{M,z'}\vert}_{\frac{M}{\vert \cu^-_m\vert - 1}}\\
			&\leq 3^{d(2n-m)} \cdot \epsilon^{-1} \cdot \bracket{(\Delta_{M,z})^2}_{\frac{M}{\vert \cu^-_m\vert - 1}} \\
			&\leq 20 \lambda  \cdot 3^{(4d+2)n} \cdot 3^{-2dm} \cdot \epsilon^{-2}.
		\end{split}
	\end{equation}
	Here we use Cauchy--Schwarz inequality from the second line to the third line, and insert the estimate \eqref{eq.DeltaMzL2Diagonal} in the last line. Finally, we also gain the factor  $3^{-2dm}$.
	
	We combine the estimates \eqref{eq.FreeIIIOffDiagonal1} and \eqref{eq.FreeIIIOffDiagonal2} and obtain an upper bound for the off-diagonal terms
	\begin{equation}\label{eq.FreeIIIOffDiagonal}
		\begin{split}
			&\Ll\vert \sum_{\substack{z,z' \in \Z_{m,n}, z \neq z'\\ \dist(z, \partial \cu_m) > 3^n, \dist(z', \partial \cu_m) > 3^n}} \bracket{\Delta_{M,z}\Delta_{M,z'}}_{\cu^-_m \setminus \{y_1\}, M} \Rr\vert   \\
			&\leq 20 \lambda \cdot 3^{2d(m-n)} \cdot 3^{(4d+2)n} \cdot 3^{-2dm} \cdot \epsilon^{-2}\\
			&= 20 \lambda \cdot 3^{(2d+2)n}  \cdot \epsilon^{-2}. 
		\end{split}
	\end{equation}

	The error \eqref{eq.FreeIIIOffDiagonal} from the off-diagonal terms is the leading order compared to   \eqref{eq.FreeIIIDiagonal} from the diagonal terms. We put them back to \eqref{eq.FreeIIIRough} and \eqref{eq.DensityPerturbation}, and obtain the estimates for the contribution from $\mathbf{III}$
	\begin{align}\label{eq.FreeIII}
		\mathbf{III} \leq 20 \lambda \cdot 3^{-m} \cdot \Ll( 3^{(2d+2)n}  \cdot \epsilon^{-2} +  3^{(d+2)n}  \cdot \epsilon^{-1}\Rr).
	\end{align}

	\textit{Step~4: track of parameters.} We collect all the estimates from \eqref{eq.FreeI}, \eqref{eq.FreeII} and \eqref{eq.FreeIII}, and put them back to \eqref{eq.FreeDecom} to obtain
	\begin{multline*}
		\frac{1}{\vert \cu_m \vert} \sum_{b \in \overline{\cu_m^*}} \bracket{ \frac{1}{2} c_b (\pi_b (\ell_\xi + \hat{\phi}^{(\epsilon)}_{m,n,\xi}))^2}_\rho \leq  \Ll(\frac{1}{2} \xi \cdot \cc\Ll(\rho, \cu_n\Rr) \xi\Rr) \\
		\qquad + C(\lambda) \Ll( 3^{dn} \epsilon + 3^{-dn} + 3^{d(6n-m)} + 3^{(n-m)} + 3^{(2d+2)n-m} \epsilon^{-2} \Rr)\id.
	\end{multline*}
	The choice of parameters should satisfy the assumptions \eqref{eq.FreeAssumption1}, \eqref{eq.FreeAssumption2}, \eqref{eq.FreeAssumption3}, and also make the upper bound above be small. A possible choice is 
	\begin{align*}
		\epsilon = 3^{-2dn}, \qquad n = \Ll\lfloor\frac{m}{9d+3}\Rr\rfloor,
	\end{align*}
	which gives 
	\begin{align*}
		\frac{1}{\vert \cu_m \vert} \sum_{b \in \overline{\cu_m^*}} \bracket{ \frac{1}{2} c_b (\pi_b (\ell_\xi + \hat{\phi}^{(\epsilon)}_{m,n,\xi}))^2}_\rho \leq  \Ll(\frac{1}{2} \xi \cdot \cc\Ll(\rho, \cu_n\Rr) \xi\Rr) + C(\lambda)3^{- \frac{m}{12}}.
	\end{align*}
	Then we apply the uniform estimate from Proposition~\ref{eq.main1A_2}, which gives us 
	\begin{align*}
		\frac{1}{\vert \cu_m \vert} \sum_{b \in \overline{\cu_m^*}} \bracket{ \frac{1}{2} c_b (\pi_b (\ell_\xi + \hat{\phi}^{(\epsilon)}_{m,n,\xi}))^2}_\rho \leq  \frac{1}{2} \xi \cdot \cc\Ll(\rho\Rr) \xi + C(d, \lambda)3^{- \frac{(1\wedge \gamma_1)m}{12}}.
	\end{align*}
	Together with \eqref{eq.FreeLower}, this concludes \eqref{eq.rhoFreeCorrectorHomo} by setting $\gamma_4 := \frac{(1\wedge \gamma_1)}{12}$.
	
\end{proof}

\begin{proof}[Proof of Proposition~\ref{prop.StrongCanonicalEnsemble}]
	Viewing Proposition~\ref{prop.Element} and Corollary~\ref{cor.defDLimit}, for any ${v \in \F_0(\Lambda^-)}$, we have 
	\begin{align*}
		\frac{1}{\vert\Lambda\vert} \sum_{b\in\ov{\Lambda^*}} \bracket{ \frac{1}{2}c_b(\pi_b (\ell_\xi + v))^2}_\rho \geq \frac{1}{2} \xi \cdot \cc(\rho, \Lambda) \xi \geq \frac{1}{2} \xi \cdot \cc(\rho) \xi,
	\end{align*}
	thus $\mu(\Lambda, \xi) \geq 0$. Moreover, Proposition~\ref{prop.rhoFreeCorrectorHomo} implies
	\begin{align*}
		0 \leq \mu(\cu_m, \xi) \leq \frac{1}{\vert \cu_m \vert} \bracket{\sum_{b \in \overline{\cu_m^*}} \frac{1}{2} c_b (\pi_b (\ell_\xi + \hat{\phi}^{(\epsilon)}_{m,n,\xi}))^2}_\rho - \frac{1}{2} \xi \cdot \cc(\rho) \xi \leq C 3^{-\gamma_4 m},
	\end{align*}
	then we prove \eqref{eq.StrongCanonicalEnsemble} along a subsequence $\{\cu_m\}_{m \in \N_+}$. Like (4) of Proposition~\ref{prop.Element}, we can prove that $\mu(\cu_m, \xi)$ is a subadditive quantity, so Lemma~\ref{lem.WhitneySub} applies and generalizes the uniform convergence \eqref{eq.StrongCanonicalEnsemble} to general cubes $\Lambda_L$ with $\gamma_3 := \gamma_4$.
\end{proof}

\subsection{Stationary corrector}\label{subsec.Stationary}
Recall that we have two definitions of the conductivity, that is \eqref{eq.defC} for $\c(\rho)$ and \eqref{eq.defCLimit} for $\cc(\rho)$. In this part, we will establish the identity $\c(\rho) = \cc(\rho)$, so the results proved in the previous sections are indeed for the convergence to $\c(\rho)$. Secondly, we will give the concrete construction of $\Phi_{L}$ valid for \eqref{eq.mainUniform}.

We first establish several auxiliary lemmas on the stationary function. Recall that a mapping $y \mapsto f(y, \eta)$ is stationary iff $f(y, \eta) = \tau_y f(0, \eta) = f(0, \tau_y \eta)$. 

\begin{lemma}\label{lem.Stationary}
	The following properties hold.
	\begin{enumerate}
		\item For every $x, y, z \in \Zd$, we have 
		\begin{align}\label{eq.communtePT}
			\pi_{x+z, y+z} \tau_z = \tau_z \pi_{x,y}.
		\end{align}
		\item For every integer $1 \leq i \leq d$, every $\xi \in \Rd$ and every local function $f \in \F_0$, the following mapping is stationary.
		\begin{align}\label{eq.stationaryEnergy}
			y \mapsto c_{y,y+e_i} \pi_{y,y+e_i} \Ll(\ell_\xi + \sum_{x \in \Zd} \tau_x f\Rr)^2.
		\end{align}
		\item For every $\xi \in \Rd, f \in \F_0(\Lambda_L^-)$, we have 
		\begin{align}\label{eq.StatAvg}
			\sum_{i=1}^d \bracket{c_{0,e_i} \Ll(\pi_{0,e_i} \big(\ell_\xi + \frac{1}{\vert \Lambda_L \vert}\sum_{x \in \Zd} \tau_x f\big)\Rr)^2}_\rho 
			\leq \frac{1}{\vert \Lambda_L \vert}\sum_{b \in \overline{\Lambda_L^*}} \bracket{c_{b} \Ll(\pi_{b} \Ll( \ell_\xi + f \Rr)\Rr)^2}_\rho.
		\end{align}
	\end{enumerate}
\end{lemma}
\begin{proof}
	The identity \eqref{eq.communtePT} can be proved by a direct verification. To prove \eqref{eq.stationaryEnergy}, it suffices to prove 
	\begin{align}\label{eq.stationaryEnergyPf}
		c_{y,y+e_i} \pi_{y,y+e_i} \Ll(\ell_\xi + \sum_{x \in \Zd} \tau_x f\Rr)^2 = \tau_y \Ll(c_{0,0+e_i} \pi_{0,e_i} \Ll(\ell_\xi + \sum_{x \in \Zd} \tau_x f\Rr)^2\Rr).
	\end{align}
	We start from its {\rhs} and evaluate the term one by one. From the definition of $c_{y,y+e_i}$, we have ${\tau_y c_{0,0+e_i} = c_{y, y+e_i}}$. For the term $\ell_\xi$, we have 
	\begin{align*}
		\tau_y  \pi_{0,e_i} \ell_\xi &= \tau_y \Ll(-\xi \cdot e_i (\eta_{e_i} - \eta_0) \Rr) = -\xi \cdot e_i (\eta_{y + e_i} - \eta_y) = \pi_{y, y+e_i} \ell_\xi.
	\end{align*}
	For the term involving $f$, we apply \eqref{eq.communtePT}
	\begin{align*}
		\tau_y \pi_{0,e_i} \Ll(\sum_{x \in \Zd} \tau_x f\Rr) &= \pi_{y, y+e_i}  \tau_y \Ll(\sum_{x \in \Zd} \tau_x f\Rr) \\
		&= \pi_{y, y+e_i} \Ll(\sum_{x \in \Zd} \tau_{x+y} f\Rr)\\
		&= \pi_{y, y+e_i} \Ll(\sum_{x \in \Zd} \tau_{x} f\Rr).
	\end{align*}
	Then we obtain \eqref{eq.stationaryEnergyPf} and establish the stationary property.
	
	Finally, we verify the inequality \eqref{eq.StatAvg}. Notice that $\pi_{y,y+z} \tau_{x} \ell_\xi = \pi_{y,y+z} \ell_\xi$, we have 
	\begin{equation}\label{eq.StatAvgJensen}
		\begin{split}
			&\sum_{i=1}^d \bracket{c_{0,e_i} \Ll(\pi_{0,e_i} \big(\ell_\xi + \frac{1}{\vert \Lambda_L \vert}\sum_{x \in \Zd} \tau_x f\big)\Rr)^2}_\rho \\
			&= \sum_{i=1}^d \bracket{c_{0,e_i} \Ll(\pi_{0,e_i} \big(  \frac{1}{\vert \Lambda_L \vert}\sum_{x \in \Lambda_L} \tau_x (\ell_\xi + f)\big) \Rr)^2}_\rho \\
			& \leq \sum_{i=1}^d \frac{1}{\vert \Lambda_L \vert}\sum_{x \in \Lambda_L} \bracket{c_{0,e_i} \Ll(\pi_{0,e_i} \Ll(  \tau_x (\ell_\xi + f) \Rr)\Rr)^2}_\rho.
		\end{split}
	\end{equation}
	From the first line to the second line, we also use the fact that $f \in \F_0(\Lambda_L^-)$, so the derivative vanish when translation is outside $\Lambda_L$. From the second line to the third line above, we apply Jensen's inequality. Then we simplify the result
	\begin{align*}
		\bracket{c_{0,e_i} \Ll(\pi_{0,e_i} \Ll(  \tau_x (\ell_\xi + f) \Rr)\Rr)^2}_\rho &= \bracket{\tau_{-x}c_{0,e_i} \Ll(\pi_{0,e_i} \Ll(  \tau_x (\ell_\xi + f) \Rr)\Rr)^2}_\rho \\
		&= \bracket{c_{-x,-x+e_i} \Ll(\pi_{-x,-x+e_i} \Ll( \ell_\xi + f \Rr)\Rr)^2}_\rho.
	\end{align*}
	The equality in the first line comes from the stationarity of $\P_\rho$, i.e. $\bracket{\tau_x F}_\rho = \bracket{F}_\rho$ for all $x \in \Zd$, while the equality from the second line comes from \eqref{eq.communtePT}. We put this identity back to \eqref{eq.StatAvgJensen} and conclude the desired result
	\begin{align*}
		&\sum_{i=1}^d \bracket{c_{0,e_i} \Ll(\pi_{0,e_i} \big(\ell_\xi + \frac{1}{\vert \Lambda_L \vert}\sum_{x \in \Lambda_L} \tau_x f\big)\Rr)^2}_\rho \\
		&\leq \sum_{i=1}^d \frac{1}{\vert \Lambda_L \vert}\sum_{x \in \Lambda_L} \bracket{c_{-x,-x+e_i} \Ll(\pi_{-x,-x+e_i} \Ll( \ell_\xi + f \Rr)\Rr)^2}_\rho \\
		&\leq \frac{1}{\vert \Lambda_L \vert}\sum_{b \in \overline{\Lambda_L^*}} \bracket{c_{b} \Ll(\pi_{b} \Ll( \ell_\xi + f \Rr)\Rr)^2}_\rho.
	\end{align*}
\end{proof}

With the preparation of Lemma~\ref{lem.Stationary} and Proposition~\ref{prop.StrongCanonicalEnsemble}, we can now prove the main result in this subsection.
\begin{proposition}\label{prop.ProofMain}
	For every $\rho \in [0,1]$, the quantities defined in \eqref{eq.defC} and \eqref{eq.defCLimit} coincide
	\begin{align*}
		\c(\rho) = \cc(\rho).
	\end{align*}
	Moreover, let $\phi_{\Lambda, \xi}$ be an optimizer for $\mu(\Lambda, \xi)$ defined in \eqref{eq.defmu} and $\Phi_{L} \in \F^d_0(\Lambda_L)$ be defined as
	\begin{align}\label{eq.main1C_F}
		\Phi_{L} := \Ll(\frac{1}{\vert \Lambda_L \vert}\phi_{\Lambda_L, e_1}, \frac{1}{\vert \Lambda_L \vert}\phi_{\Lambda_L, e_2}, \cdots, \frac{1}{\vert \Lambda_L \vert}\phi_{\Lambda_L, e_d} \Rr),
	\end{align}
	then there exists an exponent $\gamma(d,\lambda, \r) > 0$ and a positive constant $C(d, \lambda, \r) < \infty$, such that 
	\begin{align}\label{eq.main1C_2}
		\sup_{\rho \in [0,1]}\vert \c(\rho; \Phi_{L}) - \c(\rho)\vert \leq C L^{-\gamma}.
	\end{align}
\end{proposition}
\begin{proof}
	The proof is similar to the sandwich argument in \eqref{eq.CanonicalCompare} and can be divided the into three steps. Similar arguments can be found in \cite[Theorem~B.1]{bulk}.
	
	\textit{Step 1: $\c(\rho) \leq  \cc(\rho)$. } We recall \eqref{eq.defQuadra} that 
	\begin{align*}
		\xi \cdot \c(\rho; \Phi_{L}) \xi =  \sum_{i=1}^d \bracket{c_{0,e_i}\Ll(\xi \cdot \Ll\{ e_i(\eta_{e_i} - \eta_0) - \pi_{0,e_i}(\sum_{y \in \Zd} \tau_y \Phi_{L})\Rr\}\Rr)^2}_{\rho}.
	\end{align*}
	and notice the fact that 
	\begin{align*}
		\xi \cdot x(\eta_x - \eta_0) &= - \pi_{0,x} \ell_{\xi},\\
		\xi \cdot \Phi_{L} &= \frac{1}{\vert \Lambda_L \vert}\phi_{\Lambda_L, \xi},
	\end{align*}
	where the second identity comes from the linear map $\xi \mapsto \phi_{\Lambda_L, \xi}$. Therefore, we have 
	\begin{align*}
		\xi \cdot \c(\rho; \Phi_{L}) \xi &= \sum_{i=1}^d \bracket{c_{0,e_i}\Ll(\pi_{0,e_i}\Ll(\ell_\xi + \frac{1}{\vert \Lambda_L \vert} \sum_{y \in \Zd} \tau_y \phi_{\Lambda_L, \xi}\Rr)\Rr)^2}_{\rho} \\
		&\leq \frac{1}{\vert \Lambda_L \vert} \sum_{b \in \overline{\Lambda_L^*}} \bracket{c_{b} \Ll(\pi_{b} \Ll( \ell_\xi + \phi_{ \Lambda_L, \xi} \Rr)\Rr)^2}_\rho \\
		&\leq \xi \cdot \cc(\rho) \xi + \mu(\Lambda_L, \xi),
	\end{align*}
	where we apply \eqref{eq.StatAvg} from the first line to the second line and \eqref{eq.defmu} from the second line to the third line. Using the estimate \eqref{eq.StrongCanonicalEnsemble} about $\mu(\Lambda_L, \xi)$, we obtain an important inequality chain with $C, \gamma_3$ independent of $\rho \in [0,1]$ 
	\begin{align}\label{eq.Dchain}
		\xi \cdot \c(\rho) \xi \leq  \xi \cdot\c(\rho; \Phi_{L})\xi  \leq  \xi \cdot \cc(\rho) \xi + C L^{-\gamma_3} \vert \xi \vert^2.
	\end{align}
	We take the first one and the third one, and let $L \nearrow \infty$, and obtain the desired result 
	\begin{align*}
		\c(\rho) \leq \lim_{L \to \infty } \cc(\rho, \Lambda_L) = \cc(\rho).
	\end{align*}

	\textit{Step 2: $\c(\rho) \geq  \cc(\rho)$. } We pick $F_K$ as a sequence of local functions to approximate $\c(\rho)$, and compare $v_{p,K} = \ell_p + \sum_{x \in \Zd} p \cdot \tau_x F_K$ in the functional of $\nub_*(\rho, \Lambda_L, q)$
	\begin{align*}
		\nub_*(\rho, \Lambda_L, q) &= \frac{1}{2}q\cdot \D_*^{-1}(\rho,\Lambda_L)\cdot q \\
		&\geq  \frac{1}{2 \chi(\rho)\vert\Lambda\vert}\sum_{b\in \ovs{\Lambda} }  \bracket{ (\pi_b \ell_{q})(\pi_b v_{p,K}) - \frac{1}{2} c_b(\pi_b v_{p,K})^2}_{\rho}.
	\end{align*} 
	Here we pick $F_K \in \mcl F_0(\Lambda_L)$. Applying the stationarity of mapping in \eqref{eq.stationaryEnergy} and \eqref{eq.MeanZeroLocal}, we obtain that 
	\begin{align*}
		\frac{1}{2}q\cdot \D_*^{-1}(\rho,\Lambda_L) q \geq p \cdot q - \frac{1}{4\chi(\rho)} p \cdot \c(\rho;F_K) p.
	\end{align*}
	Letting $L,K \to \infty$, this yields 
	\begin{align*}
		\frac{1}{2}q\cdot \D^{-1}(\rho)q \geq p \cdot q - \frac{1}{4\chi(\rho)} p \cdot \c(\rho) p.
	\end{align*}
	By taking $q = \D(\rho)p$ and recall $2\chi(\rho)\D(\rho) = \cc(\rho)$ from \eqref{eq.defCLimit}, we conclude $\c(\rho) \geq  \cc(\rho)$.
	
	\textit{Step 3: Error estimate. } Once we identify that $\c(\rho) = \cc(\rho)$, we go back to \eqref{eq.Dchain} to obtain
	\begin{align*}
		\xi \cdot \c(\rho) \xi \leq  \xi \cdot\c(\rho; \Phi_{L})\xi  \leq  \xi \cdot \c(\rho) \xi + C L^{-\gamma_3} \vert \xi \vert^2.
	\end{align*}
	Since $C, \gamma_3$ come from Proposition~\ref{prop.StrongCanonicalEnsemble} and are independent of $\rho$, this concludes the estimate \eqref{eq.main1C_2} by setting $\gamma := \gamma_3$.
\end{proof}

\subsection{Proof of Theorem~\ref{thm.mainUniform} and Theorem~\ref{thm.main1}}\label{subsec.pfMain}
We summarize the proof of the main theorems in this paper.
\begin{proof}[Proof of Theorem~\ref{thm.mainUniform} and Theorem~\ref{thm.main1}]
	The finite-volume approximation defined in \eqref{eq.defC2} gives a subadditive quantity $\cc(\rho, \Lambda_L)$, which defines a limit $\cc(\rho)$ in \eqref{eq.defCLimit}. Using the dual quantity, we prove the convergence rate of this approximation in Proposition~\ref{prop.GrandCanonicalEnsemble} and Proposition~\ref{prop.CanonicalEnsemble}. In Lemma~\ref{lem.Stationary}, we identify that $\cc(\rho)$ coincides with $\c(\rho)$ defined in \eqref{eq.defC}, therefore Proposition~\ref{prop.GrandCanonicalEnsemble} together with Proposition~\ref{prop.CanonicalEnsemble} proves Theorem~\ref{thm.main1}. Finally, Proposition~\ref{prop.StrongCanonicalEnsemble} removes the dependence of density of the local corrector and Lemma~\ref{lem.Stationary} proves Theorem~\ref{thm.mainUniform}, where \eqref{eq.main1C_F} gives a concrete construction of the density-free local corrector.
\end{proof}

\subsection{Application: CLT variance in gradient replacement}\label{subsec.CLT}
This subsection is an application of the results in Theorems~\ref{thm.mainUniform} and ~\ref{thm.main1}. We will calculate the CLT variance in the gradient replacement, which corresponds to \cite[Section~5]{fuy}. Especially, the homogenization theory provides an alternative approach besides Varadhan's characterization of closed form (\cite[Theorem~4.1]{fuy}), together with a convergence rate.

Let us make our setting clear.
For any finite set $\Lambda \subset \Zd$ and $M\in \N:
0\le M \le |\La|$, we denote by
\begin{align*}
	\X_{\Lambda} := \{0,1\}^\Lambda, \qquad \X_{\Lambda,M} := \Ll\{\eta\in \X_{\Lambda}; \sum_{x\in \La}\eta_x=M \Rr\}.
\end{align*}
Using the notation \eqref{eq.Restriction}, every configuration
$\eta \in \X=\{0,1\}^{\Zd}$ can be decomposed into two parts 
\begin{align}\label{eq.defConfigDecom}
	\eta = \xi \cdot \zeta, \qquad \xi := \eta_{\vert \Lambda} \in \X_{\Lambda}, \qquad \zeta := \eta_{\vert \Lambda^c} \in \X_{\Lambda^c}.
\end{align}
We let $\mathcal{L}_{\Lambda, \zeta}$ be a restriction of the generator $\mathcal{L}$
defined in \eqref{eq.Generator} on $\Lambda$ with the exterior condition
$\zeta$
\begin{align}\label{eq.defGeneratorExt}
	\mathcal{L}_{\Lambda, \zeta} f(\xi) := \sum_{b \in \Lambda^*} c_b(\xi\cdot\zeta)
	\pi_b f(\xi),
\end{align}
for a function $f$ on $\mathcal{X}_\La$. Let $\P_{\La,M,\zeta}$ be the uniform probability measure on 
$\mathcal{X}_{\La,M}$ given $\zeta$, and  we denote by $\lan\,\cdot\,\ran_{\La,M,\zeta}$
the expectation with respect to  $\P_{\La,M,\zeta}$; see Section~\ref{subsubsec.proba}.

Using the decomposition $\eta = \xi \cdot \zeta$ from \eqref{eq.defConfigDecom} on the domain $\Lambda_L = (-L/2, L/2)^d \cap \Zd$, we define the 
$\R^d$-valued functions $A_L, B_{L,\zeta}$ and $H_{L,\zeta,F}$ of $\xi$
(so those of $\eta$) as
\begin{equation}\label{eq.defALBLHL}
	\begin{split}
		A_L(\xi) &:=  \frac12 \sum_{x,y\in \Lambda_L:|x-y|=1} (\xi_y - \xi_x)(y-x)
		= - \sum_{b \in \Lambda^*} \pi_b \Ll(\sum_{x \in \Lambda_L} x\xi_x \Rr),  \\
		B_{L,\zeta}(\xi) &:= \frac12 \sum_{x,y\in \Lambda_L:|x-y|=1} W_{x,y}(\eta) (y-x)  
		= -\mathcal{L}_{\Lambda_L,\zeta} \Big( \sum_{x\in\Lambda_L} x\xi_x\Big),\\
		H_{L,\zeta,F}(\xi) &:= \sum_{x\in \Lambda_{L-r(F)-1}} \t_x (\mathcal{L} F)(\eta) 
		= \mathcal{L}_{\Lambda_L,\zeta}
		\bigg( \sum_{x\in \Lambda_{L-r(F)-1}} \t_x  F\bigg)(\xi).
	\end{split}
\end{equation}
Here $B_{L,\zeta}$
denotes the microscopic current with 
\begin{align}\label{eq.current}
	W_{x,y}(\eta) := c_{x,y}(\eta)(\eta_y - \eta_x),
\end{align}
and $F = (F_i)_{i=1}^d \in \mathcal{F}_0^d$ 
in $H_{L,\zeta,F}$ will be taken as a correction function. The correction function $F$ should be local compared to $\Lambda_L$, so we require $r(F) \leq L-1$, where $r(F)$ is the diameter of the support of $F$
\begin{align*}
	r(F) := \min \big\{r \in \N_+: F \in \mathcal{F}_0^d(\La_r)\big\}.
\end{align*}

Given $0\le M \le |\Lambda_L|$ and $\zeta\in \mathcal{X}_{\Lambda_L^c}
=\{0,1\}^{\Lambda_L^c}$, for functions $f,g:\mathcal{X}_{\La_L,M} \to \R$
such that $\bracket{f}_{\Lambda_L, M, \zeta} = \bracket{g}_{\Lambda_L, M, \zeta} = 0$, their CLT covariance and variance are respectively defined by
\begin{align}	\label{eq:CLT-v}
	\begin{aligned}
		\Delta_{L, M, \zeta}[f, g] &:= \bracket{f (- \mathcal{L}_{\Lambda_L, \zeta})^{-1}  g}_{\Lambda_L, M, \zeta},  \\
		\Delta_{L, M, \zeta}[f] &:= \Delta_{L, M, \zeta}[f,f].
	\end{aligned}
\end{align}

In the gradient replacement argument, the heuristic is that, when $F$ is well-chosen, we should have 
\begin{align*}
	D(M/\vert \Lambda_L\vert)A_L \simeq (B_{L,\zeta}-H_{L,\zeta,F}),
\end{align*}
where $D(\rho)= \c(\rho)/2\chi(\rho)$  is the diffusion matrix
defined in \eqref{eq.Einstein}. The error in this replacement is small in weak sense, so we aim to estimate the following  CLT variance 
\begin{multline}\label{eq.defQLF}
	Q_L(F;q,M,\zeta) \\
	:= 
	\frac{1}{|\Lambda_L|}\Delta_{L,M,\zeta} \Ll[q\cdot 
	\{ D(M/\vert \Lambda_L\vert)A_L-(B_{L,\zeta}-H_{L,\zeta,F})\}\Rr] 
	-  \frac12 q\cdot R(M/\vert \Lambda_L\vert;F)q,
\end{multline}
with $q \in B_1$  and $R(\rho;F) = \c(\rho;F) - \c(\rho)$ as the error defined in \eqref{eq.defR}.


Proposition~\ref{prop:7.1} below provides the convergence rate of $Q_L(F;q,M,\zeta) $, which is a quantitative refinement of the result shown in  \cite[Corollary~5.1]{fuy} and is a key for proving the hydrodynamic limit for non-gradient models. Actually, the quantity 
$Q_L(F;q,M,\zeta)$ can be divided into three pieces:
\begin{equation}\label{eq.Q6}
	\begin{split}
		Q^{(a)}_L(\tilde q, M, \zeta) &:= 	\frac{1}{|\Lambda_L|} \Delta_{L, M, \zeta}\Ll[\tilde q \cdot A_L\Rr] - 2 (\tilde q \cdot {\mathbf c}^{-1}(\hat{\rho}) \tilde q) \chi^2(\hat{\rho}),		
		\\		
		Q_L^{(b)}(F;q,M,\zeta) &: =   	\frac{1}{|\Lambda_L|}\Delta_{L,M,\zeta} \Ll[q\cdot(B_{L,\zeta}-H_{L,\zeta,F})\Rr]
		- \frac12 q\cdot {\mathbf c}(\hat{\rho};F)q,  \\  
		Q_L^{(c)}(F;q,\tilde q, M,\zeta) &: = 	\frac{1}{|\Lambda_L|} \Delta_{L,M,\zeta} \Ll[\tilde{q}\cdot A_L, 
		q\cdot(B_{L,\zeta}-H_{L,\zeta,F})\Rr]
		-  (q\cdot \tilde{q}) \chi(\hat{\rho}),			
	\end{split}
\end{equation}
with the shorthand notation $\hat{\rho} := M/\vert \Lambda_L\vert$,
for the empirical density.

\begin{proposition}\label{prop:7.1}
	By choosing $\tilde{q}=D(\hat{\rho})q$ in \eqref{eq.Q6}, 
	the following identity is valid for $Q_L(F;q,M,\zeta)$
	\begin{multline}\label{eq.QLIdentity}
		Q_L(F;q,M,\zeta) \\ = Q^{(a)}_L(D(\hat{\rho}) q, M, \zeta)
		+Q_L^{(b)}(F;q,M,\zeta) -2 Q_L^{(c)}(F;q, D(\hat{\rho})q, M,\zeta).
	\end{multline}
	
	Moreover, there exists a finite positive constant $C=C(d, \lambda, \r)$ 
	such that the quantity $Q_L(F;q,M,\zeta)$ defined in \eqref{eq.defQLF} satisfies the following estimate for all $L \in \N_+$ and $F \in \mcl F_0^d$
	\begin{multline}\label{eq.QLFBound}
		Q_L(F) :=\sup_{ \substack{q \in B_1,0 \leq M \leq \vert \Lambda_L\vert, \\ \zeta \in \mathcal{X}_{\La_L^c}}} \vert 	Q_L(F;q,M,\zeta) \vert \\
		\le C \Big(L^{-\gamma_2}+
		r(F)^d (1+r(F)^{2d} \|F\|_\infty^2) L^{-1} +
		r(F)^{2d} \| F\|_\infty L^{-d} +L^{-1}\Big).
	\end{multline}
	Here the positive exponent $\gamma_2$ is the same one as in Proposition~\ref{prop.CanonicalEnsemble} and Theorem~\ref{thm.main1}.
\end{proposition}

We study at first the error estimate for each term in \eqref{eq.QLIdentity}.

\subsubsection{CLT variance of $A_L$}
For the CLT variance of $A_L$, it was shown in Theorem 5.1 of \cite{fuy}
that
\begin{align} \label{4.1.1}
	\lim_{L, M \to \infty, M/|\La_L| \to \rho} 
	|\La_L|^{-1} \Delta_{L,M,\zeta} (\tilde q\cdot A_L)
	= 2 (\tilde q\cdot \c^{-1}(\rho)\tilde q) \chi^2(\rho),
\end{align}
uniformly in $\rho\in [0,1]$ and $\eta\in \mathcal{X}$;
recall $\zeta=\eta|_{\La_L^c}$. 
This is the key in the gradient replacement and shown based on a
characterization of closed forms given in Corollary
4.1 of \cite{fuy}.  

The computation of $Q_L^{(a)}$, as  the error in \eqref{4.1.1},
is indeed deeply related to the dual quantity studied in \eqref{eq.DualCanonical}, 
and this makes Varadhan's lemma avoidable.  
It is, in a sense, hidden in the variational formula for  $\Da_*(\Lambda_L, M)$.
In other words, our dual computation well fits to computing the CLT variance 
of  $A_L$.  This matches with the dual computation in (5.7) in the proof of 
Theorem 5.1 in \cite{fuy},
but our computation is presented at higher level on the configuration space.
We can write the error $Q_L^{(a)}$ of the CLT variance of $A_L$ defined by 
\eqref{eq.Q6} concretely in terms of the dual quantity as in \eqref{eq:Q6LMq} below.
It is important that we can give such an exact formula before taking the limit 
so that one can directly apply our estimate to obtain its convergence rate. 
Note that, to prove the hydrodynamic limit formulated in 
Theorem~\ref{thm.HLQuant},
we need the uniformity in the density $\rho \in [0,1]$ as in Corollary 5.1
in \cite{fuy}.

The decay estimate for $Q_L^{(a)}(q,M,\zeta)$ is stated as follows.
This is one of the main achievements of this paper.

\begin{proposition}\label{Theorem 5.1}
	There exists a finite positive constant $C=C(d, \lambda, \r)$ such that
	\begin{align}  \label{eq:Q(4)}
		\sup_{q \in B_1,0 \leq M \leq \vert \Lambda_L\vert, \zeta \in \mathcal{X}_{\La_L^c}} |Q_L^{(a)}(q,M,\zeta)|\le CL^{-\gamma_2}.
	\end{align}
	Here the positive exponent $\gamma_2$ is the same one as in Proposition~\ref{prop.CanonicalEnsemble} and Theorem~\ref{thm.main1}.
\end{proposition}

\begin{proof}	
	Let us implement the discussion in previous sections to $Q_L^{(a)}$. From \eqref{eq.defALBLHL} and \eqref{eq.defPiAffine}, the quantity $q \cdot A_L$ satisfies the identity
	\begin{align*}
		q \cdot A_L =  - \sum_{ b= \{x,y\} \in (\Lambda_L)^*} \pi_b \ell_q.
	\end{align*}
	Then by variational formula of the dual quantity $\widehat{D}_*(\Lambda_L, M)$ defined in \eqref{eq.DualCanonical}, the optimizer $u(\Lambda_L, q)$  satisfies that
	\begin{align*}
		\sum_{ b \in (\Lambda_L)^*} c_b \pi_b u(\Lambda_L, q) = 	\sum_{ b \in (\Lambda_L)^*} \pi_b \ell_q. 
	\end{align*}
	Recall the generator $\mathcal{L}_{\Lambda_L, \zeta}$ defined in \eqref{eq.defGeneratorExt}, thus the quantity $(- \mathcal{L}_{\Lambda_L, \zeta})^{-1} (q \cdot A_L)$ is nothing but $u(\Lambda_L, q)$ defined in (1) of Proposition~\ref{prop.Element} (see also Remark~\ref{rmk.DualCanonical}) because of the following identity 	\footnote{Here we abuse the notation because the generator defined in \eqref{eq.defGeneratorExt}, i.e. (5.1) in \cite{fuy}, is slightly different from \eqref{eq.GeneratorDomain}. Nevertheless, this tiny difference does no harm; see Remark~\ref{rmk.Domain}.}
	\begin{align}\label{eq.recover_u}
		(- \mathcal{L}_{\Lambda_L, \zeta})^{-1} (q \cdot A_L) = (\mathcal{L}_{\Lambda_L, \zeta})^{-1}\Ll(\sum_{ b \in (\Lambda_L)^*} \pi_b \ell_q\Rr) = u(\Lambda_L, q).
	\end{align}
	Therefore, we put this result back to $\Delta_{L, M, \zeta}[q \cdot A_L]$
	and obtain
	\begin{align*}
		\Delta_{L, M, \zeta}[q \cdot A_L] &=  \bracket{(q \cdot A_L) (- \mathcal{L}_{\Lambda_L, \zeta})^{-1}  (q \cdot A_L)}_{\Lambda_L, M, \zeta}\\
		&=\bracket{(q \cdot A_L)  u(\Lambda_L, q)}_{\Lambda_L, M, \zeta}\\
		&= - \sum_{ b \in (\Lambda_L)^*}  \bracket{(\pi_b \ell_q) u(\Lambda_L, q)}_{\Lambda_L, M, \zeta}\\
		&= \frac{1}{2} \sum_{ b \in (\Lambda_L)^*}  \bracket{(\pi_b \ell_q) (\pi_b u(\Lambda_L, q))}_{\Lambda_L, M, \zeta}\\
		&= \frac{1}{2} \sum_{ b \in (\Lambda_L)^*}  \bracket{c_b (\pi_b u(\Lambda_L, q))^2}_{\Lambda_L, M, \zeta}.
	\end{align*}
	Here from the third line to the fourth line, we use the identity $\pi_b \pi_b = -2 \pi_b$ and integration by part. 
	From the fourth line to the fifth line, we use the variational formula of $\widehat{D}_*(\Lambda_L, M)$ once again rewritten similar to \eqref{eq.bilinearNu}. This concludes that 
	\begin{align*}
		\vert \Lambda_L \vert^{-1} \Delta_{L, M, \zeta}[q \cdot A_L] = \Ll(2\chi(\hat{\rho})\Rr) \Ll( \frac{1}{2} q \cdot \widehat{D}^{-1}_*(\Lambda_L, M) q\Rr) = \chi(\hat{\rho}) q \cdot \widehat{D}^{-1}_*(\Lambda_L, M) q .
	\end{align*}
	We put this result back to \eqref{eq.Q6}, and obtain that 
	\begin{align} \label{eq:Q6LMq}
		Q^{(a)}_L(q, M, \zeta) & = \chi(\hat{\rho}) q \cdot \Ll(\widehat{D}^{-1}_*(\Lambda_L, M) - D^{-1}(\hat{\rho})\Rr) q  \\
		& = \chi(\hat{\rho}) q \cdot \widehat{D}^{-1}_*(\Lambda_L, M)
		\Ll(D(\hat{\rho}) - \widehat{D}_*(\Lambda_L, M) \Rr) D^{-1}(\hat{\rho}) q
		\notag\\
		&= \frac{1}{2} q \cdot \widehat{D}^{-1}_*(\Lambda_L, M)
		\Ll(\c(\hat{\rho}) - \ca_*(\Lambda_L, M) \Rr) D^{-1}(\hat{\rho}) q. \notag
	\end{align}
	Apply the quantitative homogenization of canonical ensemble in 
	\eqref{eq.main1B} (see also Proposition~\ref{prop.CanonicalEnsemble} and $\cc(\rho) = \c(\rho)$ proved in Proposition~\ref{prop.ProofMain}), 
	noting the uniform positivity of matrices $\widehat{D}_*(\Lambda_L, M)$
	and $D(\hat{\rho})$ in (2) of Proposition~\ref{prop.Element}, we obtain the desired result \eqref{eq:Q(4)}.
\end{proof}

\subsubsection{CLT variance of $B_{L,\zeta}-H_{L,\zeta,F}$ and covariance between
	$A_L$ and $B_{L,\zeta}-H_{L,\zeta,F}$}
\label{sec:5.3-D}	
The quantitative estimates of $Q_L^{(b)}$ and $Q_L^{(c)}$ are given in 
the following lemma, which is a refinement of Proposition 5.1 of \cite{fuy}.

\begin{lemma}\label{Proposition 5.1}
	There exists a finite positive constant $C=C(d, \lambda, \r)$ such that
	\begin{align}
		&  |Q_L^{(b)}(F;q,M,\zeta)| \le  C r(F)^d (1+r(F)^{2d}\|F\|_{\infty}^2) L^{-1},  \label{5.1} 
		\\  
		&  |Q_L^{(c)}(F;q,\tilde{q},M,\zeta)| \le C r(F)^{2d} \| F\|_\infty L^{-d} +CL^{-1},
		\label{5.2}
	\end{align}
	uniformly in  $q, \tilde{q} \in \R^d: |q| = |\tilde{q}|=1$, $0\le M\le |\La_L|$
	and  $\zeta\in\mathcal{X}_{\La_L^c}$.
\end{lemma}
\begin{proof}
	The proof of this lemma is exactly the same as
	that of Proposition 5.2 of \cite{Fu24}, which discusses
	the Kawasaki dynamics adding the effect of creation and annihilation of 
	particles called the Glauber effect, but the Glauber part 
	does not play any role in its proof. We give a sketch of it here. Actually, we can consider a function 
	\begin{align}\label{eq.vqF}
		v(\Lambda_L,q,F) :=  \sum_{x\in\Lambda_L} (q \cdot x) \xi_x + \sum_{x\in \Lambda_{L-r(F)-1}} q \cdot \t_x F,
	\end{align}
	then the term $B_{L,\zeta}-H_{L,\zeta,F}$ can be represented as  
	\begin{align*}
		q\cdot(B_{L,\zeta}-H_{L,\zeta,F}) = -\mathcal{L}_{\Lambda_L,\zeta} v(\Lambda_L,q,F).
	\end{align*}
	We can thus reformulate $Q_L^{(b)}(F;q,M,\zeta)$ in \eqref{eq.Q6} as 
	\begin{align*}
		Q_L^{(b)}(F;q,M,\zeta) = \frac{1}{|\Lambda_L|} \bracket{v(\Lambda_L,q,F)\Ll(-\mathcal{L}_{\Lambda_L,\zeta} v(\Lambda_L,q,F)\Rr)}_{\Lambda_L, M, \zeta} - \frac12 q\cdot {\mathbf c}(\hat{\rho};F)q.
	\end{align*}
	Then the calculation is quite similar as the subadditive quantity in \eqref{eq.defC2}, and the error comes from the local equivalence of ensembles and the boundary layer.
	
	Concerning the term $Q_L^{(c)}(F;q,\tilde{q},M,\zeta)$, we also have a representation using \eqref{eq.vqF} and \eqref{eq.recover_u}
	\begin{align*}
		Q_L^{(c)}(F;q,\tilde{q},M,\zeta) &= \frac{1}{|\Lambda_L|} \bracket{u(\Lambda_L,q)\Ll(-\mathcal{L}_{\Lambda_L,\zeta} v(\Lambda_L,\tilde{q},F)\Rr)}_{\Lambda_L, M, \zeta}
		-  (q\cdot \tilde{q}) \chi(\hat{\rho})\\
		&= \frac{1}{2 |\Lambda_L|} \sum_{ b \in (\Lambda_L)^*}  \bracket{(\pi_b \ell_q) (\pi_b v(\Lambda_L,\tilde{q},F))}_{\Lambda_L, M, \zeta} -  (q\cdot \tilde{q}) \chi(\hat{\rho}).
	\end{align*}
	The second line comes from the variational formula of $u(\Lambda_L,q)$; see \eqref{eq.defNu}. Since $v(\Lambda_L,\tilde{q},F) \in \ell_{q, \Lambda_L} + \F_0^{-}(\Lambda_L)$, the first term will yield $(q\cdot \tilde{q}) \chi(\hat{\rho})$ under $\P_{\hat{\rho}}$; see \eqref{eq.MeanZeroLocal}. The error just comes from the local equivalence of ensembles.

\end{proof}

\subsubsection{CLT variance of $D(\hat{\rho}) A_{L} -(B_{L,\zeta}-H_{L,\zeta,F})$}

Proposition~\ref{prop:7.1} is shown by combining Proposition \ref{Theorem 5.1} and
Lemma \ref{Proposition 5.1}, and gives a quantitative refinement of
Corollary 5.1 of \cite{fuy}.

\begin{proof}[Proof of Proposition~\ref{prop:7.1}]
	Omitting $\hat\rho$, we simply write $D, \c,
	\c(F)$, $R(F)$ and $\chi$ for $D(\hat\rho)$, $\c(\hat\rho)$,
	$\c(\hat\rho;F)$, $R(\hat\rho;F)$ and $\chi(\hat\rho)$, respectively.
	Then, noting that $D$ is symmetric, we have
	\begin{align*}
		|\La_L&|^{-1}  \Delta_{L,M,\zeta} \big[q\cdot 
		\{ DA_{L} -(B_{L,\zeta}-H_{L,\zeta,F})\}\big] \\
		& = |\La_L|^{-1} \Big\{ \Delta_{L,M,\zeta} \big[Dq\cdot A_L \big]
		-2 \Delta_{L,M,\zeta} \big[Dq\cdot A_L,
		q\cdot (B_{L,\zeta}-H_{L,\zeta,F})  \big]  \\
		& \hskip 40mm
		+ \Delta_{L,M,\zeta} \big[q\cdot (B_{L,\zeta}-H_{L,\zeta,F}) \big] \Big\}
		\\
		& = \big\{Q_L^{(a)}(Dq, M,\zeta) + 2 \big(Dq\cdot \c^{-1}Dq\big) 
		\, \chi^2\big\}  - 2 \big\{ Q_L^{(c)}(F;q,Dq, M,\zeta) 
		+ \big(q\cdot D q\big) \, \chi
		\big\} \\
		&  \quad + \big\{ Q_L^{(b)}(F;q, M,\zeta) + \frac12 \big(q\cdot \c(F)
		q\big)\big\}.
	\end{align*}
	However, by $D=  \c/2\chi$ and $\c(F)- \c
	=R(F)$, the term except for $Q^{(a)}, Q^{(b)}, Q^{(c)}$ on the right-hand side
	is rewritten as
	\begin{align*}
		&2 \big(Dq\cdot \c^{-1}Dq\big) 
		\, \chi^2 -2 \big(q\cdot Dq\big) \, \chi
		+ \frac12 \big(q\cdot \c(F)q\big)\\
		&= \frac{1}{2} q \cdot \c q - q \cdot \c q + \frac12 \big(q\cdot \c(F)q\big) \\
		&= \frac12 q \cdot R(F)q.
	\end{align*}
	This proves the identity \eqref{eq.QLIdentity}. Afterwards, the bound for $Q_L(F)$ follows from \eqref{eq:Q(4)} (combined
	with the boundedness of $D(\rho)$),	\eqref{5.1} and \eqref{5.2} by changing the constant $C>0$ if necessary.
\end{proof}

\section{Quantitative hydrodynamic limit}\label{sec.hy}

The main task of this section is to establish a quantitative hydrodynamic limit for non-gradient models. We introduce the settings in Section~\ref{subsec.hydro_setting}, then reduce Theorem~\ref{thm.HLQuant} to the relative entropy in Section~\ref{subsec.hydro_from_entropy}. The main part of the proof follows \cite{fuy}, and we reformulate a self-contained version in Section~\ref{subsec.relative_entropy}. The main improvement is a quantitative version of gradient replacement, which does not require the characterization of the closed forms (Varadhan's lemma), but brings a convergence rate. Some preparation about the CLT variance is done in Section~\ref{subsec.CLT}.

\subsection{Setting of hydrodynamic limit}\label{subsec.hydro_setting}

\subsubsection{Dynamics}
As briefly explained in Section 1.1,  the hydrodynamic limit was studied in 
\cite{fuy} for the non-gradient Kawasaki dynamics on the lattice torus 
$\T_N^d$.  It is described as a Markov process $\eta^N(t) = \{\eta_x^N(t), 
x \in \T_N^d\}$ on the configuration space $\mathcal{X}_N=\{0,1\}^{\T_N^d}$,
governed by the infinitesimal generator
\begin{align*}
	\mathcal{L}_N = N^2\mathcal{L},
\end{align*}
where $\mathcal{L}$ is the operator defined 
by \eqref{eq.Generator} replacing ${\mathbb Z}^d$ with $\T_N^d$, i.e., 
\begin{align}\label{eq.Generator-torus}
	\L := \sum_{b \in (\T_N^d)^*} c_b(\eta) \pi_b = \frac{1}{2}\sum_{x,y \in \T_N^d: \vert x - y\vert = 1} c_{x,y}(\eta) \pi_{x,y}.
\end{align}
The hydrodynamic limit aims to study
the asymptotic behavior as  $N \to \infty$ of
the macroscopic empirical mass distribution $\rho^N(t, \d v)$ of $\eta^N(t)$
defined by \eqref{eq.empricalMeasure}, i.e.,
\begin{align*}
	\rho^N(t, \d v) = N^{-d} \sum_{x \in \Td_N} \eta^N_x(t) \delta_{x/N}(\d v), 
\end{align*}
for $v \in \Td$,
and the nonlinear diffusion equation \eqref{eq.nonlinear}, i.e., 
\begin{align*}
	\partial_t \rho(t, v) = \nabla \cdot(D(\rho(t, v)) \nabla \rho(t, v)), \qquad (t, v) \in \R_+ \times \Td,
\end{align*}
was derived in the limit.  We always assume Hypothesis \ref{hyp} for the
jump rates $\{c_b=c_{x,y}\}$ and Hypothesis \ref{hyp:1.2} for the initial value
$\rho_0$ of  \eqref{eq.nonlinear}, given in Section \ref{sec:1}.

\subsubsection{Minimizing sequence $\{\Phi_n\}$ for $\c(\rho)$}
\label{sec:7.1.5}

The following proposition is a summary from the results in 
Lemma \ref{lem.supNorm} and Proposition~\ref{prop.ProofMain}
(see also Theorem~\ref{thm.mainUniform}). This proposition is a refinement of \eqref{eq.RQualitative}, originally shown in \cite{fuy}.
Recall  the variational formula \eqref{eq.defC}
for $\c(\rho)$ and \eqref{eq.defR}  for $R(\rho;F)$.
We denote by $r(F)$ the diameter of support of $F$ that,
\begin{align}\label{eq.defSupport}
	r(F) := \min \big\{L \in \N_+: F \in \mathcal{F}_0^d(\La_L)\big\}.
\end{align}

\begin{proposition} \label{prop:5.4}
	There exists an exponent $\gamma(d,\lambda, \r) > 0$ (same as in Theorem~\ref{thm.mainUniform}) and a positive constant $C(d, \lambda, \r) < \infty$, and a sequence of local functions $\Phi_n(\eta)\in 
	\mathcal{F}_0^d$ on $\mathcal{X}$, such that
	\begin{equation} \label{eq:Fn}
		r(\Phi_n)\le n, \quad \sup_{\rho\in [0,1]} | R(\rho; \Phi_n) | \le C n^{- \gamma},
		\quad \|\Phi_n\|_\infty \le C n^2 \log n.
	\end{equation}
\end{proposition}

\subsubsection{Relative entropy and local equilibrium state}
Let $h_N(f|\psi)$ be the relative entropy per volume for two probability
densities $f$ and $\psi$ with respect to $\nu^N\equiv \nu^N_{1/2}$ defined as 
\begin{align}\label{eq.defRelativeEntropy}
	h_N(f|\psi) := N^{-d}\int_{\mathcal{X}_N} f \log(f/\psi) \, \d\nu^N.
\end{align}
Here $\nu_{1/2}^N$ is the Bernoulli product measure of parameter $\frac{1}{2}$ on $\X_N$.

The local equilibrium state of second order approximation
$\psi_t^N$ is defined by
\begin{align}  \label{eq:2psit}
	\psi^N_t(\eta)  & \equiv \psi_{\la(t,\cdot),F_N}^N(\eta)  \\  & 
	:= \frac{1}{Z_{\la(t,\cdot),F_N}} \exp\bigg\{ \sum_{x\in\T_N^d} \la(t,x/N)
	\eta_x + \frac1{N} \sum_{x\in\T_N^d} \nabla \la(t,x/N) \cdot \tau_x F_N(\eta)
	\bigg\},   \notag
\end{align}
where $Z_{\la(t,\cdot),F_N}$ is the normalization constant with respect to
the Bernoulli product measure $\nu^N$, and $ 
\nabla\la=(\partial_{v_i}\la)_{i=1}^d$ is the spatial gradient, and the local function $F_N \in \mathcal{F}_0^d$ serves a correction term.  

Let $f_t^N$ stand for
the density of the distribution of $\eta^N(t)$ on $\mathcal{X}_N$ under $\nu^N$, and we consider its relative entropy with respect to the corrected local equilibrium state 
\begin{equation}  \label{eq:hNt}
	h_N(t) := h_N(f^N_t|\psi^N_t).
\end{equation}
Recall \eqref{eq.defRelativeEntropy} for the definition of $h_N(\cdot|\cdot)$.

The density $\psi_t^N$ can be seen as a good  local equilibrium approximation of $\rho(t,v)$ only when the functions $\lambda(t,v)$ and $F_N$ are well-chosen. Let $\bar\la: (0,1)\to\R$ be a function defined by 
\begin{equation}  \label{eq.bar_la}
	\bar{\la}(\rho) = \log \{\rho/(1-\rho)\}, \quad \rho \in (0,1).  
\end{equation}  
We also consider $\bar{\rho}: \R\to (0,1)$ as the inverse function of the mapping $ \rho \mapsto \bar{\la}(\rho)$, i.e.,
\begin{equation}\label{eq.bar_rho}
	\bar{\rho}(\la) = e^\la / (e^\la +1), \quad \la \in \R.  
\end{equation}

\begin{condition}\label{condition.good}
	Recall the solution $\rho(t,v)$ of the hydrodynamic equation \eqref{eq.nonlinear} and $\Phi_n$ defined in Proposition~\ref{prop:5.4}. We define the following choice of functions
	\begin{align}\label{eq.choice_la}
		\la(t,v) = \bar\la(\rho(t,v)), \qquad 
	\end{align}
	and
	\begin{align}\label{eq.choice_FN}
		F_N = \Phi_{n(N)},  \qquad n(N)
		:=\lfloor N^{s_1} \rfloor, \qquad s_1 \in \Ll(0, \frac{1}{8(d+2)^2}\Rr).
	\end{align}
\end{condition}
Recall Hypothesis \ref{hyp:1.2} for the initial value $\rho_0$ of the equation \eqref{eq.nonlinear} (i.e., smooth and $0<\rho_0(v)<1$) assumed for  
Theorem~\ref{thm.HLQuant}.  Then, by the maximum principle, we have 
$0<\rho(t,v)<1$; see below Hypothesis \ref{hyp:1.2}.  Therefore, $\la(t,v)$
is well-defined.

Roughly, $F_N$ aims to minimize $\sup_{\rho\in [0,1]} | R(\rho; \Phi_n) |$, and its support is mesoscopic compared to $N$. The choice of scale $s_1$ will be clear later (see \eqref{eq.s123} below).

In the following paragraphs, because several estimates are more general, we usually assume the approximation $\psi_t^N$ without specific choices of $\lambda(t, v), F_N$. The Condition~\ref{condition.good} enters the proof in the end.

\subsubsection{Probability measures}
Following the convention $\psi^N_t(\eta) \equiv \psi_{\la(t,\cdot),F_N}^N(\eta)$ in \eqref{eq:2psit}, we identify $\psi_{t,0}^N=\psi_{\la(t,\cdot),0}^N$ by taking $F_N=0$. Then we define the following probability measures on $\mathcal{X}_N$, 
\begin{align}\label{eq.measures_psi}
	\d P^{\psi_t^N} := \psi_t^N \, \d\nu^N, \qquad \d P^{\psi_{t,0}^N} := \psi_{t,0}^N \d\nu^N, \qquad P \equiv P^{f_t^N} := f_t^N \, \d \nu^N,
\end{align}
and the expectations under them by $E^{\psi_t^N}$, $E^{\psi_{t,0}^N}$ and $E \equiv E^{f_t^N}$, respectively. Notice that the measure $P$ and the expectation $E$ are canonical as they are associated to $(\eta^N_t)_{t \geq 0}$.

\subsubsection{Sobolev spaces on $\Td$}
Let us recall the Sobolev spaces $H^{\a} \equiv H^{\a}(\T^d)$ on $\T^d$
for $\a\in \R$.  Associated with the Laplacian $\De$ on $\T^d$,
we have eigenfunctions $e_{\bf n}$ of $-\De$ indexed by
${\bf n} = (n_1,\ldots,n_d) \in {\mathbb Z}^d$
\begin{align*}
	-\De e_{\bf n} = \la_{\bf n} e_{\bf n}, \qquad \la_{\bf n} := \pi^2|{\bf n}|^2 = \pi^2 \sum_{i=1}^d n_i^2.
\end{align*}
Let $\{e_{\bf n}\}_{{\bf n}\in {\mathbb Z}^d}$ form a complete orthonormal system
of $L^2(\T^d)$; see Section 2.2 of \cite{Fu24} for the concrete expression
of $e_{\bf n}$ using the trigonometric functions. Then, the norm $\|g\|_{H^\a}$ is defined first for $g\in C^\infty(\T^d)$ by
\begin{align}\label{eq:Ha}
	\|g\|_{H^{\a}}^2 := \lan (-\De+1)^{\a}g,g\ran
	= \sum_{{\bf n}\in \Zd} (\la_{\bf n}+1)^{\a} \lan g, e_{\bf n}\ran^2,
\end{align}
where $\lan\cdot,\cdot\ran$ denotes the inner product of $L^2(\T^d)$.
The Sobolev space $H^\a(\T^d)$ is the completion of $C^\infty(\T^d)$
under the norm $\|\cdot\|_{H^\a}$.

\subsection{Hydrodynamic limit via relative entropy}\label{subsec.hydro_from_entropy}

As shown in \cite{fuy}, the key to prove the hydrodynamic limit is the decay of $h_N(t)$ defined in \eqref{eq:hNt}.

\begin{proposition}  \label{Corollary 2.1}
	Under Condition~\ref{condition.good}, we assume that 
	$h_N(0) \le C_0N^{-\k_0}$ for some $C_0>0$ and $\k_0>0$,
	and Hypothesis \ref{hyp:1.2} for the initial value $\rho_0$.  Then, for every $T > 0$, there exist two finite positive constants
	$C_{\ref{Corollary 2.1}}(C_0, T, \rho_0, \k_0, d,\lambda, \r)$ and $\k_{\ref{Corollary 2.1}}(\k_0, d,\lambda, \r)$ such that
	\begin{align}\label{eq.entropy_decay}
		\forall t \in [0,T], \qquad h_N(t) \le C_{\ref{Corollary 2.1}} N^{-\k_{\ref{Corollary 2.1}}}.
	\end{align}
\end{proposition}

The proof of Proposition~\ref{Corollary 2.1} is postponed to Section~\ref{subsec.relative_entropy}. In this part, we admit it and deduce Theorem~\ref{thm.HLQuant} via the concentration inequality and the entropy inequality.

We prepare at first two lemmas. 

\begin{lemma} \label{lem:Add-2}
	Under the same setting as Theorem~\ref{thm.HLQuant}, for every $T>0$, there exist two finite positive constants
	$C_{\ref{lem:Add-2}}(C_0, \k_0, T)$ and $\k_{\ref{lem:Add-2}}(C_0, \k_0)$ such that
	$$
	\sup_{t\in [0,T]}
	E\Big[|\lan \rho^N(t),\phi\ran- \lan \rho(t),\phi\ran_N|^2\Big]
	\le C_{\ref{lem:Add-2}} \|\phi\|_\infty^2 N^{-\k_{\ref{lem:Add-2}}},
	$$
	where the two inner products are defined as
	\begin{equation}  \label{eq:2rhophiN}
		\begin{split}
			\lan \rho(t),\phi\ran_N &:= \frac1{N^d} \sum_{x\in \T_N^d}
			\rho(t,x/N) \phi(x/N), \\
			\lan \rho^N(t),\phi\ran &:= \int_{\T^d} \phi(v)\rho^N(t,\d v).
		\end{split}
	\end{equation}
\end{lemma}

\begin{proof}
	By the entropy inequality \cite[Appendix~1.8]{kipnis1998scaling}, we have
	\begin{equation}  \label{eq:3-ent-Q}
		E\Big[|\lan \rho^N(t),\phi\ran- \lan \rho(t),\phi\ran_N|^2\Big] 
		\le \frac{4}{N^d} \log E^{\psi_{t,0}^N}\Big[ e^{X_N^2/4}\Big] 
		+ 4 h_N(f_t^N|\psi_{t,0}^N),
	\end{equation}
	where $X_N$ is defined by
	\begin{align*}
		X_N :=  N^{d/2} \{ \lan \rho^N,\phi\ran- \lan \rho(t),\phi\ran_N\}
		=  N^{-d/2} \sum_{x\in \T_N^d} \zeta_x \phi(x/N),
	\end{align*}
	and $\zeta_x := \eta_x -\rho(t,x/N)$.
	
	\smallskip
	To estimate the entropy in \eqref{eq:3-ent-Q}, 
	we compare $\psi_t^N = \psi_{\la(t,\cdot),F_N}^N$ and $\psi_{t,0}^N=
	\psi_{\la(t,\cdot),0}^N$ recalling \eqref{eq:2psit}:
	$$
	\psi_{\la(t,\cdot),F_N}^N = \frac{Z_{\la(t,\cdot),0}}{Z_{\la(t,\cdot),F_N}}
	e^{R(\eta)} \psi_{\la(t,\cdot),0}^N, 
	$$
	where
	$$
	R(\eta) = \frac1N \sum_{x\in \T_N^d}  \nabla \la(t,x/N) \cdot \t_x F_N(\eta).
	$$
	We have a simple bound
	\begin{align}  \label{eq:2.6-a}
		|R(\eta)| \le N^{d-1} \|\nabla\la(t)||_\infty \|F_N\|_\infty
		\le C_2 N^{d-\k_1}, 
	\end{align}
	for some $0<\k_1<1$ (indeed, $\k_1=1-2s_1-\e, \e>0$), 
	by $\sup_{t\in [0,T]}\|\nabla\la(t)
	\|_\infty<\infty$ and Proposition \ref{prop:5.4}. Recall
	that $F_N = \Phi_{n(N)}$  with the choice of parameters in \eqref{eq.choice_FN}. This also implies that
	\begin{align*}
		0< \frac{Z_{\la(t,\cdot),0}}{Z_{\la(t,\cdot),F_N}} \le e^{\|R\|_\infty}.
	\end{align*}
	Thus, in \eqref{eq:3-ent-Q}, we can estimate it as
	\begin{equation}  \label{eq:1.3-Q}
		\begin{split}
			h_N(f_t^N|\psi_{t,0}^N) & = N^{-d}\int_{\mathcal{X}_N} \Big( \log \frac{f_t^N}{\psi_t^N} 
			+ \log \frac{\psi_t^N}{\psi_{t,0}^N} \Big) f_t^N d\nu^N  \\
			& \le h_N(f_t^N|\psi_t^N) + 2 N^{-d} \|R\|_\infty \\ 
			& \le  C_{\ref{Corollary 2.1}} N^{-\k_{\ref{Corollary 2.1}}} + 2C_2 N^{-\k_1}.
		\end{split}
	\end{equation}
	In the last line, we make use of the  the assumption  $h_N(0) \le C_0N^{-\k_0}$ in Theorem \ref{thm.HLQuant}, together with Proposition~\ref{Corollary 2.1}.
	
	\smallskip
	On the other hand, for the first term in \eqref{eq:3-ent-Q},
	since $\{\zeta_x\}$ are independent and mean $0$
	under $P^{\psi_{t,0}^N}$, by Hoeffding's inequality (see Theorem 2.8 in
	\cite{boucheron2013concentration}, take $N^{d/2}s$ for $t$ and note $|\zeta_x\phi(\tfrac{x}N)|
	\le \|\phi\|_\infty$), we obtain that
	\begin{align*}
		P^{\psi_{t,0}^N}\big(|X_N|\ge s\big) \le 2 e^{-s^2/(2\|\phi\|_\infty^2)},
	\end{align*}
	for all $s>0$ and $\phi\not\equiv 0$.  Thus, we have
	\begin{align}  \label{eq:1.5-Q}
		E^{\psi_{t,0}^N}\Big[e^{X_N^2/4}\Big]
		& = \sum_{i=0}^\infty E^{\psi_{t,0}^N} 
		\Big[e^{X_N^2/4}; i \le |X_N|< i+1\Big]  \\
		& \le \sum_{i=0}^\infty e^{(i+1)^2/4}  2 e^{-i^2/(2\|\phi\|_\infty^2)}
		=: C_3 <\infty,  \notag
	\end{align}
	uniformly for $\phi$ with $\|\phi\|_\infty=1$.
	
	Summarizing \eqref{eq:3-ent-Q}, \eqref{eq:1.3-Q}, \eqref{eq:1.5-Q}, we have
	\begin{equation*}
		E\Big[|\lan \rho^N(t),\phi\ran- \lan \rho(t),\phi\ran_N|^2\Big] 
		\le \frac{4\log C_3}{N^d} + 4 \big( C_{\ref{Corollary 2.1}} N^{-\k_{\ref{Corollary 2.1}}} 
		+ 2 C_2 N^{-\k_1} \big) 
		\le C_4 N^{-\k_{\ref{Corollary 2.1}} \wedge \k_1},
	\end{equation*}
	where the constant $C_4 = C_{4,T}>0$ is uniform in $\phi$ such that
	$\|\phi\|_\infty= 1$. 
	
	For general $\phi (\not\equiv 0)$, taking $\phi/\|\phi\|_\infty$ for
	$\phi$, we obtain
	\begin{equation*}
		E\Big[|\lan \rho^N(t),\phi\ran- \lan \rho(t),\phi\ran_N|^2\Big] 
		\le C_4 \|\phi\|_\infty^2 N^{-\k_{\ref{lem:Add-2}}},
	\end{equation*}
	by writing $\k_{\ref{lem:Add-2}} := \k_{\ref{Corollary 2.1}} \wedge \k_1$.
	This concludes the proof of the lemma.
\end{proof}

Next, for $\b\in (0,1]$, let us consider the H\"older norms 
$\|g\|_{C^\b} = \|g\|_{C^\b(\T^d)}$ of a function $g$ on $\T^d$:
\begin{equation*}  
	\|g\|_{C^\b} := \|g\|_\infty + \sup_{v_1\not=v_2} \frac{|g(v_1)-g(v_2)|}{|v_1-v_2|^\b},
\end{equation*}
where $|v|\in [0,\sqrt{d}/2]$ is  the Euclidean norm
defined for $v\in \T^d$ modulo $1$ in each component.
The next lemma is elementary and we refer to \cite[Lemma~2.3]{Fu24} for its proof.

\begin{lemma} \label{lem:Add-3}
	For every $\b\in (0,1]$ and functions $\rho, \phi \in C^\b(\T^d)$,
	we have
	$$
	|\lan \rho,\phi\ran_N -\lan \rho,\phi\ran|
	\le d^{\b/2} (2N)^{-\b} \|\rho\cdot\phi\|_{C^\b(\T^d)}.
	$$
\end{lemma}

We are at the position to give the proof of Theorem \ref{thm.HLQuant}.
\begin{proof}[Proof of Theorem \ref{thm.HLQuant} via Proposition~\ref{Corollary 2.1}]
	By \eqref{eq:Ha}, one can rewrite and estimate as
	\begin{align}  \label{eq:2.10-A}
		E\Big[\|\rho^N(t)-\rho(t)\|_{H^{-\a}}^2\Big]
		& = \sum_{{\bf n}\in \Zd} (\la_{\bf n}+1)^{-\a} 
		E\Big[\lan \rho^N(t)-\rho(t), e_{\bf n}\ran^2\Big] \\
		&  
		\le  2 \sum_{{\bf n}\in \Zd} (\la_{\bf n}+1)^{-\a} 
		E\Big[\big( \lan \rho^N(t), e_{\bf n}\ran-
		\lan \rho(t), e_{\bf n}\ran_N\big)^2\Big] \notag \\
		&\quad + 2 \sum_{{\bf n}\in \Zd} (\la_{\bf n}+1)^{-\a} \big( \lan \rho(t), e_{\bf n}\ran_N-
		\lan \rho(t), e_{\bf n}\ran\big)^2.  \notag
	\end{align}
	By Lemma \ref{lem:Add-2} and noting 
	$\| e_{\bf n}\|_\infty \le (\sqrt{2})^d$
	for all ${\bf n}$, the first term on the {\rhs} of \eqref{eq:2.10-A} is bounded by
	\begin{align}  \label{eq:thm1-1}
		C_{\ref{lem:Add-2}}  N^{-\k_{\ref{lem:Add-2}}} \sum_{{\bf n}\in \Zd} (\la_{\bf n}+1)^{-\a} 
		\le C_2 N^{-\k_{\ref{lem:Add-2}}},  
	\end{align}
	with $C_2=C_{2,T,\a}>0$ if $\a>d/2$.
	On the other hand, by Lemma \ref{lem:Add-3}, the second term  on the {\rhs} of \eqref{eq:2.10-A}  is bounded by
	\begin{align}  \label{eq:2-F}
		2 \sum_{{\bf n}\in \Zd} (\la_{\bf n}+1)^{-\a} d^\b (2N)^{-2\b }
		\|\rho(t)e_{\bf n}\|_{C^\b(\T^d)}^2,
	\end{align}
	for every $\b\in (0,1]$.
	However, since $\|\rho(t)e_{\bf n}\|_{C^\b(\T^d)} \le 
	\|\rho(t)\|_{C^\b(\T^d)} \|e_{\bf n}\|_{C^\b(\T^d)}$,
	\begin{align*}
		\|e_{\bf n}\|_{C^\b(\T^d)}\le \|e_{\bf n}\|_\infty
		+ \|\nabla e_{\bf n}\|_\infty^\b (2 \|e_{\bf n}\|_\infty)^{1-\b}
		\le C_5(|{\bf n}|^\b +1), 
	\end{align*}
	and $\sup_{t\in [0,T]} \|\rho(t)\|_{C^\b(\T^d)}<\infty$ due to the smoothness 
	of $\rho(t)$, the series \eqref{eq:2-F} is bounded by
	\begin{align}  \label{eq:second}
		C_6 N^{-2\b} \sum_{{\bf n}\in \Zd} (\la_{\bf n}+1)^{-\a} 
		(|{\bf n}|^\b +1)^2 \le C_7 N^{-2\b},
	\end{align}
	if $\a> \b+d/2$ holds. Notice the assumption $\a>d/2$, we can find $\b>0$ small enough satisfying this condition.
	
	Thus, from \eqref{eq:2.10-A}, \eqref{eq:thm1-1} and
	\eqref{eq:second}, we obtain the conclusion 
	of Theorem~\ref{thm.HLQuant}  by taking
	$\k = \k_{\ref{lem:Add-2}} \wedge 2\b$ with $\b \in (0, (\a-d/2)\wedge 1)$.
\end{proof}

Our method and estimates also apply to the non-gradient 
Glauber--Kawasaki dynamics.  
This is discussed in \cite{Fu24}.

\subsection{Decay of relative entropy}
\label{subsec.relative_entropy}

We prove Proposition~\ref{Corollary 2.1} in this subsection. The sketch follows \cite{fuy}, and we outline the procedure with the highlight of improvement. The main idea of proof is to calculate the time derivative $\partial_t h_N(t)$. It is done in Section~\ref{subsubsec.derivative_ht}, then the error bound is discussed in the the following paragraphs, which mainly consists of 3 parts.
\begin{itemize}
	\item Section~\ref{subsubsec.one_blcok} implements a refined one-block estimate, where the functions are evaluated using the local particle density.
	\item Section~\ref{subsubsec.grad_replacement} is devoted to the gradient replacement. This part is the main improvement, and makes use of the results from homogenization. Especially, the inputs are Theorem~\ref{thm.mainUniform} (Proposition~\ref{prop:5.4}) and the CLT variance in Proposition~\ref{prop:7.1}. 
	\item Section~\ref{subsubsec.LDP} studies the error of large deviation.
\end{itemize}
Finally, Gronwall's inequality is applied in Section~\ref{subsubsec.conclusion}, and the parameters of several mesoscopic scales are well chosen to conclude the proof. 

The proofs of several lemmas and proposition are postponed to
Appendix \ref{sec:4-C}. Note that ``the two-blocks estimate'' is unnecessary in the relative entropy method.

\subsubsection{Time derivative of the relative entropy $h_N(t)$}  
\label{subsubsec.derivative_ht}

We follow at first \cite[Lemma 3.1]{fuy} to develop the time derivative $\partial_t h_N(t)$. The error was of order $o(1)$, and we aim to refine it the present work. 

Recall $\nabla \la = (\partial_{v_i} \la)_{1\le i \le d}$.  Moreover,
we denote 
\begin{align}\label{eq.def_la}
	\nabla^2 \la = (\partial_{v_i}\partial_{v_j}\la)_{1 \le i,j \le d}, \qquad \nabla^3 \la = (\partial_{v_i}\partial_{v_j}\partial_{v_k} 
	\la)_{1 \le i,j,k \le d}, \qquad \dot{\la} = \partial \la/\partial t.
\end{align}
and set
\begin{align}\label{eq.def_la_norm}
	\| \nabla \la\|_{3,\infty} := \|\nabla\la\|_\infty + \|\nabla^2\la\|_\infty
	+ \|\nabla^3\la\|_\infty.
\end{align}

\begin{lemma} 
	\label{Lemma 3.1}
	Assume $\la \in C^{1,3}([0,T]\times\T^d)$, then we have
	\begin{align}\label{eq.derivative_h}
		\partial_t h_N(t) \le E^{f^N_t}[\Om_1 + \Om_2]
		+ N^{-d} \sum_{x\in\T_N^d} \dot{\la}(t,x/N)\rho(t,x/N) + Q_N^{En}(\la,F_N),
	\end{align}
	where  $\Om_1 = \Om_1(\eta)$  and  $\Om_2 = \Om_2(\eta)$  are
	defined respectively by
	\begin{align}  \label{eq:Om1}
		\Om_1(\eta) := & - \frac{N^{1-d}}2 \sum_{x,y \in \T_N^d: x\sim y} c_{x,y} \Om_{x,y}, 
		\\   \label{eq:Om2}
		\Om_2(\eta) := & - N^{-d} \sum_{x\in\T_N^d} \dot{\la}(t,x/N) \eta_x 
		+ \frac{N^{-d}}4 \sum_{x,y \in \T_N^d:x\sim y} c_{x,y} \Om_{x,y}^2   
		\\  \notag
		& -\frac{N^{-d}}4 \sum_{x,y \in \T_N^d:x\sim y} c_{x,y} \sum_{i,j=1}^d
		\partial_{v_i}\partial_{v_j} \la(t,x/N) (y_i -x_i)(y_j-x_j)(\eta_y-\eta_x) 
		\\  \notag
		&   + \frac{N^{-d}}2 \sum_{x,y \in \T_N^d:x\sim y} c_{x,y} \sum_{i,j=1}^d
		\partial_{v_i}\partial_{v_j}\la(t,x/N)  \pi_{x,y}\Ll(\sum_{z\in\T_N^d}(z_j-x_j)
		\tau_z F_{N,i}\Rr).
	\end{align}
	Here $F_{N,i}$ is the $i$-th component of $F_N$ and $\Om_{x,y}$ is defined as follows
	\begin{align}\label{eq:Omxy}
		\Om_{x,y}(\eta) &: =    \nabla\la(t,x/N) \cdot
		\bigg((y - x)(\eta_y-\eta_x)- \pi_{x,y}\Big(\sum_{z\in\T_N^d}
		\tau_z F_N\Big)\bigg).
	\end{align}
	There exists a finite positive constant $C(d,\lambda)$ depending on the dimension and the uniform ellipticity, such that the following estimate holds for the error term $Q_N^{En}(\la,F_N)$ 
	\begin{align}  \label{eq:R(la,F)}
		|Q_N^{En}(\la,F_N)| \le &
		CN^{-1} \Big(1+ \|\nabla\la\|_{3,\infty}\Big)^3
		\Big(1+r(F_N)^{d+2}\|F_N\|_\infty\Big)^3 \\
		& + CN^{-1} \|\nabla \dot{\la}\|_\infty \|F_N\|_\infty + 
		CN^{-1} \|\dot{\la}\|_\infty \|\nabla\la\|_\infty r(F_N)^{d}\|F_N\|_\infty,  \notag
	\end{align}
	as long as the functions $\la$ and $F_N$ satisfy the bound
	\begin{align}  \label{eq:2.F} 
		N^{-1} r(F_N)^d \|\nabla\la\|_\infty \| F_N\|_\infty \le 1.
	\end{align}
\end{lemma}
\begin{remark}
	The constant $1$ on the right-hand side of \eqref{eq:2.F} may be replaced
	by any other constant.
\end{remark}

The proof is postponed to Appendix~\ref{sec:4-C}. At least, we read the decay of $|Q_N^{En}(\la,F_N)|$ in \eqref{eq:R(la,F)} when $N \to \infty$. We then analyze $\Omega_1, \Omega_2$ in the following paragraphs.

\subsubsection{$\Omega_2$ by the refined one-block estimate}
\label{subsubsec.one_blcok}

We now formulate the ``refined'' one-block estimate, that is, a fine error 
estimate for the replacement
of the microscopic function $\Om_2$ of order $O(1)$ 
by its ensemble average.

We define the local sample average of $\eta$ over the box $x+\Lambda_{L}$
\begin{align}  \label{eq:3.saeta}
	\bar\eta_x^L 
	:= \frac{1}{|\La_L|} \sum_{y \in x +\Lambda_{L}} \eta_y.
\end{align}
Here the $x +\Lambda_{L}$ is regarded as its embedding in $\Td_N$ according to the context, and we take $L$ as an odd integer for convenience. Then, for every local function $\Om\in \mathcal{F}_0$, we define the error function
in the replacement of $\t_x\Om$ by its ensemble average at density
$\bar\eta_x^L$ as
\begin{align}  \label{eq:hatOm}
	\hat{\Om}_x^L(\eta) := \t_x \Om(\eta) - \lan \Om \ran(\bar\eta_x^L),
\end{align}
where we denote $\lan \,\cdot\, \ran (\rho)$ by $\lan \,\cdot\, \ran_\rho
= \Er[\,\cdot\,]$, $\rho\in [0,1]$. 


The error of the one-block estimate is mall thanks to the  local ergodicity with respect to the integral of $P^{f^N_t}, t\in [0,T]$
\begin{proposition}[Refined one-block estimate]
	\label{One-block}
	Let $a_{t,x}$ be deterministic coefficients which satisfy
	\begin{equation} \label{BG:assn:atx}  
		\forall t \in [0,T],  x \in \T_N^d, \qquad |a_{t,x}| \leq M.
	\end{equation}
	We assume
	\begin{equation} \label{3.26-A}
		L^{(d+2)/2} \le \de N
	\end{equation}
	for some $\de>0$ small enough (the choice of $\de$ is determined later in
	Lemma \ref{l:L10:6})  and for all $N$ large enough.  Then, for every exponent
	$\epsilon \in (0,1)$, there exists  $C=C(T,\epsilon) >0$ such that for 
	every $\Om\in \mathcal{F}_0$ satisfying (recall the definition \eqref{eq.defSupport})
	\begin{equation} \label{3.9-C}
		r(\Om) \le L^\epsilon,
	\end{equation}
	we have the following estimate for all $t\in [0,T]$ and $N$ large enough
	\begin{multline}\label{One-block:est}
		E \bigg[ \bigg | \int_0^t  N^{-d}
		\sum_{x \in \T_N^d} a_{s,x} \hat{\Om}^L_{x}(\eta^N(s)) \, \d s\bigg | \bigg] \\
		\leq C M (\|\Om\|_\infty +1)\big(N^{-1} L^{(d+2)/2} 
		+ L^{-d} r(\Om)^d \big).
	\end{multline}
\end{proposition}

The proof of Proposition \ref{One-block} is given in Appendix \ref{sec:4-C}
as a combination of a known estimate, which is formulated
in Lemma \ref{l:L10:6}, and the equivalence of ensembles with precise
convergence rate \eqref{eq:3EE}.

Proposition \ref{One-block} can be applied to the specific term $\Om_2$ 
obtained in Lemma~\ref{Lemma 3.1} to replace $\int_0^t E^{f^N_s}[\Om_2] \, \d s$.  The following corollary
summarizes the error for this replacement. Roughly, \eqref{3.Q-Om21} also exhibits a decay when $N \to \infty$.

\begin{corollary}
	\label{One-block-cor}
	For $t\in [0,T]$, the integral $\int_0^t E^{f^N_s}[\Om_2] \, \d s$ can be replaced by
	\begin{align}  \label{eq:hatc-D}
		\int_0^t E^{f^N_s} & \bigg[N^{-d} \sum_{x\in\T_N^d} \Big\{ -\dot{\la}(s,x/N)
		\bar{\eta}_{x}^L  \\
		& \phantom{\sum_{x\in\Ga_N}} 
		+ \frac12 \left(\nabla\la(s,x/N) \cdot
		\c(\bar{\eta}_{x}^L;F_N)\nabla\la(s,x/N)\right)
		\Big\} \bigg] \, \d s,  \notag
	\end{align}
	where $\c(\rho;F_N)$ is defined in \eqref{eq.defQuadra}. Moreover, given a constant $\epsilon\in (0,1)$, and assume the following condition for the function $F_N\in \mathcal{F}_0^d$ (with $\r>0$ as the range of the rate $c_{x,y}$ in Hypothesis~\ref{hyp}) 
	\begin{equation} \label{3.12-C}
		r(F_N) +  \r +1 \le L^\epsilon,
	\end{equation}
	then there exists a positive finite constant $C=C({T,\|\la\|_\infty}, \epsilon)$ such that the error for this replacement is dominated by
	\begin{align}  \label{3.Q-Om21}
		& Q_{N,L}^{\Om_2,1}(\la,F_N):= 
		C\big(N^{-1}L^{(d+2)/2} + L^{-d}(1+r(F_N)^d)\big)  \\
		& \qquad \times
		\Big( \|\dot{\la}\|_\infty+ \|\nabla\la\|_\infty^2 (1+r(F_N)^{d}\|F_N\|_\infty)^2
		+ \|\nabla^2\la\|_\infty (1+r(F_N)^{d+1}\|F_N\|_\infty)\Big).
		\notag
	\end{align}
	
\end{corollary}

\begin{proof}
	We apply Proposition \ref{One-block} to each term of $\Om_2$.  
	First, note that the ensemble averages of the third and fourth terms of 
	$\Om_2$ vanish:
	$$
	\lan c_{x,y}(\eta_x-\eta_y)\ran_\rho=0, \quad \lan c_{x,y}\pi_{x,y}\Phi\ran_\rho=0
	$$
	for all $\Phi=\Phi(\eta)\in \mathcal{F}_0$. 
	
	Then, \eqref{3.Q-Om21} is obtained by gathering all errors
	given by \eqref{One-block:est} for each term of $\Om_2$.
	Indeed, take $M=\|\dot{\la}\|_\infty$ for the first term of $\Om_2$, 
	$M(\| \Om\|_\infty +1) \le C \|\nabla\la\|_\infty^2 (1+r(F_N)^{d}\|F_N\|_\infty)^2$
	for the second term, $M(\| \Om\|_\infty +1) \le C \|\nabla^2\la\|_\infty$ 
	for the third term, and
	$M(\| \Om\|_\infty +1) \le C \|\nabla^2\la\|_\infty(1+r(F_N)^{d+1}\|F_N\|_\infty)$
	for the fourth term.
	
	The condition \eqref{3.12-C} ensures \eqref{3.9-C} for each choice of $\Om$.
	In \eqref{3.Q-Om21}, $r(F_N)^d$ needs to be $(r(F_N) + \r+1)^d$, but
	this difference is absorbed to the constant $C$ by changing it.
\end{proof}

\subsubsection{$\Omega_1$ by the gradient replacement}
\label{subsubsec.grad_replacement}

This part is a highlight of the non-gradient and quantitative method.
The term  $\Om_1$ in Lemma \ref{Lemma 3.1} looks like of the order $O(N)$.  
We need the gradient replacement
to eliminate the diverging factor in this term, with proper error estimates;
see Proposition \ref{Theorem 3.2} and Lemma \ref{Lemma 3.4}.
The main inputs comes from Section~\ref{subsec.CLT}, where we state the decay estimates
for the CLT variances.  In particular,  Propositions \ref{Theorem 5.1} and
\ref{prop:5.4} are new and essential to show the quantitative result.

The following proposition  is a refinement of Theorem 3.2 of \cite{fuy}. We recall that $A_L=A_L(\xi)$
is the $\R^d$-valued function defined by \eqref{eq.defALBLHL},
and  \eqref{eq.defR} for $R(\rho;F) = \c(\rho;F)  - \c(\rho)$,
and \eqref{eq.defC} and \eqref{eq.defQuadra}  for $\c(\rho)$ and
$\c(\rho;F)$, respectively.

\begin{proposition}[Gradient replacement]
	\label{Theorem 3.2}
	For every $T>0$, there exists a positive constant $C(d, \r, T) < \infty$,  such that the following estimates hold for all $t\in [0,T]$ and $\b > 0$:
	\begin{align*}
		\int_0^t &  E^{f^N_s} \!  \Bigg[\Om_1 +  N^{1-d}
		\sum_{x\in \T_N^d}  \nabla\la(s,x/N) \cdot  D(\bar{\eta}_x^L)
		L^{-d} \t_x A_L    \\ 
		&  \qquad\qquad  -\frac{\b d N^{-d}}{2}  
		\sum_{x\in\T_N^d} \nabla\la(s,x/N) \cdot
		R(\bar{\eta}_x^L;F_N) \nabla\la(s,x/N) \Bigg]  
		\, \d s \\
		&   \le 
		\frac{C}{\b} + \b^2 Q_{N,L}^{(1)}(\la,F_N)
		+  \b Q_L^{(2)}(\la,F_N),
	\end{align*}
	and the upper bound for $ Q_{N,L}^{(1)}(\la,F_N)$
	\begin{equation}  \label{3.Q1}
		| Q_{N,L}^{(1)}(\la,F_N)| \le C N^{-1} L^{(2d+4)} \|\nabla\la\|_\infty^3
		(1+ r(F_N)^{3d}\|F_N\|_\infty^3),
	\end{equation}
	and the domination for $Q_L^{(2)}(\la,F_N)$
	\begin{equation}  \label{eq:3.26}
		|Q_L^{(2)}(\la,F_N)| \le dT \|\nabla\la\|_\infty^2 Q_L(F_N).
	\end{equation}
	Here the upper bound $Q_L(F_N)$ appears in the computation 
	of the CLT variance in Proposition~\ref{prop:7.1}.
\end{proposition}

To prove Proposition~\ref{Theorem 3.2}, the first step is to reduce a non-equilibrium problem 
into a static problem under the canonical equilibrium measure, so we can make use of the results in Section~\ref{subsec.CLT}. We need the following key lemma, which is sometimes called Kipnis--Varadhan estimate or It\^o--Tanaka trick. It is a result similar to Lemmas 3.3 (the case ${\mathbf U}(\rho)=0$)
and 6.1 of \cite{fuy}, but providing a fine error estimate.

Recall the notation $\mathcal{X}_\La, \mathcal{X}_{\La,M}, \P_{\La,M,\zeta} $
and $\lan\,\cdot\,\ran_{\La,M,\zeta}$ introduced at the beginning of Section
\ref{subsec.CLT}, and $\L_{\La_L,\zeta}$ as the generator defined by \eqref{eq.defGeneratorExt}.  The following lemma will be used for the proof of Proposition \ref{Theorem 3.2} taking $n=d$ and for the proof of Lemma \ref{Lemma 3.4} taking $n=1$.

\begin{lemma}
	\label{Lemma 3.3}
	Let  $n\in \N$, $J(t,v) = \{J_i(t,v)\}_{i=1}^n \in C^\infty([0,T]\times\T^d, \R^n)$,  and  ${\mathbf U}(\rho) = 
	\{ {\mathbf U}_{ij}(\rho)\}_{1\le i,j \le n} \in C([0,1],\R^{n \times n})$, and	$G(\eta) = \{ G_i(\eta)\}_{i=1}^n\in \mathcal{F}_0^n$ be functions satisfying
	\begin{align*}
		\forall \zeta\in \mathcal{X}_{\La_L^c},  0 \le M \le |\La_L|, \qquad \lan G \ran_{\La_L,M,\zeta} = 0.
	\end{align*}  
	There exists a positive constant $C(d, \r, T) < \infty$
	such that for every $t\in [0,T]$ and $\b > 0$, we have an estimate
	\begin{align*}
		& \int_0^t  E^{f_s^N} \bigg[ N^{1-d}
		\sum_{x\in\T_N^d} J(s,x/N) \cdot \t_x G(\eta)
		- d\b N^{-d} \sum_{x\in\T_N^d} J(s,x/N) \cdot {\mathbf U}(\bar{\eta}_x^L)
		J(s,x/N)  \bigg] \, \d s \\
		&\le d \b\, t \sup_{|q| \le \|J\|_\infty } \sup_{0 \le M \le |\La_L|,\zeta\in \mathcal{X}_{\La_L^c}} 
		\Big( \, |\La_L| \big\lan (q\cdot G) \,
		(-\L_{\La_L,\zeta})^{-1} (q\cdot G), \big\ran_{\La_L,M,\zeta}  
		- q \cdot {\mathbf U}(M/ |\La_L|)q \Big) \\
		& \qquad + \frac{C}{\b} + \b^2 Q_{N,L}^{(1)}(J,G).
	\end{align*}
	The error term $Q_{N,L}^{(1)}(J,G)$ has a bound
	\begin{align}  \label{eq:3QNell}
		|Q_{N,L}^{(1)}(J,G)| 
		\le C N^{-1} L^{(2d+4)} \| G\|_\infty^3 \| J\|_\infty^3.
	\end{align}
\end{lemma}

Lemma \ref{Lemma 3.3} is a combination of essentially known results,
providing fine error estimates, so the proof is postponed to 
Appendix \ref{sec:4-C}. 

\begin{proof}[Proof of Proposition \ref{Theorem 3.2}]
	
	To rewrite the function in the expectation in Proposition \ref{Theorem 3.2},
	we recall three $\R^d$-valued functions  $A_{L}(\xi)$, 
	$B_{L,\zeta}(\xi)$, and $H_{L,\zeta,F}(\xi)$ of 
	$\eta=\xi\cdot\zeta$ introduced
	in \eqref{eq.defALBLHL} and  $\xi = \eta|_{\La_L}$ and $\zeta = \eta|_{\La_L^c}$.
	
	We now apply Lemma \ref{Lemma 3.3} taking
	$n = d$  and
	\begin{align}  \label{eq:3-G}
		\begin{aligned}
			J(t,v) & = \nabla \la(t,v), \\
			G(\eta) & = 
			L^{-d} \left\{D(\bar{\eta}_L) A_L(\xi) - B_{L,\zeta}(\xi) + H_{L,\zeta,F_N}(\xi) \right\},  \\
			{\mathbf U}(\rho) & = \frac12\left( \c(\rho;F_N) - \c(\rho)\right) = \frac12 R(\rho;F_N).
		\end{aligned}
	\end{align}
	Note that, under the above choice, the expressions under
	$\int_0^t E^{f^N_s}[\cdots]\d s$ in Proposition \ref{Theorem 3.2} and
	Lemma \ref{Lemma 3.3} coincide, by recalling that $\Om_1$ is defined 
	by \eqref{eq:Om1} and \eqref{eq:Omxy}.

	Then, we have $\| G\|_\infty \le C (1+r(F_N)^{d}\|F_N\|_\infty)$.
	Indeed, the function
	$D=D(\rho)$ is bounded (recall \eqref{eq:1.9}),
	$\| A_L\|_\infty \le C L^d$, $\| B_{L,\zeta} \|_\infty \le C L^d$
	since $W_{x,y}$ is bounded, and from
	\begin{align*}
		H_{L,\zeta,F_N} = \sum_{\{x,y\}\in \La_L^*} c_{xy}(\xi\cdot\zeta) \pi_{xy}
		\Big( \sum_{z\in \La_{L-r(F_N)-1}} \t_z F_N\Big),
	\end{align*}
	we have
	\begin{align*}
		\| H_{L,\zeta,F_N} \|_\infty \le C L^d r(F_N)^d \| F_N\|_\infty. 
	\end{align*}
	Therefore, by \eqref{eq:3QNell}, the error $Q_{N,L}^{(1)}(J,G)$ with $J$ and $G$ determined as above has the estimate 
	\eqref{3.Q1}.
	
	Thus, the proof of Proposition \ref{Theorem 3.2} is concluded 
	with the help of Proposition~\ref{prop:7.1}, by observing that
	\begin{multline}  \label{eq:4.11-P}
		Q_L^{(2)}(\|J\|_\infty,F_N) \\
		:=
		d T \sup_{|q| \le \|J\|_\infty } \sup_{0 \le M \le |\La_L|,\zeta\in \mathcal{X}_{\La_L^c}} 
		\Big\{ L^d \lan (q\cdot G), 
		(-\L_{\La_L,\zeta})^{-1} (q\cdot G) \ran_{\La_L,M,\zeta}
		- q \cdot {\mathbf U}(M/L^d)q \Big\}  
	\end{multline}
	appearing from Lemma  \ref{Lemma 3.3}
	is bounded by  $ d T \| J\|_\infty^2 Q_L(F_N)$ with $Q_L(F_N)$ given in 
	Proposition~\ref{prop:7.1}.  Taking $J=\nabla\la$ and
	denoting $Q_L^{(2)}(\la,F_N) := \| \nabla\la\|_\infty^2 Q_L(F_N)$, 
	we obtain \eqref{eq:3.26} and conclude the proof of Proposition \ref{Theorem 3.2}.
\end{proof}

In the first displayed formula in Proposition \ref{Theorem 3.2}, the term
next to $\Om_1$ still looks like diverging order $O(N)$
from the front factor $N^{1-d}$. In the following lemma, we evaluate it and reduce it into another 
term of order $O(1)$.
\begin{lemma}
	\label{Lemma 3.4}
	Set
	\begin{align*}
		Q_{N,L}^{(3)}(\la) :=
		\int_0^t E^{f^N_s} &  \bigg[
		N^{1-d} \sum_{x\in \T_N^d}  \nabla \la(s,x/N) \cdot  D(\bar{\eta}_x^L)
		L^{-d} \t_x A_L(\eta)     \\ 
		& \qquad  + N^{-d} \sum_{x\in\T_N^d} \sum_{i,j=1}^d \partial_{v_i}\partial_{v_j}
		\la(s,x/N)
		P_{ij}(\bar{\eta}_x^L) \bigg] \, \d s,
	\end{align*}
	where we have 
	\begin{align}\label{eq.defP}
		\rho \in [0,1], \quad P(\rho) \equiv \{P_{ij}(\rho)\}_{1\le i,j \le d} := \int_0^\rho D(\rho') \, \d\rho' .
	\end{align}
	There exists a positive constant $C(d, \r, T) < \infty$ such that for every  $t\in [0,T]$ and $\b>0$
	the following upper bound holds for $Q_{N,L}^{(3)}(\la)$:
	\begin{align*}
		Q_{N,L}^{(3)}(\la)\le C (L^{-1}\|\nabla^2\la\|_\infty +N^{-1}\|\nabla^3\la\|_\infty)
		+C\b L^{-1} \|\nabla\la\|_\infty^2   + 
		\frac{C}\b + C\b^2 N^{-1}L^{2d-2}  \|\nabla\la\|_\infty^3. 
	\end{align*}
\end{lemma}
It is a refined version of \cite[Lemma~3.4]{fuy}. 
The outline, following that of \cite[Lemma~4.3]{Fu24}, is given in Appendix \ref{sec:4-C} by applying 
Lemma \ref{Lemma 3.3}.

\subsubsection{Large deviation estimate}
\label{subsubsec.LDP}

Here, we have finished the step of the one-block estimate and the gradient replacement, so let us summarize the outcome of Lemma \ref{Lemma 3.1}, Corollary \ref{One-block-cor} and Proposition \ref{Theorem 3.2} together with Lemma~\ref{Lemma 3.4}  in the form of the following Lemma \ref{lem.before_LD}. We recall $h_N(t)$ defined in \eqref{eq:hNt}  and $R(\rho;F)$ defined in \eqref{eq.defR}, and the matrix $P(\rho)$ defined in \eqref{eq.defP}.

\begin{lemma}
	\label{lem.before_LD} 
	Assume $\la \in C^{1,3}([0,T]\times\T^d)$, $\rho(t,v) = \bar{\rho}(\la(t,v))$ as defined in \eqref{eq.bar_rho},  $F_N \in\mathcal{F}_0^d$ satisfying the conditions \eqref{eq:2.F}, \eqref{3.26-A} and \eqref{3.12-C}, 
	then there exists a finite constant $C(d, \r, \rho, \lambda, T) < \infty$, such that the following estimate holds for every $\b>0$
	\begin{multline}  \label{eq:6.1-h}
		h_N(t)   \le  h_N(0)+ \int_0^t  E^{f^N_s}[W_s]\, \d s  + C(\b+1) \,\|\nabla\la\|_\infty^2 \, \sup_{\rho \in [0,1]} |R(\rho; F_N)|   \\
		+ \int_0^t  \Ll(Q_N^{En}(\la,F_N)+ Q_{N,L,\b}^{\Om_1}(\la,F_N) + Q_{N,L}^{\Om_2}(\la,F_N)\Rr) \, \d s.
	\end{multline}
	On the {\rhs} of \eqref{eq:6.1-h}, the function $W_t$ is defined by
	\begin{align*}
		W_t(\eta) = & -  N^{-d}  \sum_{x\in\T_N^d} \dot{\la}(t,x/N)
		\{ \bar{\eta}_x^L - \rho(t,x/N)\}   
		\\  
		&  +  N^{-d} \sum_{x\in\T_N^d} \Tr \left(\nabla^2\la(t,x/N)
		\{ P(\bar{\eta}_x^L) - P(\rho(t,x/N))\}\right)  
		\\  
		&  +  \frac{N^{-d}}{2}  \sum_{x\in\T_N^d} \nabla\la(t,x/N) \cdot
		\{\c(\bar{\eta}_x^L) - \c(\rho(t,x/N))\}
		\nabla\la(t,x/N).
	\end{align*}
	The term $Q_N^{En}(\la,F_N)$ is as in Lemma \ref{Lemma 3.1}, and the term $Q_{N,L,\b}^{\Om_1}(\la,F_N)$ has a decomposition as
	\begin{align}  \label{3.QOM1}
		Q_{N,L,\b}^{\Om_1}(\la,F_N) := \frac{C}{\b}+\b^2 Q_{N,L}^{(1)}(\la,F_N)
		+\b Q_L^{(2)}(\la,F_N)+Q_{N,L}^{(3)} (\la).
	\end{align}
	See Proposition \ref{Theorem 3.2} for $Q_{N,L}^{(1)}$, $Q_L^{(2)}$ and Lemma \ref{Lemma 3.4} for $Q_{N,L}^{(3)}$.
	The error $Q_{N,L}^{\Om_2}(\la,F_N)$ is estimated as
	\begin{align}  \label{3.QOM2}
		&  |Q_{N,L}^{\Om_2} (\la,F_N)| \\
		& \le  C\big(N^{-1}L^{(d+2)/2} 
		+ L^{-d}(1+r(F_N)^d)\big) \notag   \\
		& \quad \times
		\Big( \|\dot{\la}\|_\infty+ \|\nabla\la\|_\infty^2 (1+r(F_N)^{d}\|F_N\|_\infty)^2
		+ \|\nabla^2\la\|_\infty (1+r(F_N)^{d+1}\|F_N\|_\infty)\Big)     \notag \\
		& \;\; + C N^{-1}(\|\nabla^3\la\|_\infty
		+ \|\nabla^2\la\|_\infty\|\nabla\la\|_\infty
		+ \|\nabla\la\|_\infty^2 \|\nabla \la \|_\infty).  \notag
	\end{align}
\end{lemma}

\begin{proof}
	The proof is similar to \cite[Lemma~3.2]{fuy}.
	
	First, we integrate the estimate for $\partial_t h_N(t)$ shown in Lemma
	\ref{Lemma 3.1} on the time integral $[0,t]$.  Then, for $\Om_1$, we apply 
	Proposition \ref{Theorem 3.2} (with an error $\frac{C}{\b}+\b^2 Q_{N,L}^{(1)}
	+\b Q_L^{(2)}$) and Lemma \ref{Lemma 3.4}
	(with an error $Q_{N,L}^{(3)}$), and obtain a 
	replacement of $\int_0^t E^{f^N_s}[ \Om_1] \, \d s$
	by the following integral:
	\begin{align*}
		\int_0^t &  \bigg\{  E^{f^N_s} \bigg[ N^{-d}
		\sum_{x\in\T_N^d} \Tr \left( \nabla^2\la(s,x/N) P(\bar{\eta}_{x}^L)\right)
		\bigg]   \\
		& \quad   +  E^{f^N_s} \bigg[  \frac{\b N^{-d}}{2}
		\sum_{x\in\T_N^d} \nabla\la(s,x/N) \cdot
		R(\bar{\eta}_{x}^L;F_N)\nabla\la(s,x/N) \bigg]  \bigg\} \, \d s,
	\end{align*}
	with an error $Q_{N,L,\b}^{\Om_1}(\la,F_N)$ defined in \eqref{3.QOM1}.
	The second term in the above formula is bounded by 
	$C \b\,\|\nabla\la\|_\infty^2 \, \sup_{\rho} |R(\rho; F_N)|$ and this is counted
	on the right-hand side of \eqref{eq:6.1-h} for $h_N(t)$.

	For $\Om_2$, by Corollary \ref{One-block-cor}, $\int_0^t E^{f^N_s}[\Om_2] \, \d s$ 
	can be replaced by the expression in \eqref{eq:hatc-D}:
	\begin{align*}  
		\int_0^t E^{f^N_s} & \bigg[N^{-d} \sum_{x\in\T_N^d} \Big\{ -\dot{\la}(s,x/N)
		\bar{\eta}_{x}^L  \\
		& \phantom{\sum_{x\in\Ga_N}} 
		+ \frac12 \nabla\la(s,x/N) \cdot
		\c(\bar{\eta}_{x}^L;F_N)\nabla\la(s,x/N)
		\Big\} \bigg] \, \d s ,  \notag
	\end{align*}
	with the error $Q_{N,L}^{\Om_2,1}(\la,F_N)$ defined in \eqref{3.Q-Om21}.
	This gives the first part
	in the error estimate of $|Q_{N,L}^{\Om_2}(\la,F_N)|$.
	Note that $\c(\bar{\eta}_{x}^L;F_N)$ in the above expression
	(i.e.\ \eqref{eq:hatc-D}) can be
	replaced by $\c(\bar{\eta}_{x}^L)$ with an error 
	$C \,\|\nabla\la\|_\infty^2 \, \sup_{\rho} |R(\rho; F_N)|$.
	
	The reason that we have the second part in the error estimate
	of $|Q_{N,L}^{\Om_2}(\la,F_N)|$ in \eqref{3.QOM2} is as follows.  
	Recalling that $\rho(t,v) = \bar{\rho}(\la(t,v))$, we have the identity
	\begin{align}  \label{3.IBP}
		-\int_{\T^d} \Tr \big(\nabla^2\la(t,v) P(\rho(t,v))\big) \d v  
		=  \frac1{2} \int_{\T^d} \nabla\la(t,v) \cdot \c(\rho(t,v))
		\nabla\la(t,v)\d v.
	\end{align}
	This is shown by integration by parts and noting 
	\begin{align*}
		\text{LHS of \eqref{3.IBP}}   
		&= \sum_{i,j=1}^d \int_{\T^d} - (\partial_{v_i} \partial_{v_j} \lambda)(t,v) P_{i,j}(\rho(t,v)) \, \d v \\
		&= \sum_{i,j=1}^d \int_{\T^d} \partial_{v_j}\lambda(t,v)  \partial_{v_i}  P_{i,j}(\rho(t,v)) \, \d v \\
		&= \sum_{i,j=1}^d \int_{\T^d} \partial_{v_j}\lambda(t,v)  \partial_{\rho}  P_{i,j}(\rho(t,v)) \partial_{\lambda} \rho(t,v) \partial_{v_i}\lambda(t,v) \, \d v \\
		&=  \sum_{i,j=1}^d \int_{\T^d} \partial_{v_j}\lambda(t,v)  D_{i,j}(\rho(t,v)) \chi(\rho(t,v)) \partial_{v_i}\lambda(t,v)\, \d v \\
		&= \text{RHS of \eqref{3.IBP}}   .
	\end{align*}
	The third line makes use of the chain rule. The fourth line makes use of $\frac{\partial \bar{\rho}}{\partial\la} = \chi(\rho)$ in \eqref{eq.bar_rho}, the definition of $P(\rho(t,v))$ in \eqref{eq.defP}. The last relies on the Einstein relation \eqref{eq.Einstein}.

	The terms with $\rho(t,x/N)$ in the second and third terms of $W$
	cancels by this identity with an error bounded by
	\begin{align*}
		C N^{-1}(\|\nabla^3\la\|_\infty+
		\|\nabla^2\la\|_\infty \| \nabla \rho\|_\infty 
		+ \|\nabla^2\la\|_\infty\|\nabla\la\|_\infty
		+ \|\nabla\la\|_\infty^2 \|\nabla \rho \|_\infty),
	\end{align*}
	which appears when we discretize the above two integrals in \eqref{3.IBP}; 
	note that $P, \c$ are bounded and $P, \c \in C^\infty((0,1))$ as proved in \cite{bernardin}.
\end{proof}

We aim to treat the term $E^{f^N_t}[W_t]$ in Lemma \ref{lem.before_LD} via a large deviation estimate. Using the entropy inequality, it can be reduced to the expectation under
${P^{\psi^N_t} = \psi^N_t\, \d \nu^N}$: 
\begin{equation}  \label{eq:3-ent}
	\forall \delta > 0, \qquad E^{f^N_t}[W_t] \le \frac{1}{\de N^d} \log E^{\psi^N_t}\Ll[ e^{\de N^d W_t}\Rr] 
	+ \frac{1}{\de}  h_N(t).
\end{equation}
We study the asymptotic behavior of
the first term on the right-hand side via a large deviation type upper 
bound with an appropriate error estimate under
$\psi^N_t\, \d \nu^N$; see Lemma~\ref{Theorem 3.3} below.

For  $\la(\cdot) \in C^1(\T^d)$  and  $F = F(\eta) \in \mathcal{F}_0^d$,
dropping the $t$-dependence in \eqref{eq:2psit}, the local equilibrium state  
$\psi_{\la(\cdot),F}^N(\eta)  \d \nu^N$ of second order approximation
is a probability measure on  $\mathcal{X}_N$ defined by
\begin{equation} \label{B.8}
	\psi_{\la(\cdot),F}^N (\eta) = Z^{-1}   
	\exp\bigg\{ \sum_{x\in \T_N^d} \la(x/N) \eta_x
	+ \frac1{N} \sum_{x\in\T_N^d} \nabla\la(x/N) \cdot \tau_x F(\eta) \bigg\}, 
\end{equation}
for  $\eta\in\mathcal{X}_N$,  where   $Z = Z_{\la(\cdot),F,N}$  is the 
normalization constant with respect to $\nu^N$.  Then we
have the following large deviation type upper bound for  
$\psi_{\la(\cdot),F}^N\, \d \nu^N$ with the error estimate.


Recall that $\nu^N_{\rho}$ is the the product Bernoulli measure of parameter $\rho$ on ${\X_N = \{0,1\}^{\T_N^d}}$. We define that
\begin{align*}
	\bar{\nu}^N_{\lambda} := \nu^N_{\rho(\lambda)},
\end{align*} 
with $\rho=\bar\rho(\la)$ defined as in \eqref{eq.bar_rho}. Then the rate  function for the large deviation principle for the Bernoulli measure $\bar{\nu}^N_{\lambda}$ is denoted by  $I(u;\la)$, namely, for $u\in [0,1]$,
\begin{align}\label{2.4} 
	\begin{aligned}
		& I(u;\la) := - \la  u - q(u) + p(\la),    \\
		& p(\la) := \log (e^\la + 1), \quad
		q(u)   := -\{u\log u +(1-u)\log(1-u)\}.
	\end{aligned}
\end{align}

\begin{lemma}[Appendix~B of \cite{Fu24}] 
	\label{Theorem 3.3}
	For every  $G(v,u) \in C^1(\T^d \times [0,1])$, we define its sum with respect to the local density 
	\begin{equation} \label{eq:3.tildeG}
		\widetilde{G}_{N,L}(\eta)
		:= \sum_{x\in \T_N^d}  G(x/N,\bar{\eta}_x^{L}).
	\end{equation}
	The following large deviation estimate then holds
	\begin{align*} 
		& N^{-d}  \log E^{\psi_{\la(\cdot),F}^N} \Ll[ \exp \widetilde{G}_{N,L}(\eta) \Rr]  \\
		\le &  \sup_{u(v) \in C(\T^d ; [0,1])}
		\int_{\T^d} \{ G(v,u(v)) 
		-I(u(v);\la(v))\} \, \d v + Q_{N,L}^{LD}(\la,F;G).
	\end{align*}
	The error term $Q_{N,L}^{LD}(\la,F;G)$ has an estimate:
	\begin{align*}
		|Q_{N,L}^{LD}(\la,F;G)| \le & C N^{-1} \|\nabla\la\|_\infty \| F\|_\infty
		+  C N^{-1}L  \Big( \|\nabla\la\|_\infty 
		+  \|\partial_v G\|_\infty \Big) \\
		& + C L^{-d} \Big(\log L +  \|\la\|_\infty
		+   \|\partial_u G\|_\infty \Big),
	\end{align*}
	where $\partial_v G$ and $\partial_u G$ are respectively the derivatives with respect to its first and second variable.
\end{lemma}
\begin{proof}
	It suffices to take $G_1=0, G_2=G$ and $K=1$ in Appendix~B of \cite{Fu24}.
\end{proof}

We then apply the large deviation estimate to $h_N(t)$. Recall Lemma \ref{Lemma 3.4} for $P(\rho)$, 
and \eqref{2.4} for $I(u;\la)$. We define $g_{\delta}(t)$ as follows:
\begin{align*}
	& g_{\de}(t) := \sup_{u(v) \in C(\T^d;[0,1])} 
	\int_{\T^d}
	\{\de \cdot \si(u(v);t,v) - I(u(v);\la(t,v)) \}\, \d v,  
	\\
	&  \si(u;t,v) :=  -\dot{\la}(t,v)\{ u-\rho(t,v)\} +
	\Tr \left(\nabla^2\la(t,v) \{ P(u) - P(\rho(t,v))\}\right)
	\phantom{\frac12}      \\
	&  \qquad\qquad\qquad \qquad\qquad 
	+ \frac12 \nabla\la(t,v) \cdot
	\{\c(u) - \c(\rho(t,v))\} \nabla\la(t,v).
\end{align*}

The following proposition gives the detail of the error terms  in Theorem~2.1 of \cite{fuy}.
\begin{proposition}\label{prop.after_LD}
	Under the same setting as Lemma~\ref{lem.before_LD}, there exists a constant ${C(d, \r, \rho, \lambda, T) < \infty}$  such that for every $\b > 0$ and $\delta > 0$, we have
	\begin{multline}\label{2.5} 
		h_N(t)   \le  h_N(0) + \frac{1}{\de} \int_0^t g_{\delta}(s)\,\d s
		+ \frac{1}{\de} \int_0^t h_N(s)\,\d s  \\  + C(\b+1) \,\|\nabla\la\|_\infty^2 \, \sup_{\rho \in [0,1]} |R(\rho; F_N)| +Q_{N,L,\de}^{LD}(\la,F_N)  \\
		+ \int_0^t  \Ll(Q_N^{En}(\la,F_N)+ Q_{N,L,\b}^{\Om_1}(\la,F_N) + Q_{N,L}^{\Om_2}(\la,F_N)\Rr) \, \d s.
	\end{multline}
	Here the term $Q_{N,L,\de}^{LD}(\la,F_N)$ is defined using $ Q_{N,L}^{LD}$ in Lemma~\ref{Theorem 3.3}
	\begin{align*}
		Q_{N,L,\de}^{LD}(\la,F_N) := \int_0^T Q_{N,L}^{LD}(\la(s),F_N; \delta \cdot \si(u;t,v)) \, \d s.
	\end{align*}
\end{proposition}
\begin{proof}
	The result is immediate by combing Lemma \ref{lem.before_LD}, 
	the entropy inequality \eqref{eq:3-ent} and Lemma \ref{Theorem 3.3}. In view of $W_t$ given in Lemma \ref{lem.before_LD}, we set $G$ in Lemma \ref{Theorem 3.3} as
	\begin{align*}
		G(v, u) = \delta \cdot \si(u;t,v).
	\end{align*}
\end{proof}
\begin{remark}
	Under the setting $\la(t,v) = \bar\la(\rho(t,v))$, there exists $\de_0 > 0$, such that for every $\de \in (0,\de_0)$, we have 
	\begin{align}\label{eq.g_negatve}
		\forall t \in [0,T], \qquad g_{\delta}(t) \leq 0.
	\end{align}
	This observation \eqref{eq.g_negatve} was proved in \cite[(2.6)]{fuy}.
\end{remark}

\subsubsection{Concluding the proof of Proposition \ref{Corollary 2.1}}
\label{subsubsec.conclusion}

\begin{proof}[Proof of Proposition \ref{Corollary 2.1}]
	Firstly, we choose the parameters $L, N$ as the setting of Lemma~\ref{lem.before_LD} and Proposition~\ref{prop.after_LD}, and assume \eqref{eq.choice_la} in Condition~\ref{condition.good}. The constant $C>0$ in this proof may change from line to line and may depend on $\rho_0$.

	\textit{Step~1: relative entropy estimate.} 
	In the bound $h_N(t)$ obtained in Proposition~\ref{lem.before_LD}, thanks to the observation \eqref{eq.g_negatve}, the term $g_{\delta}$ can be dropped. The Gronwall's inequality applies and we obtain the estimate
	\begin{multline}  \label{eq:3.hNt}
		0\le  h_N(t) \le  e^{t/\de} \Big( h_N(0)+ C(\b+1) \,\|\nabla\la\|_\infty^2 \, \sup_{\rho \in [0,1]} |R(\rho; F_N)| +Q_{N,L,\de}^{LD}(\la,F_N)  \\
		+ \int_0^t  \Ll(Q_N^{En}(\la,F_N)+ Q_{N,L,\b}^{\Om_1}(\la,F_N) + Q_{N,L}^{\Om_2}(\la,F_N)\Rr) \, \d s		\Big),
	\end{multline}
	for every $\b>0$, $t\in [0,T]$ and $0 < \de < \de_0$.  We then treat the bounds term by term.
	
	
	\smallskip
	
	\textit{Term~1: $Q_{N,L,\b}^{\Om_1}(\la,F_N)$.} In \eqref{eq:3.hNt}, we recall \eqref{3.QOM1} that the first error $Q_{N,L,\b}^{\Om_1}(\la,F_N)$ for the microscopic current of non-gradient type adjusting with the corrector $F_N$ consists of four terms as 
	\begin{align*}
		Q_{N,L,\b}^{\Om_1}(\la,F_N)=
		\frac{C}{\b}+\b^2 Q_{N,L}^{(1)}(\la,F_N)
		+\b Q_L^{(2)}(\la,F_N)+Q_{N,L}^{(3)}(\la).
	\end{align*}
	Here, the term $Q_{N,L}^{(1)}$ has a bound (see Proposition \ref{Theorem 3.2})
	\begin{align} \label{eq:Q.1}
		|Q_{N,L}^{(1)}(\la,F_N)| \le C N^{-1} L^{2d+4}(1+r(F_N)^{3d} \|F_N\|_\infty^3).
	\end{align}
	Note that $\la(t,v)$ is smooth, in particular,
	${\sup_{t\in [0,T]} \|\nabla\la(t)\|_\infty<\infty}$.  The term $Q_L^{(2)}$ is estimated by 
	\eqref{eq:3.26} as
	\begin{align}\label{eq:Q.2}
		|Q_L^{(2)}(\la,F_N)| \le 
		C (1+ r(F_N)^{2d}) (1+r(F_N)^{2d}\|F_N\|_\infty^2) L^{-1}+ C L^{-\gamma_2}.
	\end{align}
	The term $Q_{N,L}^{(3)}$ does not depend on $F_N$ and 	it is estimated in Lemma \ref{Lemma 3.4}
	\begin{align} \label{eq:Q.5}
		\vert Q_{N,L}^{(3)}(\la) \vert\le C(L^{-1}+N^{-1}) +C\b L^{-1} + 
		\frac{C}\b + C\b^2 N^{-1}L^{2d-2}.
	\end{align}

	\smallskip
	
	\textit{Term~2: $Q_{N,L}^{\Om_2}(\la,F_N)$.}
	The error $Q_{N,L}^{\Om_2}(\la,F_N)$ for the gradient term is bounded as
	\begin{align} \label{eq:Q.5-b}
		|Q_{N,L}^{\Om_2}(\la,F_N)|
		\le  & C \big(N^{-1} L^{(d+2)/2} + L^{-d} (1+r(F_N)^d)\big)\\
		& \times
		(1+r(F_N)^{2d} \|F_N\|_\infty^2 + r(F_N)^{d+1} \|F_N\|_\infty) + CN^{-1}. \notag
	\end{align}
	This estimate is provided in Lemma \ref{lem.before_LD}.
	
	\smallskip
	
	\textit{Term~3: $Q_{N,L,\de}^{LD}(\la,F_N)$.}
	The error $Q_{N,L,\de}^{LD}(\la,F_N)$ appearing in a large deviation bound 
	and is estimated as
	\begin{align} \label{eq:Q.6}
		|Q_{N,L,\de}^{LD}(\la,F_N)| \le C N^{-1} (L+ \| F_N\|_\infty)+ C L^{-d} \log L.
	\end{align}
	This estimate is shown in Lemma \ref{Theorem 3.3}.

	\smallskip
	
	\textit{Term~4: $Q_N^{En}(\la,F_N)$.}
	The error $Q_N^{En}(\la,F_N)$ in the entropy calculation, especially in the time derivative of $h_N(t)$, has a bound
	\begin{align} \label{eq:R.1}
		|Q_N^{En}(\la,F_N)| \le C N^{-1} (1+r(F_N)^{d+2} \|F_N\|_\infty)^3,
	\end{align}
	if the condition \eqref{eq:2.F} is satisfied;
	see 	Lemma \ref{Lemma 3.1}.

	\medskip
	\textit{Step~2: choice of parameters.} Now,  according to Condition~\ref{condition.good}, we set
	\begin{align*}
		1\ll n=n(N) \ll L=L(N) \ll N,  \qquad F_N =\Phi_{n(N)},
	\end{align*}
	based on Proposition \ref{prop:5.4}.	
	More precisely, we set $n, L, \beta$ as mesoscopic scales
	\begin{equation}\label{eq.s123}
		n:=\lfloor N^{s_1} \rfloor < L := \lfloor N^{s_2} \rfloor< N, \quad 0<s_1<s_2<1, \qquad \b:=n^{s_3}, \quad s_3 > 0,
	\end{equation}
	with $s_1, s_2, s_3$ to be determined. Then, from \eqref{eq:3.hNt} and the estimates \eqref{eq:Q.1}--\eqref{eq:R.1} stated above,
	we obtain that
	\begin{align*}
		h_N(t) \le  C \Big( h_N(0) & + n^{-s_3} + n^{4d+6} N^{-1}L^{2d+4} 
		+ n^{s_3} L^{-\gamma}   \\
		& + n^{4s+6} L^{-1} + L^{-d}\log L + n^{s_3} n^{-\gamma_2}  \Big),
	\end{align*}
	where $\gamma, \gamma_2$ are in Theorem~\ref{thm.mainUniform} and Proposition~\ref{prop:5.4}.
	Since $h_N(0)\le C_1 N^{-\kappa_1}$, choosing 
	$s_3 \in (0,\gamma)$, $s_2= \frac{1}{2d+5}$ and then
	$s_1\in (0, \frac{s_2}{4d+6})$ small enough, one can derive
	\begin{align}  \label{eq:7.24}
		h_N(t) \le C N^{-\kappa},
	\end{align}
	for some $C, \kappa>0$.  Note that the conditions \eqref{eq:2.F}, 
	\eqref{3.26-A} and \eqref{3.12-C} in  Proposition \ref{lem.before_LD}
	are all satisfied under the above choice of $s_1, s_2$ for $n, L$ in \eqref{eq.s123}.
\end{proof}

\section{Extension to disordered lattice gas}\label{sec.disorder}

We have worked with the non-gradient model where the jump rate is non-constant but a deterministic function on the configuration space. One may wonder whether the method is effective when the external randomness comes into the jump rate function. These models are known as \emph{the disordered lattice gas} or \emph{the exclusion process in random/inhomogeneous environment}. In the literature, \cite{quastel2006disorder, faggionato2003disorder} considered the lattice gas with disorder on site; \cite{gonccalves2008scaling, jaranonhomogeneous, jara2006} studied the cases where the jump rate depends on the random environment and combines the homogenization theory to obtain the diffusion matrix; \cite{faggionato2008,faggionato2022} further relaxed the underlying graph to the supercritical percolation and other stationary random graphs. It is not immediate to cover the quantitative homogenization results in all these models, because there are various ways to pose the disorder and sometimes the jump rate also degenerates. Here we just give one typical example to illustrate that our proof still works in the presence of external disorder. This model can be seen as a lattice gas with disorder on bonds, where randomness is introduced without breaking the spatial homogeneous property, the uniform ellipticity and the product Bernoulli measure is still an invariant measure. The argument will give a quantitative convergence rate (with a random fluctuation) of finite-volume conductivity.
\subsection{Model and hypothesis}
The notations in this new model is almost the same of the original one. However, instead of the generator \eqref{eq.Generator} defined by a collection of functions $\{c_b(\eta)\}_{b\in(\Zd)^*}$, we now consider a collection of random functions:
\begin{align*}
	c:\ \Omega&\rightarrow \F_0^{(\Zd)^*},
	\\ \omega&\mapsto \{c_b^\omega(\cdot)\}_{b\in(\Zd)^*},
\end{align*}
on some probability space $(\Omega, \gil,\mathbf P)$ satisfying the following hypothesis.
\begin{hypothesis}\label{hyp2} The following conditions are supposed for  $c:\Omega\rightarrow \F_0^{(\Zd)^*}$.
	\begin{enumerate}
		\item Non-degenerate and local: there exists a positive integer $\r$ and a positive number $\lambda>1$ such that for every $\omega$, the function $c^\omega_{x,y}(\eta)$ depends only on $\{\eta_z: \vert z - x\vert \leq \r\}$, and is bounded on two sides $1 \leq c^\omega_{x,y}(\eta) \leq \lambda$.
		\item Detailed balance under Bernoulli measures: for every $\omega$, the function $c^\omega_{x,y}(\eta)$ does not depend on $\{\eta_x, \eta_y\}$.
		\item Spatially homogeneous: 
		the joint distribution of $(c_{x,x+e}^\omega)_{x\in\Zd, e \in U}$ is the same as that of $(\tau_x c_{0,e}^\omega)_{x\in\Zd, e \in U}$, where both of them are viewed as a family of random functions taking value in $\F_0$.
		\item Unit range dependence: for any two edge sets $E,F \subset (\Zd)^*$ such that if there is no adjacent pair $b\in E,\ b'\in F$ sharing a  common vertex, then the functions $\{c_b^\cdot\}_{b\in E}$ and $\{c_b^\cdot\}_{b\in F}$ are independent.
	\end{enumerate}
\end{hypothesis}
We specially mention that given each sample $\omega\in\Omega$, the functions $\{c_b^\omega\}_{b\in(\Zd)^*}$ satisfy the same properties as that in Hypothesis~\ref{hyp} except the spatial homogeneous property, which is replaced by the equality in distribution. Meanwhile, the detailed balance condition ensures that the product Bernoulli measure $\operatorname{Ber}(\rho)^{\otimes \Zd}$ is still an invariant measure for the Kawasaki dynamics of jump rate $\{c_b^\omega\}_{b\in(\Zd)^*}$ for every $\rho \in [0,1]$. 

\subsection{Quantitative stochastic homogenization}
We use the renormalization approach to establish the quantitative homogenization result. Given $\omega\in\Omega$, we define the \emph{quenched} subadditive quantities $\nub(\omega,\rho,\Lambda,p)$ and $\nub_*(\omega,\rho,\Lambda,q)$  as in \eqref{eq.defNu}
\begin{equation}\label{eq.defNuOmega}
	\begin{split}
		\nub(\omega, \rho,\Lambda,p) &:= \inf_{v\in\ell_{p,\Lambda^+} +\F_0(\Lambda^-)} \Ll\{ \frac{1}{2 \chi(\rho)\vert\Lambda\vert} \sum_{b\in\ov{\Lambda^*}} \bracket{ \frac{1}{2}c^\omega_b(\pi_b v)^2}_\rho \Rr\},\\
		\nub_*(\omega, \rho,\Lambda,q) &:=\sup_{v \in  \F_0} \Ll\{ \frac{1}{2 \chi(\rho)\vert\Lambda\vert}\sum_{b\in\ov{\Lambda^*}}  \bracket{ (\pi_b \ell_{q})(\pi_b v) - \frac{1}{2} c^\omega_b(\pi_b v)^2}_{\rho}\Rr\}.
	\end{split}
\end{equation}
Then for each $\omega$, (1)-(4) in Proposition~\ref{prop.Element} still hold except \eqref{eq.subadditivenu} and \eqref{eq.subadditivenu*}. This gives the definition of the quenched diffusion matrix and conductivity
\begin{align*}
	\nub(\omega, \rho,\Lambda,p) &=\frac{1}{2}p\cdot \D(\omega, \rho,\Lambda) p, \qquad \cc(\omega, \rho,\Lambda) = 2\chi(\rho) \D(\omega, \rho,\Lambda),\\
	\nub_*(\omega, \rho,\Lambda,q) &=\frac{1}{2}q\cdot \D_*^{-1}(\omega, \rho,\Lambda) q, \qquad \cc_*(\omega, \rho,\Lambda) = 2\chi(\rho) \D_*(\omega, \rho,\Lambda).   
\end{align*}
The results \eqref{eq.subadditivenu} and \eqref{eq.subadditivenu*} are missing, because now for every $\Lambda \subset \Zd$ and $z \in \Zd$ we only have $\nub(\omega, \rho, z+\Lambda,p) \stackrel{(d)}{=} \nub(\omega, \rho, \Lambda,p)$  (from (3) of Hypothesis~\ref{hyp2}) instead of $\nub(\rho, z+\Lambda,p) = \nub(\rho, \Lambda,p)$ in Proposition~\ref{prop.Element}. Nevertheless, denoting by $\mathbf E$ the expectation associated to $(\Omega, \gil,\mathbf P)$, we recover
\begin{equation}\label{eq.subadditiveEmu}
	\begin{split}
		\mathbf E[\nub(\cdot,\rho,\cu_{m+1},p)]&\leq\mathbf E[\nub(\cdot,\rho,\cu_{m},p)],\\
		\mathbf E[\nub_*(\cdot,\rho,\cu_{m+1},q)]&\leq\mathbf E[\nub_*(\cdot,\rho,\cu_{m},q)].
	\end{split}
\end{equation}
The monotone property allows us to define the limit
\begin{equation*}
	\D(\rho) :=\lim_{m\rightarrow\infty}\mathbf E[\D(\cdot,\rho,\cu_m)], \qquad \cc(\rho) :=2\chi(\rho)\D(\rho).
\end{equation*}

Our goal is to prove the quantitative convergence rate of the finite-volume quenched conductivity matrix to this limit. Because $ \cc(\omega, \rho,\Lambda)$ depends on the random environment $\omega$, we need to measure the random fluctuation and we use the following notation introduced in \cite[Appendix A]{AKMbook}: for a random variable $X$ in $(\Omega, \gil,\mathbf P)$ and exponents $s,\theta\in(0,\infty)$, we define
\begin{equation}\label{eq.defOs}
	X\leq\mathcal O_s(\theta) \Longleftrightarrow 	\mathbf E[\exp((\theta^{-1}X_+)^s)]\leq 2,
\end{equation}
where $X_+:=\max\{X,0\}$. Roughly, \eqref{eq.defOs} tells that $X$ is concentrated with a stretched exponential tail. More properties of $\mathcal O_s$ can be found in \cite[Appendix A]{AKMbook}. Using this notation, we state our result as follows.
\begin{theorem}\label{thm.disorder}
	Fix $t\in(0,d)$, there exists $\beta(d,\lambda)>0$, $C(t,d,\lambda)<\infty$ such that for any $\rho\in(0,1)$ and any $m\in\N_+$:
	\begin{equation}\label{eq.disorder}
		|\cc(\rho)-\cc(\omega, \rho,\cu_m)|+|\cc(\rho)-\cc_*(\omega, \rho,\cu_m)|\leq\left(C3^{-m\beta (d-t)}+\mathcal O_1(C3^{-mt})\right).
	\end{equation}
\end{theorem}

\begin{proof}
	We follow \cite[Theorem~2.4]{AKMbook} to handle the fluctuation part. We combine the quenched subadditive quantity defined in \eqref{eq.defNuOmega} to obtain the quenched master quantity 
	\begin{align}\label{eq.defJOmega}
		J(\omega, \rho,\Lambda,p,q):=\nub(\omega, \rho,\Lambda,p)+\nub_*(\omega, \rho,\Lambda,q)- p\cdot q.
	\end{align}
	Then for each $\omega$ fixed, Lemmas~\ref{lem.repJ} and \ref{lem.ratecontrolJ} remain valid and especially we have  
	\begin{align*}
		\Ll\vert \cc(\omega, \rho,  \cu_m) - \cc_*(\omega, \rho,  \cu_m) \Rr\vert \leq C(d,\lambda)\chi(\rho)\Ll(\sup_{|p|=1}J(\omega,\rho,\cu_m,p, \D(\rho) p)^{\frac{1}{2}}\Rr).
	\end{align*}
	
	One can verify that $J(\omega, \rho,\Lambda,p,q)$ is a subadditive quantity. The key step of the proof relies on the observation in \cite[Lemma~A.7]{AKMbook}, which breaks the control of subadditive quantities with mixing condition into its mean and the random fluctuation. It also applies to the mapping $\Lambda \mapsto J(\omega, \rho,\Lambda,p,q)$ in our setting: for every $p \in B_1$, there exists a constant $C$ independent of $\rho$ such that 
	\begin{align}\label{eq.KeyDisorder}
		J(\omega,\rho,\cu_m,p, \D(\rho) p) \leq 2\mathbf{E}[J(\cdot,\rho,\cu_n,p, \D (\rho) p)] + \mathcal O_1(C \lambda 3^{-(m-n)d}).
	\end{align}
	Here we have the freedom to choose the parameter $n \in (0,m) \cap \N_+$, which will determine $t$ in \eqref{eq.disorder} by setting $t := d\Ll(1 - \frac{n}{m}\Rr) \in (0,d)$.

	Therefore, it suffices to study of the decay of the mean part $\mathbf{E}[J(\cdot,\rho,\cu_n,p, \D (\rho) p)]$ and the proof is quite similar to what we have done in Section~\ref{sec.Rate}, where the main steps are Lemmas~\ref{lem.Jestimate} and ~\ref{lem.mainVariance}, and Proposition~\ref{prop.L2Flat}. They can be carried to the disordered setting for the following reasons.
	\begin{enumerate}
		\item Lemma~\ref{lem.Jestimate} depends on the modified Caccioppoli inequality \eqref{eq.Caccioppoli}, which is valid for each $\omega \in \Omega$ because  $\{c_b^\omega\}_{b\in(\Zd)^*}$  satisfies the uniform ellipticity by (1) of Hypothesis~\ref{hyp2}, and the underlying invariant measure is still Bernoulli measure by (2) of Hypothesis~\ref{hyp2}.
		\item The variance decay in Lemma~\ref{lem.mainVariance} uses the spatial independence, which is ensured by (3) and (4) of  Hypothesis~\ref{hyp2}. More precisely, let $v(\omega, \rho, \cu_n,  p, q)$ be the optimizer of $J(\omega,\rho,\cu_n,p, q)$ like \eqref{eq.varJ}, we aim to estimate 
		\begin{align*}
			&\frac{1}{2 \chi(\rho)}\mathbf E\left[\bracket{\Ll\vert\frac{1}{\vert \cu_n \vert} \sum_{x \in \cu_n} \nabla_x \Ll(v(\omega, \rho, \cu_n,  p, q) - \ell_{D_n^{-1}q - p}\Rr)\Rr\vert^2}_{\rho}\right] \\
			&\leq C 3^{-\beta n} + C\sum_{k=0}^{n-1}3^{-\beta(n-k)}\tau_k,
		\end{align*}
		where the quantity $D_n(\rho)$ is defined as
		\begin{align*}
			D_n(\rho) :=\mathbf E[\D_*(\cdot,\rho,\cu_n)^{-1}]^{-1},
		\end{align*}
		and $\tau_n$ is defined as
		\begin{equation*}
			\tau_n=\sup_{p,q\in B_1} \mathbf E[J(\cdot,\rho,\cu_n,p,q)- J(\cdot,\rho,\cu_{n+1},p,q)].
		\end{equation*}
		We write $v_n$ as a shorthand for $v(\omega, \rho, \cu_n,  p, q)$ and $v_{n-1, z}$ as a shorthand of $v(\omega, \rho, z+\cu_{n-1},  p, q)$, then we have the decomposition
		\begin{align*}
			&\mathbf E\left[\bracket{\frac{1}{2 \chi(\rho)}\Ll\vert\frac{1}{\vert \cu_n \vert} \sum_{x \in \cu_n} \nabla_x \Ll(v_n - \ell_{D_n^{-1}q - p}\Rr)\Rr\vert^2}_{\rho}\right]^\frac{1}{2}\\
			& \leq \mathbf E\left[\bracket{\frac{1}{2 \chi(\rho)} \Ll\vert\frac{1}{\vert \cu_n \vert}\sum_{z\in\Z_{n,n-1}} \sum_{x \in z+\cu_{n-1}} \nabla_x \Ll(v_{n-1, z} - \ell_{D_{n-1}^{-1}q - p}\Rr)\Rr\vert^2}_{\rho}\right]^\frac{1}{2}\\
			& \quad + \mathbf E \left[\bracket{\frac{1}{2 \chi(\rho)}\Ll\vert\frac{1}{\vert \cu_n \vert}\sum_{z\in\Z_{n,n-1}} \sum_{x \in z+\cu_{n-1}} \nabla_x \left(  v_{n-1,z} -  v_n\right)\Rr\vert^2}_{\rho}\right]^\frac{1}{2}\\
			& \quad +\mathbf \lambda|D_n^{-1}q-D_{n-1}^{-1}q|.
		\end{align*}
		The key is the cancellation in the first term under $\mathbf E[\bracket{\cdots}_\rho]$. After expanding the sum, we not only use the finite-range dependence of  $\eta \mapsto c^{\omega}_b(\eta)$ to obtain the independence over $\mathbb P_\rho$, but we also use the unit-range dependence of $\mathbf P$. 
		
		\item The $L^2$-flatness estimate in Proposition~\ref{prop.L2Flat} should also be carried under the expectation $\mathbf E$. It relies on the weighted multiscale Poincar\'e inequality \eqref{eq.WMPoincare}, which does not involve the jump rate. Then we further develop it, and the variance decay in Lemma~\ref{lem.mainVariance} applies to the typical case \eqref{eq.L2Case11Weight}, which has been discussed as above. For the atypical case \eqref{eq.L2Case12}, the $L^\infty$ estimate also applies since it only requires the log-Sobolev inequality (see the proof in Appendix~\ref{appendix.B}), which remains valid thanks to the uniform ellipticity of the jump rate by (1) of Hypothesis~\ref{hyp2} .
	\end{enumerate}
\end{proof}
One can deduce further results built on Theorem~\ref{thm.disorder}.

\appendix
\section{Sub-additivity and Whitney inequality}
A lot of results on quantitative homogenization are stated for the triadic cubes, but they actually hold for the general domain with reasonable boundary regularity, and such generalization only relies on the sub-additivity and a nice decomposition. We state this observation under $\Rd$ setting, and one can adapt it easily in lattice models. In the following statement, a triadic cube $Q$ in $\Rd$ is an open set of type $z + (-\frac{3^m}{2}, \frac{3^m}{2})^d$, where $z \in \Rd$ is called its \emph{center} and $3^m$ with $m \in \mathbb{Z}$ is called its \emph{size} and is denoted by  $\size(Q)$. Especially, we denote by $\cu_m = (-\frac{3^m}{2}, \frac{3^m}{2})^d$ as the cube centered at $0$ and of size $3^m$ in this section. We also denote by $\vert U \vert$ the $\Rd$-Lebesgue measure for Borel set $U$, and by $\sigma(\mcl C)$ the area of the $(d-1)$-dimensional surface $\mcl C$.

\begin{lemma}\label{lem.WhitneySub} Let the quantity $\nu$ be defined on the bounded open sets of $\Rd$ with Lipschitz boundary and satisfy the following properties.
	\begin{enumerate}
		\item (Spatial homogeneous) For any $z \in \Rd$ and open set $U \subset \Rd$, $\nu(z+U) = \nu(U)$.
		\item (Sub-additivity) Given disjoint open sets $\{U_i\}_{1 \leq i \leq n}$ such that they partition $U$ in the sense  that $U_1, \cdots, U_n \subset U$ and  $\vert U \setminus (\bigcup_{i=1}^n U_i)\vert = 0$, then we have
		\begin{align}\label{eq.WhitneySubadditive}
			\nu(U) \leq \sum_{i=1}^n \frac{\vert U_i \vert}{\vert U \vert}\nu(U_i).
		\end{align}
		\item (Decay for triadic cubes) There exist two finite positive constants $C, \alpha$, such that for any triadic cube $\cu_m$, we have 
		\begin{align}\label{eq.WhitneyDecayCube}
			\nu(\cu_m) \leq C(3^{- \alpha m} \wedge 1).
		\end{align}
	\end{enumerate}
	Then there exists a positive constant $C'$, such that for any open set $U$ circumscribed by piecewise $C^1$ surfaces, we have
	\begin{align}
		\nu(U)  \leq C'  \Ll(\frac{\sigma( \partial U ) \diam(U)^{(1-\alpha) \vee 0}}{\vert U \vert} \wedge 1\Rr).
	\end{align}
\end{lemma}
\begin{remark}
	As an application, for the cube $\Lambda_L = (-\frac{L}{2}, \frac{L}{2})^d$ with $L \geq 1$, we have $\nu(\Lambda_L)  \leq C' L^{-(\alpha \wedge 1)}$ and it gives a polynomial decay in function of the diameter. 
\end{remark}

\begin{proof}
	The case for $\diam(U) \leq 1$ is trivial and we focus on the case that $U$ of large diameter. The proof relies on the standard Whitney decomposition, which will give a family of open triadic cubes $\{Q_j\}_{j \geq 0}$ of nice properties; see \cite[Appendix J]{Grafakos} for its construction and proof. In our proof, the useful properties from such decomposition  are
	\begin{itemize}
		\item $\{Q_j\}_{j \geq 0}$ are disjoint and they partition $U$ in the sense of (2) in Lemma~\ref{lem.WhitneySub};
		\item $\sqrt{d}\size(Q_j) \leq \dist(Q_j, \partial U) \leq 4\sqrt{d}\size(Q_j)$.
	\end{itemize}
	Applying the sub-additivity to $\{Q_j\}_{j \geq 0}$ and by passing to the limit, we have 
	\begin{align}\label{eq.SubWhitney}
		\nu(U) \leq \sum_{i=1}^\infty \frac{\vert Q_i \vert}{\vert U \vert}\nu(Q_i).
	\end{align}
	Let $m$ be the positive integer such that 
	\begin{align}\label{eq.WhitneyDiameter}
		3^{m-1} \leq \diam(U) \leq 3^m,
	\end{align}
	then we classify the cubes by their sizes 
	\begin{align*}
		I_k := \{i \in \N_+: \size(Q_i) = 3^k\}.
	\end{align*}
	Using the second properties listed above about the decomposition, all the cubes in $I_k$ stay in distance $4\sqrt{d} 3^k$ from the boundary, then we have 
	\begin{align}\label{eq.QIkInclusion}
		\bigcup_{i \in I_k} Q_i \subset \{x \in U: \dist(x, \partial U) \leq 5\sqrt{d} 3^k\}.
	\end{align}
	Since $U$ admits piece-wise $C^1$ boundary, we denote by $\partial U = \bigcup_{j=1}^m \mcl C_j$,
	where different pieces $\mcl C_j$ only have intersection of null set. Then we have the following volume estimate 
	\begin{multline}\label{eq.WhitneyVolume}
		\sum_{i\in I_k} \vert Q_i\vert \leq \left\vert  \{x \in U: \dist(x, \partial U) \leq 5\sqrt{d} 3^k\} \right\vert  \leq \sum_{j=1}^m \left\vert  \{x \in U: \dist(x, \partial \mcl C_j) \leq 5\sqrt{d} 3^k\} \right\vert\\
		\leq 10 \sqrt{d} 3^k  \sum_{j=1}^m \sigma(\mcl C_j) = 10 \sqrt{d} 3^k \sigma(\partial U).
	\end{multline}
	Then we put these estimates back to the cubes from Whitney decomposition \eqref{eq.SubWhitney} to obtain
	\begin{align*}
		\nu(U) &\leq \sum_{k=-\infty}^0 \frac{1}{\vert U \vert} \Ll(\sum_{i\in I_k} \vert Q_i \vert\Rr) \nu(\cu_k)  +  \sum_{k=1}^m \frac{1}{\vert U \vert} \Ll(\sum_{i\in I_k} \vert Q_i \vert\Rr) \nu(\cu_k)\\
		&\leq \frac{C \sigma(\partial U)}{\vert U \vert} \sum_{k=-\infty}^0 3^{k} + \frac{C \sigma(\partial U)}{\vert U \vert} \sum_{k=1}^m 3^{(1-\alpha) k}\\
		&\leq \frac{C \sigma(\partial U)}{\vert U \vert} 3^{m(1-\alpha) \vee 0}.
	\end{align*}
	Here apply \eqref{eq.WhitneyDecayCube} and \eqref{eq.WhitneyVolume} from the first line to the second line, then we use \eqref{eq.WhitneyDiameter} to conclude the proof.
\end{proof}

\section{Proof of $L^\infty$ norm using mixing time}\label{appendix.B}
In this part, we prove Lemma~\ref{eq.supNorm}. Such estimate should be generally valid for Markov chain, and we state the case of Kawasaki dynamics.

\begin{lemma}
	There exists a constant $C(\lambda,d) < \infty$, such that for any connected domain $\Lambda \subset \Zd$ of diameter $L$ and any two functions  $u,f: \X \to \R$ satisfying
	\begin{align*}
		\L_\Lambda u = f,
	\end{align*}
	then for any $N\in \N_+$ we have 
	\begin{align}\label{eq.supNorm2}
		\norm{u - \bracket{u}_{\Lambda, N}}_{\infty} \leq C L^{2} \log L \norm{f}_{\infty}.
	\end{align}
\end{lemma}
\begin{proof}
	Let $P_t(\cdot, \cdot):\X \times \X \to \R_+$ be the transition probability of the continuous-time Kawasaki dynamics generated by $\L_\Lambda$ with reflect boundary condition, which has $\mu_{\Lambda, N} = \P_{\Lambda, N}$ as its stationary measure for any $N \in \N_+$. Then we use the Duhamel's formula 
	\begin{align}\label{eq.Duhamel}
		u(\eta) - \bracket{u}_{\Lambda, N} = \int_{0}^\infty (P_t f)(\eta) - \bracket{f}_{\Lambda, N} \, \d t,
	\end{align}
	which is well-defined thanks to the spectral gap and exponential decay of the mapping $t \mapsto \var_{\Lambda, N}[P_t f]$.
	
	Then we estimate the $L^\infty$ norm. We notice that 
	\begin{align}\label{eq.TV}
		\norm{(P_t f)(\eta) - \bracket{f}_{\Lambda, N}}_{\infty} &= \sup_{\eta \in \X} \Ll(\sum_{\eta' \in \X} P_t(\eta, \eta') f(\eta')- \sum_{\eta' \in \X} \mu_{\Lambda, N} (\eta')f(\eta') \Rr)  \nonumber \\
		&\leq 2 \sup_{\eta \in \X}\norm{P_t(\eta, \cdot) - \mu_{\Lambda, N} }_{TV} \norm{f}_{\infty}.
	\end{align}
	Because we assume reflect boundary condition for the Kawasaki dynamics associated to $\L_\Lambda$, the term $P_t(\eta, \eta')$ is non-zero only when 
	\begin{align*}
		\forall x \in \Lambda^c, \qquad \eta_x  =  \eta'_x. 
	\end{align*}
	Thus, the sum $\sum_{\eta' \in \X}$ is actually over a finite set and well-defined. Then we denote by $\norm{\cdot}_{TV}$ the total variation distance and make use of its definition \cite[(4.7)]{LP17}. By convention in \cite[(4.22),(4.30),(4.31)]{LP17}, we also define 
	\begin{align*}
		d(t) &:= \sup_{\eta \in \X}\norm{P_t(\eta, \cdot) - \mu_{\Lambda, N} }_{TV}, \\
		\tx &:= \inf\Ll\{t\in \R_+: d(t) < \frac{1}{4}\Rr\}.
	\end{align*} 
	Then $t \mapsto d(t)$ is decreasing and satisfies (see \cite[(4.33)]{LP17})
	\begin{align*}
		\forall n \in \N, \qquad d(n \tx) \leq 2^{-n}.
	\end{align*}
	We put this observation and \eqref{eq.TV} back to \eqref{eq.Duhamel}
	\begin{align*}
		\norm{(P_t f)(\eta) - \bracket{f}_{\Lambda, N}}_{\infty} &\leq \sum_{n=0}^\infty \int_{n \tx}^{(n+1) \tx} \norm{(P_t f)(\eta) - \bracket{f}_{\Lambda, N}}_{\infty} \, \d t\\
		&\leq \sum_{n=0}^\infty 2 d(n \tx)  \norm{f}_{\infty} \tx\\
		&\leq 2\sum_{n=0}^\infty 2^{-n} \norm{f}_{\infty} \tx\\
		&= 2 \norm{f}_{\infty}\tx.
	\end{align*}
	One classical method to obtain the mixing time of Markov process is the log-Sobolev inequality; see \cite[(1.8)]{DS96}. In our case, we use the log-Sobolev inequality on general Kawasaki dynamics developed in \cite[Theorem~3]{LuYau}; see also \cite[(2)]{morris16} and discussion there. This gives us 
	\begin{align}\label{eq.Tmix}
		\tx \leq C(d,\lambda) L^2 \log \log {L^d \choose N} \leq \tilde C(d,\lambda) L^2 \log L,
	\end{align}
	which yields \eqref{eq.supNorm2}. 
\end{proof}
\begin{remark}
	The steps before \eqref{eq.Tmix} is standard for all the reversible Markov processes. We apply the log-Sobolev inequality for the mixing time, because the jump rate in this paper depends on the local configuration. The spectral gap inequality \cite[Theorem~2]{LuYau} is also enough to conclude \eqref{eq.Tmix} with a slightly larger bound $L^{d+2} \log L$. We remark some  other recent progresses \cite{oliveira13, lacoin16, morris16} on the mixing time of the exclusion processes (i.e. the constant-speed Kawasaki dynamics). Their generalization on the non-gradient models is an interesting but challenging question. 
\end{remark}

\begin{proof}[Proof of Lemma~\ref{lem.supNorm}]
	We take $u(\Lambda, q)$ for example, which is the solution of 
	\begin{align*}
		\L_{\Lambda} u(\Lambda, q) = \sum_{b \in \overline{\Lambda_L^*}} \pi_b \ell_q.
	\end{align*}
	Therefore, we apply \eqref{eq.supNorm2} with $\norm{\sum_{b \in \overline{\Lambda_L^*}} \pi_b \ell_q}_{\infty} \leq C(\lambda, d)L^d$. The case for $v(\rho, \Lambda, \xi)$ can be done similarly.
\end{proof}
\begin{remark}\label{rmk.Lp}
	For $2\le p < \infty$, we have $\| u- \langle u \rangle_{\Lambda,N}\|_p \le C L^2 \|f\|_p$ without $\log L$.  Indeed, we may apply Riesz--Thorin interpolation theorem for two inequalities $\| P_t f- \langle f \rangle_{\Lambda,N}\|_\infty \le 2 \|f\|_\infty$ and $\| P_t f- \langle f \rangle_{\Lambda,N}\|_2
	\le e^{-c L^{-2}t} \|f\|_2$ for some $c>0$, which follows from the spectral gap estimate for $\mathcal{L}_\Lambda$.
\end{remark}

\section{Proofs of lemmas and proposition in Section \ref{subsec.relative_entropy}}
\label{sec:4-C}

Here we complete the proofs of 
Lemma \ref{Lemma 3.1} (calculation of $\partial_t h_N(t)$), 
Proposition \ref{One-block} (refined one-block estimate),
Lemma \ref{Lemma 3.3} (a key lemma) and 
Lemma \ref{Lemma 3.4} (rewriting $O(N)$-looking term to $O(1)$)
postponed from Sections \ref{subsubsec.derivative_ht}, \ref{subsubsec.one_blcok} and 
\ref{subsubsec.grad_replacement}.  These are gathered here, since 
they are shown based on known results with some modification
providing fine error estimates.

\subsection{Proof of Lemma \ref{Lemma 3.1} (calculation of $\partial_t h_N(t)$)}
\label{sec:4.1-C}
Throughout this proof, we use the following shorthand notations for simplicity 
\begin{align*}
	F \equiv F_N, \qquad \psi_t \equiv \psi_{t}^N, \qquad f_t \equiv f_t^N, \qquad Z_t \equiv Z_{\la(t,\cdot),F_N}.
\end{align*}
Recall that $h_N(t)$ is the relative entropy per volume defined by \eqref{eq:hNt} 
with respect to the local equilibrium $\psi_t = \psi_{\la(t,\cdot),F}^N$
in \eqref{eq:2psit}.    Then, we apply the estimate
\begin{equation}
	\partial_t h_N(t) \le N^{-d}
	\int_{\X_N} \psi_t^{-1} \left(\mathcal{L}_N^* \psi_t - \partial_t \psi_t
	\right) \cdot f_t\, \d \nu^N ,     \label{3.1}
\end{equation}
which holds for a large class of Markovian models, where $\mathcal{L}_N^*$ 
denotes the dual operator of the generator  $ \mathcal{L}_N (=N^2\mathcal{L})$  
with respect to 
the measure $ \nu^N$; see \cite{Yau} Lemma 1.
The proof is divided into four steps.

{\it Step 1.} (First term in \eqref{3.1}) 
The integrand on the right-hand side of  (\ref{3.1}) was already computed 
in the proof of Lemma 3.1 of \cite{fuy},
but only with the error estimate $o(1)$.  Let us record the computation (3.2) in
\cite{fuy} here to make
the error term clear.  Noting the symmetry of the Kawasaki generator $\L$:
$\L^*=\L$,
\begin{align}  \label{3.2}
	& \qquad  N^{-d} \psi_t^{-1} N^2 \L \psi_t    \\
	= &  \frac{N^{2-d}}2  \sum_{x\sim y} c_{x,y} \bigg[ \exp 
	\Big\{(\la(t,x/N) - \la(t,y/N)) (\eta_y-\eta_x)
	\phantom{\sum_{z\in\T_N^d}}   \notag  \\
	&\qquad\qquad \qquad\qquad \qquad      
	+ \frac1N \pi_{x,y} \sum_{z\in\T_N^d} \left(\nabla\la(t,z/N)\cdot\tau_z F
	\right) \Big\} - 1 \bigg]    
	\notag \\  
	= & - \frac{N^{1-d}}2 \sum_{x\sim y} c_{x,y} \Om_{x,y}
	+ \frac{N^{-d}}{4} \sum_{x\sim y} c_{x,y} \Om_{x,y}^2  \notag \\
	& - \frac{N^{-d}}{4} \sum_{x\sim y}  c_{x,y} \sum_{i,j=1}^d
	\partial_{v_i}\partial_{v_j}\la(t,x/N) (y_i -x_i)(y_j-x_j)(\eta_y-\eta_x) \notag \\
	& + \frac{N^{-d}}2 \sum_{x\sim y} c_{x,y} \sum_{i,j=1}^d
	\partial_{v_i}\partial_{v_j}\la(t,x/N)  \pi_{x,y}\left(\sum_{z\in\T_N^d}(z_j-x_j)
	\tau_z F_i\right)          + Q_{1,N},   \notag
\end{align}
where the error term $Q_{1,N}\equiv Q_{1,N}(\la,F)$ is defined as the difference
between the expressions in the middle two lines and the last three lines
in \eqref{3.2}.  Recall \eqref{eq:Omxy} for $\Om_{x,y}$.

This error is estimated as 
\begin{align}  \label{eq:R(la,F)-1}
	|Q_{1,N}| \le &
	CN^{-1} \Big(1+ \|\nabla\la\|_{3,\infty}\Big)^3
	\Big(1+r(F)^{d+2}\|F\|_\infty\Big)^3,
\end{align}
if the condition \eqref{eq:2.F} 
and accordingly $|\e_N| \le C$ below are satisfied for some $C>0$. Here, we use $\e_N\equiv \e_{N,x,y}$ to denote the term given in the braces in \eqref{3.2}. Indeed, to show \eqref{eq:R(la,F)-1}, we expand the second to the third lines
of \eqref{3.2} as
$$
\exp\{\e_N\}-1 = \e_N+ \tfrac12\e_N^2 + O(\e_N^3), \quad |\e_N| \le C.
$$
Then, by Taylor's formula, we have
\begin{align*}
	\la(t,x/N) - \la(t,y/N)= -&\tfrac1N  \nabla\la(t,x/N)\cdot (y-x)\\
	& -\tfrac1{2N^2}(y-x)\cdot \nabla^2\la(t,x/N) (y-x)+ 
	O\Big( N^{-3} {\|\nabla^3\la\|_\infty} \Big),
\end{align*}
and therefore
$$
\e_N= -\tfrac1N \Om_{x,y} + \tfrac1N \bar\Om_{x,y}
- \tfrac1{2N^2}(y-x)\cdot \nabla^2\la(t,x/N)(y-x)
(\eta_y-\eta_x) + O\Big( N^{-3} {\|\nabla^3\la\|_\infty}\Big),
$$
where
\begin{align*}
	& \bar \Om_{x,y}  \equiv \bar\Om_{x,y}(\eta) 
	=  \pi_{x,y}\sum_{z\in\T_N^d}\left(\nabla\la(t,z/N) 
	- \nabla\la(t,x/N)\right)\cdot\tau_z F) \\
	& = N^{-1} \left( \nabla^2\la(t,x/N)\pi_{x,y}\sum_{z\in\T_N^d}(z-x)\right)\cdot\tau_z F 
	+ O\left(N^{-2} \|\nabla^3\la\|_\infty \left|\pi_{x,y}\sum_{z\in\T_N^d}
	(z-x)^{\otimes 2}
	\t_z F\right|\right),
\end{align*}
where $x^{\otimes 2}=(x_{i}x_j)_{1\leq i,j\leq d}$ is a $d\times d$ matrix. 
Thus, $Q_{1,N}\equiv Q_{1,N}(\la,F)$ is given by
\begin{align*}
	Q_{1,N} = & \frac{N^{2-d}}2 \sum_{x\sim y} c_{x,y} 
	O\Big( N^{-3} \|\nabla^3\la\|_\infty  (1+r(F)^{d+2}\|F\|_\infty)\Big)  \\
	& + \frac{N^{2-d}}4 \sum_{x\sim y} c_{x,y}  
	\big(\e_N^2- N^{-2}\Om_{x,y}^2 \big)
	+ O\left(N^{2-d} \sum_{x\sim y} c_{x,y} \e_N^3 \right).
\end{align*}
The first term (its absolute value) is bounded by
$$
CN^{-1} \|\nabla^3\la\|_\infty(1+r(F)^{d+2}\|F\|_\infty).
$$
By noting that $|\Om_{x,y}|\le \|\nabla\la\|_\infty 
(1+r(F)^{d}\|F\|_\infty)$,
the second is bounded by
\begin{align*}
	&\left|  \frac{N^{2-d}}4 \sum_{x\sim y} c_{x,y}  
	\big(\e_N+ \tfrac1{N}\Om_{x,y} \big) \big(\e_N- \tfrac1{N}\Om_{x,y} \big) \right| \\
	& \le CN^{-1} \Big( 1+ \|\nabla\la\|_{3,\infty}\Big)^2
	\Big(1+r(F)^{d+2}\|F\|_\infty\Big)^2,
\end{align*}
see the same estimate given in Section 9.1 of \cite{Fu24}, if necessary.
Similarly, the third is bounded by
\begin{align*}
	CN^2  \sup_{x,y,t} |\e_N^3|
	\le CN^{-1} \Big( 1+ \|\nabla\la\|_{3,\infty}\Big)^3
	\Big(1+r(F)^{d+2}\|F\|_\infty\Big)^3,
\end{align*}
also see Section 9.1 of \cite{Fu24}, if necessary.
Summarizing these estimates for three terms of $Q_{1,N}(\la,F)$, we obtain
\eqref{eq:R(la,F)-1} for $Q_{1,N}(\la,F)$.

{\it Step 2.} (Second term in  \eqref{3.1})
Recalling \eqref{eq:2psit} for $\psi_t\, (=\psi_t^N)$, the second term in the integrand
on the right-hand side of \eqref{3.1} is rewritten as
\begin{align}  \label{3.3}
	N^{-d} \psi_t^{-1}  \partial_t \psi_t
	=  &- N^{-d} Z_t^{-1} \partial_t Z_t  \\
	&  + N^{-d} \partial_t
	\left\{ \sum_{x\in\T_N^d} \la(t,x/N) \eta_x 
	+ \frac1{N} \sum_{x\in\T_N^d} \nabla\la(t,x/N)\cdot
	\tau_x F  \right\}.    \notag
\end{align}
From \eqref{3.2} and \eqref{3.3}, it follows that
\begin{equation}
	N^{-d} \psi_t^{-1} \left(\mathcal{L}_N^* \psi_t - \partial_t \psi_t \right)
	= \Om_1(\eta) + \Om_2(\eta) + a(t) +\sum_{i=1}^2 Q_{i,N},  \label{3.4}
\end{equation}
where $Q_{1,N}$ is given in Step 1,
\begin{align*}
	Q_{2,N} \equiv Q_{2,N}(\la,F) = -N^{-d-1}\sum_{x\in \T_N^d} \nabla\dot{\la}(t,x/N)\cdot\t_x F,
\end{align*}
and
\begin{equation}
	a(t) := N^{-d} Z_t^{-1} \partial_t Z_t
	= E^{\psi_t}\left[ N^{-d} \sum_{x\in\T_N^d} \dot{\la}(t,x/N)\eta_x \right]
	+ O\Big(N^{-1} \|\nabla\dot{\la}\|_\infty \|F\|_\infty\Big).     \label{3.5}
\end{equation}
Recall \eqref{eq:Om1} for $\Om_1$.
The last equality for $a(t)$  is seen from \eqref{3.3} by noting that
$$
E^{\psi_t}\left[ N^{-d} \psi_t^{-1} \partial_t\psi_t \right] = 0
$$
and
\begin{equation}  \label{eq:R3}
	| Q_{2,N}| \le N^{-1} \|\nabla\dot{\la}\|_\infty \|F\|_\infty.
\end{equation}
Thus,
$$
a(t) = N^{-d} \sum_{x\in\T_N^d} \dot{\la}(t,x/N)\rho(t,x/N)+Q_{3,N} +
O\Big(N^{-1} \|\nabla\dot{\la}\|_\infty \|F\|_\infty\Big),
$$
where the error $Q_{3,N} \equiv Q_{3,N}(\la,F)$ is given by
$$
Q_{3,N} = N^{-d} \sum_{x\in\T_N^d} \dot{\la}(t,x/N)
\Big( E^{\psi_t}\left[ \eta_x \right]-\rho(t,x/N)\Big).
$$
We only need to estimate $|E^{\psi_t}\left[ \eta_x \right] 
-\rho(t,x/N)|$.

{\it Step 3.} (Proof of the error estimate  \eqref{eq:R(la,F)})
To give an estimate for $Q_{3,N} \equiv Q_{3,N}(\la,F)$, we note the 
$r$-Markov property of $\psi_t$ (or $P^{\psi_t}$) with $r=r(F)$
(cf.\  Lemma A.1 of \cite{Fu24}) to see
\begin{equation}  \label{m-etax}
	E^{\psi_t}[\eta_x] = E^{\psi_t}\left[E^{P_{\{ x\}}^{\om,F}} [\eta_x] \right],
\end{equation}
where $P_{\{ x\}}^{\om,F}(\eta_x) = P^{\psi_t}(\{\zeta\in \mathcal{X}_N;
\zeta_x=\eta_x\}|\mathcal{F}_{\{x\}^c})(\om)$ for $\eta_x\in \{0,1\}$
and $\om \in \{0,1\}^{\bar\La \setminus \{ x\}}$, $\bar\La = \La_{r(F), x}$.
Then,
\begin{align}\label{mm-etax}
	E^{P_{\{ x\}}^{\om,F}} [\eta_x] 
	& = \widetilde Z_{\om,F}^{-1} A_{\om,F},
\end{align}
where
\begin{align*}
	A_{\om,F} & := \sum_{\eta_x \in \{0,1\}} \eta_x
	e^{\la(t,x/N) \eta_x + \frac1N \sum_{y \in \T_N^d:\, \text{supp}\, \t_y F\subset 
			\bar\La} (\nabla\la(t,y/N)\cdot\t_y F (\eta_x\cdot\om))}, \\
	\widetilde Z_{\om,F} & := \sum_{\eta_x\in \{0,1\}} 
	e^{\la(t,x/N) \eta_x + \frac1N \sum_{y \in \T_N^d:\, \text{supp}\, \t_y F\subset 
			\bar\La} (\nabla\la(t,y/N)\cdot\t_y F (\eta_x\cdot\om))}.
\end{align*}

Noting that $E^{P_{\{ x\}}^{\om,0}} [\eta_x] = \rho(t,x/N)$ taking $F\equiv 0$,
we have
$$
E^{P_{\{ x\}}^{\om,F}} [\eta_x] - \rho(t,x/N)
=  \widetilde Z_{\om,F}^{-1} A_{\om,F}  -  \widetilde Z_{\om,0}^{-1} A_{\om,0}.
$$
A simple estimate shows
\begin{align*}
	& \widetilde Z_{0,\omega} e^{-\frac{c}N r(F)^d \|\nabla\la(t,\cdot)\|_\infty \| F\|_\infty}
	\le \widetilde Z_{\om,F} \le 
	\widetilde Z_{0,\omega} e^{\frac{c}N r(F)^d \|\nabla\la(t,\cdot)\|_\infty \| F\|_\infty}, \\
	& A_{0,\omega} e^{-\frac{c}N r(F)^d \|\nabla\la(t,\cdot)\|_\infty \| F\|_\infty}
	\le A_{\om,F} \le 
	A_{0,\omega} e^{\frac{c}N r(F)^d \|\nabla\la(t,\cdot)\|_\infty \| F\|_\infty},
\end{align*}
for some $c>0$.
Therefore, taking the expectation in $\om$ under $P^{\psi_t}$, we obtain
by \eqref{m-etax}
\begin{align*}
	|E^{\psi_t}\left[ \eta_x \right] -\rho(t,x/N)|
	& = \left| E^{\psi_t}\left[ \frac{A_{0,\omega}}{\widetilde Z_{\om,F}}
	\Big( \frac{A_{\om,F}}{A_{0,\omega}} - \frac{\widetilde Z_{\om,F}}{\widetilde Z_{0,\omega}}\Big) \right] \right| \\
	& \le C_1 \Big| e^{\frac{c}N r(F)^d \|\nabla\la(t,\cdot)\|_\infty \| F\|_\infty}
	-1 \bigg| \\
	& \le C_2 N^{-1} r(F)^d \|\nabla\la\|_\infty \| F\|_\infty,
\end{align*}
if $F$ satisfies the condition \eqref{eq:2.F}, i.e., $N^{-1} r(F)^d 
\|\nabla\la\|_\infty \| F\|_\infty \le 1$; note that $0\le A_{0,\omega}/\widetilde Z_{\om,F}
= A_{0,\omega}/\widetilde Z_{0,\omega} \cdot \widetilde Z_{0,\omega}/\widetilde Z_{\om,F}\le e^c$
under this condition.  This leads to the bound on 
$Q_{3,N} \equiv Q_{3,N}(\la,F)$:
\begin{align}  \label{eq:3.R4}
	|Q_{3,N}| \le C N^{-1}  \|\dot{\la}\|_\infty \|\nabla\la\|_\infty 
	r(F)^{d}\|F\|_\infty.
\end{align}
By \eqref{eq:R(la,F)-1},  \eqref{eq:R3} and \eqref{eq:3.R4},
we obtain \eqref{eq:R(la,F)}, and this concludes the proof of Lemma \ref{Lemma 3.1}.

\subsection{Proof of Proposition \ref{One-block} (refined one-block estimate)}
\label{sec:4.2-C}

We follow the argument in Section 3 of \cite{FvMST}, where the Glauber-Kawasaki dynamics $\eta^N(t)$ was considered.  Note that although the gradient condition was assumed in \cite{FvMST}, it was not used in Section 3.  

Recalling $\hat\Om_x$ in \eqref{eq:hatOm}, we decompose as 
\begin{align} \label{BG:pf:0}
	&E \bigg[ \bigg |\int_0^t N^{-d} \sum_{x \in \T_N^d} a_{s,x} \hat{\Om}_x 
	\d s \bigg |\bigg]
	\\\leq & E\bigg[ \bigg |\int_0^t N^{-d} 
	\sum_{x \in \T_N^d} a_{s,x} m_x \d s \bigg |\bigg]  
	+ E\bigg[ \bigg |\int_0^t N^{-d} \sum_{x \in \T_N^d} a_{s,x} E^{\nu_{1/2}}
	[\hat{\Om}_x \mid \bar\eta_x^L] \d s \bigg |\bigg],  \notag
\end{align}
where $\nu_{1/2}$ is the Bernoulli measure with mean
$\rho=1/2$ on $\mathcal{X}=\{0,1\}^{{\mathbb Z}^d}$ or on 
$\mathcal{X}_N=\{0,1\}^{\T_N^d}$ and
\begin{align*}
	m_x \equiv m_x(\eta)
	:= \hat{\Om}_x - E^{\nu_{1/2}}[\hat{\Om}_x \mid \bar\eta_x^L].
\end{align*}
Precisely saying, the three expectations in \eqref{BG:pf:0} are taken for the
process $\eta^N(\cdot)$, so that $\hat{\Om}_x$ in the left-hand side
and $m_x$ in the right-hand side are understood as $\hat{\Om}_x(\eta^N(s))$ and
$m_x(\eta^N(s))$, respectively, and also we replace $\eta$ to
$\eta^N(s)$ in the conditional expectation, which is a function of
$\eta$, in the last expectation.

Lemma 3.1 of \cite{FvMST} shown for the Glauber-Kawasaki dynamics
applies for the first term in the right-hand side of \eqref{BG:pf:0}.
Recall we assume \eqref{BG:assn:atx}  for $a_{t,x}$: $|a_{t,x}|\le M$,
also \eqref{3.26-A} for $L$.

\begin{lemma} {\rm (cf.\ Lemma 3.1 of \cite{FvMST}, Lemma 9.1 of \cite{Fu24})} 
	\label{l:L10:6}
	Let $\gamma = \gamma(N) > 0$ be given and assume 
	$\gamma M L^{d+2} \le \de N^2$ for some $\de>0$ small enough.
	Then, for every $t\in [0,T]$ and $N$ large enough
	\begin{equation}  \label{eq:8.15-C}
		E\left[ \left |\int_0^t N^{-d} \sum_{x \in \T_N^d} a_{s,x} m_x \d s \right |\right]
		\leq C \left( \frac 1\gamma 
		+ \frac{ \gamma  L^{d+2} }{ N^2 } M^{2} \| \Om\|_\infty^2\right),
	\end{equation}
	for some $C=C_T>0$.  
	In particular, choosing 
	$\ga= N L^{-(d+2)/2}/ M(\| \Om\|_\infty+1)$, the right-hand side is
	bounded by
	$$
	2CN^{-1} L^{(d+2)/2} M(\| \Om\|_\infty+1).
	$$
	For this choice of $\ga$, the above condition for $\ga$ is satisfied if 
	the condition \eqref{3.26-A} stated in Proposition \ref{One-block} holds.
\end{lemma}

\begin{proof}
	We may adjust the proof of Lemma 3.1 of \cite{FvMST} in our setting.
	The differences are two-fold.  The generator in \cite{FvMST} was
	$N^2 \L + KL_G$ with Glauber generator $L_G$ with time change factor $K\ge 1$.
	In our case, $K=0$.  The second difference is the assumption on $a_{t,x}$.
	We assume \eqref{BG:assn:atx}: $|a_{t,x}| \le M$, while it was
	$|a_{t,x}| \le CK^\th/[t(T-t)]^\Theta$ for some $\th\ge 0, \Theta\in [0,1)$
	depending on the Glauber speed $K$ and also $t\in [0,T]$.  Our assumption
	\eqref{BG:assn:atx} is simpler. We only point out the necessary changes in the proof.
	
	First by applying the entropy inequality with respect to the Bernoulli
	measure $\nu^N=\nu_{1/2}^N$ on $\mathcal{X}_N$, 
	as in (3.3) of \cite{FvMST}, we obtain
	\begin{align} \label{BG:pf:5}
		& E\left[ \left |\int_0^t N^{-d} \sum_{x \in \T_N^d} a_{s,x} m_x \d s \right |\right]\\
		&\qquad  \leq \frac{\log 2}\gamma + \frac{N^{-d}}\gamma 
		\log E^{\nu^N} \left[ \exp \left| \gamma \int_0^t \sum_{x \in \T_N^d} 
		a_{s,x} m_x  \d s \right| \right].   \notag
	\end{align}
	
	To estimate the second term on the right-hand side of \eqref{BG:pf:5},
	instead of the assumption (1.19) in \cite{FvMST}
	for $a_{t,x}$ (i.e.\ $|a_{t,x}| \le CK^\th/[t(T-t)]^\Theta$, $\th\ge 0,
	\Theta\in [0,1)$), we have \eqref{BG:assn:atx}:
	$|a_{t,x}| \le M$ for all $t\in [0,T], \, x \in \T_N^d$.
	In other words, we may take $\Theta=0$ and $M$ for $CK^\th$ in
	(1.19) of \cite{FvMST}.  Furthermore, letting $\e= \e(N)\downarrow 0$,
	the first term in the concluding estimate in  Lemma 3.1 of \cite{FvMST}
	becomes negligible, since $\e^{1-\Theta}K^\th=\e M/C\to 0$.  
	
	Then, since there is no Glauber part (i.e.\ $K=0$), we do not have the term 
	$K\lan -L_G\psi,\psi\ran_{\nu_\b}$ in (3.7) (recall $\nu_\b=\nu^N$)
	so that no $C'KN^d$ in (3.8) in
	the proof of Lemma 3.1  of \cite{FvMST}.  This $K$ gave the second term
	$\frac{K}\gamma$ in the concluding estimate in Lemma 3.1.
	Instead, we only have $\frac{\log 2}\gamma$ from \eqref{BG:pf:5}
	and this gives the first term $\frac{C}\gamma$ in our estimate \eqref{eq:8.15-C}.
	
	For the third term in the concluding estimate in  Lemma 3.1
	($\e^{2\Theta}=\e^0=1$), at the last part of its proof, we had
	$$
	\langle (-\L_{L, x})^{-1} (a_{t,x} m_x)a_{t,x} m_x \rangle_{x+\Lambda_L,j}
	\leq \frac1{\operatorname{gap}(j,L)} \sup_{t,x,\eta} | a_{t,x} m_x |^2, 
	$$
	where $\operatorname{gap}(j,L) \geq C / L^2$ as indicated above this
	estimate in \cite{FvMST}.  Thus, under our assumption
	\eqref{BG:assn:atx} on $a_{t,x}$, estimating as
	$$
	\sup_{t,x,\eta} |a_{t,x}m_x|^2 \le (2M)^2 \|\Om\|_\infty^2,
	$$
	we obtain the estimate \eqref{eq:8.15-C}.
	
	Note that the condition `$\gamma M L^{d+2} \le \de N^2$
	for some $\de>0$ small enough' was used for the
	denominator in (3.10) of \cite{FvMST} (with $M$ instead of $CK^\th$ and 
	$\e^{-\Theta}=1$), derived by Rayleigh estimate,
	to stay uniformly positive, say $\ge 1/2$.  This determines how small
	$\de>0$ needs to be.  This is the condition (3.2) in \cite{FvMST} with
	$\e^{-\Theta}= \e^0=1$ and $M/C$ instead of $K^\th$.
	
	The second assertion with the special choice of $\ga$ is straightforward.
\end{proof}

Let us continue the proof of Proposition \ref{One-block}.
For the second term in \eqref{BG:pf:0}, we apply Theorem 4.1 of \cite{CM} 
(taking $\La=\La_L$) which provides the equivalence of ensembles
for general Gibbs measures with precise convergence rate:
For any $\e\in (0,1/4)$, there exists $C=C_\e > 0$ such that for all $N$ 
sufficiently large
\begin{equation} \label{eq:3EE}
	\sup_{\rho\in [0,1]} \max_{x \in \T_N^d} \sup_{\eta \in \mathcal X_N} \left| 
	E^{\nu_{\rho}}[\t_x \Om \mid \bar\eta_x^L]
	- \lan \Om \ran (\bar\eta_x^L)  \right| 
	\leq \frac C{L^{d}} |\supp\Omega| \ \| \Om\|_\infty.
\end{equation}

Combining the two estimates given in Lemma \ref{l:L10:6} and 
\eqref{eq:3EE} with $\rho=1/2$ and $|
\supp\Omega| \le r(\Om)^d \le L^{d(1-4\e)}$
and writing the exponent as  $a= 1-4\e\in (0,1)$, 
the proof of Proposition \ref{One-block} is concluded.

\subsection{Proof of Lemma \ref{Lemma 3.3} (a key lemma)}
\label{sec:4.3-C}

Recall that $J(t,v) \in C^\infty([0,T]\times\T^d, \R^n)$, 
$G(\eta) \in \mathcal{F}_0^n$ and  ${\mathbf U}(\rho) \in C([0,1],\R^n \otimes \R^n)$
are given.  For a box $\La_L$, we decompose $\eta$ as 
$\xi = \eta|_{\La_L}$ and $\zeta = \eta|_{\La_L^c}$.
We write $G$ as $G(\eta)= G_{\zeta}(\xi)$ regarding it as a function of
$\xi$.  We assume $\lan G_\zeta \ran_{\La_L,M} = 0$ for every  
$\zeta\in \mathcal{X}_{\La_L^c}$ and $M: 0 \le M \le |\La_L|$.  
The proof is divided into
three steps.  It is given by combining that of Lemma 6.1 of \cite{fuy}
and the calculation developed in that of Lemma 3.1
of \cite{FvMST} for Glauber-Kawasaki dynamics.

{\it Step 1.} (Application of entropy inequality and Feynman-Kac formula) 
Setting 
$$ 
W_{N,t}(\eta) = N \sum_{x \in \T_N^d} J(t,x/N) \cdot \tau_xG - 
\b \sum_{x \in \T_N^d} J(t,x/N) \cdot {\mathbf U}(\bar{\eta}_x^L)J(t,x/N),
$$ 
by the entropy inequality 
noting that $h_N(P^{f_0}|P^{\nu^N}) \le C$ (recall the normalized entropy defined by \eqref{eq.defRelativeEntropy}) ($P^{f_0}$ denotes the distribution
on the path space $D([0,T],\mathcal{X}_N)$ of $\eta^N(\cdot)$ with the
initial distribution $f_0^N\, \d \nu^N$) at the process level, we have 
\begin{align*}
	\int_0^t E^{f_s^N} [ N^{-d}  W_{N,s}(\eta) ]\d s 
	& \le \frac1{\b} N^{-d} \log {E^{\nu^N} \left[e^{\int_0^t \b W_{N,s}(\eta^N(s))\d s} 
		\right]}  + \frac{C}{\b},
\end{align*}
for every $\b>0$; recall $\nu^N=\nu_{1/2}^N$ and $E^{\nu^N}[\,\cdot\,]$
means the expectation under the process $\eta^N(\cdot)$ with the initial
distribution $\nu^N$.
Then, by Feynman--Kac formula, similar to (3.6) of \cite{FvMST}
(with $\e=0, \e^{1-\Theta}=0$ in the sense that we explained in the proof of 
Lemma \ref{l:L10:6} in Section \ref{sec:4.2-C})  shown for the
Glauber--Kawasaki dynamics, the right-hand side is bounded by
\begin{align}  \label{eq:9.18-E}
	\frac1{\b} N^{-d}\int_0^t  \Om_{N,\b}(s)\d s + \frac{C}{\b},
\end{align}
where 
\begin{equation}  \label{eq:9.19-Om}
	\Om_{N,\b}(t) 
	= \sup_{\psi: \int\psi^2d\nu^N=1}
	\Big\{E^{\nu^N} [\b W_{N,t} \psi^2] - N^2E^{\nu^N} [\psi (-\mathcal{L} \psi)]
	\Big\}.
\end{equation}
Recall that the generator of the process $\eta^N(\cdot)$ is $\L_N=N^2\L$.

{\it Step 2.} (Error estimate for $\Om_{N,\b}(t)$)
Since  $\Om_{N,\b}(t)$ is the same quantity as (6.3) of \cite{fuy},
the same bound as \cite{fuy} holds 
for $\Om_{N,\b}(t)$ until the very end of the proof of Lemma 6.1 of \cite{fuy}.
The only difference is that,
in (6.7) of \cite{fuy}, the limit was taken as $N\to\infty$, but here we have to
derive an error estimate. 

Indeed, the first term in \eqref{eq:9.18-E} is bounded by $t$ times 
the expression on line 4, p.28 of
\cite{fuy}, i.e.\ by adjusting the notation, it is bounded by
\begin{equation}  \label{eq:9.19-Q}
	t \sup_{|q| \le \|J\|_\infty } \sup_{\zeta\in \mathcal{X}_{\La_L^c}}
	\sup_{M: 0\le M \le |\La_L|}
	\left[ \frac{N^2}{\b |(\La_L)^*|} \Om_{N,\b,q,G}^{M,L,\zeta}- \b  
	\, q \cdot {\mathbf U}(M/|\La_L|)q \right], 
\end{equation}
where $\Om_{N,\b,q,G}^{M,L,\zeta}$ is the largest eigenvalue of the symmetric
operator $\L_{\La_L,\zeta}+V$ with $V=N^{-1}\b |(\La_L)^*|  \, 
q\cdot G_\zeta$ in $L^2(\mathcal{X}_{\La_L,M}, \nu_{\La_L,M})$;
recall \eqref{eq.defGeneratorExt} and its above for $\mathcal{L}_{\Lambda_L, \zeta}$
and $\nu_{\La_L,M}$.

By Theorem 6.1 of \cite{fuy} (Rayleigh-Schr\"odinger bound for 
general Markov generator), $\Om_{N,\b,q,G}^{M,L,\zeta}$ is bounded as
\begin{equation}  \label{eq:9.21-P}
	\Om_{N,\b,q,G}^{M,L,\zeta} \le
	\lan V  (-\L_{\La_L,\zeta})^{-1} V\ran_{\La_L,M}
	+ Q_{N,L,\b}(V),
\end{equation}
where the upper bound of the error is given by 
$$
Q_{N,L,\b}(V) := 
4 \Vert V \Vert_{\infty}^3
\lbr \sup_{f:\lan f\ran_{\La_L,M} =0}
\frac {\lan f^2\ran_{\La_L,M}}{\lan -f \L_{\La_L,\zeta} f\ran_{\La_L,M}}\rbr^2.
$$
Since we have a spectral gap estimate for the Kawasaki generator:
\begin{align}  \label{eq:3-SG}
	\lan -f \L_{\La_L,\zeta} f\ran_{\La_L,M} \ge C_0 L^{-2}
	\lan f^2\ran_{\La_L,M},
\end{align}
for $f:\lan f\ran_{\La_L,M} =0$, 
with a constant $C_0>0$ independent of $\zeta, M, L$ (see \cite{LuYau} or
Appendix A of \cite{fuy}), recalling the definition of $V$,
we have a bound for the error term:
\begin{align*}
	Q_{N,L,\b}(V)
	\le 4N^{-3}\b^3 |(\La_L)^*|^3 \|q\cdot G\|_\infty^3 
	\, \big(C_0^{-1} L^2\big)^2.
\end{align*}
Therefore, in \eqref{eq:9.19-Q},     
\begin{align*}
	t \frac{N^2}{\b |(\La_L)^*|} Q_{N,L,\b}(V)    
	\le C_T N^{-1}\b^2 L^{(2d+4)} \| G\|_\infty^3 \| J\|_\infty^3,
\end{align*}
for $q: |q| \le \|J\|_\infty$.  This is the estimate \eqref{eq:3QNell}
for $Q_{N,L}^{(1)}(J,G)$ in Lemma \ref{Lemma 3.3} multiplied by $\b^2$.  
The rest in \eqref{eq:9.19-Q} is, by \eqref{eq:9.21-P},
\begin{align*}
	&t  \sup_{|q| \le \|J\|_\infty } \sup_{\zeta\in \mathcal{X}_{\La_L^c}}
	\sup_{M: 0\le M \le |\La_L|}
	\left[ \frac{N^2}{\b |(\La_L)^*|} \lan V (-\L_{\La_L,\zeta})^{-1} V\ran_{\La_L,M}
	- \b  \, q \cdot {\mathbf U}(M/|\La_L|)q \right].
\end{align*}
This coincides with the variational term in Lemma \ref{Lemma 3.3} 
recalling the definition of $V$ and noting that 
$|(\La_L)^*| \leq  d |\La_L|$.  Thus, we complete the proof 
of Lemma \ref{Lemma 3.3}.

\subsection{Proof of Lemma \ref{Lemma 3.4} (rewriting $O(N)$-looking term 
	to $O(1)$)}
\label{sec:4.4-C}

We essentially follow the proof of Lemma 3.4 of \cite{fuy}, 
but make it finer.  Indeed, it is very similar to the proof of Lemma 4.3 of \cite{Fu24}
(Section 9.4) for the Glauber-Kawasaki dynamics.  The difference is only that 
the term $C/\b$ in the concluding estimate for $Q_{N,L}^{(3)}$ was $CK/\b$
for the Glauber-Kawasaki case.  

This difference is caused when we apply the key lemma, i.e., Lemma 4.2 in
\cite{Fu24} and Lemma \ref{Lemma 3.3} in our case.  Comparing these two
lemmas, the difference is only that the error was $CK/\b$ in Lemma 4.2 in
\cite{Fu24}, while it is $C/\b$ in our Lemma \ref{Lemma 3.3}.
Therefore, the only difference in the proof is that
the estimate (9.25) in \cite{Fu24} should be replaced by
\begin{align}  \label{eq:3.7-0}
	& \b T \|\nabla \la\|_\infty^2 d \, |\La_{L+1}|  
	\sup_{0\leq M\leq |\Lambda_{L+1}|,\zeta\in \mathcal{X}_{\La_{L+1}^c} }
	\lan G (-\L_{\La_{L+1},\zeta})^{-1} G 
	\ran_{\La_{L+1}, M}  \\
	& \qquad \qquad
	+ \frac{C}{\b} + \b^2 Q_{N,L+1}^{(1)}(\nabla\la,G)
	\notag
\end{align}
in our setting.  Otherwise,  everything is the same and we may simply
copy all the proof of Lemma 4.3 of \cite{Fu24} as it is.  In this way, 
we can complete the proof of Lemma~\ref{Lemma 3.4}.

\section*{Acknowledgements}
The research of CG is supported by the National Key R\&D Program of China (No. 2023YFA1010400) and the National Natural Science Foundation of China (Nos. 12301166, 12595284, and 12595280). The research of TF is supported by International Scientists Project of BJNSF (No. IS23007). Part of this project was developed when CG and HW attended the workshop ``Probability and Statistical Physics'' at TSIMF. CG thanks Jean-Christophe Mourrat to mention this problem during his thesis. We would like to thank Yuval Peres and Shangjie Yang for helpful discussions, and thank Scott Armstrong to share the recent progress in high-contrast homogenization. We also thank the referees for helpful comments and suggestions.

\bibliographystyle{abbrv}
\bibliography{KawasakiRef}
\newpage

\end{document}